\documentclass[11pt, reqno]{amsart}
\usepackage[dvipsnames]{xcolor}
\usepackage{subfiles}
\usepackage[utf8]{inputenc}
\usepackage{amscd,amsmath,amssymb,amsthm,yfonts,mathrsfs,hhline,tikz,dsfont,tikz-cd,mathabx,esvect,accents,dsfont}
\usepackage{wasysym}
\usetikzlibrary{arrows, decorations.markings, decorations.pathreplacing, patterns, calc, intersections, through}
\tikzset{>=latex}
\tikzcdset{arrow style=math font}
\usetikzlibrary{decorations.pathmorphing, intersections, shapes}

\usepackage[bookmarks=true, bookmarksopen=true,%
bookmarksdepth=3,bookmarksopenlevel=2,%
colorlinks=true,%
linkcolor=blue,%
citecolor=blue,%
filecolor=blue,%
menucolor=blue,%
urlcolor=blue]{hyperref}

\usepackage[centering,margin=1.2in]{geometry}
\usepackage[matrix,arrow,curve]{xy}
\global\binoppenalty=10000
\usepackage{todonotes,setspace} 
\usepackage{verbatim}
\usepackage{graphicx}
\usepackage[capitalise]{cleveref}
\DeclareMathAlphabet{\mathpzc}{OT1}{pzc}{m}{it}
\newtheorem{theorem}{Theorem}[section]
\newtheorem*{maintheorem}{Main Theorem}
\newtheorem{lemma}[theorem]{Lemma}
\newtheorem{convention}[theorem]{Convention}
\newtheorem{prop}[theorem]{Proposition}
\newtheorem{cor}[theorem]{Corollary}
\newtheorem{conjecture}[theorem]{Conjecture}

\theoremstyle{definition}
\newtheorem{defn}[theorem]{Definition}
\newtheorem{remark}[theorem]{Remark}
\newtheorem{example}[theorem]{Example}%[subsection]
\newtheorem{notation}[theorem]{Notation}

\numberwithin{equation}{section}

\def\beq{\begin{equation}}
\def\eeq{\end{equation}}

\newcommand{\longra}{\longrightarrow}
\newcommand{\into}{\hookrightarrow}
\newcommand{\onto}{\twoheadrightarrow}
\newcommand{\hooklongrightarrow}{\lhook\joinrel\longrightarrow}

\newcommand{\hr}[1]{\left(#1\right)} % round, aligned
\newcommand{\hm}[1]{\left|#1\right|} % modulo, aligned
\newcommand{\ha}[1]{\left\langle#1\right\rangle} % angle, aligned
\newcommand{\hs}[1]{\left[#1\right]} % square, aligned
\newcommand{\hc}[1]{\left\{#1\right\}} % calligraphic, aligned
\def\le{\leqslant}
\def\ge{\geqslant}

\def\Ac{\mathcal A}
\def\Ad{\operatorname{Ad}}
\def\Agt{\mathfrak A}
\def\Aut{\operatorname{Aut}}

\def\b{\mathfrak b}

\def\bi{\mathbf{i}}
\def\bs{\boldsymbol}

\def\C{\mathbb C}
\def\Cc{\mathcal C}
\def\Ch{\operatorname{Ch}}
\def\Cl{\operatorname{Cl}}
\def\cluster{\mathcal B}

\def\Cut{\mathcal{C}}
\def\cut{\mathrm{cut}}

\def\decS{\mathbb S}
\def\dim{\operatorname{dim}}
\def\Dbb{\mathbb D}
\def\Dc{\mathcal D}
\def\Dcr{\mathscr D}

\def\Diffeo{\operatorname{Diffeo}}
\def\Dres{\Dc_{\mathrm{res}}}

\def\eff{\xi}
\def\eps{\varepsilon}

\def\Eff{\Xi}

\def\Fcr{\mathscr F}
\def\Frac{\operatorname{Frac}}

\def\gr{\operatorname{gr}}

\def\Gc{\mathcal G}
\def\Glue{\mathcal{G}}
\newcommand\Grm{\mathrm G}

\def\Hc{\mathcal H}

\def\i{\mathbf i}
\def\icox{\bi_{-c,c}}

\def\ifull{\bi_{\overline{w}_0,w_0}}
\def\ihalf{\bi_{\overline{w}_0,c}}

\def\isf{\bi_{\overline{w}_0,c^*,c}}
\def\isym{\bi_{\overline{w}_0,c^*}}
\def\Id{\operatorname{Id}}
\def\iiota{\iota}

\def\Kc{\mathcal K}

\def\la{\lambda}
\def\len{\mathrm{len}}

\def\loc{\mathrm{loc}}
\def\longra{\longrightarrow}
\def\La{\Lambda}
\def\Lbb{\mathbb L}
\def\Lbba{\Lbb^\Ac}
\def\Lbbx{\Lbb^\Xc}
\def\Lc{\mathcal L}

\newcommand\Mor{\operatorname{Mor}}

\def\Oc{\mathcal O}

\def\pt{\mathrm{pt}}
\def\Pbb{\mathbb P}
\def\Pc{\mathcal P}
\def\Prm{\mathrm P}
\newcommand\Pt{\operatorname{Pt}}
\newcommand\Ptoh{\widehat\Pt \vphantom{\Pt}^o}

\def\Q{\mathbb Q}
\def\Qc{\mathcal Q}
\def\Qcl{{\bs Q}}
\def\Qcone{Q_{\mathrm{cone}}^n}

\def\Qcyl{Q_{\mathrm{cyl}}^n}
\def\Qf{\bs{\mathrm Q}}
\def\Qsf{Q_{\mathrm{sf}}^n}
\def\Qop{\mathbf{Q}}
\def\rat{\mathrm{rat}}

\def\res{\operatorname{res}}

\def\Rc{\mathcal R}
\def\Res{\mathrm{Res}}

\def\sym{\mathrm{sym}}

\def\tri{\underline\Delta}

\def\Tc{\mathcal T}
\def\Tca{\Tc^\Ac}
\def\Tcx{\Tc^\Xc}
\def\Toda{\mathrm{Toda}}
\def\Uc{\mathcal U}

\def\Vcr{\mathscr V}

\def\Wbb{\mathbb{W}}
\def\Wc{\mathcal W}

\def\Xc{\mathcal X}

\def\yrm{\mathrm{y}}

\def\Z{\mathbb Z}

\def\quiver{Q}

\title[Algebraic modular functor conjecture in type $A_n$ Teichm\"uller theory]{The algebraic modular functor conjecture in type $A_n$ quantum Teichm\"uller theory}

\author[Gus Schrader]{Gus Schrader}
\author[Alexander Shapiro]{Alexander Shapiro}
\begin{document}

\maketitle

\begin{abstract}
Fock and Goncharov introduced a quantization of higher Teichm\"uller theory using cluster Poisson varieties and their noncommutative deformations, associating to a complex semisimple Lie group $G$ and a marked surface $S$ a quantum algebra $\Lbb_{G,S}$ equipped with an action of the surface mapping class group. They conjectured that these quantizations form an algebraic analog of a modular functor: cutting a surface along a simple closed curve should correspond to a canonical gluing isomorphism for the associated algebras. In this paper we prove this conjecture for $G =\mathrm{PGL}_{n+1}$. Our approach requires two extensions of the Fock–Goncharov framework: (1) enhanced moduli spaces incorporating additional boundary data, providing algebro-geometric analogs of Fenchel–Nielsen twist coordinates; and (2) the residue universal Laurent ring, a refinement of the quantum universal Laurent ring obtained by localizing and imposing residue conditions. Using these tools, we construct canonical cutting isomorphisms that are equivariant under mapping class group actions and suffice to reconstruct the entire algebra $\Lbb_{G,S}$ from data associated to the cut surface.
\end{abstract}

\setcounter{tocdepth}{1}
\tableofcontents

\section{Introduction}
%$$
%\mathbb T, \Gamma, \bs\Gamma, \Gabb, \mathds{L}, \Lbbx, \cluster
%$$

In a series of papers roughly two decades ago, Fock and Goncharov described a precise sense in which higher rank Teichm\"uller theory admits a \emph{quantization}.  Their approach to constructing this quantization involved two key insights. The first of these is that if $S$ is a surface with a nonempty collection of interior marked points and $G=\mathrm{PGL}_{n}(\mathbb{C})$, the \emph{moduli space of decorated $G$-local systems on $S$ } provides a natural setting for an algebro-geometric approach to rank $n$ Teichm\"uller theory. Indeed, in~\cite{FG06b} it is shown that these moduli spaces are \emph{cluster Poisson varieties}: they admit `cluster atlases' whose charts are split algebraic tori with distinguished coordinate frames, and whose gluing maps are composites of certain special subtraction-free birational maps known as cluster mutations. Moreover, the natural action of the surface mapping class group $\Gamma_S$ on the moduli space is realized in these coordinates as a composite of cluster transformations. Any cluster Poisson variety has a well-defined totally positive locus, and for the moduli spaces above with $G=\mathrm{PGL}_2$, Fock and Goncharov show that a certain symplectic leaf in this locus is canonically identified with the classical Teichm\"uller space for the surface obtained by regarding the interior marked points as punctures. This identification is compatible with the mapping class group actions on both sides. 

The second key insight, developed in~\cite{FG09a,FG09b}, is that cluster Poisson varieties admit `chartwise' non-commutative deformations, in which the combinatorics of the cluster atlas is used to construct something like a `non-commutative scheme.' In more detail, one assigns to each chart the atlas a non-commutative `quantum torus algebra' $\Tc^q$ defined over $\mathbb{Z}[q^{\pm1}]$, whose $q=1$ specialization recovers the coordinate ring of the chart in question. And for a pair of charts related by a single mutation, it was explained in~\cite{FG09a} that there exists a non-commutative analog of the classical gluing map, defined using the quantum dilogarithm function. Crucially, they show that any relation between classical mutations lifts to a relation between the corresponding quantum ones, ensuring the consistency of the gluing at the quantum level. Taking `global sections' of the resulting object, we get a single non-commutative algebra $\Lbb$ known as the \emph{quantum universal Laurent ring} associated to the cluster Poisson variety, which consists of elements that are skew Laurent polynomials in every quantum torus in the cluster atlas. Moreover, any symmetry of the classical variety which can be realized via cluster transformations automatically extends to an automorphism of $\Lbb$.

In this way, given a pair $(G,S)$ as above Fock and Goncharov construct in~\cite{FG09a} a non-commutative algebra $\Lbb_{G,S}$ carrying an action of the mapping class group $\Gamma_S$ by quantum cluster transformations. Yet several natural structures and properties of these quantizations have remained conjectural since their definition. First, associated to each simple closed curve $c$ on S and finite dimensional representation $V$ of $G$, there is a canonical function $\chi_V(c)$ on the classical moduli space computing the trace of the monodromy of the local system along $c$ in the representation $V$. In~\cite{FG09a},  it was conjectured that there should exist corresponding elements $\chi^q_V(c)$ of the quantum universal Laurent ring $\Lbb_{G,S}$, and that for a fixed curve $c$ these elements should generate a commutative subalgebra $R_G(c)$ in $\Lbb_{G,S}$ isomorphic to a copy of the representation ring of $G$.

In the same paper, Fock and Goncharov also made a deeper conjecture: based on the physical interpretation of the moduli space of flat connections on a surface as parametrizing solutions of the equations of motion in Chern-Simons theory, they suggested that the moduli spaces of decorated local systems and their quantizations should provide an algebraic avatar of a topological quantum field theory. A little more precisely, they conjectured that for fixed $G$, the assignment $S\rightsquigarrow \Lbb_{G,S}$ should behave like an algebraic analog of a modular functor in 2D conformal field theory in the following sense.

It is immediate from the construction of $\Lbb_{G,S}$ that each interior `puncture' point $p$ on $S$ gives rise to a central (and thus mutation-invariant) subalgebra $R_T(p)\subset\Lbb_{G,S}$ isomorphic to the representation ring of the maximal torus $T$ of $G$. Less obviously, whenever $p$ is not the only special point on its connected component, there is also an action of the Weyl group $W_G$ on $\Lbb_{G,S}$ by cluster transformations, which restricts to the reflection representation on the central subalgebra $R_T(p)$.  Hence the subalgebra of invariants $R_G(p) = R_T(p)^{W_p}$ gives a canonical copy of the representation ring of $G$ associated to each such puncture. 

Now suppose that $c$ is a simple closed curve on $S$, and $S'$ the surface obtained by cutting $S$ open along $c$, and $S^\circ$ the surface obtained by shrinking the corresponding boundary circles $c_\pm$ on $S'$ to punctures $p_\pm$. Then there is a central ideal $\mathcal{I}_c\subset\Lbb_{G,S'}$ generated by the relations $\chi_{p_+}=\chi^*_{p_-}$ identifying the central element in $R_T(p_+)$ corresponding to a character $\chi$ with the dual character in $R_T(p_-)$. Assuming that neither of $p_\pm$ is the only special point on its connected component in $S'$, we have an action of the product of Weyl groups $W_{p_+}\times W_{p_-}$, and the ideal $\mathcal{I}_c$ is preserved by the diagonal subgroup $W(c)$.
The centralizer $\Gamma_{S;c}$ of the Dehn twist along $c$ in the mapping class group $\Gamma_S$ can be identified with the quotient of the mapping class group $\Gamma_{S'}$ by the central subgroup generated by the product $\tau_{c_+}\tau_{c_-}$ of Dehn twists along $c_\pm$.

\begin{conjecture}
\label{conj:intro-MF}[Algebraic modular functor conjecture for $G=\mathrm{PGL}_{n+1}$; see~\cite{FG09a} Section 6.2, ~\cite{GS19} Conjectures 2.27 and 2.29.]
\label{conj:MF}
For any essential simple closed curve $c$ as above, there is a canonical subalgebra $R_G(c)\subset \Lbb_{G,S}$, and an isomorphism of algebras
\begin{align}
\label{eq:intro-eta}
\eta_c \colon \Lbb_{G,S}^{R_G(c)} \simeq \left( \Lbb_{G,S'}/\mathcal{I}_c \right)^{W(c)}
\end{align}
where  $\Lbb_{G,S}^{R_G(c)}$ denotes the centralizer of $R_G(c)$ in $\Lbb_{G,S}$. The map~\eqref{eq:intro-eta} is equivariant with respect to the action of $\Gamma_{S;c}$ by cluster transformations on both sides, and restricts to an isomorphism $R_G(c)\simeq R_G(p)$.
\end{conjecture}

The main result of this paper is a proof this conjecture. Our proof involves developing extensions of both of the original key insights of Fock and Goncharov: the classical limit of our story is most naturally formulated in terms of a new version of the moduli space of decorated local systems incorporating extra data associated to the boundary components $c_\pm$ created when cutting $S$ along $c$, while to construct its quantization we need to go beyond the standard theory of cluster algebras and work with a new object we call the \emph{residue universal Laurent ring.}

The need to work with an enhanced version of the moduli spaces can be understood by analogy with classical Teich\"muller theory for bordered Riemann surfaces with parametrized geodesic boundaries of non-fixed length, as considered by Mirzakhani~\cite{Mir07} and Alekseev--Meinrenken~\cite{AM24}. Indeed, for each pants decomposition of a surface $S$ with $s$ parametrized geodesic boundaries (i.e. a presentation of $S$ as the collection of 3-holed spheres glued along a collection of pairs of boundary curves $c_i$) there is a corresponding system of Fenchel--Nielsen length/twist coordinates $(l_i,\tau_i)_{i=1}^{3g-3+2s}$ on $\mathrm{Teich}_{\mathrm{bordered}}(S)$ which are canonically conjugate with respect to the Weil--Petersson form. The coordinates $l_i$ parametrizes the geodesic length of the boundary curve $c_i$, while the dual coordinates $\tau_i$ are the twist parameters used to glue the pieces together.  In this context, the gluing map amounts to a \emph{symplectic reduction}: in order to glue the pants along a pair of geodesic boundaries $c_\pm$ we require $l_{c_+}=l_{c_-}$, while  simultaneously twisting both boundaries $c_\pm$ by the same small angle results in the same hyperbolic structure after gluing. 

In the original algebro-geometric approach of Fock--Goncharov, the length coordinates $l_i$ associated to a boundary component $c_+$ has a direct interpretation in terms of the Poisson-central subalgebra $R_T(p)$ for the surface where we shrink $c_+$ to a puncture, but the dual twist coordinate is necessarily absent, since it does not Poisson commute with $l_i$ . In Definition~\ref{def:Pdiamond} we introduce a new kind of algebro-geometric datum associated to boundary components which will allows us to incorporate analogs of the twist coordinates, and therefore construct well-behaved gluing maps. The components to which we associate such data we call \emph{tacked circles}. We construct a cluster Poisson structure on this new moduli space $\Pc^{\diamond}_{G,S}$, allowing us to quantize it following the procedure of Fock--Goncharov. And in parallel to the classical gluing construction sketched above, given a surface $S'$ with pair of tacked circles $c_\pm$ and a diffeomorphism $\phi \colon c_+\simeq c_-$ we show there is a canonical quantum Hamiltonian reduction $\Lbb^{c}_{G,S'}$ of the corresponding quantum universal Laurent ring which carries an action of the Weyl group $W(c_\pm)$ as in Conjecture~\ref{conj:MF}. 

However it turns out that the symplectically reduced universal Laurent ring $\Lbb^{c}_{G,S'}$ is still not sufficient to reconstruct the universal Laurent ring $\Lbb_{G,S}$ associated to the glued surface. To achieve this, we define the residue universal Laurent ring $\Lbb_{G,S';\phi}$ associated to the gluing datum $(S',\phi)$ by localizing $\Lbb_{G,S'}$ at a canonical mutation-invariant collection of divisors $Ø(c_\pm)$, taking invariants for $W(c_\pm)$, and then passing to the subalgebra consisting of elements having only simple poles at each such divisor with residues satisfying a certain condition --- see Definition~\ref{def:Tres} for details. With these definitions, we in fact prove the following stronger result allowing us to completely reconstruct $\Lbb_{G,S}$ in terms of data associated to the cut surface $S'$:

\begin{maintheorem}

\label{thm:main-intro}
Let $(G,S,c)$ be as in Conjecture~\ref{conj:intro-MF}. Then there is a $\Gamma_{S;c}$ equivariant algebra isomorphism
\begin{align}
\label{eq:intro-full-eta}
\eta_c\colon \Lbb_{G,S}\simeq \Lbb_{G,S';\phi}.
\end{align}
\end{maintheorem}
Let us now give a sketch of how we prove Theorem~\ref{thm:main-intro}. In the Fock--Goncharov theory, the most important charts in the cluster atlas for moduli spaces of decorated local systems on $S$ are those associated to \emph{ideal triangulations} of $S$. Our construction of the gluing isomorphism~\eqref{eq:intro-full-eta} when $\mathrm{rank}(G)>1$ involves different charts specially adapted to a simple closed curve $c$ of the form described in Conjecture~\ref{conj:intro-MF}, where each component of $S'$ has at least one boundary component or special point apart from the tacked circles $c_\pm$.  Any such curve on S can be isolated in an ideal cylinder formed by two ideal triangles. Given the choice of such a cylinder,  we construct a special cluster chart in which the action of the Dehn twist by cluster transformations localizes: for $G=PGL_{n+1}$, it is given by a composite of just $n$ commuting mutations occuring within a quantum sub-torus of rank $2n$. In fact, the quantum universal Laurent ring associated to this subquiver is precisely the quantized coordinate ring of the phase space for the $GL_{n+1}$ open Toda integrable system, whose commuting quantum Hamiltonians provide a candidate for the commutative subalgebra $R_G(c)$ associated to $c$.  In Theorem~\ref{thm:alg-MF-coherence} we prove that the $\chi^q_V(c)$ defined this way do not depend on the choice of isolating ideal cylinder, and are thus indeed canonically associated to the curve $c$.

In Section~\ref{sec:AlgWhit} we develop an algebraic version of the Whittaker transform for the $GL_{n+1}$ $q$-difference open Toda chain, identifying the corresponding local universal Laurent ring with a variant of the spherical part of the nil double affine Hecke algebra for $GL_{n+1}$. Then using a variant of the residue characterization of Hecke algebras developed by Ginzburg, Kapranov, and Vasserot in~\cite{GKV97}, we construct in Theorem~\ref{thm:alg-MF} a cutting isomorphism $\eta_c(\tri)\colon\Lbb_{G,S}\simeq \Lbb_{G,S;\phi}$ associated to each isolating ideal cylinder for $c$. Finally, we show in Theorem~\ref{thm:alg-MF-coherence} that these isomorphisms do not depend on the auxiliary choice of isolating cylinder, and in Theorem~\ref{thm:alg-MF-equivariance} that the resulting canonically defined isomorphism~\eqref{eq:intro-full-eta} is equivariant with respect to the action of $\Gamma_{S;c}$.

Working with the residue universal Laurent rings opens the door to extending Conjecture~\ref{conj:intro-MF} and Theorem~\ref{thm:main-intro} to the exceptional case that the surface $S'$ obtained by cutting $S$ along $c$  has a connected component containing a single boundary circle $c_-$. Although there is no action of the Weyl group $W(c_\pm)$ on $\Lbb^c_{S'}$ by cluster transformations in this case, we can use the trick of first cutting $S$ along an auxiliary cycle $a$, cutting along $c$, and gluing back along $a$ in order to define (see Definition~\ref{def:exceptional-lres}) an analog of the residue Laurent ring in this setting. Note that this would not be possible at the level of the cutting map in the original conjecture, which is defined only on a centralizer subalgebra in $\Lbb_S$. In Theorem~\ref{thm:secondlast} we show that when $g(S_-)>1$ the resulting isomorphism does not depend on the choice of auxiliary curve $a$, and establish its $\Gamma_{S;c}$ equivariance in Theorem~\ref{thm:last}. 

The corresponding results in the remaining corner case $g(S_-)=1$ require different techniques, and will be treated in the companion paper~\cite{SS25}, where given a simple closed curve $b$ on $S$ which crosses the cutting curve $c$, we derive explicit formulas for the image of the elements $\chi^q_V(b)$ under the cutting map $\eta_c$.

\subsection*{Acknowledgements}
We want to thank {M.\,Bershtein, F.\,Bonahon, A.\,Braverman, K.\,De Commer, P.\,Di Francesco, V.\,Fock, I.\,Frenkel, M.\,Gekhtman, A.\,Goncharov, I.\,Ip, R.\,Kedem, S.\,Khoroshkin, D.\,Jordan, M.\,Shapiro, N.\,Reshetikhin, J.\,Teschner, and C.\,Voigt} for many helpful discussions during the long gestation of the article.  The work of G.S. was supported by the NSF under grant DMS 2302624. The work of A.S. has been supported by the European Research Council under the European Union’s Horizon 2020 research and innovation programme under grant agreement No 948885 and by the Royal Society University Research Fellowship.
%\phantom{.}\blue{M.\,Bershtein, F.\,Bonahon, A.\,Braverman, K.\,De Commer, P.\,Di Francesco, V.\,Fock, I.\,Frenkel, M.\,Gekhtman, A.\,Goncharov, I.\,Ip, R.\,Kedem, S.\,Khoroshkin, D.\,Jordan, M.\,Shapiro, N.\,Reshetikhin, J.\,Teschner, C.\,Voigt.} 
%\red{Katherine and Ira for tolerating us?}

\section{Decorated surfaces and Ptolemy groupoids}
\label{sec:ptolemy}

\subsection{Moduli spaces of decorated local systems}

In this section we recall the setup of stratified surfaces and decorated local systems following~\cite{JLSS21}. The moduli spaces of decorated local systems considered in~\cite{FG06b} and~\cite{GS19} can then be recovered as particular cases of this construction.

\begin{defn}%FG06b
A \emph{stratified surface} is an oriented surface $S$ together with a finite collection of non-intersecting simple curves embedded into $S$ in such a way, that for any curve $c$ in the collection we have $\partial c \subset \partial S$ and $c \smallsetminus \partial c \subset S \smallsetminus \partial S$. We refer to the curves as \emph{walls} and denote their union by $C \subset S$. Connected components of $S \smallsetminus C$ are called \emph{regions.}
\end{defn}

%A \emph{gate} is an open segment in $\partial S \smallsetminus \partial\Cc$, where $\partial \Cc = \bigcup_{C_i \in \Cc} \partial C_i$. We require that gates do not intersect, and that each connected component of $\partial S \smallsetminus \partial\Cc$ either does not contain a gate or is covered by gates up to a finite subset of points. We denote gates by thin segments of $\partial S$.

\begin{convention}
    We will always follow the convention that the orientation of $S$ induces an orientation of its boundary $\partial S$ via the following recipe: if $\vec n$ is an inwards pointing tangent vector at a point $x\in \partial S$, and $\vec v$ a positively oriented tangent vector in the tangent space to $\partial S$ at x, then $(\vec v,\vec n)$ forms a positively-oriented frame in  $T_xS$. When drawing surfaces on the page, we will use the convention that the ordered pair consisting of a right-horizontal and an up-vertical vector forms a positive frame.
\end{convention}

Fix a connected reductive group $G$, an embedding $B \into G$ of a Borel subgroup, and a projection $B \onto T$ onto the quotient $T = B/N$, where $N$ is the unipotent radical of $B$. A \emph{$(G,T)$-coloring} of a stratified surface $S$ is an assignment of a ``color'' $G$ or $T$ to every region of $S$ in such a way that no two regions sharing a wall have the same color. Note that a coloring is defined by the color of any of the regions, and thus there are either two $(G,T)$-colorings of $S$ or none.

\begin{defn}
A \emph{decorated surface} $\decS$ is a stratified surface $S$ together with a $(G,T)$-coloring. We denote the union of all regions colored $G$ by $\decS_G$ and that of all regions colored $T$ by $\decS_T$.
%, and write $|\decS_G|, |\decS_T|$ for the number of $G$- and $T$-regions respectively.
\end{defn}
 
\begin{defn}
A stratified surface $S$ is \emph{marked} if it is endowed with a possibly empty collection $M \subset S \smallsetminus C$ of \emph{marked} points. A decorated surface $\decS$ is \emph{marked}\footnote{Marked points on a decorated surface play the same role as gates in~\cite{JLSS21}.} if the underlying stratified one is. Furthermore, we set $M_G = M \cap \decS_G$ and $M_T = M \cap \decS_T$.
\end{defn}

\begin{example}
On Figure~\ref{fig:triang-ex} we show examples of decorated surfaces, where $\decS_G$ is shown in purple and $\decS_T$ in yellow.

\begin{figure}[h]
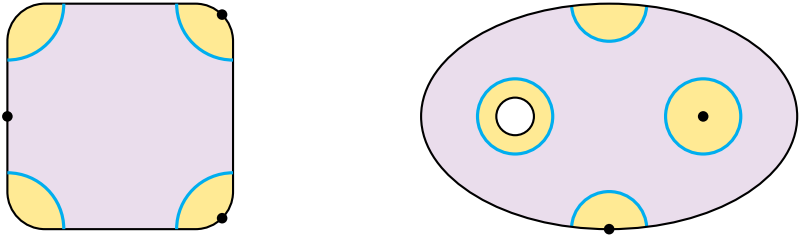

\subfile{triang-ex.tex}
\caption{Examples of decorated surfaces.}
\label{fig:triang-ex}
\end{figure}
\end{example}

\begin{defn}
Let $\decS_1$ and $\decS_2$ be decorated surfaces.  A \emph{decorated embedding} $\iota \colon \decS_1 \hookrightarrow \decS_2$ is an oriented embedding of surfaces which respects walls, regions, colors, and marked points. A \emph{decorated isotopy} is a continuous path through the space of decorated embeddings.
\end{defn}

In what follows, we shall only consider decorated surfaces up to decorated isotopies.

Recall that a data of a (Betti) $G$-local system is equivalent to that of a principal $G$-bundle with a flat connection. A reduction of a local system to a subgroup $H$ amounts to the specification of an $H$-subbundle preserved by the parallel transport. We will use this notion for the Borel $B$ considered as a subgroup of $G\times T$-via the map $B \into G\times T$ obtained as the composition of the diagonal embedding of $B$ into $B \times B$ with the product of the embedding $B \into G$ and the projection $B \onto T$.

%\red{Discussion of the different moduli spaces here seems kind of confusing}
%Let us denote
%\begin{align*}
%\partial\decS_T &= \partial S \cap \decS_T, \hspace{1cm} M_T = M \cap \decS_T, \\
%\partial\decS_G &= \partial S \cap \decS_G, \hspace{1cm} M_G = M \cap \decS_G.
%\end{align*}

%\begin{defn}
%We say that a marked stratified surface $\decS$ \emph{admits a triangulation} if there exists a collection $E = (e_1, \dots, e_k)$ of arcs embedded into $S$ such that
%\begin{itemize}
%\item $\decS \smallsetminus E$ is a disjoint union of triangles $\Dbb_3$;
%\item endpoints of any arc belong to $\partial\decS_T \cup M_T$;
%\item a pair of arcs may only intersect at a common endpoint $v \in M_T$;
%\item let $v$ be an endpoint of an arc, and $c$ be a connected component of $\partial\decS_T$ which contains~$v$. Then $v \in c \cap M$ unless $c \cap M = \varnothing$.
%\end{itemize}
%If $\decS$ admits a triangulation we refer to arcs $e_i \in E$ as \emph{diagonals.}
%\end{defn}
%
%\begin{defn}
%A marked stratified surface is \emph{simple} if it admits a triangulation, $M_G = \varnothing$, each $T$-region contains at most one marked point, and each marked point lies on the boundary $\partial S$ unless the $T$-region it belongs to is internal.
%\end{defn}

\begin{defn}
\label{def:decloc}
Let $\decS$ be a decorated surface. A \emph{decorated local system} $\Lc$ on $\decS$ is a triple $(\Lc_G,\Lc_T,\Lc_B)$, consisting of a $G$-local system $\Lc_G$ on $\decS_G$, a $T$-local system $\Lc_T$ on $\decS_T$, and a reduction $\Lc_B$ to $B$ of the $G\times T$-local system $\Lc_G\times \Lc_T$ over $C$. A decorated local system $\Lc$ is \emph{framed} if it is endowed with a trivialization of $\Lc_G$ and $\Lc_T$ at each point in $M_G$ and $M_T$ respectively. A morphism of framed, decorated local systems is a pair of morphisms between the corresponding $G,T$-local systems on $\decS_G,\decS_T$ respecting the $B$-reduction and trivialization data. 
\end{defn}

\begin{defn}
The \emph{(framed) decorated character stack} is the moduli stack of (framed) decorated local systems on $\decS$. We denote the decorated character stack by $\Ch(\decS)$ and the framed one by $\Ch_\bullet(\decS)$.
%In other words,
%$$
%\Ch(\decS) = \underline\Lc(\decS)/\Gc(\decS) \qquad\text{and}\qquad \Ch_\bullet(\decS) = \underline\Lc_\bullet(\decS)/\Gc(\decS)
%$$
%are respectively the quotient stacks of spaces $\underline\Lc(\decS)$ and $\underline\Lc_\bullet(\decS)$ by the group $\Gc(\decS) = G^{|\decS_G|} \times T^{|\decS_T|}$.
% If $\Ch(\decS)$ or $\Ch_\bullet(\decS)$ is an affine variety, we will also call it a \emph{moduli space} of (framed) decorated local systems.
% \red{I'd delete this last sentence. The moduli spaces we deal with are pretty much never affine.}
\end{defn}

The stack $\Ch_\bullet(\decS)$ carries a natural action of the group $\Gc_M = G^{\hm{M_G}} \times T^{\hm{M_T}}$, which changes trivialization at marked points.
%\footnote{We warn the reader of a notational clash with~\cite{JLSS21}, the group $\Gc_M$ was denoted $\Gc_\decS$ in \emph{loc.\,cit.}}
In particular, we have
$$
\Ch(\decS) \simeq \Ch_\bullet(\decS)/\Gc_M.
$$
We refer the reader to~\cite{IO23} for a comprehensive discussion of character stacks.

\begin{example}
\label{ex:basic-disks}
The decorated character stacks corresponding to decorated surfaces on Figure~\ref{fig:monogons}, read left to right as follows: $G/N$, $G/B$, $\pt/N$, $\pt/B$, where $\pt$ denotes a point. For the reader's convenience, let us briefly review how to arrive at this description, starting from Definition~\ref{def:decloc}. Suppose that $\mathcal{L}$ is a framed decorated local system on the surface shown in the left pane of Figure~\ref{fig:monogons}. Choose a path connecting the two marked points and intersecting the blue domain wall at a single point $p$. Parallel transporting the trivializations  at the two marked points along this path, we obtain a distinguished basepoint $e_p$ in the fiber at $p$ of the $G\times T$ local system $\mathcal{L}_G\times\mathcal{L}_T$ on the wall. The data of a $B$-reduction of the latter local system is that of a $B$-orbit $Bf_p$ in this fiber. Since the fiber is a $G\times T$ torsor, there is a well-defined $B$-coset in $G\times T$ given by $(g,t)B$, where $(g,t)\in G\times T$ is the unique element such that $(g,t)f_p=e_p$. Applying the same construction to an isomorphic decorated framed local system $\mathcal{L}'\simeq\mathcal{L}$ evidently results in the same such coset. Finally, since $T$ normalizes $N$, we can identify $(G \times T)/B \simeq G/N$ via $(g,t)B\mapsto gt^{-1}N$. 
The other three stacks are obtained from $G/N$ via taking further quotients by $G$ on the left and $T$ on the right.
\begin{figure}[h]
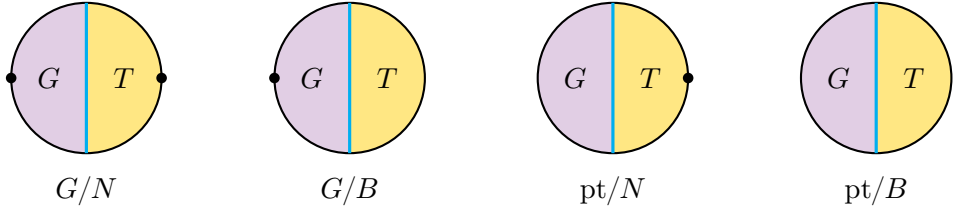

\subfile{monogons.tex}
\caption{Decorated monogons.}
\label{fig:monogons}
\end{figure}
\end{example}

\begin{example}
\label{ex:4-disks}
Let $\pi \colon G/N \to G/B$ denote the natural projection. Then the decorated character stacks corresponding to surfaces shown on Figure~\ref{fig:4-disks} are non-canonically isomorphic to the ones parametrizing the following data:
\begin{enumerate}
\item pairs $(g,\bar F) \in G \times G/N$, such that $g\pi(\bar F) = \pi(\bar F)$;
\item pairs $(g, F) \in G \times G/B$, such that $gF = F$;
\item pairs $(g,\bar F) \in G \times G/N$, such that $g \bar F = \bar F$. This forces $g$ to be unipotent;
\item pairs $(g, F) \in G \times G/B$, such that $gF = F$ and $g$ is unipotent.
\end{enumerate}

For the reader's convenience, we again sketch how to obtain the explicit description of the moduli space (1). 
% Let $P$ be such a decorated framed local system, and fix two arbitrary points in the fibers of the $G,T$ local systems over the special points, so that given $P'\simeq P$ there is a unique automorphism of $P$ identifying    
In the left pane of Figure~\ref{fig:4-disks}, choose a line segment connecting the two marked points, intersecting the domain wall at a single point $p$. Using parallel transport of the $G$ and $T$ local systems, we get this way a distinguished point in the fiber over $p$ of the $G\times T$-local system on the wall.  This distinguished point allows us to associate to the loop around the wall an element $(g,t)\in G\times T$ recording the corresponding monodromy, along with a $B$-coset $(h,s)B$ in $G\times T$ preserved by $(g,t)$: we have $(gh,ts)b = (h,s)$ for some $b\in B$. To reconcile this description with that given in (1), note that
in particular the $B$-coset $\bar{F}=hB$ in $G$ is preserved by $g$, and moreover that the semisimple part of $b$ is exactly $t^{-1}$, that is $b=nt^{-1}$ with $n\in N$ and $gh n = ht $. Since $G/N$ is a principal $T$-bundle over $G/B$, the element $t$ can be recovered as the monodromy of $hN\in G/N$ under $g$: it is the unique element of $T$ such that $ghN = htN$. Hence the data in (1) indeed suffice to recover the original framed decorated local system up to isomorphism.

As was noted in~\cite{JLSS21}, the moduli space $(2)$ is the quotient of $(1)$ by the $T$-action changing the trivialization, whereas $(3) \subset (1)$ is the preimage of the identity with respect to a multiplicative $T$-valued moment map $\mu \colon (1) \to T^*$. Accordingly, $(4)$ is the multiplicative Hamiltonian reduction of $(1)$.

\begin{figure}[h]
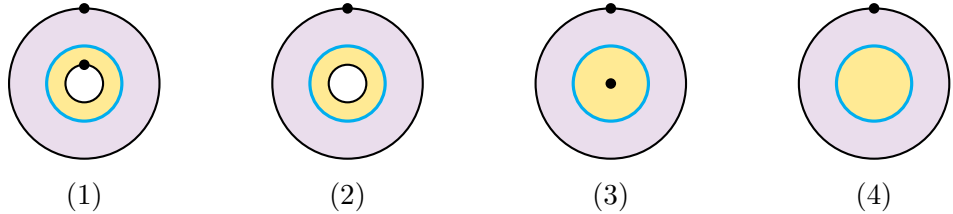

\subfile{4-disks.tex}
\caption{Decorated cylinders and disks.}
\label{fig:4-disks}
\end{figure}
\end{example}

%\begin{example}
%The decorated surfaces shown on Figure~\ref{fig:disks-2} are almost identical to those in the previous example, except for we replaced $G$-trivializations with pairs of boundary $T$-disks, each with a $T$-trivialization. The resulting decorated character stacks are obtained from those above by adding to each pair $(g,F)$ or $(g, \bar F)$ two more framed flags $\bar F_1, \bar F_2 \in G/N$ and quotenting the resulting quadruple by the group $G$. As was shown in~\cite{FG06b, GS19}, the moduli space (2) can be endowed with the structure of a cluster Poisson variety if $G$ is of adjoint type, and the moduli space (3) can be endowed with the structure of a cluster $K_2$-variety if $G$ is simple. In what follows, we shall also define structures of a cluster Poisson and a cluster $K_2$-varieties on the moduli space~(1).
%
%
%\begin{figure}[h]
% \subfile{disks-2.tex}
%\caption{Stratified cylinders and disks with $M_G = \varnothing$.}
%\label{fig:disks-2}
%\end{figure}
%
%\end{example}

\begin{defn}
A region of a decorated surface $\decS$ is said to be \emph{boundary} if its intersection with $\partial\decS$ is non-empty, and the region can be retracted onto this intersection while preserving $\partial\decS$. A region is \emph{internal} if it does not intersect $\partial\decS$. A \emph{decorated $n$-gon} $\Dbb_n$ is a decorated disk with a single $G$-region and $n$ boundary $T$-disks.
%We refer to $\Dbb_n \smallsetminus \partial\Dbb_n$ as an \emph{open decorated $n$-gon.}
\end{defn}

One can see boundary $T$-annuli on the first two pictures of Figure~\ref{fig:4-disks}, and internal $T$-disks on the second two. Decorated monogons and a decorated triangle are shown respectively on Figures~\ref{fig:monogons} and~\ref{fig:triangle}. Note that a boundary region is necessarily either an annulus or a disk.
%We say that two boundary $T$-disks are \emph{adjacent} if they intersect the same connected component of $\partial S$, and one immediately follows the other as we traverse this connected component.

\begin{figure}[h]
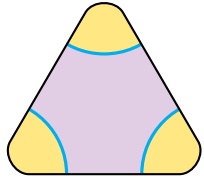

\subfile{triangle.tex}
\caption{Decorated triangle.}
\label{fig:triangle}
\end{figure}

We say that a decorated surface is \emph{simple}\footnote{We warn the reader that our definition of a simple decorated surface is more general than the one given in~\cite{JLSS21}.} if each of its $T$-regions is either a boundary $T$-disk with a single marked point, an interior $T$-disk with a single marked point, or a boundary $T$-annulus with at most one marked point, the unique marked point in a $T$-region is boundary if the region is such, and the $G$-region contains no marked points. In the remainder of the article, we will only work with simple decorated surfaces.

\begin{defn}
\label{def:marked}
    A \emph{marked surface} is an oriented surface $S$ together with a subset $M\subset \partial S$ of boundary marked points, a subset $P$ of interior marked points (which following~\cite{FG06b} we refer to as \emph{punctures}), and a subset $\Cc = \hc{c_j}$ of connected components of $\partial S$, each of which comes with a distinguished point $t_j \in c_j$. We refer to the elements of $\Cc$ as \emph{tacked circles}, to the points $t_j$ as their \emph{tacks.} A \emph{special point} on a marked surface is one which is either a boundary marked point, a puncture, or a tack. We denote the set of special points by $V = V(S)$.
\end{defn}
% Marked surfaces without tacked circles as defined here are the same as the marked surfaces defined in ~\cite{FG06b}. 
% Note that if $S$ is a marked surface in the sense of Definition~\ref{def:marked}, then $S_\circ$ is one in the sense of~\cite{FG06b}.

% We refer to these diamonds as \emph{tacks,} and to connected components of $\partial S$ containing them as \emph{tacked circles.}

From a simple decorated surface $\decS$ one can produce a marked surface $S$ using the following recipe:
\begin{itemize}
    \item Retract each boundary or interior $T$-disk to the unique marked point it contains, thereby creating respectively a boundary marked point or a puncture on $S$;
    \item Retract each boundary $T$-annulus containing a marked point to its intersection with $\partial\decS$ and replace the marked point with a tack, thereby creating a tacked circle on $S$;
    \item Retract each boundary $T$-annulus without marked points to its intersection with $\partial\decS$ and glue in a disk with an interior marked point, thereby creating a puncture on $S$.
\end{itemize}
We show an example of a decorated surface and the corresponding marked surface on Figure~\ref{fig:marked}. In order to graphically disambiguate boundary marked points from tacks, we draw the latter as diamonds.

\begin{figure}[h]
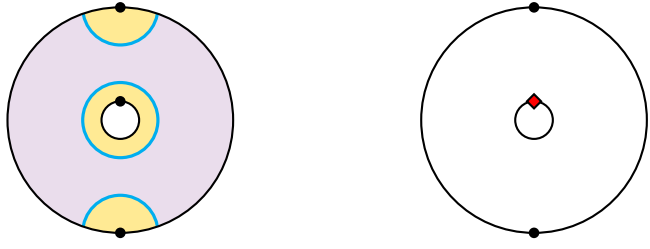

\subfile{marked.tex}
\caption{Decorated surface and marked surface.}
\label{fig:marked}
\end{figure}

Note that the above recipe does not distinguish the interior marked points on a decorated surface from its boundary circles without marked points, and thus different decorated surfaces can be assigned the same marked surface. 

\begin{defn}
Given a marked surface $S$, we define its \emph{$\mathcal{A}$-lift} $\decS_{\mathcal{A}}$ to be the decorated surface where all punctures are replaced by internal $T$-disks with a single marked point, and its \emph{$\mathcal{P}$-lift}  $\decS_{\mathcal{P}}$  to be the one where all punctures are replaced by boundary circles without marked points.
\end{defn}

\begin{notation}
    Given a marked surface $S$, we write $S_\circ$ for the marked surface obtained from $S$ by replacing every tacked circle of $S$ by a puncture. By construction, $S_\circ$ is then a marked surface in the sense of~\cite{FG06b}.
\end{notation}

\begin{defn}
Given a marked surface $S$, we denote by $\Diffeo^+(S)$ the group of diffeomorphisms of $S$ which preserve its orientation, as well as the sets of punctures, marked points, and tacked circles. We write $\Diffeo_0^+(S)$ for the connected component of the identity in $\Diffeo^+(S)$. The \emph{mapping class group} of $S$ is the quotient
$$
\Gamma_S = \Diffeo^+(S)/\Diffeo_0^+(S).
$$
\end{defn}

\begin{defn}
An \emph{arc} on a marked surface $S$ a curve $a \colon [0,1] \rightarrow S$ satisfying the following conditions:
\begin{itemize}
	\item $a(t_1) = a(t_2)$ implies $\hc{t_1,t_2} = \hc{0,1}$;
	\item $\partial a \in V(S)$ and $(a \smallsetminus \partial a) \cap V(S) = \varnothing$;
	\item $a$ does not retract onto $\partial a$.
\end{itemize}
\end{defn}

We shall only consider arcs up to the action of $\Diffeo_0^+(S)$ and abuse notation by referring to their $\Diffeo_0^+(S)$-orbits by arcs as well.

\begin{defn}
Two arcs are \emph{compatible} if they do not intersect in $S \smallsetminus V(S)$. An \emph{ideal triangulation} of a marked surface $S$ is a maximal collection of distinct, pairwise compatible arcs. The arcs of a triangulation cut the surface S into \emph{ideal triangles} and are often referred to as \emph{edges.} A triangle is \emph{self-folded} if two of its edges coincide, and is \emph{special} if one of its edges is a tacked circle. An edge $e$ of a triangle is \emph{boundary} if $e \subset \partial S$.
\end{defn}

Ideal triangulations are always considered up to the action of $\Diffeo_0^+(S)$. In what follows, we  will often write ``triangulation'' instead of `ideal triangulation' whenever it does not cause confusion.

\begin{defn}
We say that a marked surface $S$ is \emph{admissible} if the surface $S_\circ$ admits an ideal triangulation.
\end{defn}

Let $g(S)$ denote the genus of $S$. Then it is not hard to see that $S$ is admissible if and only if it satisfies one of the following:
\begin{itemize}
\item $g(S)>0$ and $V(S) \ne \varnothing$;
\item $g(S)=0$ and $|V(S)| \ge 3$;
\item $g(S)=0$, $|V(S)| = 2$, and $|M(S)|=1$.
\end{itemize}

% A \emph{flip} $F_e$ of a non-boundary edge $e$ is the replacement of $e$ with the other diagonal of the unique 4-gon containing $e$ as a diagonal.

% Assume that $\decS$ is such and write  $S$ for the underlying oriented surface. Let us retract each boundary $T$-disk of $\decS$ to its marked point and each boundary $T$-annulus to its intersection with $\partial S$.
%, and shrink to punctures those connected components of $\partial S$ which do not contain marked points.
% In order to disambiguate marked points on $\partial S$ which came from boundary $T$-disks from those which came from boundary $T$-annuli, we draw the latter as diamonds, see Figure~\ref{fig:marked}. We refer to these diamonds as \emph{tacks,} and to connected components of $\partial S$ containing them as \emph{tacked circles.}

% The stack $\Ch(\decS)$ parametrizes $G$-local systems on the $G$-region of $\decS$ together with a flat section of the associated $G/B$-bundle over every $T$-region. \red{[Wait, didn't we say the decoration is a flat section of the $G\times T/B$ bundle?]} 

Let $\decS$ be a simple decorated surface, $D$ its boundary $T$-disk, and $D'$ the boundary $T$-disk immediately following $D$ as we traverse $\partial S$ in the positive direction, note that $D$ and $D'$ may coincide. Consider a framed decorated local system $\Lc$ on $\decS$ and set $s, s'$ to be the sections of the associated $G/N$-bundles over the walls $c \subset \partial D$, $c' \subset \partial D'$. Furthermore, let $\rho(s)$ to be the parallel transport of $s$ along the path traversing $\partial S$ in the positive direction from $c$ to $c'$, so that $\rho(s)$ is a section of the $G/N$ bundle over $c'$. Recall that a pair of framed flags $\bar F_1, \bar F_2 \in G/N$ is said to be in relative position $u \in W$, where $W$ is the Weyl group of $G$, if its $G$-coset can be presented as $NtuN \subset G\backslash(G/N)^2\simeq N\backslash G/N$ for some $t\in T$. The framed decorated local system $\Lc$ is \emph{generic at $D$} if $\rho(s)$ and $s'$ are in relative position $w_0$, with the latter being the longest element of the Weyl group $W$. We say that $\Lc$ is \emph{generic} if it is such at each boundary $T$-disk. Given a simple decorated surface $\decS$, we define $\Ch^\circ_\bullet(\decS) \subset \Ch_\bullet(\decS)$ to be the substack which parametrizes generic framed decorated local systems. We then set $\Ch^\circ(\decS) = p(\Ch^\circ_\bullet(\decS))$, where $p \colon \Ch_\bullet(\decS) \to \Ch(\decS)$ is the natural projection. We remark that $\Ch^\circ(\decS)$ and $\Ch^\circ_\bullet(\decS)$ are in fact quasi-projective varieties, see~\cite{IO23}.

\begin{remark}
\label{rem:cluster-structure-ex}
If $S$ is an admissible marked surface $S$ without tacked circles, then the following is known:\begin{itemize}
\item[a)]Suppose $G$ is simply connected, and let $\decS_{\mathcal A}$ be the decorated surface obtained as the $\mathcal{A}$-lift of $S$. % Assume that every $T$-region of $\decS$ is a disk, internal or boundary, with a single marked point, and the group $G$ is simply connected. 
Then $\Ch^\circ_\bullet(\decS_{\mathcal A})$ is (non-canonically) isomorphic to the moduli space $\Ac_{G,S}$ of \emph{twisted} framed local systems on $S$, see Section~\ref{subsec:cluster-coord}, which carries a natural pre-symplectic form. In~\cite{FG06b, GS19}, there was constructed a cluster $K_2$-variety $\Ac_{\Qcl(G,S)}$ sharing with $\Ac_{G,S}$ a subset of toric charts, on which the pre-symplectic form of $\Ac_{\Qcl(G,S)}$ coincides with that of $\Ac_{G,S}$. It is expected that the ring of functions on the cluster $K_2$-variety coincides with that on $\Ac_{G,S}$, however we do not know of a reference where this statement is proven in full generality.
\item[b)] If $G$ is of adjoint type, i.e. has has trivial center, then by~\cite[Lemma 3.13]{GS19}, the moduli space $\Ch^\circ_\bullet(\decS)$ coincides with the moduli space $\Pc_{G,S}$ of {\it framed $G$-local systems with pinnings} introduced in {\it loc.\,cit.} In~\cite{FG06b, GS19},  there was constructed a cluster Poisson variety $\Pc_{\Qcl(G,S)}$ sharing with $\Pc_{G,S}$ a subset of toric charts, on which the Poisson structure of $\Pc_{\Qcl(G,S)}$ coincides with that of $\Pc_{G,S}$. The isomorphism $\Oc\big(\Pc_{\Qcl(G,S)}\big) \simeq \Oc(\Pc_{G,S})$ of coordinate rings was then proven in~\cite{She22}, by showing that the complement to the union of joint toric charts is of codimension at least two in $\Pc_{G,S}$.
%A similar statement is expected in the case of $\mathcal{A}_{G,S}$, although we do not know of a reference. }
\end{itemize}
\end{remark}

% \begin{defn}
% \label{def:marked}
% \red{Do we want to allow punctures here?} A \emph{marked surface} is an oriented surface $S$ together with a collection $P \subset S$ of punctures, a subset $M \subset \partial S$ of boundary marked points, and a subset $\Cc \subset \pi_0(\partial S)$ of tacked circles, where each tacked circle contains a unique marked point. We write $\mathbb{M}$ for the union of $M$ with the set of marked points on all tacked circles of $S$.

% We also write $S_\circ$ the marked surface obtained from $S$ by replacing every tacked circle of the latter by a puncture. 
% \end{defn}

% Note that if $S$ is a marked surface in the sense of Definition~\ref{def:marked}, then $S_\circ$ is one in the sense of~\cite{FG06b}.

% \begin{defn}
% Let $g(S)$ denote the genus of $S$. We say that a marked surface $S$ is \emph{admissible} if one of the following holds:
% \begin{itemize}
% \item $g(S)>1$ and $\hm{P(S_\circ) \sqcup M(S_\circ)} \ge 1$;
% \item $g(S)=0$ and $\hm{P(S_\circ) \sqcup M(S_\circ)} \ge 3$;
% \item $g(S)=0$ and $\hm{P(S_\circ)} = \hm{M(S_\circ)} = 1$.
% \end{itemize}
% \end{defn}

% It is easy to see that admissible marked surfaces are in bijection with admissible decorated surfaces without internal $T$-disks. \red{?} In what follows we will always assume that a marked surface is admissible.

Most important for us will be the class of simple decorated surfaces whose $T$-regions are boundary disks and annuli containing one marked point each. In this case, all special points of the corresponding marked surface are boundary marked points and tacks. As indicated in Example~\ref{ex:4-disks}, the moduli spaces for all other simple decorated surfaces can be recovered as various kinds of reductions of those for surfaces of this type.
\begin{defn}
\label{def:Pdiamond}
Let $S$ be a marked surface without punctures and $\decS$ be the corresponding decorated surface. We define the moduli space
$$
\Pc_{G,S}^\diamond = \Ch_\bullet^\circ(\decS).
$$
\end{defn}

In this text we will only consider Lie groups of Dynkin type $A_n$, whose simply connected form is $SL_{n+1}$ and whose adjoint form is $PGL_{n+1}$. For a marked surface $S$ without punctures but possibly with tacked circles, in Section~\ref{sec:cluster-coord} we will construct, in the sense of Remark~\ref{rem:cluster-structure-ex}, the structure of a non-degenerate cluster ensemble on the moduli spaces $\hr{\Pc_{SL_{n+1},S}^\diamond,\Pc_{PGL_{n+1},S}^\diamond}$. Among the symplectic toric charts in its atlas are those introduced in~\cite{BK23}.

% \red{Do we need to worry about adjoint versus simply-connected forms here? I.e. have $\mathscr{A}$-type coordinates on the space with $G=SL_{n+1}$, and $\mathscr{X}$-type ones for the $G=PGL_{n+1}$ version?}

%Recall that central characters of the group $G$ are in bijection with regular semisimple conjugacy classes. Let $\bs\la = \hc{\la_p \,|\, p \in P(S)}$ be a collection of central characters of $G$. Then
%$$
%\underline\Lc_\bullet^\circ(\decS) = \bigsqcup_{\bs\la} \underline\Lc_{\bullet, \bs\la}^\circ(\decS),
%$$
%where $\underline\Lc_{\bullet, \bs\la}^\circ(\decS)$ is the subspace of framed decorated local systems on $\decS$, such that for every puncture $p \in P(S)$ the monodromy around $p$ belongs to the conjugacy class defined by $\la_p$. We denote the corresponding moduli space by
%$$
%\Ch_{\bullet,\bs\la}^\circ(\decS) = \underline\Lc_{\bullet, \bs\la}^\circ(\decS)/\Gc_\decS.
%$$
%
%\begin{defn}
%Let $S$ be a marked surface and $\decS$ be the corresponding decorated surface. We define the moduli space
%$$
%\Pc_{G,S,\bs\la}^\diamond = \Ch_{\bullet,\bs\la}^\circ(\decS).
%$$
%\end{defn}

\subsection{Ptolemy groupoids}
% \red{In this section, $S$ will be a marked surface without punctures, although possibly with tacked circles.}
% \begin{defn}
% Consider a marked surface $S$. Let $\Diffeo^+(S)$ be the group of diffeomorphisms of $S$ which preserve the orientation, as well as the sets of punctures, marked points, and tacked circles. Denote by $\Diffeo_0^+(S)$ the connected component of identity in $\Diffeo^+(S)$. The \emph{mapping class group} of $S$ is the quotient
% $$
% \Gamma_S = \Diffeo^+(S)/\Diffeo_0^+(S).
% $$
% \end{defn}

% \begin{defn}
% Let $S$ be a marked surface. An \emph{ideal triangulation} of $S$ is a finite collection of triangles covering $S$, such that the set of vertices coincides with the full set of marked points $\mathbb{M}$, and any pair of distinct triangles intersect by a possibly empty union of edges and vertices. A triangle is \emph{self-folded} if two of its edges coincide, and is \emph{special} if one of its edges is a tacked circle. An edge $e$ of a triangle is \emph{boundary} if $e \subset \partial S$. A \emph{flip} $F_e$ of a non-boundary edge $e$ is the replacement of $e$ with the other diagonal of the unique 4-gon containing $e$ as a diagonal.
% \end{defn}
\begin{defn}
Let $S$ be a marked surface with a triangulation $\tri$, and $e$ be a non-boundary edge in $\tri$. We say that the triangulation $\tri'$ is related to $\tri$ by a \emph{flip} $F_e$ at $e$ if $\tri'$ is obtained from $\tri$ by replacing $e$ with the other diagonal of the unique 4-gon containing $e$ as a diagonal.
\end{defn}

% In what follows, we always consider ideal triangulations of $S$ up to an action of $\Diffeo_0^+(S)$. 
% The notion of a flip naturally descends to $\Diffeo_0^+(S)$-orbits of ideal triangulations. 

\begin{defn}
The \emph{Ptolemy complex} $\Pt_2(S)$ of a marked surface $S$ is a 2-groupoid, whose objects are ideal triangulations,\footnote{Let us point out that our definition of the Ptolemy complex differs from the one given in~\cite{Tes05}, where the objects are ideal triangulations whose edges are labelled.} 1-morphisms are sequences of flips, and 2-morphisms are compositions of
\begin{itemize}
\item \emph{bigons,} $F_e^2 = \Id$;
\item \emph{4-gons,} $F_e F_f = F_f F_e$, where $e$ and $f$ do not share a vertex;
\item \emph{pentagons,} see Figure~\ref{fig:pentagon}.
\end{itemize}
\end{defn}

\begin{remark}
The Ptolemy complex can be presented in more topological terms. Namely, consider a 3-manifold $M_S$ defined as $M_S = S \times [0,1]/(\partial S \times [0,1])$. Then 1-morphisms in the Ptolemy complex are the 3d-triangulations of $M$, i.e.\ presentations of $M$ as a union of tetrahedra, while pentagon relations are the 2-3 Pachner moves.
\end{remark}

\begin{defn}
The \emph{Ptolemy groupoid} $\Pt(S)$ is the 1-truncation of the Ptolemy complex. 
\end{defn}
Concretely, this means that $\Pt(S)$ is the 1-category with the same objects as the Ptolemy complex $\Pt_2(S)$, but with 1-morphisms given by isomorphism classes of $1$-morphisms in $\Pt_2(S)$. 
\begin{figure}[h]
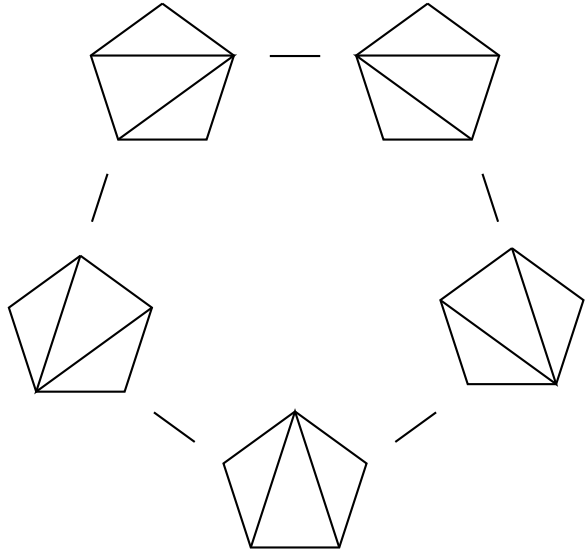

\subfile{pentagon.tex}
\caption{Pentagon relation.}
\label{fig:pentagon}
\end{figure}

\begin{prop}[\cite{CF99}]
The Ptolemy complex is connected and simply connected:
\begin{enumerate}
\item for any pair of objects, $\tri_1,\tri_2$, there exists a 1-morphism $F \colon \tri_1 \to \tri_2$;
\item for any pair of 1-morphisms, $F,F' \colon \tri_1 \to \tri_2$, there exists a 2-morphism $\Phi \colon F \Rightarrow F'$.
\end{enumerate}
\end{prop}

\begin{cor}
The Ptolemy groupoid is connected and thin: there exists a unique morphism between any pair of objects.
\end{cor}

\begin{defn}
Let $\Gamma$ be a group acting on a groupoid $\mathrm{G}$ via endofunctors $A_\gamma$, $\gamma \in \Gamma$. The \emph{enhancement} of $\Grm$ by $\Gamma$ is the groupoid  $\widehat\Grm$ whose objects are the same as those of $\Grm$, and morphisms $\Mor(O_1,O_2)$ between objects $O_1$ and $O_2$ are the formal compositions $A_\gamma F$, where $F \in \Mor(O_1,A_\gamma^{-1}(O_2))$. Composition in $\widehat\Grm$ is defined by
\begin{align}
\label{eq:enhanced-mult}
A_{\gamma_2} F_2 \circ A_{\gamma_1} F_1 =  A_{\gamma_2 \gamma_1} A_{\gamma_1}^{-1}(F_2) F_1,
\end{align}
where $F_1 \in \Mor(O_1,A_{\gamma_1}^{-1}(O_2))$ and $F_2 \in \Mor(O_2,A_{\gamma_2}^{-1}(O_3))$.
\end{defn}

We apply this construction to the natural action via endofunctors of the mapping class group $\Gamma_S$ on $\Pt(S)$. Namely, an element $\gamma \in \Gamma_S$ acts via an endofunctor $A_\gamma$ defined by $A_\gamma(\tri) = \gamma(\tri)$ and $A_\gamma(F_{\tri_2,\tri_1}) = F_{\gamma(\tri_2), \gamma(\tri_1)}$, where $F_{\tri_2,\tri_1}$ is the unique morphism between $\tri_1$ and $\tri_2$. This action of $\Gamma_S$ evidently descends from one on the Ptolemy complex.
%We also recall that the stabilizer in $\Gamma_S$ of any triangulation $\tri$ is trivial, see~\red{BLAH}.

\begin{defn}
The \emph{enhanced Ptolemy groupoid} $\widehat\Pt(S)$ is the enhancement of the Ptolemy groupoid $\Pt(S)$ by the mapping class group $\Gamma_S$.
\end{defn}

\begin{cor}
\label{cor:MCG-pi1}
For any object $\tri$ of $\widehat\Pt(S)$, we have an injective homomorphism
$$
%\label{eq:MCG-act}
\Gamma_S \hooklongrightarrow \Aut(\tri) \simeq \pi_1\big(\widehat\Pt(S)\big), \qquad \gamma \longmapsto A_\gamma F_{\gamma^{-1}(\tri),\tri}.
$$
\end{cor}

%\blue{Gus attempt to spell out definition, feel free delete/comment out if superfluous: let us denote by $[\tri_1,\tri_2]$ the unique morphism in $\Pt(S)$ from $\tri_2$ to $\tri_1$. Then the set of morphisms from $\tri_2$ to $\tri_1$ in $\widehat\Pt(S)$ consists of pairs $(A_\gamma,[\tri',\tri_2])$ where $\gamma\in \Gamma_S$ satisfies $\gamma(\tri')=\tri_1$. Composition is defined by
%$$
%(A_\beta,[\tri'',\tri_2])\circ(A_\gamma,[\tri',\tri_1]) := (A_{\beta\gamma},[\gamma^{-1}(\tri''),\tri_1])
%$$}
%Hence for any object $\tri$ in $\widehat\Pt(S)$ we get a homomorphism $\Gamma_S \to \End(\tri)$ given by
%\beq
%\label{eq:MCG-act}
%\gamma \mapsto A_\gamma \circ [\gamma^{-1}(\tri),\tri],
%\eeq
%where $[\gamma^{-1}(\tri),\tri]$ is the unique morphism from $\tri$ to $\gamma^{-1}(\tri)$ in the thin Ptolemy groupoid.

\subsection{Cut and glue functors}
\label{subsec:cut-n-glue}
In this section, $S$ will be a marked surface without puntures, but possibly with tacked circles.
\begin{defn}
We say that an arc $a$ of a triangulation $\tri$ of $S$ is a \emph{shadow} of a simple closed curve $c \subset S$ if $a$ and $c$ are isotopic as free loops. An ideal triangulation $\tri$ of $S$ is called \emph{$c$-admissible} if it contains a shadow $a$ of $c$ which is not a tacked circle.
\end{defn}
{Note in particular that shadows of a curve $c$ need not be isotopic as arcs of the triangulation, i.e.\ relative to their basepoint, and can indeed  be based at different special points on $S$.}
 In any given triangulation, a  simple closed curve $c$ can have at most 2 shadows. If $c$ has two shadows, they bound a cylinder $C \subset S$ covered by a pair of triangles, with $\pi_1(C) = \ha{c}$. Note that the shadows of $c$ are also shadows of each other.

\begin{defn}
A simple closed curve $c \subset S$ is said to be \emph{isolated} in a triangulation $\tri$ if the latter contains two shadows of $c$. We call the subsurface $C \subset S$ bounded by the shadows of $c$ an \emph{isolating cylinder} for $c$, and refer to $\tri$ as a $c$\emph{-isolating triangulation} of $S$.
\end{defn}

% <<<<<<< Updated upstream
% Consider a collection $\bs c$ of  simple closed curves on $S$. In what follows we will consider the full subgroupoids
% =======
Consider a collection $\bs c$ of non-intersecting, non-homotopic  simple closed curves on $S$. We consider the following full sub-groupoids 
% >>>>>>> Stashed changes
$$
\Pt^o_{\bs c}(S) \subset \Pt_{\bs c}(S) \subset \Pt(S),
$$
where the objects of $\Pt_{\bs c}(S)$ are ideal triangulations admissible with respect to at least one curve in $\bs c$ and those of $\Pt^o_{\bs c}(S)$ are ideal triangulations admissible with respect to every curve in $\bs c$.
%We will also make use of full subgroupoids
%$$
%\Pt_{|\bs c|}(S) \subset \Pt_{\bs c}(S)
%\qquad\text{and}\qquad
%\Pt^o_{|\bs c|}(S) \subset \Pt^o_{\bs c}(S),
%$$
%where the objects of $\Pt_{|\bs c|}(S)$ are such ideal triangulations where at least one curve in $\bs c$ is isolated and those of $\Pt^o_{\bs c}(S)$ are ideal triangulations where every curve in $\bs c$ is isolated.

Let $\tau_c$ be the Dehn twist along an oriented simple closed curve $c \subset S$. Clearly, the cyclic subgroup $\ha{\tau_c} \subset \Gamma_S$ is independent of the orientation of $c$. Given a collection $\bs c$ of  simple closed curves on $S$, we write
$$
\Gamma_{S;\bs c} = C_{\Gamma_S}\ha{\tau_c \,|\, c \in \bs c}
$$
for
the centralizer in $\Gamma_S$ of the subgroup generated by the Dehn twists $\tau_c$, $c \in \bs c$.
Then elements of $\Gamma_{S;\bs c}$ preserve the isotopy classes of all curves in $\bs c$. We define $\widehat\Pt_{\bs c}(S)$ and $\widehat\Pt^o_{\bs c}(S)$ to be the $\Gamma_{S;\bs c}$-enhanced groupoids $\Pt_{\bs c}(S)$ and $\Pt^o_{\bs c}(S)$ respectively.

Now, let $S'$ be a possibly disconnected marked surface, and $c_\pm \in \Cc(S')$ be a pair of oriented tacked circles
% , with orientation induced by that of $S'$, 
with tacks $t_\pm \in c_\pm$. Given an orientation-compatible homeomorphism $\phi \colon c_+ \to c_-$ such that $\phi(t_+) = t_-$, we consider an oriented surface $S$ obtained from $S'$ by identifying $c_+$ and $c_-$ along $\phi$, and forgetting the data of the tacks $t_\pm$. {We endow the curve $c\subset S$ obtained as the common image of the $c_\pm$ with the orientation given by that of the curve $c_+$. }

% \red{[I don't get the stuff about erasing; the image of $c_\pm$ is now just contained in the interior of S, right?] }
% Let $\Pt_\phi(S')$ be the groupoid obtained from $\Pt_{c_\pm}(S')$ by identifying all objects in the same orbit of the cyclic group $\ha{D_{c_+}D_{c_-}}$ generated by the product of two Dehn twists, and all morphisms with the same source and target. 
 Let $\Pt_\phi(S')$ be the thin groupoid whose set of objects is obtained from that of $\Pt_{c_\pm}(S')$ by identifying all objects in the same orbit of the cyclic group $\ha{\tau_{c_+}\tau_{c_-}}$ generated by the product of two Dehn twists.
 Gluing along $\phi$ induces a surjection from $\Gamma_{S'}$ to $\Gamma_{S;c}$ whose kernel is $\ha{\tau_{c_+}\tau_{c_-}}$, and we set
$$
\Gamma_{S';\phi} = \Gamma_{S'}/\ha{\tau_{c_+}\tau_{c_-}}\simeq \Gamma_{S;c}.
$$
This group acts on $\Pt_\phi(S')$, and we define $\widehat\Pt_\phi(S')$ to be the $\Gamma_{S';\phi}$-enhanced groupoid $\Pt_\phi(S')$. The same construction can obviously be carried out for a collection of pairs of tacked circles on $S'$.

\begin{prop}
\label{prop:cut-glue}
Let $\bs c$ be a collection of pairwise non-intersecting, non-homotopic simple closed curves on $S$. Then we have a pair of mutually-inverse \emph{cutting} and \emph{gluing} functors:
\begin{align*}
\Cut_{\bs c} &\colon \Ptoh_{\bs c}(S) \longra \widehat\Pt_{\bs\phi}(S'), \\
\Glue_{\bs\phi} &\colon \widehat\Pt_{\bs\phi}(S') \longra \Ptoh_{\bs c}(S).
\end{align*}
\end{prop}

\begin{proof}
% Note that it suffices to construct a pair of inverse \emph{cut} and \emph{glue} functors:
% \begin{align*}
% \Cut_{\bs c} \colon \Ptoh_{\bs c}(S) \longra \widehat\Pt_{\bs\phi}(S'), \\
% \Glue_{\bs\phi} \colon \widehat\Pt_{\bs\phi}(S') \longra \Ptoh_{\bs c}(S).
% \end{align*}
% <<<<<<< Updated upstream
We treat the case of a single curve $c$, with the general case being obtained by iterating the construction.
To specify functors as in the statement of the Proposition, it suffices to define $\Gamma_{S;c}$-equivariant maps $\Cut_c$ and $\Glue_\phi$ on objects of $\Pt_c(S)$ and $\Pt_\phi(S')$, since both groupoids are connected and thin. Let $\tri'$ be an object in $\widehat\Pt_\phi(S')$. Without loss of generality we may assume that it is $c_+$-admissible. Then $\tri'$ contains the isolating cylinder $C$ of $c_+$, which in turn has 4 arcs: $c_+$, its shadow $a_+$, and a pair of non-boundary arcs, $e_1,e_2$, where $e_1$ precedes $e_2$ as we go around the tack on $c_+$ in a positive direction defined by the orientation of $S'$.  Let us identify tacked circles $c_+$ and $c_-$ along $\phi$, and retract $C$ onto $a_+$ in such a way that $e_1$ retracts to the marked point of $a_+$, arriving at a triangulation $\tri$ in $\widehat\Pt_c(S)$. We illustrate this on Figure~\ref{fig:local-gluing-surfaces}, where the homeomorphism $\phi$ is indicated by horizontal yellow segments connecting tacked circles $c_\pm$. This procedure yields the same ideal triangulation $\tri$ if the initial one $\tri'$ is replaced by its image under the diagonal product of Dehn twists $\tau_{c_+}\tau_{c_-}$. Hence it descends to a map on objects of $\Pt_\phi(S')$ given by $\Glue_\phi(\tri') = \tri$. This map is $\Gamma_{S;c}$-equivariant, and well-defined on ideal triangulations which are both $c_+$- and $c_-$-admissible.

\begin{figure}[h]
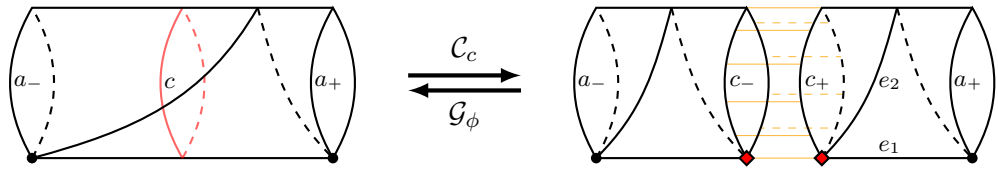

\subfile{fig-gluing-cyl.tex}
\caption{Cutting and gluing functors.}
\label{fig:local-gluing-surfaces}
\end{figure}

On the other hand, let $\tri$ be an object in $\widehat\Pt_c(S)$, and $a$ be a shadow of $c$ {equipped with an orientation.} Cut the surface $S$ along $a$, thus obtaining a pair of new boundary arcs $a_\pm$. We suppose that the orientation of $c$ is chosen such that we may assume that $a_+$ is a shadow of $c$. We then turn it into a tacked circle and rename it into $c_+$. Now, we glue a cylinder $C$ with one marked point on each side along $a_-$, denote by $c_-$ the arc of the cylinder which is a new boundary arc, and turn $c_-$ into a tacked circle, thereby recovering the surface $S'$ along with a partial triangulation. We complete it to a triangulation $\tri'$ by taking the unique triangulation of $C$ such that $\mathcal{G}_\phi(\tri')=\tri$, and define the $\Gamma_{S;c}$-equivariant map $\Cut_c$ by setting $\Cut_c(\tri) = \tri'$. This completes the construction of the pair of inverse functors $\mathcal{G}_\phi$ and $\mathcal{C}_c$.
\end{proof}

We finish this section with the following observation. Let $\bs c \subset \Cc(S)$ and $\bs\phi$ be as in Proposition~\ref{prop:cut-glue}. Given a decomposition $\bs c = \bs c_1 \sqcup \bs c_2$, denote by $S_{\bs c_i}$ the result of cutting $S$ along $\bs c_i$, and let $\bs\phi_i$ be such that $S$ is the result of gluing $S_{\bs c_i}$ by $\bs\phi_i$.
Then the equalities
$$
\Cc_{\bs c} = \Cc_{\bs c_1} \Cc_{\bs c_2} = \Cc_{\bs c_2} \Cc_{\bs c_1}
\qquad\text{and}\qquad
\Gc_{\bs\phi} = \Gc_{\bs\phi_1} \Gc_{\bs\phi_2} = \Gc_{\bs\phi_2} \Gc_{\bs\phi_1}
$$
follow from the construction of functors $\Cc_c$ and $\Gc_\phi$. To aid the reader in parsing these formulas, we write down the source and target groupoids of thefactors in the composition $\Cc_{\bs c_2} \Cc_{\bs c_1}$. We have
\begin{align*}
&\Cc_{\bs c_1} \colon \widehat\Pt_{\bs c}(S) \longra \Ptoh_{\bs c_2;\bs\phi_1}(S_{\bs c_1}), \\
&\Cc_{\bs c_2} \colon \Ptoh_{\bs c_2;\bs\phi_1}(S_{\bs c_1}) \longra \widehat\Pt_{\bs\phi}(S_{\bs c}),
\end{align*}
where $\Pt^o_{\bs c_2;\bs\phi_1}(S_{\bs c_1})$ is the full subgroupoid of $\Pt_{\bs\phi_1}(S_{\bs c_1})$, whose objects are (equivalence classes of) triangulations admissible with respect to every curve in $\bs c_2$, and $\Ptoh_{\bs c_2;\bs\phi_1}(S_{\bs c_1})$ is the $\Gamma_{S;\bs c}$-enhancement of $\Pt^o_{c_2;\bs\phi_1}(S_{\bs c_1})$.

\subsection{Umbral moves}
\label{subsec:umbral}
In this section we investigate the passage from one $c$-isolating triangulation to another in the Ptolemy groupoid. The key to understanding this is the following local transformation on isolating cylinders. 

\begin{defn}
\label{def:umbral}
Let $C$ be an isolating cylinder of $c$, $e_0$ a non-boundary boundary shadow of $c$, and $e_1$, $e_2$ the internal arcs, labelled in such a way that the cycle $e_0\rightarrow e_1\rightarrow e_2\rightarrow e_0$ inside each triangle of $C$ agrees with the orientation of the surface.
%such that $e_2$ follows $e_1$ as we go around the marked point of $e_0$ in a direction consistent with the orientation of $C$. 
%We call the following two sequences of flips
%$$
%F^+_{C;e_0} = F_{e_2}F_{e_1}F_{e_0},
%\qquad
%F^-_{C;e_0} = F_{e_1}F_{e_2}F_{e_0}
%$$
%the \emph{umbral moves.} 
We call the following sequence of flips
$$
U_{C;e_0} = F_{e_2}F_{e_1}F_{e_0}
$$
the \emph{umbral move} at $e_0$.
%\begin{figure}
%\label{fig:umbral}
%\includegraphics[scale=.2]{umbral}
%\caption{The umbral move $U_{C;e_0}$.}
%\end{figure}
\begin{figure}[h]
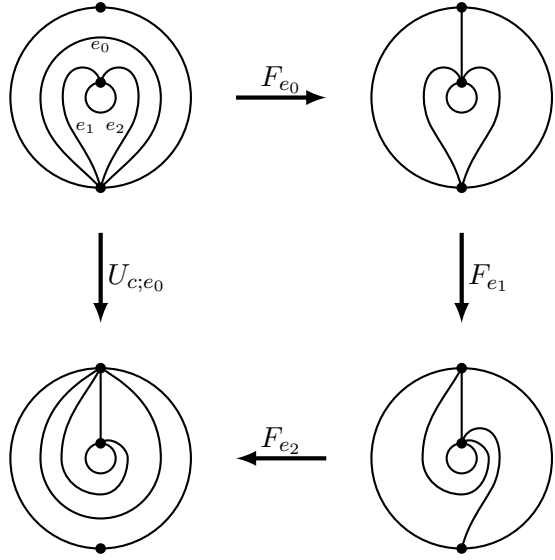

\subfile{fig-umbral.tex}
\caption{The umbral move $U_{C;e_0}$.}
\label{fig:umbral}
\end{figure}

The umbral move is illustrated in Figure~\ref{fig:umbral}.

\begin{remark}
There is also an \emph{opposite umbral move} given by the sequence of flips 
$$U^{op}_{C;e_0} = F_{e_1}F_{e_2}F_{e_0}.
$$ 
The triangulation $U^{op}_{C;e_0}(\tri)$ is simply the image of $U_{C;e_0}(\tri)$ under the Dehn twist along $c$, and for this reason we choose to work only with the umbral moves themselves and not their opposite counterparts.
\end{remark}
Suppose we cut $S$ along $c$ to produce a surface $S'$ with two new tacked circles $c_\pm$. Given a triangulation $\tri$ of $S$ containing an isolating cylinder $C$ for $c$,  write $e_\pm$ for the corresponding edges bounding $C$. It is straightforward to see that under the cutting functor $\mathcal{C}_c$, the umbral move $U_{C;e_\pm}$ is intertwined with the (unique) umbral move $U_{C_\pm}$ on $S'$ for the cylinder $C_\pm$ having $c_\pm\subset\partial S'$ as one of its boundaries.

\end{defn}

%\begin{prop}
%\label{prop:3-3}
%Let $C, C_\pm$ be cylinders isolating a curve $c \subset S$, $\tri$ be an object in $\Pt_C(S)$, and $\tri_\pm = F^\pm_{C;e_0}(\tri)$ objects in $\Pt_{C_\pm}(S)$. Then the following diagram is commutative:
%$$
%\begin{tikzcd}
%\Hc(\tri) \arrow{r}{\eta_{C}(\tri)} \arrow[swap]{d}{\Hc\hr{F^\pm_{C;e_0}}} & \Hc_c(\tri) \arrow{d}{\Hc_c\hr{F^\pm_{C;e_0}}} \\
%\Hc(\tri_\pm) \arrow[swap]{r}{\eta_{C_\pm}(\tri_\pm)} & \Hc_c(\tri_\pm)
%\end{tikzcd}
%$$
%\end{prop}

%\begin{prop}
%\label{prop:3-3}
%Suppose $\tri$ is an ideal triangulation containing an isolating cylinder $C$ for the curve $c$, and $e$ is a non-boundary shadow of $c$. Write $\tri' = F_{C;e}(\tri)$ for the triangulation obtained by applying the umbral move at $e$ to $\tri$. Then the following diagram is commutative:
%$$
%\begin{tikzcd}
%\Hc(\tri) \arrow{r}{\eta_{\tri}} \arrow[swap]{d}{\Hc\hr{F_{C;e}}} & \Hc_c(\tri) \arrow{d}{\Hc_c\hr{F_{C;e}}} \\
%\Hc(\tri') \arrow[swap]{r}{\eta_{\tri'}} & \Hc_c(\tri')
%\end{tikzcd}
%$$
%\end{prop}

%We prove Proposition~\ref{prop:3-3} in Section~\ref{sec:3-to-3}. \red{Not really, we only prove the algebraic version. We still need to construct $\Hc_c$ in Section 10, and use that the only operators that commute with $\Lbb$ are scalars. Btw, is it obvious in the cut rep-n?}

\begin{lemma}
\label{lem:re-iso}
Any morphism in $\Pt_{\hm{c}}(S)$ can be factored into a sequence of umbral moves and flips preserving the isolating cylinder of $c$.
\end{lemma}

\begin{proof}
Since the cutting functor intertwines umbral moves on $S$ with those on the cut surface $S'$, it suffices to prove the statement for objects of the groupoid $\Pt^o_{c_+,c_-}(S')$ associated to $S'$.
It is convenient to introduce the surface $S_c^\circ$ obtained from $S'$ by shrinking the tacked circles $c_\pm$ to a pair of punctures $p_\pm$. Denote by $\Pt^o_{\hm{p_\pm}}(S_c^\circ)$ the full subgroupoid of $\Pt(S_c^\circ)$, such that in any object of the former both punctures $p_\pm$ are internal punctures of self-folded triangles. 
 Recall that any $\tri$ in $\Pt^o_{c_+,c_-}(S')$ contains isolating cylinders $C_\pm$ for both tacked circles $c_\pm$. So replacing the $C_\pm$ by self-folded triangles $\Delta_\pm$, we obtain an object $\tri^\circ$ of  $\Pt^o_{\hm{p_\pm}}(S_c^\circ)$. In what follows we say that triangles $\Delta_\pm$ are \emph{isolating} punctures $p_\pm$. Denote by $v_\pm$ the other two vertices of $\Delta_\pm$, and let $b_\pm$ and $d_\pm$ be their boundary and internal arcs respectively, so that $\partial b_\pm = v_\pm$, and $\partial d_\pm = \hc{v_\pm, p_\pm}$, see Figure~\ref{fig:fpm}.
 
Now observe that if $\tri_1,\tri_2$ are related by performing an umbral move $U_{C_\pm}$ on $S'$, then   $\tri^\circ_1,\tri^\circ_2$ are related by the composite $U_\pm = F_{d_\pm} F_{b_\pm}$ of two flips on $S_c^\circ$.
From an object of $\Pt^o_{\hm{p_\pm}}(S_c^\circ)$ one can recover an object of $\Pt^o_{c_+,c_-}(S')$ up to a power of Dehn twist along $c$, and the latter can be realized using flips preserving the isolating cylinder.
% Hence to prove the Lemma it suffices to show that any two objects of $\Pt^o_{\hm{p_\pm}}(S_c^\circ)$ can be connected by a sequence of moves $F_\pm$ as above and flips that preserve both self-folded triangles $t_\pm$. 
Hence the Lemma will follow if we can prove that 
% 
% 
% We define $\Cc_c^\circ(\tri)$ to be the object in $\Pt^o_{\hm{p_\pm}}(S_c^\circ)$ obtained from $\Cc_c(\tri)$ by replacing cylinders $C_\pm$ with self-folded triangles $t_\pm$. Denote by $v_\pm$ the other two vertices of $t_\pm$, and let $b_\pm$ and $d_\pm$ be their boundary and internal arcs respectively, so that $\partial b_\pm = v_\pm$, and $\partial d_\pm = \hc{v_\pm, p_\pm}$. Images under $\Cc_c^\circ$ of ombromanic moves at the cylinders $C_\pm$ in $\tri$ become the compositions $F_\pm = F_{d_\pm} F_{b_\pm}$ of two flips in $\Cc_c^\circ(\tri)$.  Now the desired statement is equivalent to the following: 
any morphism in $\Pt^o_{\hm{p_\pm}}(S_c^\circ)$ can be factored into a sequence of moves $F_\pm$ and flips outside of triangles isolating $p_\pm$. 

Given a pair of objects $\tri, \tri'$ in $\Pt^o_{\hm{p_\pm}}(S_c^\circ)$, let us superimpose arcs $d'_\pm$ with $\tri$. We will assume that the number of intersection points
$$
X_i^j = (d_i \smallsetminus \partial d_i) \cap (d'_j \smallsetminus \partial d'_j), \qquad i,j \in \hc{+,-},
$$
is minimal in the isotopy classes of $d_i$ and $d'_j$, and set
$$
X_i = X_i^+ \sqcup X_i^-,
\qquad
X^j = X_+^j \sqcup X_-^j,
\qquad\text{and}\qquad
X = X_+^+ \sqcup X_+^- \sqcup X_-^+ \sqcup X_-^-.
$$
The proof is by induction on $\hm{X}$. If $\hm{X}=0$, using flips outside of triangles $\Delta_\pm$ we can turn $\tri$ into a triangulation $\tri''$ containing triangles $\Delta''_i$, $i=\pm$, with edges $b_i, e_i^1,e_i^2$, such that $\partial e_i^1 = \partial e_i^2 = \hc{v_i,v'_i}$ and both $e_i^1,e_i^2$ are isotopic to the concatenation of $d_i$ and $d'_i$, see Figure~\ref{fig:fpm}. Note that possibly $v_i=v'_i$, in which case the $e_i^1,e_i^2$ are loops. Then the flip $F_{b_i}$ replaces $b_i$ with $d'_i$, and flip $F_{d_i}$ replaces $d_i$ with $b'_i$. Thus, performing moves $U_\pm$ we change $\tri''$ into a triangulation containing triangles $\Delta'_\pm$. The latter can be turned into $\tri'$ using only flips outside of triangles $\Delta'_\pm$, which proves the base of induction.

\begin{figure}[h]
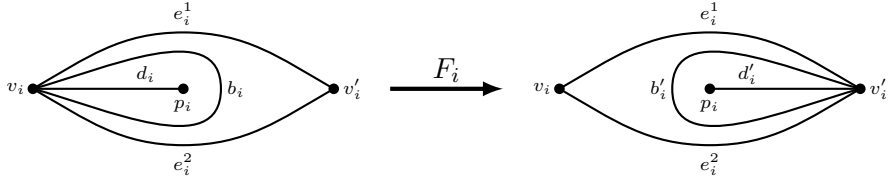

\subfile{fpm.tex}
\caption{Moves $F_i$, $i=\pm$ on $S_c^\circ$ covered by the umbral moves on $S'$.}
\label{fig:fpm}
\end{figure}

Note that a self-folded triangle is determined by its internal edge. In the remainder of the proof we shall discuss transformations of the latter and with the understanding that the corresponding self-folded triangles change accordingly. Denote by $x_i$ the last point in $X^i$ one encounters traversing $d'_i$ from $v'_i$ to $p_i$. First, let $x_i \in X_i^i$ for $i=+$ or $i=-$. Without loss of generality we assume $x_+ \in X_+^+$. Consider (the isotopy class of) the arc $d''_+$ which follows $d_+$ from $v_+$ to $x_+$ and $d'_+$ from $x_+$ to $p_+$. The arc $d''_+$ intersects $d_\pm$ only at the end-points, and by the base of induction we may replace the pair $(d_+,d_-)$ with $(d''_+,d_-)$. Since $\hm{d''_+ \cap (d'_+ \cup d'_-)} \le \hm{X_+ \smallsetminus x_+}$ we may now apply the induction hypothesis. 

It remains to consider the case in which $x_+ \in X_-^+$ and $x_- \in X_+^-$. Without loss of generality we assume that $\hm{X_+} \ge \hm{X_-}$. Consider (the isotopy class of) the arc $d''_+$ which follows $d_-$ from $v_-$ to $x_+$ and $d'_+$ from $x_+$ to $p_+$. As before, $d''_+$ only intersects $d_\pm$ at the end-points, and we can replace $(d_+,d_-)$ with $(d''_+,d_-)$. Observing that $\hm{d''_+ \cap (d'_+ \cup d'_-)} \le \hm{X_- \smallsetminus x_+}$, we can apply the induction hypothesis once again. This completes the proof.
\end{proof}

\section{Quantum cluster varieties and cluster modular groupoids}

\label{sec:qcv}
In this section we review the notion of quantum cluster varieties. Our exposition mostly follows that in~\cite{FG09a, FG09b, GS19}.

\subsection{Seeds, quivers, and quantum tori}
In this paper we only use skew-symmetric quantum cluster algebras, which we incorporate in the definition of a seed.

\begin{defn}
A \emph{seed} is a datum $\Theta=\hr{I, I_*, \La, (\cdot,\cdot)_\La,\hc{e_i}}$ where
\begin{itemize}
\item $I$ is a finite set;
\item $I_* \subset I$ is a subset;
\item $\La$ is a lattice;
\item $(\cdot,\cdot)_\La$ is a skew-symmetric bilinear $\frac{1}{2}\mathbb{Z}$-valued form on $\La$;
\item $\hc{e_i \,|\, i \in I}$ is a basis for the lattice $\La$, with $(e_i,e_j)_\La\in\mathbb{Z}$ whenever $i\in I\setminus I_0$.
\end{itemize}
The sets $I_*$ and $I_m = I\setminus I_0$ are called the \emph{frozen} and \emph{mutable} subsets of $I$ respectively.
\end{defn}

The data of the last bullet point is equivalent to that of an isomorphism $\bs e \colon \Z^{|I|} \simeq \La$. In particular, given a pair of seeds $\Theta,\Theta'$ with the same data $\hr{I,I_*}$, we get a canonical isomorphism of abelian groups (not necessarily isometry of lattices) $\bs{e'} \bs e^{-1} \colon \La \simeq \La'$. In what follows, we will skip the subscript $\La$ in the notation $(\cdot,\cdot)_\La$ when the lattice is clear from the context.

\begin{defn}
 We say that seeds $\Theta$ and $\Theta'$ are \emph{equivalent} if they have the same label sets $(I,I_*)$ and the canonical isomorphism $\La \simeq \La'$ is an isometry, that is $(e_i,e_j)_\La = (e'_i, e'_j)_{\La'}$ for all $i,j \in I$. We define a \emph{quiver} to be an equivalence class of seeds.
\end{defn}

The quiver $Q$ associated to a seed $\Theta$ can be visualized as a directed graph with vertices labelled by the set $I$ and arrows with weights defined by the matrix $\eps = \hr{\eps_{ij}}$, where $\eps_{ij} = (e_i,e_j)$. More precisely, if $\eps_{ij} \in \frac12\Z_{>0}$ we draw $\lfloor\eps_{ij}\rfloor$ solid and $2(\eps_{ij} - \lfloor\eps_{ij}\rfloor)$ dashed arrows from vertex $i$ to vertex $j$. The vertices corresponding to the subset $I_*$ are called \emph{frozen} and are drawn as squares, while those corresponding to $I_m$ are referred to as \emph{mutable} and are drawn as circles.

The pair $\hr{\Lambda,(\cdot,\cdot)_\La}$ determines a \emph{quantum torus algebra} $\Tcx_\Lambda$, which is defined to be the free $\Z[q^{\pm1}]$-module spanned by $\hc{Y_\la \,|\, \la\in\La}$, with the multiplication rule
\begin{align}
\label{eq:Y-mult}
q^{(\la,\mu)} Y_\lambda Y_\mu = Y_{\lambda+\mu}.
\end{align}
Specializing $q=1$ we recover the coordinate ring of the split algebraic torus $\mathbb{T}_\Lambda$ with character lattice $\Lambda$.

A basis $\hc{e_i}$ of the lattice $\La$ gives rise to a distinguished system of generators for $\Tcx_\Lambda$, namely
%$$
%\Tc^q_\La \simeq T^q_Q = \Z[q^{\pm1}]\ha{X_i \,|\, i \in I} / \ha{X_j X_i = q^{2\eps_{ij}} X_i X_j}
%$$
the elements $Y_i=Y_{e_i}$. This way we obtain a \emph{quantum cluster $\Xc$-chart}
\begin{align}
\label{eq:qtor-presentation}
\Tcx_Q = \Z[q^{\pm1}]\ha{Y_i^{\pm1} \,|\, i \in I} / \ha{q^{\eps_{jk}}Y_jY_k = q^{\eps_{kj}}Y_kY_j} \simeq \Tcx_\La.
\end{align}
The generators $Y_i$ are called the \emph{quantum cluster $\Xc$-variables}. We note that this presentation of $\Tcx_Q$ depends only on the quiver and not on the choice of the representative seed.

Given a quantum torus with generators $Y_1,\ldots, Y_d$ one can pass back to the lattice representation by using the \emph{Weyl ordering}. First, define an algebra anti-involution $*$ by
$$
*Y_i = Y_i, \quad *q = q^{-1}.
$$
Then given any exponent vector $(n_1,\ldots, n_d)\in\mathbb{Z}^d$, we define the Weyl ordered product
\begin{align}
\label{eq:weyl-order}
:\!Y_1^{n_1}\cdots Y_d^{n_d}\!:
\end{align}
to be the unique $*$-invariant monomial of the form $q^kY_1^{n_1}\cdots Y_d^{n_d}$.

%The quantum torus $\mathcal{T}_\La^q$ has an ``infinitesimal'' version, which is the Heisenberg-type Lie algebra
%$$
%\Heis_\La=\La_\C \oplus \C c,
%$$
%where we write $\La_\C$ for the vector space $\La\otimes_\Z\C$, with generators $\hc{y_\la \,|\, \la\in\La}$ and $c$, Lie bracket
%$$
%[y_\la,y_\mu]=\frac{(\la,\mu)}{2\pi i}c,
%$$
%and the $*$-structure
%$$
%*y_\la=y_\la, \qquad *c=c.
%$$

\subsection{Compact quantum dilogarithm}
Recall the (compact) quantum dilogarithm function
%\footnote{Various normalizations and reparameterizations of this function are also known under the names of \emph{$q$-exponent}, \emph{$q$-Gamma function}, or \emph{$q$-Pochhammer symbol.}}
\beq
\label{eq:compact-qdl}
\Psi_q(Y) = \prod_{m=0}^\infty \frac{1}{1+q^{2m+1}Y} \in \mathbb{Z}((q))[[Y]].
\eeq
It is related to the infinite $q$-Pochhammer symbol by
$$
\Psi_q(Y)  = (-qY;q^2)^{-1}_\infty.
$$
The Taylor coefficients of $\Psi_q(Y)$ are given by series expansions of rational functions of $q$, so we can regard the compact quantum dilogarithm as an element of the ring $\mathbb{Q}(q)[[Y]]$:
\beq
\label{eq:qpochinv}
\Psi_q(Y) = \sum_{n=0}^\infty \frac{q^nY^n}{(q^2-1)(q^4-1)\cdots (q^{2n}-1)}.%\in \mathbb{Z}[[q]][[Y]]\subset \mathbb{Q}(q)[[Y]].
\eeq
Its reciprocal also has a similar Taylor series:
\beq
\label{eq:qpoch}
\Psi_q(Y)^{-1} = \sum_{n=0}^\infty \frac{q^{n^2}Y^n}{(1-q^2)(1-q^4)\cdots (1-q^{2n})}.%\in \mathbb{Z}[[q]][[Y]]\subset \mathbb{Q}(q)[[Y]].
\eeq
In particular, when thought of as an element of $\mathbb{Q}(q)[[Y]]$, $\Psi_q(Y)$ satisfies the inversion relation
$$
\Psi_{1/q}(Y) = \frac{1}{\Psi_{q}(Y)}.
$$
It is immediate from the infinite product representation~\eqref{eq:compact-qdl} that the compact quantum dilogarithm satisfies the $q$-difference equation
\beq
\label{eq:q-Gamma}
\Psi_q(q^2Y) = (1+qY)\Psi_q(Y).
\eeq
Another important property we will need is the \emph{pentagon identity} for $\Psi_q$, which takes the form of the following identity in the ring of noncommutative formal power series with $\mathbb{Q}(q)$-coefficients:
$$
UV= q^2VU \implies \Psi_q(U)\Psi_q(V) = \Psi_q(V)\Psi_q(qUV)\Psi_q(U).
$$
Note that the middle factor in the triple product can be written using the Weyl ordering~\eqref{eq:weyl-order} as $\Psi_q(:UV:)$.

We also recall the $q$-binomial theorem in the form
$$
\sum_{n\ge0}\frac{(a;q^2)_n}{(q^2;q^2)_n}z^n = \frac{(az;q^2)_{\infty}}{(z;q^2)_{\infty}}\in\mathbb{Q}[[a,z,q]].
$$

\subsection{Quantum cluster varieties}

Let $\Theta,\Theta'$ be seeds representing quivers $Q,Q'$ with the same label sets $(I,I_*)$.
%(possibly with central characters).
We say that the quiver $Q'$ is the \emph{mutation of $Q$ in direction $k\in I_m$,} and write $Q' = \mu_k(Q)$, if the map 
\beq
\label{eq:mon-mut}
\mu_k \colon \Lambda \longra \Lambda', \qquad e_i \longmapsto 
\begin{cases}
-e'_k &\text{if} \; i=k, \\
%e'_i + \max\{(e_i,e_k),0\}e'_k &\text{if} \; i \ne k
e'_i + [\eps_{ik}]_+ e'_k &\text{if} \; i \ne k
\end{cases}
\eeq
is an isometry, where
$$
[a]_+ = \max\hc{a,0}.
$$
%(intertwining the central characters).
%
%Let $\Theta$ be a seed, and $k \in I \setminus I_0$ a mutable vertex of the corresponding quiver $Q$. Then one obtains a new seed, $\mu_k(\Theta)$, called the \emph{mutation of $\Theta$ in direction $k$}, by changing the basis~$\hc{e_i}$ while the rest of the data remains the same. The new basis $\{e_i'\}$ is defined by
%$$
%e'_i = 
%\begin{cases}
%-e_k &\text{if} \; i=k, \\
%e_i + [\eps_{ik}]_+e_k &\text{if} \; i \ne k,
%\end{cases}
%$$
%where $[a]_+=\max(a,0)$. Although the bases $\hc{e_i}$ and $\hc{\mu_k^2(e_i)}$ do not necessarily coincide, the seeds $\Theta$ and $\mu_k^2(\Theta)$ are equivalent. Hence the above formula descends to give a  well-defined notion of the mutation of the quiver in direction $k$. 
The \emph{mutation class} of a quiver $Q$, which we denote by $\bs Q$, is the set of all quivers that can be obtained from $Q$ by some finite sequence of mutations.

\begin{remark}
Consider a seed $\Theta_0$ with lattice $\La_0$. For every other seed $\Theta$ mutation equivalent to $\Theta_0$ we may identify its lattice $\La$ with $\La_0$ via the canonical isomorphism $\bs e_0 \bs e^{-1}$. Then the seed mutation $\mu_k \colon \Theta \to \Theta'$ can be understood as a change of basis in $\Lambda_0$ defined by the rule~\eqref{eq:mon-mut}.
\end{remark}

To each quiver mutation $\mu_k$ we associate an isomorphism of quantum tori
\begin{align}
\label{eq:monpart}
\mu'_k \colon \Tcx_Q \longra \Tcx_{\mu_k(Q)}, \qquad Y_{e_i} \longmapsto Y_{\mu_k(e_i)}.
\end{align}

The \emph{quantum cluster $\Xc$-mutation} is the isomorphism 
\beq
\label{eq:def-Xmutation}
\mu^q_k \colon \Frac(\Tcx_{Q}) \longra  \Frac(\Tcx_{Q'}), \qquad f \longmapsto \Psi_q\hr{Y_k'} \mu'_k(f) \Psi_q\hr{Y_k'}^{-1},
%\mu^q_k =  \Ad_{\Psi_q\hr{Y_k'}}\circ\mu_k'.
\eeq
where $\Frac(\Tcx_Q)$ denotes the skew fraction field of the Ore domain $\Tcx_Q$.
%
%\red{$$
%\mu^q_k = \Ad_{\Psi_q\hr{Y_{-e_k}}},
%$$}
The fact that conjugation by $\Psi_q\hr{Y'_{k}}$ yields a birational automorphism follows form the integrality of the form~$\omega_\La$ and the difference equation~\eqref{eq:q-Gamma}. Equivalently, we can rewrite the factorization above as
\begin{align}
\label{eq:sharp-fact}
\mu^q_k = \mu'_k \circ \mu^\sharp, \quad \mu^\sharp(f) = \Psi(Y_{k}^{-1})f\Psi(Y_{k}^{-1})^{-1}.
\end{align}

The specialization of~\eqref{eq:def-Xmutation} at $q=1$ gives a birational map of tori $\mu_k:{\Tc}_Q\dashrightarrow\Tc_{Q'}$ known as the classical cluster $\Xc$-mutation in direction $k$.

\begin{remark}
As noted above, the requirement that the linear map $\mu_k$ from~\eqref{eq:mon-mut} is an isometry depends only on the quivers $Q,Q'$ and not on the choice of seeds representing them. In particular, we get an equivalent definition if we replace $\mu_k$ with
\beq
\label{eq:neg-mut}
\mu^-_k \colon \Lambda \longra \Lambda', \qquad e_i \longmapsto 
\begin{cases}
-e'_k &\text{if} \; i=k, \\
e'_i + [\eps_{ki}]_+ e'_k &\text{if} \; i \ne k,
%e'_i + \max\{(e_k,e_i),0\}e'_k &\text{if} \; i \ne k,
\end{cases}
\eeq
obtained from~\eqref{eq:mon-mut} by postcomposing it with the isometry of $\Lambda'$ given by $e_i'\mapsto e_i'+(e_k,e_i)e_k'$. In terms of $\mu_k^-$, the quantum cluster transformation $\mu_k^q$ defined by formula~\eqref{eq:def-Xmutation} factors as
\beq
\label{eq:Xmut-negative}
\mu_k^q(f) = \mu_k^-\hr{\Psi_q(Y_k)^{-1}f\Psi_q(Y_k)}.
\eeq
\end{remark}

%The collection of quantum charts $\Tcx_Q$, $Q \in \bs Q$, together with quantum cluster $\Xc$-mutations is often referred to as the \emph{quantum cluster $\Xc$-variety}. We regard the quantum charts as the quantized algebras of functions on the toric charts in the atlas for the classical cluster Poisson variety.
%, isomorphic to $(\C^\times)^{|I|}$, endowed with log-canonical Poisson brackets
%$$
%\hc{Y_j, Y_k} = \eps_{jk} Y_j Y_k.
%$$

%The quantum charts form an $|I_m|$-regular tree, and the cluster mutations quantize the gluing data between adjacent charts. The universally Laurent ring is the quantum analog of the ring of global functions on the cluster Poisson variety. Unless otherwise specified, in what follows we will simply write ``cluster variety'' for quantum cluster $\Xc$-variety --- the same applies to variables, charts, mutations, etc.

%\begin{defn}
%The \emph{universally Laurent ring} $\Lbbx_\Qcl$ is defined as
%$$
%\Lbbx_\Qcl = \bigcap_{Q \in \Qcl} \Tcx_Q \subset \Frac(\Tcx_\La),
%$$
%where for every $Q \in \Qcl$ we identify $\Tcx_Q \simeq \Tcx_\La$ using the isomorphism ~\eqref{eq:qtor-presentation}. Elements of $\Lbbx_\Qcl$ are said to be \emph{universally Laurent.}
%\end{defn}

\begin{defn}
\label{defn:Laurent-ring}
Let $\quiver$ be a quiver and $\Tcx_\quiver$ the corresponding quantum torus. The \emph{universally Laurent ring} $\Lbbx_\quiver$ based at $\quiver$ is the subring in $\Tcx_\quiver$ consisting of elements $X\in\Tcx_\quiver$ such that for any finite-length sequence of mutable directions $(k_1,\ldots, k_l)$, we have
$$
\mu^q_{k_l}\cdots \mu^q_{k_1}(X)\in \Tcx_{Q'}, \qquad\text{where}\qquad Q' = \mu_{k_l}\cdots \mu_{k_1}(\quiver).
$$
Elements of $\Lbbx_\quiver$ are said to be \emph{universally Laurent.}
\end{defn}
%\red{
%\begin{remark}
%We caution the reader that in the presence of \emph{compatible pairs} introduced in Section~\ref{subsec:ensemble}, symbols $\Tc_Q^q$ and $\Lbb^q_\Qcl$ will denote different, but closely related, objects, see Notation~\ref{not:a-tori}.
%\end{remark}
%}

\begin{defn}
\label{defn:perm}
Consider a pair of quivers $Q,Q'$ with possibly different label sets $I,I'$, lattices $\La,\La'$, and bases $\hc{e_i}, \hc{e'_i}$. Let $\sigma \colon I\rightarrow I'$ be a bijection such that $\sigma(I_m)=I'_m$. We say that an isometry 
$$
\varsigma \colon \La \to \La'
$$
is a \emph{quiver permutation} if $\varsigma(e_i) = e'_{\sigma(i)}$ for all $i\in I$. If the equality $\varsigma(e_i) = e'_{\sigma(i)}$ holds only for the mutable directions $i \in I_m$, then $\varsigma$ is called a \emph{quiver quasi-permutation.} An isometry $\mu \colon \La \to \La'$ is a \emph{(quasi-)cluster quiver transformation} if it can be factored into a composition of quiver mutations and quiver (quasi-)permutation.
\end{defn}

\begin{remark}
The notion of a quasi-cluster quiver transformation is independent of the choice of particular seeds representing the two quivers, as is the case with cluster mutations.
\end{remark}

\begin{defn}
Let $Q,Q'$ be a pair of quivers, and $\Tcx_Q, \Tcx_{Q'}$ the corresponding quantum tori. An algebra isomorphism $\Tcx_Q \to \Tcx_{Q'}$ is a \emph{(quasi-)permutation} if it is given by  $Y_{e_j} \to Y_{\varsigma(e_j)}$ for some quiver (quasi-)permutation $\varsigma$. A \emph{(quasi-)cluster transformation}
$$
\mu \colon \Tcx_Q \dashrightarrow \Tcx_{Q'}
$$
is an isomorphism $\Frac(\Tcx_Q) \simeq \Frac(\Tcx_{Q'})$, which can be factored into a composition of quantum cluster mutations and a (quasi-)permutation.
\end{defn}

By Definition~\ref{defn:Laurent-ring}, an element of the non-commutative Laurent polynomial ring $\Tcx_Q$ is universally Laurent if its image under any finite sequence of quantum cluster mutations remains a (non-commutative) Laurent polynomial. Equivalently, it is universal Laurent if it stays a Laurent polynomial under any quasi-cluster transformation. Note that a quasi-cluster transformation $\mu \colon \Tcx_Q \dashrightarrow \Tcx_{Q'}$ defines an isomorphism $\Lbbx_Q \simeq \Lbbx_{Q'}$. In particular, if $Q = Q'$, then $\mu$ is an automorphism of the universally Laurent ring $\Lbbx_Q$.

%For $q \in \C^\times$, we may consider the $\C$-algebra given by the corresponding specialization of $\Lbbx_\quiver \otimes_\Z \C$. We abuse notation and denote this algebra by the same symbol $\Lbbx_\quiver$. Given $b\in\mathbb{R}_{>0}$ consider the specializations
%$$
%q = e^{\pi ib^2}, \qquad \tilde q = e^{\pi ib^{-2}},
%$$
%so that $q,\tilde q$ are complex numbers on the unit circle, and a quantum torus $\Tc_Q^{\tilde\Xc}$ defined by formula~\eqref{eq:qtor-presentation} with $q$ replaced with $\tilde q$ and generators $Y_i$ replaced with $\widetilde Y_i$. Equip the $\C$-vector space
%$$
%\Tc_Q^{\Xc,\tilde\Xc} = \Tcx_Q \otimes_\C \Tc_Q^{\tilde\Xc}
%$$
%with an algebra structure, such that the canonical maps
%\begin{align*}
%&\Tcx_Q \to \Tc_Q^{\Xc,\tilde\Xc}, \qquad f\mapsto f\otimes 1, \\
%&\Tc_Q^{\tilde\Xc} \to \Tc_Q^{\Xc,\tilde\Xc}, \qquad \tilde f \mapsto 1 \otimes \tilde f
%\end{align*}
%are algebra homomorphisms, and the cross relations between generators $Y_\lambda,\widetilde Y_\lambda$ of the two factors are given by
%$$
%\label{eq:annoying-sign}
%e^{2\pi i(\lambda,\mu) } Y_\lambda \widetilde Y_{\mu} = \widetilde Y_{\mu} Y_\lambda.
%$$
%Since $(\la,\mu) \in \frac12\Z$, we see that the two factors commute up to an integer multiple of $-1$.
%
%\begin{defn}
%The \emph{modular double} $\Lbb^{\Xc,\tilde\Xc}_Q$ of the universally Laurent ring is the vector space
%$$
%\Lbb^{\Xc,\tilde\Xc}_Q = \Lbbx_Q \otimes_\C \Lbb^{\tilde\Xc}_Q \subset \Tc_Q^{\Xc,\tilde\Xc}
%$$
%with the algebra structure inherited from that of $\Tc_Q^{\Xc,\tilde\Xc}$.
%\end{defn}
%

\subsection{Tropicalization}
\label{subsec:trop-var}
%It turns out that the equivalence of quasi-cluster transformation can be tested by means of \emph{tropical} cluster variables, which we now briefly review. 
Given a quiver $\quiver$, the \emph{universal semifield} $\Pbb_\mathrm{univ}(\quiver) \subset \Q_+(\bs x)$ consists of all nonzero rational functions of classical cluster variables $\bs x = \hc{x_i, \,|\, i \in I}$ having subtraction-free expressions. We also write $\Pbb_\mathrm{trop}(\quiver)$ for the \emph{tropical semifield} on variables $\bs{\yrm} = \hc{\yrm_i, \,|\, i \in I}$, which is  the abelian multiplicative group freely generated by $\bs\yrm$ endowed with the addition
$$
\prod_{i\in I}{\mathrm y}_i^{a_i}\oplus \prod_{i\in I}\yrm_i^{b_i} = \prod_{i\in I}\yrm_i^{\min(a_i,b_i)}.
$$
We call $\bs\yrm$ the \emph{tropical $\Xc$-variables} associated to quiver $\quiver$. The canonical homomorphism of semifields $
\pi_t \colon \Pbb_{\mathrm{univ}}(\bs x) \to \Pbb_{\mathrm{trop}}(\bs\yrm)$, defined by $\pi_t(x_i) = \yrm_i$ and $\pi_t(a) = 1$ for all $a \in \Q_+$, induces a mutation rule for the tropical $\Xc$-variables, which can be described as follows. Let us say that a Laurent monomial in $\bs\yrm$ is positive (resp.\ negative) if it is not 1 and all its exponents are nonnegative (resp.\ nonpositive). 

\begin{prop}[\cite{DWZ10,Nag10,Pla11}]
Let $\quiver$ be a quiver and $(\bs x,\bs\yrm)$ the corresponding sets of classical and tropical $\Xc$-variables. Then for any quiver $\quiver'$ mutation equivalent to $\quiver$, each Laurent monomial $\yrm_i'=\pi_t(x'_i)$ in $\bs\yrm$ is either positive or negative.
\end{prop}
By this sign-coherence, once we have fixed an initial quiver $\quiver_0$ we get a well-defined \emph{tropical sign} $\sigma = \sigma_{\quiver_0}$ of each tropical cluster variable $\yrm_i$, defined by $\sigma(\yrm_i)=1$ if $\yrm_i$ is a positive Laurent monomial in tropical cluster variables of $\quiver_0$, and $-1$ otherwise. Now if $\quiver$ is a quiver mutation equivalent to $\quiver_0$, and $Q' = \mu_k(Q)$, the mutation rule for tropical $\Xc$-variables can be conveniently expressed in terms of the tropical sign via
$$
\yrm'_i =
\begin{cases}
(\yrm_{k})^{-1}, & i=k \\
\yrm_i(\yrm_k)^{[\sigma(\yrm_k)\eps_{ki}]_+}, & i\neq k.
\end{cases}
$$
%$$
%\upsilon''_i =\begin{cases}
%(\upsilon'_{k})^{-1}, \quad i=k \\
%\upsilon'_i(\upsilon'_k)^{[\sigma(\upsilon'_k)\epsilon'_{ki}]_+}, \quad i\neq k.
%\end{cases}
%$$

In fact, the question of whether two composites of quantum mutations coincide can be decided just based on the corresponding tropical cluster transformations.

\begin{theorem}[\cite{FG09a,Kel11,KN11}]
\label{trop-criterion}
%Let $ c = (i_1,\ldots, i_l), d=(j_1,\ldots, j_m)$ be two sequences of mutable directions.
Write $c^q,d^q$ for quantum cluster transformations obtained as composites of cluster mutations~\eqref{eq:def-Xmutation}, $c$ and $d$ for their $q=1$ specializations, and $c_t,d_t$ for their tropical counterparts. Then
$$
c^q = d^q \iff c = d \iff c_t = d_t.
$$
\end{theorem}
\begin{remark}
\label{rmk:trop-determination}
Thanks to Theorem~\ref{trop-criterion}, it follows that a cluster in the cluster atlas is completely determined by its set of tropical variables.
\end{remark}

%\begin{theorem}[\cite{FG09a,Kel11,KN11}]
%\label{trop-criterion}
%Let $c^q,d^q$ be a pair of quantum quasi-cluster transformations, $c$ and $d$ be their classical versions, and $c_t,d_t$ their tropical counterparts. Then 
%$$
%c^q = d^q \iff c = d \iff c_t = d_t.
%$$
%\end{theorem}
%
%\begin{remark}
%In \emph{loc.\,cit.\,}the statement is proven for cluster, rather than quasi-cluster, transformations, but the proof of the quasi- case works verbatim.
%\end{remark}

The following lemma is useful for computing the effect of sequences of mutations at the level of tropical cluster variables.

\begin{lemma}
\label{tropical-chain}
Let $\quiver$ be a quiver mutation equivalent to a reference quiver $\quiver_0$, $\yrm$ be the tropical cluster variables of $\quiver$, and $i_0,\ldots,,i_{l+1}$ be a sequence of vertices such that 
\begin{enumerate}
\item $\sigma(\yrm_i) = 1$ for all $r \ge 1$, where $\sigma$ is the tropical sign with respect to the quiver $Q_0$;
\item for each $1\leq r \leq l$, the only outward pointing arrows from vertex~$i_r$ in the quiver $\mu_{i_{r-1}}\circ\cdots\circ \mu_{i_1}(\quiver)$ are to vertices $i_{r-1}$ and $i_{r+1}$.  
\end{enumerate}
Then in the quiver $\quiver' = \mu_{i_{l}}\circ\cdots\circ \mu_{i_1}(\quiver)$ we have
$$
\yrm'_{i_0} = \yrm_{i_0}\yrm_{i_1}, \qquad
\yrm'_{i_s} = \yrm_{i_{s+1}}, \qquad
\yrm'_{i_l} = \prod_{r=1}^l \yrm_{i_r}^{-1}, \qquad
\yrm'_{i_{l+1}} = \prod_{r=1}^{l+1} \yrm_{i_r}
$$
for all $1 \le s < l$, while no other tropical variables are affected.
%\begin{align}
% &\upsilon''_{i_r} = \upsilon'_{i_{r+1}}, \quad 1\leq j\leq l,\\
% & \upsilon''_{i_0}\mapsto \upsilon'_{i_0}\upsilon'_{i_1}, \qquad \upsilon''_{i_{l}}\mapsto \left(\prod_{r=1}^{l} \upsilon'_{i_r}\right)^{-1}, \qquad \upsilon_{i_{l+1}}''\mapsto \prod_{r=1}^{l+1} \upsilon'_{i_r}. 
%\end{align}
\end{lemma}
%\begin{remark}
%Lemma~\ref{tropical-chain} also has an obvious ``opposite'' analog describing the effect of mutating at a similar chain of vertices with negative tropical signs.
%\end{remark}

\subsection{Compatible pairs, $\Ac$-variables, and central characters.}
\label{subsec:ensemble}

In this section we recall the enhancement cluster seeds to compatible pairs, which will allow us to work with cluster $\Ac$-variables.
%
%, and central characters, and is important for the discussion of cluster modular groupoids and representations of quantum cluster varieties, see Example~\ref{ex:central-characters}.

Given a lattice $\Lambda$ with skew-form $(\cdot,\cdot)$ and a ring $R\supset \mathbb{Z}$, we define the $R$-module
$$
\La_R = \Lambda\otimes_\Z R,
$$
and endow it with an $R$-valued skew-form obtained from $(\cdot,\cdot)_\La$ by extension of scalars. We abuse notation and denote the skew-form on $\La_R$ by the same symbol $(\cdot,\cdot)$.

Given a seed  $\Theta=\hr{I, I_*, \La, (\cdot,\cdot)_\La,\hc{e_i}}$ be a seed, we write
$$
\Lambda^{\mathrm{mut}} = \bigoplus_{i\in I_m}\mathbb{Z}e_i\subseteq\Lambda.
$$

\begin{defn}
Let  $\Theta=\hr{I, I_*, \La, (\cdot,\cdot)_\La,\hc{e_i}}$ be a seed. An extension of $\Theta$ to a \emph{compatible pair} is the data of a lattice $\Xi$ such that 
$$
\Lambda\subseteq\Xi\subset \Lambda_{\mathbb{Q}},
$$
together with a basis $\hc{\eff_i}_{i\in I}$ for $\Xi$ satisfying
%for $\Xi$ we say that $\Theta$ and $\hc{\eff_i}$ form a \emph{compatible pair} if
\beq
\label{eq:partial-dual}
(e_i,\eff_j) = \delta_{ij} \qquad\text{for all}\qquad i \in I_m, \; j \in I.
\eeq
\end{defn}

A compatible pair determines an integer $I\times I$ \emph{ensemble matrix} $b = (b_{ij})$ by expanding the basis $\hc{e_j}$ into the basis $\hc{\eff_i}$:
\beq
\label{eq:ensemble-mat-def}
e_j = \sum_{i \in I} b_{ij}\eff_i, \qquad j \in I.
\eeq
We also obtain a $\Q$-valued skew-symmetric $I \times I$ matrix $\kappa$, whose entries are the values of the form on the basis $\hc{\eff_i}$:
$$
\kappa_{ij} = (\eff_i,\eff_j), \qquad i,j \in I.
$$

\begin{remark}
\label{rmk:ensemble-def}
Here are some remarks on how the definition above fits in with related ones in the literature.
\begin{enumerate}
\item Condition~\eqref{eq:partial-dual} implies that the form $(\cdot,\cdot)$ defines perfect pairing
$$
\bigoplus_{i\in I^m}\mathbb{Z}\xi_i\otimes \Lambda^{\mathrm{mut}} \rightarrow \mathbb{Z},
$$
and moreover that 
$$
\left(\Lambda^{\mathrm{mut}}_{\mathbb{Q}}\right)^\perp\cap \Xi = \bigoplus_{j\in I^*}\mathbb{Z}\xi_j.
$$
In particular, the existence of a compatible pair implies that $\Lambda^{\mathrm{mut}}$ has trivial intersection with the kernel $Z_\La$ of the form $(\cdot,\cdot)$ on $\Lambda_{\mathbb{Q}}$.
%so that~\eqref{eq:xi-ses}
%becomes
%$$
%\begin{tikzcd}
%0\arrow{r}& \bigoplus_{j\in I^*}\mathbb{Z}\xi_j\arrow{r}&  \Xi \arrow{r}& \bigoplus_{i\in I^m}\mathbb{Z}\xi_i\arrow{r}& 0.
%\end{tikzcd}
%$$
%Hence a seed $\Theta$ can be extended to a compatible pair if and only if the kernel $Z_\La$ of the form $(\cdot,\cdot)$ has trivial intersection with $\Lambda^{\mathrm{mut}}_{\mathbb{Q}}$. 
%If $Z_\La = 0$ and $I_* = \varnothing$, there exists a unique basis $\hc{\eff_i}$ which forms a compatible pair with $\Theta$.

 \item Denote by $\tilde b$ be the $I\times I_m$ submatrix of $b$. Evaluating the functional $(\cdot,\eff_k)$ on both sides of~\eqref{eq:ensemble-mat-def}, we see that
$$
\tilde b^t \kappa = \hs{\mathrm{Id}_{I_m}|0}.
$$
In particular, for any $d \in \Z$ such that the matrix $d\kappa$ has integer entries, $(\tilde b,d\kappa)$ is a compatible pair in the sense of~\cite{BZ05} with all multipliers given by $d_j=d$.
\item Now evaluating the functional $(e_i,\cdot)$ on both sides of~\eqref{eq:ensemble-mat-def}, with $i \in I_m$, we get that
$$
b_{ij}=\eps_{ij} \qquad\text{for all}\qquad i \in I_m, \; j \in I.
$$
Hence the rows of any ensemble matrix labelled by non-frozen directions coincide with the corresponding rows of the $\eps$-matrix.
 % \item Let $M=\Hom(\Lambda,\mathbb{Z})$ be the dual lattice to $\Lambda$, and $e_i^*\in M$ the basis dual to $(e_i)$. The skew-form determines maps $p_1:\Lambda_{uf}\rightarrow M$ and $p_2:\Lambda\rightarrow M/\Lambda_{uf}^\perp\simeq \Hom(\Lambda_{uf},\mathbb{Z})$. The data of a compatible pair gives rise to a map
    % $$
    % p:\Lambda\rightarrow M,\quad p(e_i)=\sum_r p_{ir}e_r^*
    % $$
    % whose restriction to $\Lambda_{uf}$ coincides with $p_1$
\end{enumerate}
\end{remark}

\begin{defn}
Compatible pairs $(\Theta, \hc{\eff_i})$ and $(\Theta', \hc{\eff'_i})$ are \emph{equivalent} if the underlying seeds $\Theta, \Theta'$ are  the induced isometry $\Xi\rightarrow\Xi'$ sends $\xi_i$ to $\xi_i'$. A quiver $Q$ \emph{enhanced to a compatible pair} $(Q,\hc{\eff_i})$ is the class of pairs equivalent to $(\Theta, \hc{\eff_i})$.
\end{defn}

%\red{This is a terrible move, but I don't really know what else to do.
%
%\begin{notation}
%\label{not:a-tori}
%If the quiver $Q$ is enhanced to a compatible pair, we will denote the quantum torus and the universally Laurent algebra by $\Tcxx$ and $\Lbbx$ respectively, and reserve notations $\Tc$ and $\Lbb$ for their analogues defined by $\hc{\eff_i}$.
%\end{notation}
%}

Consider a quiver $Q$ enhanced to a compatible pair $(Q,\hc{\eff_i})$. Given an integer $k$, such that $k\kappa_{ij} \in \Z$ for all $i,j \in I$ (unless otherwise specified, we will always choose the minimal such $k$), we define the quantum torus
$$
\Tca_Q = \Z[q^{\pm1/d}] \langle Y_{\pm\eff_i} \,|\, i \in I \rangle / \langle q^{\kappa_{jk}}Y_{\eff_j}Y_{\eff_k} = q^{\kappa_{kj}}Y_{\eff_k}Y_{\eff_j} \rangle.
$$
As before, the algebra $\Tca_Q$ only depends on the enhanced quiver rather than on a particular seed. Note that the quantum torus $\Tcx_Q$ is a $\Z[q^{\pm 1}]$-subalgebra of $\Tca_Q$.

%$$
%\Tc_A^q = \Z[q^{\pm1/d}] \langle Y_{\pm\eff_i} \,|\, i \in I \rangle / \langle q^{\kappa_{jk}}Y_{\eff_j}Y_{\eff_k} = q^{\kappa_{kj}}Y_{\eff_k}Y_{\eff_j} \rangle.
%$$
%As before, the algebra $\Tc_A^q$ only depends on the equivalence class $(Q,A)$ rather than on a particular representative. Note that the quantum torus $\Tc_\La^q$ is a $\Z[q^{\pm 1}]$-subalgebra of $\Tc_A^q$.

\begin{defn}
The elements $Y_{\eff_j}$ of $\Tca_Q$ are called the \emph{quantum cluster $\Ac$-variables}\footnote{Here we follow the terminology of~\cite{FG06b}. Note however that it is common in cluster literature to refer to our $\Ac$- and $\Xc$-variables as the $\mathrm x$- and $\mathrm y$-variables respectively.} associated to the compatible pair $(Q,\hc{\eff_i})$.
\end{defn}

Recall the notion of two seeds being related by mutation defined in \eqref{eq:mon-mut}. In the case that the seeds $\Theta,\Theta'$ are both enhanced to compatible pairs, we impose the additional requirement that under the extension of the map $\mu_k \colon \Eff \to \Eff'$ from~\eqref{eq:mon-mut} we have
\beq
\label{eq:f-mut}
\mu_k(\eff_i) =
\begin{cases}
%-\eff'_k +\sum_{j\neq k} \max\hc{b_{jk},0}\eff_j'&\text{if} \; i=k, \\
-\eff'_k +\sum_{j\neq k} [b_{jk}]_+ \eff_j'&\text{if} \; i=k, \\
\eff'_i  &\text{if} \; i \ne k.
\end{cases}
\eeq
One easily computes that the latter is equivalent to the following standard relation between the ensemble matrices:
\beq
\label{eq:ensemble-mat-change}
b'_{ij} =
\begin{cases}
-b_{ij}, &\text{if} \;\; k=i \;\;\text{or}\;\; k=j, \\
b_{ij}+\frac{|b_{ik}|b_{kj} + b_{ik}|b_{kj}|}{2}, &\text{otherwise.}
\end{cases}
\eeq
Since the pairing between $\Eff$ and $\Z\ha{e_i \,|\, i \in I_m}$ takes integer values by~\eqref{eq:partial-dual}, the quantum mutation maps~\eqref{eq:def-Xmutation} extend to well-defined isomorphisms
\beq
\label{eq:q-A-iso}
\mu^q_k \colon \Frac(\Tca_Q) \longra  \Frac(\Tca_{\mu_k(Q)}).
\eeq
We can use formula~\eqref{eq:def-Xmutation} to compute the evolution of the quantum $\Ac$-variables under mutation:
\beq
\label{eq:def-a-mutation}
\begin{aligned}
\mu_k^q(Y_{\eff_k}) &= \Psi_q(Y_{e_k'}) Y_{\mu_k(\eff_k)} \Psi_q(Y_{e_k'})^{-1}\\
&=Y_{\mu_k(\eff_k)} + Y_{\mu_k(\eff_k)+e_k'}\\
&=Y_{-\eff_k'+\eff^+_k} + Y_{-\eff_k'+\eff^-_k},
\end{aligned}
\eeq
where
$$
\eff_k^{\pm} = \sum_{j \in I}[\pm b_{jk}]_+\eff'_j.
$$
This matches the standard mutation rule for quantum $\mathcal{A}$-variables given in formula (4.23) of ~\cite{BZ05}. 

%{Similarly, we say that two compatible pairs are related by a quasi-permutation $\varsigma \colon \Xi\rightarrow \Xi'$ if $\varsigma(e_i) = e'_i$ for all $i\in I_m$.}

\begin{remark}
\label{rmk:classical-ensemble}
In the classical $q\rightarrow 1$ limit, the $Y_{\eff_k}$ become coordinates $A_k$ on the cluster chart of a cluster $K_2$-variety, also known as the cluster $\Ac$-variety. Similarly, generators $Y_{e_k}$ become coordinates $X_k$ on the cluster chart of a cluster Poisson variety, also called the cluster $\Xc$-variety. The classical limit $\mu_k^{q=1}$ of the mutation map $\mu_k^q$ is expressed in terms of classical $\Ac$-coordinates as the involutive birational transformation
\beq
\label{eq:A-mut-def}
\mu_k^{q=1}(A'_i) = \begin{cases} A_i \quad & i\neq k\\
A_k^{-1}\left(\prod\limits_{b_{jk}>0}A_j^{b_{jk}} +  \prod\limits_{b_{jk}<0}A_j^{-b_{jk}} \right) \quad & i=k,
\end{cases}
\eeq
and the classical ensemble map reads
$$
X_i = \prod_{r \in I}A_r^{\eps_{ri}}.
$$
\end{remark}

Using~\eqref{eq:q-A-iso}, we can now define the $K_2$-version of the universally Laurent ring, the \emph{quantum upper cluster algebra} $\Lbba_Q \subset \Tca_Q$, which consists of all elements $A \in \Tca_Q$ that stay Laurent polynomials under any finite sequence of quantum cluster mutations. As before, for any pair of mutation equivalent quivers $Q,Q'$ we have a canonical isomorphism $\Lbba_Q \simeq \Lbba_{Q'}$ of their quantum upper cluster algebras. Note that the universally Laurent algebra $\Lbbx_Q$ is contained in $\Lbba_Q$ as a $\Z[q^{\pm1}]$-subring, and we have the equality $\Lbbx_Q = \Lbba_Q$ if and only if $\det(b) = \pm1$. By the quantum Laurent phenomenon, see \cite[Corollary 5.2]{BZ05}, all quantum $\Ac$-variables (from all cluster charts) are elements of the upper cluster algebra. The subalgebra of $\Lbba_Q$ generated by them is called the \emph{quantum cluster algebra.}

Berenstein and Zelevinsky prove the quantum Laurent phenomenon by means of the following useful criterion for when an element of $\Tca_Q$ lies in $\Lbba_Q$. Given a quiver $Q$, let $Q^{[k]}$ denote the one obtained from $Q$ by replacing its frozen subset $I_*$ with $I_*^{[k]} = I \smallsetminus k$. In other words, in $Q^{[k]}$ we freeze all basis vectors but the one labelled $k$. Then the \emph{upper bound} $\Uc_Q^\Ac$ is defined as
$$
\Uc_Q^\Ac = \bigcap_{k \in I_m} \Lbba_{Q^{[k]}} \subset \Tca_Q.
%= \bigcap_{k \in I_m} \big(\Tca_Q \cap \Tca_{\mu_k(Q)} \big)
$$

\begin{theorem}[\cite{BZ05}, Theorem 5.1]
\label{thm:1-step}
The upper bound $\Uc_Q^\Ac$ depends only on the mutation class of the pair $(Q,\hc{\eff_i})$: we have
$$
\Lbba_Q = \Uc_Q^\Ac.
$$
Hence an element $f\in \Tca_Q$ is universally Laurent if and only if it remains Laurent under all 1-step mutations.
\end{theorem}

We record the following standard corollary of this Theorem:

\begin{cor}
\label{cor:1-step}
For any quiver $Q$ (which need not admit an extension to a compatible pair), we have
$$
\Lbbx_Q = \bigcap_{k \in I_m} \Lbbx_{Q^{[k]}}.
$$
\end{cor}

\begin{proof}
Consider the principal coefficient extension $\widehat Q$ of quiver $Q$, obtained by adding a frozen vertex $k_*$ for every vertex $k$ of $Q$ along with an arrow $k_* \to k$. Then the quiver $\widehat Q$ admits a compatible pair such that $(e_i,\xi_j) = \delta_{ij}$ for all $i,j \in I$. In particular it satisfies the assumption of Theorem~\ref{thm:1-step} with $\det(\hat b)=1$, so we have the following chain of equalities:
$$
\Lbbx_{\widehat Q} = \Lbba_{\widehat Q} = \bigcap_{k \in I_m} \Lbba_{\widehat Q^{[k]}} = \bigcap_{k \in I_m} \Lbbx_{\widehat Q^{[k]}}.
$$
Since $\Lbbx_Q$ is a subring of $\Lbbx_{\widehat Q}$ and $\Lbbx_Q \cap \Lbbx_{\widehat Q^{[k]}} = \Lbbx_{Q^{[k]}}$, the Corollary follows.
\end{proof}

\subsection{Coordinate and cluster modular groupoids.}
\label{subsec:cluster-group}

We now recall the \emph{cluster modular groupoid} associated to a cluster variety.

\begin{notation}
Given a quiver $Q$ we denote by $\Qcl$ the class of quivers that may be obtained from $Q$ by quasi-cluster quiver transformations. In particular, quivers from $\Qcl$ are not required to have identical label sets $I$.
\end{notation}

\begin{defn}
The \emph{(enhanced) quasi-cluster modular groupoid} $\Cl_\Qcl$ is the groupoid whose objects are quivers $Q \in \bs Q$ (enhanced to compatible pairs)
%{(possibly with central characters)}
and whose morphisms are quasi-cluster transformations of cluster tori. Composition of morphisms is given by composition of birational transformations. The \emph{quasi-cluster modular group} $\Gamma_\Qcl = \pi_1(\Cl_\Qcl)$, is the automorphism group of an object in $\Cl_\Qcl$.
\end{defn}
In other words, one can think of morphisms in $\Cl_{\Qcl}$ as formal composites of quiver mutations and quiver quasi-permutations, where we identify two formal composites if they induce the same birational isomorphism of tori.

Since all seeds we consider are enhanced with compatible pairs, in what follows omit the word ``enhanced'' (as well as the prefix ``quasi-'') and simply write \emph{cluster modular groupoid} and \emph{cluster modular group.}

%\begin{remark}
%A functor out of the cluster modular groupoid is determined completely by its values on objects and the elementary morphisms given by single cluster mutations and generalized permutations. 
%\end{remark}

\begin{remark}
\begin{enumerate}
\item
Any element of the cluster modular group defines an automorphism of the quantum upper cluster algebra $\Lbba_Q$ and the universally Laurent ring $\Lbbx_Q$.
\item We can also consider the subgroupoid  $\Cl_{\Qcl;I}$ whose objects are quivers with a \emph{fixed} label set $I$, and morphisms are quasi-cluster transformations obtained as composites of mutations and quasi-permutations between quivers with identical label sets. Since the groupoids $\Cl_{\Qcl}$ and $\Cl_{\Qcl;I}$ have the same fundamental groups, the question of which to work with is largely a question of convenience for the application at hand. 
\end{enumerate}
\end{remark}

%\green{
%\begin{notation}
%To lighten notations, we will often drop the superscript $\Ac$ from $\Tca_Q$ and $\Lbba_\Qcl$. We will also replace the class $\Qcl$ with its representative $Q$, or simply omit it whenever it does not cause confusion.
%\end{notation}
%}

The following gives another way to think about the cluster modular groupoid. We formulate it in the context of enhanced quivers, although the same construction works for ordinary ones using the $\Xc$-torus in place of the $\Ac$ one.

\begin{defn}
\label{def:cluster}
Let $\quiver$ be a quiver with label set $I$, and $\Tca_Q$ the corresponding quantum torus with its generators $\hc{A_i \,|\, i\in I}$. A \emph{cluster} in $\Tca_Q$ is a collection of elements
$$
\cluster_\mu = \hc{\mu(A'_i) \,|\, i \in I_{Q'}},
$$
where $\mu \colon \Tca_{Q'} \dashrightarrow \Tca_Q$ is a quasi-cluster transformation. We define the \emph{coordinate groupoid} $\widetilde\Cl_Q$ based at the quiver $Q$ to be the connected, simply connected groupoid whose objects are clusters in $\Tca_Q$. 
\end{defn}

%\begin{defn}
%\label{def:cluster}
%Let $\quiver$ be a quiver with label set $I$, and $\Tc^{\Ac}_\quiver$ the corresponding quantum torus with its generators $\hc{A_i \,|\, i\in I}$. We say that a collection of elements 
%$$
%\cluster = (B_i)_{i\in I},\quad B_i\in \Tc^{\Ac}_\quiver
%$$ 
%forms a \emph{cluster} in $\Tc^{\Ac}_\quiver$ if there exists  a quasi-cluster transformation $\bs\mu \colon \Tc^{\Ac}_{\quiver'}\dashrightarrow \Tc^{\Ac}_\quiver$  such that
%$$
%\cluster = \{\bs\mu(A'_i)\}_{i\in I}.
%$$
%We define the \emph{cluster atlas} $\widetilde{\Cl}_{\quiver}$ based at the quiver $\quiver$ to be the connected, simply connected groupoid with objects given by the set of all clusters in $\Tc^{\Ac}_\quiver$.  
%\end{defn}

As with universally Laurent rings, a quasi-cluster transformation $\mu \colon \Tca_{Q_1} \dashrightarrow \Tca_{Q_2}$ defines an isomorphism of coordinate groupoids $\widetilde\Cl_{Q_1} \simeq \widetilde\Cl_{Q_2}$. We emphasize that this isomorphism is not canonical, and depends on the choice of the transformation $\mu$.

%If the quiver $\quiver_2$ is related to $\quiver_1$ by a cluster transformation $\bs\mu$, then composition with $\bs\mu$ gives a bijection between the cluster atlases for $\quiver_1,\quiver_2$. We emphasize that this bijection is not canonical, and depends on the choice of cluster transformation $\bs\mu$.

For each $Q \in \Qcl$ there is a canonical functor 
$$
\pi \colon \widetilde{\Cl}_Q \longra \Cl_\Qcl
$$
defined as follows. The objects $\mathcal{B}_\mu$ of $\widetilde\Cl_Q$ are labelled by quasi-cluster transformations $\mu \colon \Tca_{Q'} \dashrightarrow \Tca_Q$, and we set $\pi(\cluster_\mu) = Q'$. At the level of morphisms,  $\pi$ sends the unique morphism between objects $\cluster_{\mu_1}$ and $\cluster_{\mu_2}$ to the quasi-cluster transformation $\mu_2\mu_1^{-1} \colon \Tca_{Q_1} \dashrightarrow \Tca_{Q_2}$. This way the composition law in the cluster atlas gets identified with composition of quasi-cluster transformations, making $\pi$ into a functor.

%By definition, objects of $\widetilde\Cl_Q$ are in bijection with cluster transformations with target $Q$. We can also identify the unique morphism between objects $\cluster_{\mu_1}$ and $\cluster_{\mu_2}$ with the quasi-cluster transformation $\mu_2\mu_1^{-1}$. Then the composition law in the cluster atlas gets identified with composition of quasi-cluster transformations. In this way we get a canonical functor
%$$
%\pi \colon \widetilde{\Cl}_Q \longra \Cl_Q
%$$
%defined on objects by sending $\cluster_\mu$ to the source of $\mu$, and on morphisms by setting $\pi([\bs\mu_2;\bs\mu_1])=\bs\mu_2^{-1}\bs\mu_1$. 

\begin{remark}
This approach to the cluster modular groupoid is natural in the case when the upper cluster algebra $\Lbba_Q$ quantizes a ring of functions on some variety $X$. In this case, each object $\cluster$ of the cluster atlas corresponds to a distinct toric coordinate chart $\phi_\cluster \colon U_\cluster \to (\C^\times)^d$ on $X$, and so we can think of the objects of the coordinate groupoid $\widetilde{\Cl}_Q$ as being labelled by these coordinate systems.
\end{remark}

%This point of view is very natural in the case that the classical universally Laurent ring $\Lbb^{\Ac}_\Qcl$ arises geometrically as the algebra $\mathbb{C}[X]$ of functions on some variety $X$. In this case, each object $\cluster$ of the cluster atlas corresponds to a distinct toric coordinate chart $\phi_\cluster :U_\cluster\rightarrow \mathbb{T}$ on $X$, and so we can think of the objects of $\widetilde{\Cl}_\Qcl$ as being labelled by these coordinate systems.
%\begin{remark}
%\begin{enumerate}
%\item

%\item The isomorphism in point (1) is non-canonical: indeed, each pair of clusters in the cluster atlas with the same underlying quiver $\quiver$ determines an algebra \emph{automorphism} of $\Lbbx_{\quiver}$. 
%\end{enumerate}
%\end{remark}

%\red{Another way to think about the cluster modular groupoid is as follows. We can extend the notion of cluster atlas from Definition~\ref{def:cluster} to that of a quasi-cluster atlas by allowing not only composites of quantum mutations of frames $\mathcal{B} = (b_i)_{i\in I}$, but arbitrary quasi-cluster transformations. This way we can construct a larger groupoid $\widetilde{\Cl}_\Qcl$ whose objects are charts in the quasi-cluster atlas for $\Qcl$, and with exactly one morphism between any pair of objects.  The cluster modular groupoid $\Cl_\Qcl$ can then be realized as the groupoid whose objects are quivers in $\Qcl$, and with 
%$$
%\Hom_{\Cl}(Q_1,Q_2) = \{(\cluster_1,\cluster_2) ~\big |~ \cluster_i ~\text{ has underlying quiver } Q_i\}.
%$$ }

%\blue{Discuss construction of functors out of }

%\input{representations}

\section{Cluster coordinates on decorated character varieties}
\label{sec:cluster-coord}

In this section we introduce a non-degenerate compatible pair related to the moduli spaces $\Pc^\diamond_{SL_{n+1},S}$  and $\Pc^\diamond_{PGL_{n+1},S}$.

More precisely, we explain in~\ref{subsec:quivers-from-graphs} how to construct enhanced quivers $Q_\Gamma$ indexed by certain planar bicolored graphs $\Gamma$ on $S$. We show in Proposition~\ref{eq:prop-graph-quiver-mut} that the enhanced quivers associated to any two such graphs are related by a sequence of quasi-cluster transformations, and hence 
our construction gives a well-defined pair of cluster varieties associated to $S$. 
Moreover, in~\ref{subsec:cluster-coord} we associate to each such graph $\Gamma$ a collection of functions $\{A_\ell(\Gamma)\},\{Y_\ell(\Gamma)\}$ on the moduli spaces $\Pc^\diamond_{SL_{n+1},S}$  and $\Pc^\diamond_{PGL_{n+1},S}$ respectively. In Propositions~\ref{prop:A-move} and~\ref{prop:A-shift} we show that the functions associated to graphs related by a quasi-cluster transformation satisfy the same relations as the generators of the classical cluster $\Ac$- and $\Xc$-tori for the corresponding quivers.

Since our class of graphs is preserved by the action of the mapping class group, we  obtain functors
$$
\Qc \colon \widehat\Pt(S) \longra \Cl_{\bs Q}
$$
from the enhanced Ptolemy groupoid $\widehat\Pt(S)$ to the cluster modular groupoid $\Cl_{\bs Q}$ of the compatible pair, where the quiver class $\bs Q = \bs Q(S)$ depends on the marked surface $S$ and the group $G$. The description of these functors, along with some examples, is given in~\ref{subsec:Pt-Cl}.

\subsection{Bicolored graphs}

Let $S$ be a marked surface with a set $\Cc$ of tacked circles, a set $P$ of punctures, and a set $M$ of marked points.

\begin{defn}
A \emph{bicolored graph} $\Gamma$ on $S$ is an embedded graph $\Gamma \subset S$ with black and white vertices in $S \smallsetminus (P \cup \partial S)$ and 1-valent boundary vertices, which avoid all special points of $S$.

We consider such graphs up to ambient isotopies $\gamma$, such that for each time-slice $\gamma_t$ there is a color-preserving isomorphism $\gamma_t(\Gamma) \simeq \Gamma$ of abstract graphs.
\end{defn}

We say that an edge of $\Gamma$ is \emph{white} or \emph{black} if both of its vertices are. We also refer to white and black edges as \emph{monochrome}. Similarly, an edge of $\Gamma$ is \emph{bicolored} if one of its vertices is black while the other one is white, and is \emph{boundary} if it has a boundary vertex. \emph{Faces} of $\Gamma$ are the connected components of $S \smallsetminus \Gamma$. A face $f$ is \emph{proper} if it is homeomorphic to a disk without marked points or punctures, and is incident to a bi-colored edge. A proper face is \emph{boundary} if its intersection with $\partial S$ is non-empty. Any non-boundary proper face is said to be \emph{internal.} We denote the set of proper faces, edges, and vertices of $\Gamma$ by $F(\Gamma)$, $E(\Gamma)$, and $V(\Gamma)$ respectively. In what follows we only work with proper faces, thus we abuse notation and write face instead of proper face unless otherwise specified.

\begin{defn}
Two bicolored graphs are \emph{equivalent} if one can be obtained from the other by applying the following transformations and their inverses:
\begin{itemize}
\item collapse a monochrome edge to a single vertex of the same color;
\item replace a 2-valent vertex and the two edges incident to it by a single edge.
\end{itemize}
\end{defn}

\begin{figure}[h]
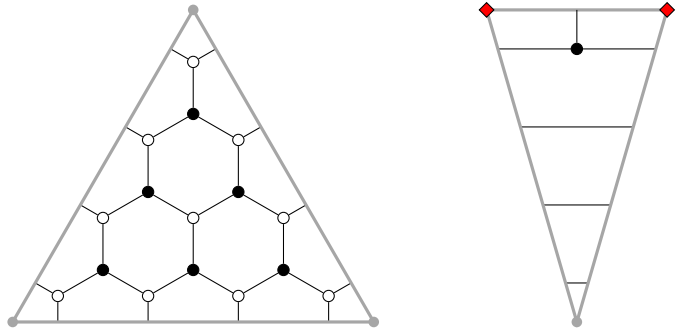

\subfile{A-graph.tex}
\caption{Bi-colored graphs of rank 3 in ideal triangles.}
\label{fig:A-graph}
\end{figure}

A bi-colored graph $\Gamma \subset S$ is \emph{trivalent} if each of its non-boundary vertices is such. Evidently, every equivalence class of bi-colored graphs contains a bipartite one and a trivalent one, and we will often use those in the sequel. We also note that there is a natural bijection between the sets of proper faces of equivalent graphs.

\begin{example}
\label{ex:tri-graph}
Given an ideal triangulation $\tri$ of a marked surface and a positive integer $n$, we define a bicolored graph $\Gamma_{\tri} \subset S$ as follows. We place the graph shown on the left of Figure~\ref{fig:A-graph} inside each non-special triangle,\footnote{In the Figure we have $n=3$. The definition of this graph can be found, for example, in~\cite{Gon17}, however we trust the reader to be able to picture the graph in question for an arbitrary number $n$.} where $n+1$ is the number of boundary vertices on each edge, and the graph on the right inside each special triangle in such a way that the pairs of boundary vertices on every internal edge of $\tri$ match, see Figure~\ref{fig:disk-graph}. The number $n$ is the \emph{rank} of the graph.
\end{example}

\begin{figure}[h]
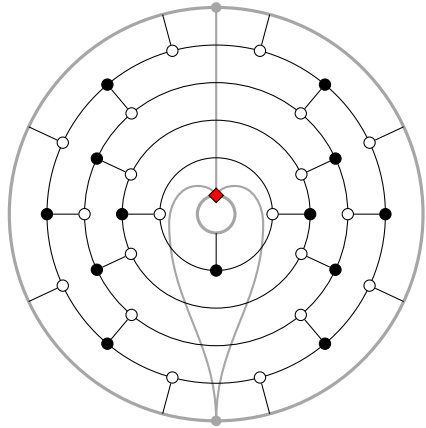

\subfile{disk-graph.tex}
\caption{Graph $\Gamma_{\tri}$ on an annulus, one boundary of which is a tacked circle while the other contains 2 marked points.}
\label{fig:disk-graph}
\end{figure}

Let $\Gamma \subset S$ be a trivalent bicolored graph on a marked surface $S$. Assume further that the graph $\Gamma$ in the vicinity of a tacked circle $c$ can be described as follows. First, there is a cycle $o$ in $\Gamma$ homotopic to $c$, which consists of white edges\footnote{We denote edges of $\Gamma$ by the letter $l$ to distinguish them from basis elements of the lattice we associate to $\Gamma$ in the sequel.} $l_j = (w_j, w_{j+1})$ for $1 \le j < k$, and bicolored edges $l_+ = (b,w_1)$, $l_- = (w_k,b)$, so that $w_{j+1}$ follows $w_j$ as we traverse $o$ in the positive direction. Second, the restriction of $\Gamma$ to the open annulus bounded by $o$ and $c$ consists of the unique boundary edge connecting $b$ to $c$, as shown on Figure~\ref{fig:twist}.

\begin{notation}
\label{not:gamma-not}
Assume that the graph $\Gamma \subset S$ is as described above in the vicinity of every tacked circle $c \in \Cc(S)$. We denote by $\Gamma_\circ$ the bicolored graph obtained from $\Gamma$ by replacing each vertex $b = b(c)$ and the three edges incident to it by a single white edge $l_k = (w_k,w_1)$.
\end{notation}

\begin{defn}
\label{defn:shift}
Assume that the bicolored graph $\Gamma$ satisfies the above conditions. Then the \emph{positive shift} $\sigma_c$ is a transformation of $\Gamma$ consisting of sliding the vertex $b$ past $w_1$ in the positive direction, see Figure~\ref{fig:twist}. The \emph{negative shift} is the transformation $\sigma_c^{-1}$.
\end{defn}

\begin{figure}[h]
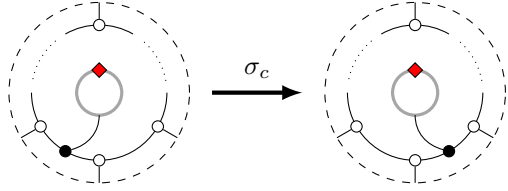

\subfile{twist.tex}
\caption{Positive shift at a tacked circle.}
\label{fig:twist}
\end{figure}

Let $\Gamma = \Gamma_{\tri}$ for some ideal triangulation $\tri$ of $S$, as in Example~\ref{ex:tri-graph}, $c \in \Cc(S)$ be a tacked circle, and $\Delta$ the special triangle of $\tri$ containing $c$. Then the positive shift $\sigma_c$ corresponds to the flip of $\tri$ at the edge of $\Delta$, which follows $c$ as we traverse $\Delta$ in the positive direction. In that case, for the Dehn twist $\tau_c$ we have $\sigma_c^k(\tri) = \tau_c(\tri)$.

Now, let us consider the universal cover $\pi \colon \widetilde S \to S$. Clearly, each marked point or puncture on $S$ lifts to a marked point on $\widetilde S$.
%Each puncture $p \in P(S)$ lifts to a boundary component on $\widetilde S$, which does not contain marked points. We shall always shrink these boundary components to marked points on $\widetilde S$, so that vertices of ideal triangulations of $S$ lift to those on $\widetilde S$.
Each tacked circle $c \in C(S)$ lifts to a \emph{tacked boundary component} of $\widetilde S$. We denote the set of such components by $C(\widetilde S)$, and observe that each of them contains infinitely many tacks. The ideal triangulations and bicolored graphs on $S$ are in bijection with those on $\widetilde S$ which are invariant with respect to the deck group $\pi_1(S)$. From now on we will only consider deck invariant triangulations and graphs on $\widetilde S$. Now, given a bicolored graph $\Gamma$ on $\widetilde S$, and a tacked circle $c \in C(S)$, we define
$$
\sigma_c(\Gamma) = \pi^{-1}\hr{\sigma_c(\pi(\Gamma))}.
$$
%\green{If $\tilde c \in \Cc(\widetilde S)$ is a tacked boundary component, we let $\sigma_{\tilde c}$ denote the same as $\sigma_{\pi(\tilde c)}$.}

\begin{defn}
A non-boundary quadrilateral face $f$ of a bipartite graph $\Gamma \subset S$ is \emph{admissible} if no two faces $\tilde f, \tilde f' \in \pi^{-1}(f)$ share an edge, and none of the vertices of $f$ is connected to a tacked circle by an edge. A face $f$ of a bipartite graph $\Gamma \subset \widetilde S$ is admissible if its projection $\pi(f)$ is.
\end{defn}

\begin{defn}
Let $f$ be an admissible face of a trivalent graph $\Gamma \subset \widetilde S$, and $f = \pi(\tilde f)$ be its projection. A \emph{square move} $\mu_f$ at the face $f$ is the transformation of $\pi(\Gamma)$, which changes the color of every vertex of $f$, see Figure~\ref{fig:face-mut}. As before, we set
$$
\mu_f(\Gamma) = \pi^{-1}\hr{\mu_f(\pi(\Gamma))}.
$$
%\green{and $\mu_{\tilde f} = \mu_f$ for any $\tilde f \in \pi^{-1}(f)$.} \red{?}
\end{defn}

\begin{figure}[h]
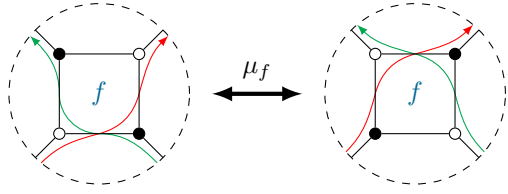

\subfile{face-mut.tex}
\caption{Square move at a face $f \in F(\Gamma)$.}
\label{fig:face-mut}
\end{figure}

\begin{defn}
A bicolored graph $\Gamma \subset S$ is \emph{ideal} if for some ideal triangulation $\tri$ of $S$ it is related to $\Gamma_{\tri}$ by a sequence of square moves and shifts at tacked circles. A graph $\Gamma \subset \widetilde S$ is \emph{ideal} if its projection $\pi(\Gamma) \subset S$ is.
\end{defn}

\begin{remark}
As explained in~\cite{FG06b}, alternatively see Section~\ref{subsec:Pt-Cl} here, for any pair of ideal triangulations $\tri, \tri'$ of $S$, the graphs $\Gamma_{\tri}$ and $\Gamma_{\tri'}$ of the same rank are related to each other by a sequence of square moves and shifts at tacked circles.
\end{remark}

In what follows we will only consider ideal bicolored graphs. Note that any such graph has the local form described before Notation~\ref{not:gamma-not} in the vicinity of each tacked circle. Let $\tri$ be an ideal triangulation of a marked surface $S$. Recall that $S_\circ$ is obtained from $S$ by replacing every tacked circle with a puncture, and let $\tri_\circ$ be the triangulation of $S_\circ$ obtained from $\tri$ by collapsing each special triangle onto one of its non-boundary edges. Then it is easy to see that $\Gamma_\circ = \Gamma_{\tri_\circ}$. For $\Gamma \subset \widetilde S$, we set
$$
\Gamma_\circ = \pi^{-1}\hr{\pi(\Gamma)_\circ}.
$$

\subsection{Zig-zag paths}

\begin{defn}
%Let $\Gamma \subset S$ be a bicolored graph on a marked surface $S$, and $\widetilde\Gamma \subset \widetilde S$ be its lift to the universal cover $\widetilde S$ of $S$.
A \emph{zig-zag path}, or a \emph{zig-zag} for short, on a bicolored graph $\Gamma \subset \widetilde S$ is a path that follows edges of $\Gamma$ turning maximally left at each black vertex, and maximally right at each white one.

\end{defn}

When drawing zig-zags we will always push them away from the vertices of $\Gamma$ so that they turn clock-wise around black vertices, counter-clockwise around white ones, and cross every bi-colored edge they follow. We will often draw projections of zig-zags onto $\pi(\Gamma)$ rather than the zig-zags themselves. Abusing notation we still refer to these projections as \emph{zig-zags,} see Figure~\ref{fig:disk-web}

Denote by $\Wc_\circ$ the union of all zig-zags of the graph $\Gamma_\circ$ on the universal cover $\widetilde S_\circ$ of $\widetilde S$.  A connected component $R$ of $\widetilde S_\circ \smallsetminus \Wc_\circ$ is called a \emph{$k$-gon} if it is a disk which contains at most one marked point, and whose boundary has non-trivial intersection with exactly $k$ zig-zags. We call each such intersection a \emph{side} of $R$ and let it inherit the orientation of the zig-zag it belongs to. Then the following properties hold for any ideal bicolored graph $\Gamma \subset \widetilde S$ of rank $n$, see~\cite{Gon17}.

\begin{enumerate}
\item Each connected component of $\widetilde S_\circ \smallsetminus \Wc_\circ$ is a polygon. No zig-zags on $\widetilde S_\circ$ self-intersect or form non-oriented bigons.
\item For any zig-zag $z$ the surface $\widetilde S_\circ \smallsetminus z$ has two connected components, one of which has exactly one marked point $m$. We will say that $z$ \emph{goes around} $m$.
%, and $z$ \emph{goes around} $m$, where $z = \pi(\tilde z)$, $m = \pi(\tilde m)$, and $\pi \colon \widetilde S_\circ \to S_\circ$ is the natural projection.
\item For each marked point $m$ on $\widetilde S_\circ$ there exist exactly $n+1$ non-intersecting zig-zags $z_0(m), \dots, z_n(m)$ that go around it. We assume the notation that the index of the zig-zags increases as we move away from $m$. Moreover, any zig-zag $z$ on $\widetilde S_\circ$ is of the form $z = z_i(m)$ for some marked point $m$ and some $0 \le i \le n$. %Indeed, if $\Gamma_\circ = \Gamma_{\tri_\circ}$ for some triangulation $\tri_\circ$ of $S_\circ$, it is easy to see that both zig-zags containing an edge of $\widetilde \Gamma_\circ$ are of the form $z_j(\tilde m)$ for some $j$ and $\tilde m$. It remains to check that this property holds under square moves, which is straightforward.
\end{enumerate}

%Denote by $z_j(m)$ the zig-zag on $\Gamma$ obtained as a projection of $z_j(\tilde m)$, where $\tilde m \in \pi^{-1}(m)$. These are well-defined since $\pi(z_j(\tilde m_1)) = \pi(z_j(\tilde m_2))$ if and only if $\pi(\tilde m_1) = \pi(\tilde m_2)$.

\begin{remark}
Let us make the following observations.
\begin{itemize}
\item Projections of zig-zags may intersect and self-intersect. For example, this happens when $S$ is a punctured torus.
\item The zig-zags $z_0(m)$ do not intersect the rest of $\Wc_\circ$. Erasing these zig-zags we recover an \emph{ideal $A_n$-web} in the sense of~\cite{Gon17}, see the right pane of Figure~\ref{fig:disk-web}.
%\item The zig-zags on $\Gamma$ are the same as on $\Gamma_\circ$, except that some of the zig-zags $z_n(m)$ may be replaced by a pair of zig-zags, one of which starts at a marked annulus and the other ends at it.
\end{itemize}
\end{remark}

%Given a reduced bi-colored graph $\Gamma \subset S$, we may consider its dual graph $\Gamma^* \subset S$. For each proper face $f$ of $\Gamma$, we draw a vertex $v_f$ of $\Gamma^*$ inside $f$, such that $v_f \in \partial S$ if $f \in F_*(\Gamma)$. For each pair of distinct faces $f_1,f_2 \in F(\Gamma)$ which share an edge $e$, we draw an edge of $e^* \in E(\Gamma^*)$ connecting the vertices $v_{f_1}, v_{f_2} \in V(\Gamma^*)$, such that $e^*$ is positioned entirely within $f_1 \cup f_2$ and only crosses $e$ once. In case $f_1,f_2 \in F_*(\Gamma)$, the edge $e^*$ goes along the boundary of $S$, and is said to be a \emph{boundary edge.}

\begin{figure}[h]
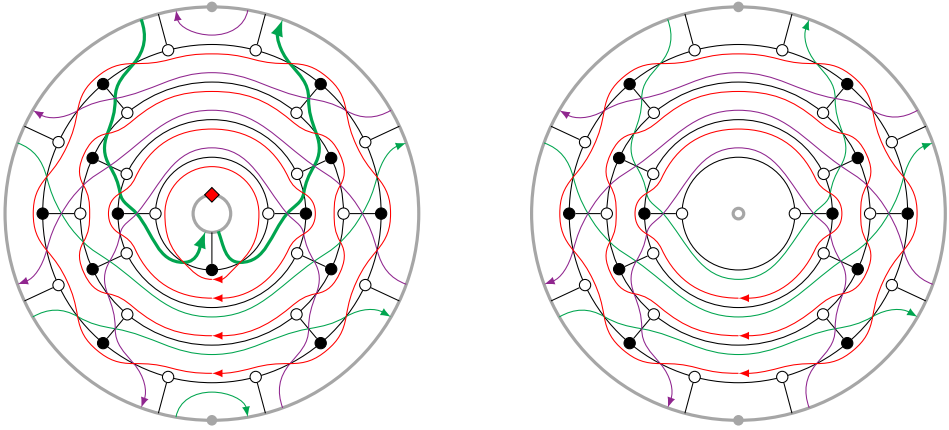

\subfile{disk-zigzag.tex}
\caption{Left pane: zig-zags on an ideal bicolored graph $\Gamma_{\tri}$ of rank 3. Right pane: the corresponding ideal $A_3$-web.}
\label{fig:disk-web}
\end{figure}

The zig-zags on a bicolored graph $\Gamma \subset \widetilde S$ are almost identical to the zig-zags on $\Gamma_\circ \subset \widetilde S_\circ$. Indeed, the only ones that can possibly differ are those that contain edges present in $\Gamma_\circ$, but not in $\Gamma$. If $l = (w_-,w_+)$ is such an edge, it is replaced in $\Gamma$ by a trivalent black vertex~$b$, edges $l_\pm = (b,w_\pm)$, and a boundary edge $l_0 = (b,v)$ with $v$ lying on a tacked boundary component $c \in \Cc(\widetilde S)$. The two zig-zags in $\Gamma_\circ$ containing $l$ are $z_0(m_{c})$ and $z_n(m)$, where $m_{c}$ is the marked point on $\widetilde S_\circ$ corresponding to $c$ on $\widetilde S$. Replacing edges $\gamma(l) \subset z_0(m_{c})$, $\gamma \in \pi_1(S)$, with pairs $\gamma(l_-), \gamma(l_+)$ we obtain a zig-zag on $\Gamma$, which we denote by $z_0(c)$. Similarly, for all $1 \le j < n$ we write $z_j(c)$ for the zig-zags $z_j(m_{c})$ considered as zig-zags on $\Gamma$. Now, let $t$ be the tack preceding $v$ as we traverse $c$ in the positive direction. The zig-zag $z_n(m)$ is split into a family of zig-zags on $\Gamma$ including the two denoted $\check{z}_\pm(t)$, which contain edges $l_\pm, l_0$. Note that $\check z_-(t)$ terminates at the vertex $v = v(t)$, while $\check{z}_+(t)$ originates at it. Let $l_1, \dots, l_k$ be all the edges of $z_n(m)$, which are not present in $\Gamma$, and $t_1, \dots, t_k$ be the corresponding tacks on $\widetilde S$. Assuming that $l_j$ precedes $l_{j+1}$ for all $1 \le j < k$ as one follows $z_n(m)$, we have $\check z_+(t_j) = \check z_-(t_{j+1})$. We will abuse notation and denote by $z_n(m)$ (or $z_n(c')$ if $m = m_{c'}$ for some tacked boundary component $c' \in \Cc(S)$) the union of zig-zags $\check{z}_-(t_1) \cup \check{z}_+(t_1) \cup \dots \cup \check z_+(t_k)$ on $\Gamma$ arising from that zig-zag on $\Gamma_\circ$. We will also abbreviate notation $\check z_+(t)$ to $\check z(t)$. The thick paths in the left pane of Figure~\ref{fig:disk-web} are (the projections of) the two parts of the zig-zag $z_3(m)$, where $m$ is (a preimage of) the bottom marked point.

\begin{lemma}
\label{lem:frame}
Let $t$ be a tack on the boundary component $c$ of $\widetilde S$, and $\Gamma \subset \widetilde S$ be a bipartite graph. Then
\begin{enumerate}
\item the zig-zag $\check z(t)$ shares a single black vertex $b_j(t)$, a single white vertex $w_j(t)$, and a single edge $l_j(t) = (b_j,w_j)$ with each of the zig-zags $z_j(c)$;
\item at the intersection point of $\check z(t)$ and $z_j(c)$ over the edge $l_j$, the respective tangent vectors to $z_j(c),\check z(t)$ form a positively-oriented frame.
\end{enumerate}
\end{lemma}

\begin{proof}
%On the surface $\widetilde S_\circ$, the tacked boundary component $c$ is replaced by a single marked point $m$. The trivalent graph $\Gamma_\circ$ contains an infinite chain of white edges, such that the whole of $\Gamma_\circ$ lies on one side of that chain, and $m$ lies on the other. The statement of the Lemma is equivalent to the following: any zig-zag $\Gamma_\circ$, which contains an edge from that chain, shares a unique black vertex, a unique white vertex, and the bicolored edge connecting the two with each of the zig-zag $z_j(m)$. The latter holds for $\Gamma_\circ=\Gamma_{\tri_\circ}$. As can be seen from Figure~\ref{fig:face-mut}, the desired property is invariant under square moves, which completes the proof.
Both statements of the Lemma hold for $\Gamma=\Gamma_{\tri}$ and are invariant under shifts $\sigma_c$ and square moves, as can be seen from Figures~\ref{fig:twist} and~\ref{fig:face-mut}.
\end{proof}

%By the same argument involving the preservation of the property in question under square moves, we get
%\begin{lemma}
%\label{lem:frame}
%Let $t$ be a tack on the boundary component $c$ of $\widetilde S$, and  $\check z(t)$ the corresponding zig-zag. Then at the intersection point of $\check z(t)$ and $z_{j-1}(c)$ over the edge $\eps_j$, the respective tangent vectors to $z_{j-1}(c),\check z(t)$ form a positively-oriented frame. 
%\end{lemma}
Here is a useful consequence of Lemma~\ref{lem:frame}:
\begin{lemma}
\label{lem:nopair}
Suppose that $t,t'$ are two tacks on boundary components $c,c'$ of $\widetilde S$, and let $l$ be an edge of a bicolored graph $\widetilde\Gamma \subset \widetilde S$ such that $l=l_i(t)=l_j(t')$. Then $t=t'$ and $i=j$.
\end{lemma}

\begin{proof}
The hypothesis of the Lemma implies $l = \check z(t) \cap z_i(c) = \check z(t') \cap z_j(c')$. Since there are exactly two zig-zags over each edge of $\widetilde \Gamma$, we conclude that the set $\{z_i(c),z_j(c'),\check z(t),\check z(t')\}$ contains only two distinct elements. Now suppose that $c\neq c'$. Then we cannot have $z_i(c)=z_j(c')$ or $\check z(t)=\check z(t')$, since elements of the former pair encircle different boundary components, and those of the latter two originate on different boundary components. Hence the only possibility is that $\check z(t')=z_i(c)$ and $\check z(t)=z_j(c')$. But this too is ruled out by Lemma~\ref{lem:frame}, which says that pairs $(z_i(c),\check z(t)) $ and $(\check z(t'),z_j(c'))$ have respectively positively- and negatively-oriented frames. Therefore we must have $c=c'$, which implies that $i=j$ since the zig-zags encircling the same boundary component do not intersect one another, and the Lemma follows.
\end{proof}

We now introduce a few more notations that will be frequently used later on. Given a bicolored graph $\Gamma \subset \widetilde S$ and a tacked boundary component $\tilde c \in \Cc(\widetilde S)$, let $Z_i(\tilde c)$ be the collection of faces of $\Gamma$, which lie between the zig-zags $z_{i-1}(\tilde c)$ and $z_i(\tilde c)$. Whenever it does not cause confusion, we denote the region of $S$ covered by the union of all faces in $Z_i(\tilde c)$ by the same symbol. Choosing a tack $\tilde t \in \tilde c$ we define $\tilde f_i^-(\tilde t) \in Z_i(\tilde c)$ to be the face incident to the edge $l_i(\tilde t)$, and $\tilde f_i^+(\tilde t) \in Z_i(\tilde c)$ to be the face incident to the edge $l_{i-1}(\tilde t)$. We then set $f_i^\pm(c) = \pi(\tilde f_i^\pm(\tilde t))$, where $c = \pi(\tilde c)$, and observe that these faces of $\pi(\Gamma)$ depend only on the tacked circle $c \in \Cc(S)$ and not on a particular lift $\tilde t$ of its tack. Finally, we define $Z_i^\pm(\tilde t)$ to be the subset of those faces of $Z_i(\tilde c)$ which lie to the right of the zig-zag $\check z_\pm(\tilde t)$, so that $\tilde f_i^+(\tilde t) \in F_i^+(\tilde t)$ and $\tilde f_i^\pm(\tilde t) \in F_i^-(\tilde t)$. In Figure~\ref{fig:notch-faces} we shade the region $Z_i^+(\tilde t)$ in green.

%Then it is easy to see that the absence of non-oriented bigons implies the inequality $\hm{Z_i(m)} \le \hm{Z_{i+1}(m)}$ for all possible $i$ and $m$. Consider a zig-zag $\check z(t)$, where $t$ is a tack on the boundary component $\tilde c = \pi^{-1}(c)$ of $\widetilde S$. Denote by $F_i = F_i(\tilde c)$ the union of faces of $\widetilde\Gamma$ contained between $z_{i-1}(\tilde c)$ and $z_i(\tilde c)$, and let $
%F_i^+ = F_i^+(t)$ be the subsets of $F_i$ given by the union of those faces lying respectively to the right of $\check z(t)$. In Figure~\ref{fig:notch-faces} we shade the region $F_i^+$ in green. We then define $\tilde f_i^- \in F_i^-$ to be the face incident to the edge $\epsilon_i(t) = \check z(t) \cap z_i(\tilde c)$, and $\tilde f_i^+ \in F_i^+$ to be the face incident to the edge $\epsilon_{i-1}(t)$. Finally, we set $f_i^\pm(c) = \pi(\tilde f_i^\pm)$. Note that the definition of $f_i^\pm = f_i^\pm(c)$ does not depend on the particular lift $t$ of the tack at $c$.

\begin{figure}[h]
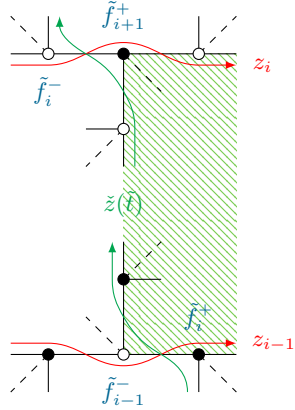

\subfile{fig-notch-faces}
\caption{Faces in $Z_i^+(\tilde t)$ for a bipartite graph. Dashed edges indicate a possibly non-zero number of edges of $\widetilde\Gamma$.}
\label{fig:notch-faces}
\end{figure}

\subsection{Quivers from ideal bicolored graphs}
\label{subsec:quivers-from-graphs}

%We shall now construct a morphism $\phi \colon \Gc_\Gamma \to \Gc_Q$ from the groupoid of bicolored graphs to the \Sasha{another shitty name}quasi-cluster modular groupoid. Let us first describe it on the level of objects.
To each ideal bipartite graph $\Gamma \subset S$ of rank $n$ we now associate a quiver $Q_\Gamma$ with seed
\beq
\label{eq:seed-gamma}
\Theta_\Gamma = \hr{I_\Gamma, I^*_\Gamma, \La_\Gamma, (\cdot,\cdot)_\Gamma, \hc{e_\ell}}.
\eeq
The first two ingredients are
$$
I_\Gamma = I_{\Gamma_\circ} \sqcup I_{\Cc(S)}
\qquad\text{and}\qquad
I^*_\Gamma = I^*_{\Gamma_\circ} \sqcup I_{\Cc(S)},
$$
where
$I_{\Gamma_\circ} = F(\Gamma)$, $I^*_{\Gamma_\circ}$ is the set of proper boundary faces of $\Gamma$, and
$$
I_{\Cc(S)} = \Cc(S) \times \hc{1, \dots, n}.
$$
\begin{remark}
Slightly more formally, we really think of elements of $I_{\Gamma_\circ}$ as pairs $(\Gamma,f)$ consisting of an ideal graph $\Gamma$ on $S$ (considered as usual up to isotopy), together with a 1-cycle $\partial f$ on $\Gamma$. In particular, given a mapping class $\tau:S\rightarrow S$ the label sets $I_\Gamma$ and $I_{\tau(\Gamma)}$ will generally be \emph{different}, in which case the two graphs will correspond to distinct objects in the cluster modular groupoid.
\end{remark}

We will define the form $(\cdot,\cdot)_\Gamma$ by specifying its values of pairs of elements of the basis $\hc{e_\ell}$. To that end let us introduce an auxilliary basis $\hc{\dot e_\ell}$ with the same label set $(I_\Gamma, I_\Gamma^*)$. The graph $\Gamma$ inherits the structure of a fat graph from the orientation of $S$, that is each vertex of $\Gamma$ is endowed with a cyclic orientation of edges incident to it. This in turn allows one to identify the lattice $\La_{\Gamma_\circ}$ with the integral first homology of the conjugate surface for $\Gamma_{\circ}$, and take the restriction $(\cdot,\cdot)_{\Gamma_\circ}$ of the form $(\cdot,\cdot)_\Gamma$ onto $\La_{\Gamma_\circ}$ to be the intersection pairing, see~\cite{GK13}. Choosing $\dot e_f$ to be the cycle $\partial f$ in $\La_{\Gamma_\circ}$ for every $f \in I_{\Gamma_\circ}$, we can depict the corresponding quiver $\dot Q_{\Gamma_\circ}$ as follows. Place a vertex $v_f$ inside each proper face $f$ of $\Gamma_\circ$. Then, for each pair of faces $f_1,f_2$ sharing a bicolored or a boundary edge $l$ draw respectively a solid or a dashed arrow $a$ of $\dot Q_{\Gamma_\circ}$ between $v_{f_1}$ and $v_{f_2}$ so that the white vertex of $l$ stays on the right as we follow $a$, see Figure~\ref{fig:net-quiver}. Thus, the pairing $(\cdot,\cdot)_{\Gamma_\circ}$ admits the following combinatorial description:
\beq
\label{eq:form-0}
(\dot e_f, \dot e_g)_{\Gamma_\circ} = s_{f,g} + \frac12 d_{f,g},
\eeq
where $s_{f,g}$ is the number of solid arrows pointing from $v_f$ to $v_g$ minus the number of those pointing from $v_g$ to $v_f$, and $b_{f,g}$ is defined similarly, but with solid arrows replaced by the dashed ones. We remark that the quiver $\dot Q_{\Gamma_\circ}$ coincides with those of a cluster chart on the moduli spaces $\Pc_{PGL_n,S}$ and $\Ac_{SL_n,S}$, see~\cite{FG06b, GS19}.

\begin{figure}[h]
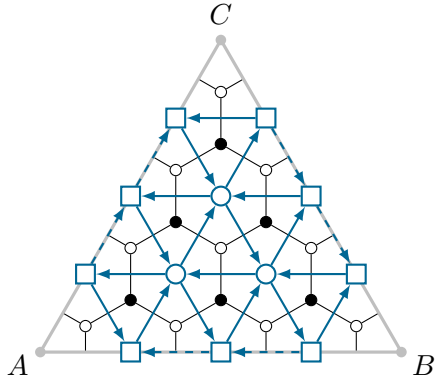

\subfile{fig-net-quiver.tex}
\caption{Quiver $\dot Q_{\Gamma}$ from the bi-colored graph on a triangle.}
\label{fig:net-quiver}
\end{figure}

Let us now describe the form $(\cdot,\cdot)_{\Gamma_\circ}$ in terms of the graph $\widetilde\Gamma = \pi^{-1}(\Gamma)$ on the universal cover $\widetilde S$ of $S$. The graph $\widetilde\Gamma$ gives rise to a pair of abelian groups
$$
\Lambda_{\widetilde\Gamma} = \bigoplus_{\tilde f \in F(\widetilde\Gamma)}\mathbb{Z} \langle e_{\tilde f} \rangle
\qquad\text{and}\qquad
\widehat\Lambda_{\widetilde\Gamma} = \prod_{\tilde f \in F(\widetilde\Gamma)}\mathbb{Z} \langle e_{\tilde f} \rangle
$$
together with a bilinear pairing
$$
(\cdot,\cdot)_{\widetilde\Gamma} \colon \widehat\Lambda_{\widetilde\Gamma} \times \Lambda_{\widetilde\Gamma} \to \frac12\Z,
$$
given by the same formula~\eqref{eq:form-0}. Since each face $\tilde f \in F(\widetilde\Gamma)$ has only finitely many bicolored or boundary edges, the pairing $(\cdot,\cdot)_{\widetilde\Gamma}$ is well-defined. Note that the pairing $(\cdot,\cdot)_{\widetilde\Gamma}$ is $\pi_1(S)$-invariant: for all $\tilde f, \tilde g \in F(\widetilde\Gamma)$ and $\gamma \in \pi_1(S)$ we have
$$
\big(\dot e_{\tilde f},\dot e_{\tilde g}\big)_{\widetilde\Gamma} = \big(\dot e_{\gamma(\tilde f)}, \dot e_{\gamma(\tilde g)}\big)_{\widetilde\Gamma}.
$$
Moreover, for any $f,g \in F(\Gamma)$ and any choice of lifts $\tilde f \in \pi^{-1}(f)$, $\tilde g \in \pi^{-1}(g)$ we have
\beq
\label{form-faces}
(\dot e_f, \dot e_g)_\Gamma = \sum_{\tilde h \in \pi^{-1}(g)} (\dot e_{\tilde f}, \dot e_{\tilde h})_{\widetilde\Gamma} = \sum_{\tilde h \in \pi^{-1}(f)} (\dot e_{\tilde h}, \dot e_{\tilde g})_{\widetilde\Gamma}.
\eeq

% In order to define the values of the from $(\dot e_{c,i}, \dot e_\ell)_\Gamma$, we need to do some preparatory work.
%Before doing so, we need the following preparatory work.
%Consider a zig-zag $\check z(t)$, where $t$ is a tack on the boundary component $\tilde c = \pi^{-1}(c)$ of $\widetilde S$. Denote by $F_i = F_i(\tilde c)$ the union of faces of $\widetilde\Gamma$ contained between $z_{i-1}(\tilde c)$ and $z_i(\tilde c)$, and let $
%F_i^+ = F_i^+(t)$ be the subsets of $F_i$ given by the union of those faces lying respectively to the right of $\check z(t)$. In Figure~\ref{fig:notch-faces} we shade the region $F_i^+$ in green. We then define $\tilde f_i^- \in F_i^-$ to be the face incident to the edge $\epsilon_i(t) = \check z(t) \cap z_i(\tilde c)$, and $\tilde f_i^+ \in F_i^+$ to be the face incident to the edge $\epsilon_{i-1}(t)$. Finally, we set $f_i^\pm(c) = \pi(\tilde f_i^\pm)$. Note that the definition of $f_i^\pm = f_i^\pm(c)$ does not depend on the particular lift $t$ of the tack at $c$.
%
%\begin{figure}[h]
%\subfile{fig-notch-faces}
%\caption{A zig-zag $\check z$. Dashed edges indicate a possibly non-zero number of edges of $\widetilde\Gamma$.}
%\label{fig:notch-faces}
%\end{figure}

%\begin{figure}[h]
%\subfile{frozen.tex}
%\caption{A notch of an ideal bicolored graph: $\epsilon_i=\sgn(s_{2i-1}) = 1$ on the left and $\epsilon_i=-\sgn(s_{2i-1}) = 1$ on the right.}
%\label{fig:frozen}
%\end{figure}

We now specify the value of the form $(\cdot,\cdot)_\Gamma$ on a pair of vectors in $\hc{\dot e_\ell}$ in the situation when one or both of them is a frozen vector $\dot e_{c,i}$ associated to a tacked circle $c\in C(S)$. For each lift $\tilde t$ of the tack on $c$ and $1 \le i \le n$, define the vector
\beq
\label{eq:inf-vector}
\dot e_{\tilde t,i} = \sum_{\tilde f \in Z_i^+(\tilde t)} \tilde f \in \widehat\Lambda_{\widetilde\Gamma}.
\eeq
The element $\dot e_{\tilde t,i}$ defines a linear functional
$$
(\dot e_{\tilde t,i}, \cdot)_{\widetilde\Gamma}\colon \Lambda_{\widetilde\Gamma}\rightarrow \frac{1}{2}\mathbb{Z},
$$
and we have $(\dot e_{\tilde t,i},\dot e_{\tilde f})_{\widetilde\Gamma}=0$ unless $\tilde f \in \hc{\tilde f_{i-1}^-(\tilde t),\tilde f_i^\pm(\tilde t),\tilde f_{i+1}^+(\tilde t)}$. Now for any $f \in F(\Gamma)$ we declare
\beq
\label{form-sea}
\hr{\dot e_{c,i}, \dot e_f}_\Gamma = - \hr{\dot e_f, \dot e_{c,i}}_\Gamma = \sum_{\tilde f \in \pi^{-1}(f)} (\dot e_{\tilde t,i}, \dot e_{\tilde f})_{\widetilde\Gamma}.
\eeq
Formula~\eqref{form-sea} can be expressed more explicitly as:
\beq
\label{form-face-notch}
\hr{\dot e_{c,i}, \dot e_f}_\Gamma = \delta_{f,f_{i+1}^+(c)} - \delta_{f,f_i^+(c)} - \delta_{f,f_-^-(c)} + \delta_{f,f_{i-1}^-(c)},
\eeq
or equivalently
$$
\hr{\dot e_{c,i}, \dot e_f}_\Gamma =
\begin{cases}
1 & \text{if} \;\; f \in \hc{f_{i+1}^+, f_{i-1}^-}, \\
-1 & \text{if} \;\; f \in\hc{f_i^+,f_i^-} \;\;\text{and}\;\; f_i^+ \ne f_i^-,\\
-2 & \text{if} \;\; f = f_i^+ = f_i^-,\\
0 & \text{otherwise},
\end{cases}
$$
where we write $f_j^\pm$ for $f_j^\pm(c)$.
% \blue{Do I get it right that the following red text is an expanded version of the second paragraph in the proof of Prop. 4.10? Personally, I was hoping that what's written there would suffice, but am happy to incorporate more details from the text below.}
%Hence the values of the pairing between $\dot e_{c,i}$ and all proper faces $\dot e_l$ only depends on the set of edges over which a tacking zig-zag $\check z(t)$ at $c$ (shown in green in Figure~\ref{fig:notch-faces}) intersects zig-zags $z_i(c),z_{i-1}(c)$ (shown there in red). 
For a tacked circle $c$ and a pair of integers $1 \le i,j \le n$, we declare
\beq
\label{al-al}
(\dot e_{c,i}, \dot e_{c,j})_\Gamma = \pm\frac12 \delta_{i\pm1,j}.
\eeq
Finally, for distinct tacked circles $c,c'\in C(S)$ and lifts $\tilde t, \tilde t' \in \partial \widetilde S$ of their tacks we define
\beq
\label{al-be}
(\dot e_{c,i}, \dot e_{c',j})_\Gamma = \sum_{\tilde f \sim Z_j^+(\tilde t')} \hr{\dot e_{\tilde t,i}, \dot e_{\tilde f}}_{\widetilde\Gamma},
\eeq
where the sum is taken over all faces in $\pi^{-1}\big(\pi\big(Z_j^+(t')\big)\big)$. In order to show that the pairing $(\cdot,\cdot)_\Gamma$ is skew-symmetric we only need the following result.

\begin{lemma}
Formula~\eqref{al-be} gives rise to a well-defined skew-symmetric pairing on $\La_{\Cc(S)}$.
\end{lemma}

\begin{proof}
For the duration of the proof we abbreviate $Z = Z_i^+(\tilde t)$ and $Z' = Z_j^+(\tilde t')$. Let $z_{i-1}$, $z_i$, $\check z$ be the zig-zags bounding $Z$, and $z'_{j-1}$, $z'_j$, $\check z'$ be those bounding $Z'$. Write $\check s_i$ for the segment of the zig-zag $\check z$ between the black vertices $b_{i-1}$, $b_i$ in the notations of Lemma~\ref{lem:frame}, and define the segment $\check s'_j$ in a similar way. Formula~\eqref{form-face-notch} implies that the expression
\beq
\label{eq:et-ef}
\sum_{\tilde f \in Z'} (\dot e_{\tilde t,i}, \dot e_{\tilde f})_{\widetilde\Gamma}
\eeq
%in the pair $\big\{\tilde f_{i-1}^-, \tilde f_i^+\big\}$ or in the pair $\big\{\tilde f_i^-, \tilde f_{i+1}^+\big\}$ one of the faces belongs to $F'$ and the other does not, so that
vanishes, unless at least one of the edges $l_i, l_{i-1}$ belongs to the boundary of the region $Z'$.

Assume first that $l_i \subset \check s'_j$. Since each edge belongs to exactly two zig-zags, we have $\check z' \in \{z_i, \check z\}$. On the other hand, $\check z \ne \check z'$ since $c \ne c'$, and we conclude that $z_i = \check z'$. Then $\tilde f_i^- \in Z'$, $\tilde f_{i+1}^+ \notin Z'$, as can be seen from Figure~\ref{fig:notch-faces}, and the pair of faces contributes $-1$ to~\eqref{eq:et-ef} by the formula~\eqref{form-face-notch}. Note that in this case the edge $l_{i-1}$ does not lie on the boundary of $Z'$. Indeed, equality $z_i = \check z'$ ensures that $l_{i-1} \not\subset \check z'$, thus if $l_{i-1}$ belonged to the said boundary we would have had either $\check z \in \{z'_{j-1},z'_j\}$, which is ruled out by the second part of Lemma~\ref{lem:frame}, or $z_{i-1} \in \{z'_{j-1},z'_j\}$, which contradicts Lemma~\ref{lem:nopair}. Therefore, the pair $\tilde f_i^+, \tilde f_{i-1}^-$ does not contribute to~\eqref{eq:et-ef}, which takes value $-1$. By a similar argument, condition $l_{i-1} \subset \check s'_j$ implies that the sum~\eqref{eq:et-ef} takes value $1$.

Now assume that $l_i$ or $l_{i-1}$ is an edge in $z'_j \cup z'_{j-1}$. Recall that neither of the zig-zags $z_i$, $z_{i-1}$ coincides with $z'_j$ or $z'_{j-1}$. Thus we have $\check z\in\{z'_j,z_{j-1}'\}$, and both edges $l_i,l_{i-1}$ belong to exactly one of the zig-zags $z'_j, z'_{j-1}$, which coincides with $\check z$. By Lemma~\ref{lem:frame}, we have $\check z' \notin \hc{z_i, z_{i-1}}$, so unless the zig-zag $\check z'$ intersects $\check z$ over an edge within the segment $\check s_i$, the contribution to~\eqref{eq:et-ef} from the pair $\tilde f_{i-1}^-, \tilde f_i^+$ cancels that from the pair $\tilde f_i^-, \tilde f_{i+1}^+$. Hence $l'_j \subset \check s_i$ or $l'_{j-1} \subset \check s_i$. Arguing as before we arrive at
\beq
\label{eq:pairing-F-gg}
\sum_{\tilde f \in Z'} (\dot e_Z, \dot e_{\tilde f})_{\widetilde\Gamma} =
\begin{cases}
1 &\text{if} \;\; l_{i-1} \subset \check s'_j \;\; \text{or} \;\; l'_j \subset \check s_i, \\
-1 &\text{if} \;\; l_i \subset \check s'_j \;\; \text{or} \;\; l'_{j-1} \subset \check s_i, \\
0 &\text{otherwise.}
\end{cases}
\eeq
Formula~\eqref{eq:pairing-F-gg} shows that the pairing $(\cdot,\cdot)_\Gamma$ on $\La_{\Cc(S)}$ is skew-symmetric if well-defined and it remains to address the latter issue.

As we have just seen, the pairing~\eqref{eq:et-ef} vanishes unless $\check z_i \in \{z'_{j-1}, z'_j\}$ or $\check z'_j \in \{z_{i-1}, z_i\}$. A lift of any of the zig-zags $\pi(\check z')$, $\pi(z'_{j-1})$, $\pi(z'_j)$ determines that of the region $\pi(Z')$, and therefore there exist at most four regions $Z'$ in the same deck orbit for which~\eqref{eq:et-ef} is nonzero. In fact, there are at most two, since the option $\check z_i = z'_j$ ensures that $\pi(\check z_i) \ne \pi(z'_{j-1})$, for otherwise the zig-zags $z'_j$ and $z'_{j-1}$ would have been deck invariant, and similarly for $\check z'_j$. This finishes the proof of the Lemma.
\end{proof}

To finish the construction of the quiver defined by the seed~\eqref{eq:seed-gamma} we now express the basis $\hc{e_\ell}$ in terms of $\hc{\dot e_\ell}$. For any $c \in \Cc(S)$ and $1 \le i \le n$, we set $e_{c,i} = \dot e_{c,i}$. If $f \in F(\Gamma)$ is such that there are no faces $\tilde f_1, \tilde f_2 \in \pi^{-1}(f)$ sharing a bicolored edge, we define $e_f = \dot e_f$. Now assume that $\tilde f_1$ and $\tilde f_2$ share an edge, and consider $\gamma \in \pi_1(S)$ such that $\gamma(\tilde f_1) = \tilde f_2$. Note that the region on $\widetilde S$ covered by the orbit $\gamma^n(\tilde f_1)$, $n \in \Z$ is bounded by a pair of zig-zags $z_{i-1}(b)$, $z_i(b)$, where $\pi(b)$ is either a puncture or a tacked circle on $S$. For $1 \le j \le n$, denote by $Z_j$ the region lying between the zig-zags $z_{j-1}(b), z_j(b)$, and by $n_j$ the number of distinct faces in $\Gamma$ arising as projections of those in $Z_j$. It is easy to see from the definition of a zig-zag that $n_j \le n_{j+1}$ for all $j$. Let $k$ be the maximal integer for which $n_k=1$, and note that $k \ge i$. Then for any $1 \le j \le k$, there is a face $f_j \in F(\Gamma)$, such that $\pi(\tilde f) = f_j$ for all $\tilde f \in Z_j$. Furthermore, there is a unique face $f_{k+1} \in F(\Gamma)$, such that any of its preimages shares bicolored edges with a pair of distinct faces $\tilde f_k, \tilde f'_k \in Z_k$. Then we set $e_{f_j} = \dot e_{f_j} + \dots + \dot e_{f_{k+1}}$ for all $1 \le j \le k$, which determines the value of the form $(\cdot,\cdot)_\Gamma$ on the basis $\hc{e_\ell}$ and completes the construction of the quiver $Q_\Gamma$.

We now enhance the quiver $Q_\Gamma$ to a compatible pair. Recall that to each boundary arc $a$ on $S$, we associated a collection of $n$ frozen vectors. If this boundary arc is a tacked circle $c$, we denoted the vectors by $e_{c,i}$, otherwise we labelled the vectors by the boundary faces $f$ adjacent to $a$. We will denote faces of the latter kind by $e_{a,i}$, where the index $i$ decreases as we traverse $a$ in accordance with the orientation induced by that of $S$, so that in $Q_{\tri}$ we have
$$
(e_{a,i},e_{a,i+1}) = \frac12
$$
for any boundary arc $a$ and any $1 \le i < n$. Now each frozen direction $\ell \in I_\Gamma^*$ has two components: the boundary arc and the index of the frozen variable on it. Given $\ell = (a,i)$ we write $\ell[1] = a$ and $\ell[2]=i$. Let $\widehat C$ be the matrix given by
$$
\widehat C_{\ell_1,\ell_2} =
\begin{cases}
\frac12\delta_{\ell_1[1],\ell_2[1]} \mathfrak{A}_{\ell_1[2],\ell_2[2]}, &\text{if}\;\; \ell_1, \ell_2 \in I_\Gamma^*, \\
0, &\text{otherwise},
\end{cases}
$$
where $\mathfrak{A}$ is the Cartan matrix of the group $G = SL_{n+1}$. Note that the matrix
\beq
\label{eq:ensemble-matrix}
b = \eps + \widehat C,
\eeq
has integer entries.

%\begin{defn}
%\label{def:alphas}
Let $\tilde t$ be a tack on a lift $\tilde c \in \Cc(\widetilde S)$ of the tacked circle $c \in \Cc(S)$, and $\tilde t'$ the next tack we encounter on $\tilde c$ as we traverse it according to its orientation. Denote by $Z_{i}(\tilde c) $ the set of faces lying between the zig-zags $z_{i-1}(\tilde c)$, $z_i(\tilde c)$, $\check z(\tilde t)$, and  $\check z(\tilde t')$. We then set
\beq
\label{eq:alpha-c}
\alpha_{c,i} = \sum_{\tilde f \in Z_{i}(\tilde c)} \dot e_{\pi(f)}% \in \Lambda_\Gamma.
\eeq
%\red{I removed a minus sign in~\eqref{eq:alpha-c}.}
If $\tilde a$ is a boundary arc incident to and oriented towards a marked point $\tilde m$, we denote by $Z_{i}(\tilde m) $ the set of faces lying between the zig-zags $z_{i-1}(\tilde m)$ and $z_i(\tilde m)$. Then we set %(note the sign difference)
\beq
\label{eq:alpha-a}
\alpha_{a,i} = -\sum_{\tilde f \in Z_{i}(\tilde x)} \dot e_{\pi(f)}% \in \Lambda_\Gamma.
\eeq
%\end{defn}

\begin{lemma}
\label{lem:compat-pair-construction}
If $S$ has only tacked circles and no punctures, the matrix~\eqref{eq:ensemble-matrix} defines an ensemble matrix in the sense of Section~\ref{subsec:ensemble}, and thereby enhances the quiver $Q_\Gamma$ to a compatible pair.
\end{lemma}

\begin{proof}
First let us show that $b$ is invertible {over $\mathbb{Q}$}.
% Given a tacked circle $c \in \Cc(S)$ and its lift $\tilde c$ let us write $Z_i(c) = \pi(Z_i(\tilde c))$ for the subset of faces in $F(\Gamma)$ appearing as projections of those in $Z_i(\tilde c)$.\red{Is this definition correct? e.g torus.} Similarly, we set $Z_i(m) = \pi(Z_i(\tilde m))$ for a marked point $m$ and its lift $\tilde m$. For a boundary arc $a$ and a number $1 \le i \le n$ we define a vector
%\beq
%\label{eq:alpha_ai}
%\alpha_{a,i} = \sum_{f \in Z_i(x)} \dot e_f,
%\eeq
%where $x \in \partial a$ is the special point towards which the arc $a$ is oriented. 
Let $b_\circ$ be the submatrix of $b$ obtained by crossing out the rows and columns labelled by vectors $e_{c,i}$, for $c$ a tacked circle. It was shown in~\cite{GS19} that the $\alpha_{c,i}$ form a $\Q$-basis of $\ker(b_\circ)$. Using~\eqref{form-face-notch} one checks that
 $$
 (\alpha_{c,i},e_{c,j})  = \mathfrak{A}_{ij},
 $$
and thus the kernel of $b$ is trivial provided $S$ has only tacked circles and no punctures.

Now we define $\Xi \subset \Lambda_\Gamma\otimes\mathbb{Q}$ to be the lattice spanned by 
\begin{align}
\label{eq:xi-def}
\xi_j = \sum_{k \in I} (b^{-1})_{kj} e_k, \quad j\in I.
\end{align}
Since $b$ itself is an integer matrix, we have by definition $\La \subset \Xi$, so it only remains to verify the condition $(e_i,\xi_j) = \delta_{ij}$ for $i \notin I_\Gamma^*$. But this follows by evaluating the functional $(e_i,\cdot)$ on both sides of~\eqref{eq:xi-def} and using the fact that $\eps_{ik} = b_{ik}$ for $i \notin I_\Gamma^*$.
\end{proof}

\begin{convention}
\label{rmk:tori-roots}
Recall that the quantum torus associated to the $\Xi$-lattice in a compatible pair is not in general defined over $\mathbb{Z}[q^{\pm1}]$, but rather over a ring obtained by adjoining the required fractional powers of $q$. For the quantum tori associated to the $\Xi$-lattice in the compatible pair associated to $(\Pc^{\diamond}_{SL_{n+1},S},\Pc^{\diamond}_{G_{n+1},S})$ in Lemma~\ref{lem:compat-pair-construction}, we will always assume that the base ring contains $(-q)^{\frac{n}{2}}$. Note that this assumption is void exactly when the element $s_{SL_{n+1}}=(-1)^{n}$ in the definition of a twisted local system is the identity.
\end{convention}

\begin{remark}
 Let $\{a_j\}_{j \in \Z/k\Z}$ be the boundary arcs comprising a connected component $B \in \pi_0(\partial S)$ of the boundary of $S$ and listed in agreement with its orientation. As was shown in~\cite{GS19}, elements of the form
$$
\sum_{j \in \Z/k\Z}^{k}(\alpha_{a_{2j},i} + \alpha_{a_{2j+1},n+1-i})
$$
form a $\Q$-basis of $\ker(\eps)$.
\end{remark}

\begin{remark}
\label{rmk:root-cas}
%One easily computes that for a tacked circle $c$ we have 
%$$
%\alpha_{c,i} = -\sum_{j=1}^n \mathfrak{A}_{ij}\xi_{c,j} = \xi_{c,i-1} -2\xi_{c,i} +\xi_{c,i+1},
%$$
One easily computes that for $b$ a tacked circle or a boundary arc
$$
\alpha_{b,i} = -\sum_{j=1}^n \mathfrak{A}_{ij}\xi_{b,j} = \xi_{b,i-1} -2\xi_{b,i} +\xi_{b,i+1},
$$
where $\xi_0 = \xi_{n+1}=0$, and the ensemble map reads
$$
e_\ell =
\begin{cases}
\sum_{i \in I} \eps_{i\ell} \xi_i, & \ell \notin I_\Gamma^*, \\
-\frac12 \alpha_\ell + \sum_{i \in I} \eps_{i\ell} \xi_i, & \ell \in I_\Gamma^*.
\end{cases}
$$
\end{remark}

We finish this section by introducing vectors $\bar e_{c,i}$ which will be convenient for calculations later in the paper. Let $\tilde t$ be a tack on a lift $\tilde c \in \Cc(\widetilde S)$ of the tacked circle $c \in \Cc(S)$. Denote by $U_{c,i} $ the set of faces  lying between the zig-zags $z_{i-1}(\tilde c)$, $z_i(\tilde c)$, $\check z_+(\tilde t)$, and  $\check z_-(\tilde t)$. We then set
\beq
\label{eq:bar-e-ci}
\bar e_{c,i} = \sum_{\tilde f \in U_{c,i}} \dot e_{\pi(\tilde f)} - e_{c,i}.
\eeq
For example, in the notations of Figure~\ref{fig:notch-lift} we have
\begin{align*}
\bar e_{c,1} &= e_{f_{12}} - e_{c,1}, \\
\bar e_{c,2} &= e_{f_{13}} + e_{f_6} + e_{f_{11}} - e_{c,2}, \\
\bar e_{c,3} &= e_{f_{14}} + e_{f_7} + e_{f_2} + e_{f_5} + e_{f_{10}} - e_{c,3}.
\end{align*}
It is easy to see that for all $1 \le i < n$ we have
$$
(\bar e_{c,i}, \bar e_{c,i+1})_\Gamma = -\frac12.
$$

\begin{notation}
From this place onwards unless specified otherwise $c$ in the frozen index $(c,i) \in I_\Gamma^*$ is a tacked circle.
\end{notation}

\subsection{Depicting quivers from triangulations}

Given an ideal triangulation $\tri$ of a marked surface $S$, we now describe the quiver $Q_{\tri} = Q_{\Gamma_{\tri}}$ in a more pictorial way. An arc of $\tri$ is \emph{special} if it precedes a tacked circle as we traverse a special triangle in a positive direction. For example, the right arc connecting the tack with the bottom marked point on Figure~\ref{fig:disk-graph} is special. Note that any special arc is shared by a special and a non-special triangle and is oriented towards the tack in the former, and away from the tack in the latter.

%Let $\Delta_c$ be a special triangle containing a tacked circle $c$ as one of its sides. In the associated triangulation $\tri_\circ$ of $S_\circ$, the tacked circle $c \in \Cc(S)$ becomes a puncture $p \in P(S_\circ)$, and the special triangle $\Delta_c$ becomes a diagonal $d = (p,q)$, possibly with $p=q$. We say that a side $(p,q)$ of a triangle $\Delta$ in $\tri_\circ$ is \emph{special} if it is oriented from $p$ to $q$, where the orientation of $\Delta$ is inherited from that of the surface $S$.

\begin{figure}[h]
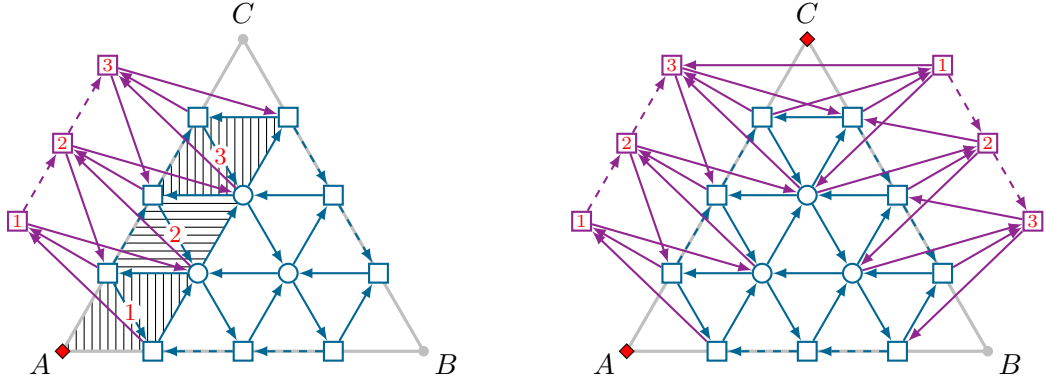

\subfile{fig-notches.tex}
\caption{Quiver for a triangle with notches.}
\label{fig:notches}
\end{figure}

Inside each non-special triangle $\Delta$ we draw the quiver $Q$ described in the beginning of this section, see Figure~\ref{fig:net-quiver}. Now, for each special arc $a$ of $\Delta$ oriented away from the tack $t$ of the tacked circle $c$, we enhance the quiver $Q$ with $n$ additional frozen vertices $v_{c,j}$ and draw dashed arrows $v_{c,j} \dashrightarrow v_{c,j+1}$ for all $1 \le j < n$. Note that every arrow of the quiver $Q$ is parallel to one of the arcs of $\Delta$. Consider the $n$ rhombi formed by the arrows of $Q$, which are adjacent to the arc of $\Delta$ preceding $a$ and whose diagonals are parallel to the arc of $\Delta$ following $a$. Let us number these rhombi from 1 to $n$, starting with the one which contains the tack $t$ as a vertex. In the left pane of Figure~\ref{fig:notches}, the edge $AB$ is special and is oriented away from the tack $A$, the 3 rhombi are adjacent to the edge $CA$ which precedes $AB$, and their diagonals are parallel to the edge $BC$ which follows $AB$. We now draw an arrow from each of the vertices of the diagonal of the $j$-th rhombus to $v_{c,j}$, and draw an edge from $v_{c,j}$ to each of the remaining vertices of the $j$-th rhombus (note, however, that the first rhombus only contains three vertices of $Q$, since the forth one is the tack $t$). Finally, let us assume that $\Delta$ contains special arcs $a,a'$, oriented away from the tacks $t,t'$ of tacked circles $c,c'$, and $a'$ follows $a$ as we travers $\Delta$. In that case we draw an additional solid arrow $v_{c,1} \to v_{c',n}$, see the right pane of Figure~\ref{fig:notches}. In what follows, the resulting quiver is denoted by $Q_\Delta$.

%\begin{figure}[h]
%\subfile{fig-square-notch.tex}
%\caption{Amalgamation in the presence of a special diagonal.}
%\label{fig:square-notch}
%\end{figure}

\begin{figure}[h]
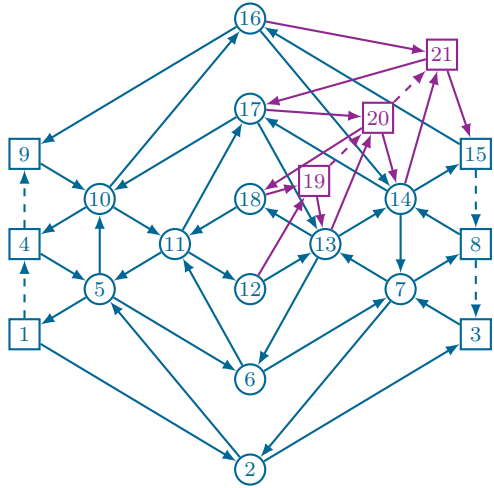

\subfile{disk-quiver.tex}
\caption{Quiver $Q_{\tri}$ from a triangulation on Figure~\ref{fig:disk-graph}.}
\label{fig:disk-quiver}
\end{figure}

Now, we are ready to describe the quiver $\dot Q_{\tri}$. For each pair of non-special triangles $\Delta, \Delta' $ in $\tri$ sharing an arc $a$, we \emph{amalgamate} the quivers $Q_{\Delta}$ and $Q_{\Delta'}$ along $a$, see e.g.~\cite{FG06a}. Namely, we identify the vertices of $\Delta$ and $\Delta'$ which lie on $a$, unfreeze the resulting $n$ vertices, and add up arrows between them, which in the case at hand amounts to simply erasing the dashed arrows. Furthermore, if $\Delta$ and $\Delta'$ are two (not necessarily distinct) non-special triangles containing arcs $a,a'$ of the same special triangle, we amalgamate the quivers $Q_{\Delta}$ and $Q_{\Delta'}$ along the vertices which lie on $a$ and $a'$. For example, the triangulation on Figure~\ref{fig:disk-graph} gives rise to the quiver shown on Figure~\ref{fig:disk-quiver}, which is obtained by amalgamating a pair of non-special triangles by both a joint arc and a self-folded triangle.
{The quivers $\dot Q_{\tri}$ have previously appeared in~\cite{BK23} in a similar context of extending the Poisson moduli space $\Pc_{G,S}$ to a symplectic one.}

\begin{figure}[h]
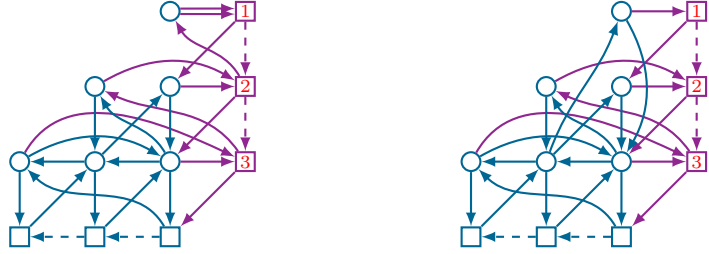

\subfile{fig-self-folded-special.tex}
\caption{Quivers $\dot Q_{\tri}$ (on the left) and $Q_{\tri}$ (on the right) for an annulus with one special and one non-special triangle.}
\label{fig:self-folded-special}
\end{figure}

Note that unless we have a non-special triangle sharing two edges with a special one, the bases $\hc{e_i}$ and $\hc{\dot e_i}$ coincide, and hence $Q_{\tri} = \dot Q_{\tri}$. Thus, it remains to treat the case of an annulus consisting of one special and one non-special triangle. Denote by $c$ the tacked circle of the former, choose a lift $\tilde t$ of its tack, and let $f_1$ be the unique face in the projection of the region $Z_1(\tilde t)$. In that case the basis $\hc{e_i}$ is obtained from $\hc{\dot e_i}$ by replacing the vector $\dot e_{f_1}$ with $e_{f_1} = \dot e_{f_1} + \dot e_{f_2}$. We show quivers $\dot Q_{\tri}$ and $Q_{\tri}$ on Figure~\ref{fig:self-folded-special}, where the vertices $v_{f_1}$, $v_{f_2}$ are the mutable vertices in the first and second top rows respectively, adjacent to the frozen ones.

%We can now define the functor
%$$
%\Qc \colon \widehat\Pt(S) \longra \Cl_{\bs Q}
%$$
%on the level of objects. Namely, we set
%$$
%\Qc(\tri) = Q(\tri).
%$$
%Before we define $\Qc$ on the level of morphisms, we will digress and discuss notched bicolored graphs. The letter can be thought of as a way to incorporate the notion of a notch into that of an ideal web, see~\cite{Gon17}.

We record one more lemma about quivers associated to triangulations.

\begin{lemma}
\label{lem:no-automorphisms}
Suppose that $S$ is a marked surface without punctures, but possibly with tacked circles, and let $\tri$ be a triangulation of $S$. Then the only permutation automorphism of the quiver $Q_{\tri}$ which is the identity on all frozen vectors $e_i$, $i\in I_\Gamma^*$ is the trivial permutation.
\end{lemma}

\begin{proof}
%Any such automorphism $\sigma$ induces a permutation $\sigma_A$ on the set of arcs in the triangulation which is identity on the set of boundary arcs. In order for the quiver permutation to be an isometry $\sigma_A$ must permute the set of arcs incident to any given boundary one. But since $S$ has only tacked circles we have a total order on the set of arcs incident to a boundary arc, and this order must be preserved by any quiver isometry. We conclude the permutation $\sigma_A$, and hence also $\sigma$, must be the identity.
Any such automorphism $\sigma$ restricts to the subquiver $Q_{\tri_\circ}$, obtained from $Q_{\tri}$ by erasing all frozen vertices from tacked circles, and we denote its restriction by $\sigma_\circ$. Let $\Delta$ be a triangle in $\tri_\circ$. Then $\sigma_\circ$ is trivial on the subquiver $Q_\Delta$ of $Q_{\tri_\circ}$ if $\Delta$ is not self-folded, and may only be permuting a pair of nodes in $Q_\Delta$ if it is. Arguing by induction on the number of trinagles in $\tri_\circ$, we conclude that $\sigma_\circ$ is a composition of permutations of pairs of nodes in the self-folded triangles of $\tri$. However, the presence of tacked boundary nodes in $Q_{\tri}$ guarantees that $\sigma$ is isometry if and only if this composition is a trivial permutation. This concludes the proof.
\end{proof}

\subsection{Mutations and shifts}

Consider an ideal bicolored graph $\Gamma$ on $S$ and a proper quadrilateral face $f \in F(\Gamma)$.

\begin{prop}
\label{eq:prop-graph-quiver-mut}
Let $\Gamma'$ be the bicolored graph obtained from $\Gamma$ by performing a square move at the face $f$. Then the following formula defines an isometry of lattices
\beq
\label{eq:e-to-e''}
\hat \mu_f \colon \Lambda_\Gamma \longra \Lambda_{\Gamma'}, \qquad e_\ell \longmapsto 
\begin{cases}
- e'_\ell &\text{if} \;\; \ell = f, \\
e'_\ell  &\text{if} \;\; \ell = (c,i) \text{ and } f=f_{i-1}^-(\Gamma), \\
 e'_\ell + \eps_{\ell f} e'_f&\text{if} \;\; \ell =(c,i) \text{ and } f=f_{i}^+(\Gamma),  \\
e'_\ell + [\eps_{\ell f}]_+ e'_f &\text{otherwise.}
\end{cases}
%\begin{cases}
%- e'_\ell &\text{if} \;\; \ell = f, \\
%e'_\ell + \max\hc{(e_\ell, e_f),0} e'_f &\text{if} \;\; \ell \in F(\Gamma) \smallsetminus \hc{f}, \\
%e'_\ell + (e_\ell, e_f) e'_f &\text{if} \;\; \ell = (c,j) \;\; \text{and} \;\; f = f_i^+(c), \\
%e'_\ell &\text{otherwise.}
%\end{cases}
\eeq
where we use the natural bijection $F(\Gamma) \simeq F(\Gamma')$ to identify the index sets labelling the bases $\{e_\ell\}$ in $\Lambda_\Gamma$ and $\{e'_\ell\}$ in $\Lambda_{\Gamma'}$.
\end{prop}

\begin{proof}
We need to show that
\beq
\label{eq:graph-isometry}
(e_{\ell_1}, e_{\ell_2})_{\Gamma} = (\hat \mu_f(e_{\ell_1}), \hat \mu_f(e_{\ell_2}))_{\Gamma'}
\eeq
for any pair $\ell_1, \ell_2 \in I_\Gamma$. When $\ell_1, \ell_2 \in F(\Gamma)$ the result follows from inspecting Figure~\ref{fig:face-mutation}, while disregarding the shaded regions. The Figure shows a neighborhood of an admissible face $\tilde f \in \pi^{-1}(f)$ on trivalent graphs $\widetilde\Gamma$, $\widetilde\Gamma'$, as well as that of the vertex $v_f$, which is shaded with pink, of the quivers $Q_\Gamma$, $Q_{\Gamma'}$. Note that the colors of the vertices of $\widetilde\Gamma$ which are not incident to the face $\tilde f$ can be arbitrary, and we only illustrate one of several possible options. The bases $\hc{e'_\ell}$ and $\hc{\dot e'_\ell}$ coincide unless a pair of faces neighboring $\tilde f$ and situated on the same side of either of the two zig-zags have the same projection in $F(\Gamma)$. In the latter case one can easily verify that formula~\eqref{eq:graph-isometry} holds for the basis $\hc{e_\ell}$, but not for $\hc{\dot e_\ell}$.

\begin{figure}[h]
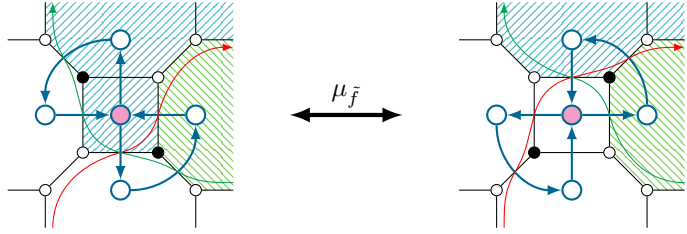

\subfile{face-mutation.tex}
\caption{Square move and quiver mutation.}
\label{fig:face-mutation}
\end{figure}

%%%%%%%%
%%%%% Following is intermediate detail level explanation:
%%%%%% original minimal detail version is preserved in commented section below.

Now we consider the case $\ell_1=e_{c,i} \in I_{\Cc(S)}$ and $\ell_2=\ell \in I_{\Gamma_\circ}$. First, assume that there is no lift $\tilde f \in \pi^{-1}(f)$ which is simultaneously incident to one of the zig-zags $z_{i-1}(\tilde c)$, $z_i(\tilde c)$, and $\check z(\tilde t)$ for a lift $\tilde c$ of $c$ and a tack $\tilde t$ on $\tilde c$. In that case the location of the intersection points of the zig-zags $z_j(c)$ and $\check z(\tilde t)$ remain unchanged and thus $(e''_{c,i},e''_\ell)=(e_{c,i},e_\ell)$ for all $\ell \in F(\Gamma)$. In particular, by~\eqref{form-face-notch} we have $(e''_{c,i},e''_f) = (e_{c,i},e_f)=0$, which in turn implies that the two sides of~\eqref{eq:graph-isometry} are equal.
Now suppose that $\tilde f$ is incident to $\check z(\tilde t)$ and to exactly one of the zig-zags $z_{i-1}(c)$, $z_i(c)$. In this case, thanks to Lemma~\ref{lem:frame} the neighborhood of $\widetilde f$ corresponds to one of the two sides of Figure~\ref{fig:face-mutation}. If the red zig-zag in the Figure is $z_i(c)$, then the square face $f$ on the left of Figure~\ref{fig:face-mutation} is the special face $f_{i+1}^+$, while the square face on the right is $f_i^-$. Similarly, if the red zig-zag is $z_{i-1}(c)$, then the square face on the left is $f_i^+$ and that on the right $f_{i-1}^-$. If the face $\ell$ is not a neighbor of $f$, then we have $\eps_{\ell f}=0$ and $\hat\mu_f(e_\ell)=e''_\ell$. We also have $(e''_{c,i},e''_\ell)=(e_{c,i},e_\ell)$ since the square move does not modify intersections of red and green zig-zags over any edge of $\ell$. Then since $\hat{\mu}_f(e_{i,c})$ differs from $e''_{i,c}$ by a multiple of $e''_f$ and $(e''_f,e''_\ell)=0$, we have $(\hat{\mu}_f(e_{i,c}),\hat{\mu}_f(e_\ell)) = (e''_{i,c},e''_\ell)= (e_{i,c},e_\ell)$ so the equality~\eqref{eq:graph-isometry} holds. It remains to consider the cases in which the face $\ell$ coincides with $f$ or is one of the faces adjacent to it. These follow from a straightforward inspection of Figure~\ref{fig:face-mutation}, which we leave to the reader. Thereby we establish the isometry~\eqref{eq:graph-isometry} when exactly one of the arguments corresponds to a proper face. 

Finally, the case $\ell_1, \ell_2 \in I_{\Cc(S)}$ when $\ell_1[1] = \ell_2[1]$ once again follows from inspecting the Figure~\ref{fig:face-mutation}, and the remaining case $\ell_1[1] \ne \ell_2[1]$ is a consequence of the formula~\eqref{al-be} and the definition of the basis $\hc{e_\ell}$.
\end{proof}

%Observing that under the square move $\mu_f \colon \Gamma \to \Gamma'$ we have $f_i^\pm(\Gamma) = f_{i\mp1}^\mp(\Gamma')$ we arrive at the following decomposition of the isometry $\hat\mu_f$.

In other words, when the graph $\Gamma'$ is obtained  from $\Gamma$ by a square move at a special face of the latter of the form $f=f_i^+(c)$ with $i\leq n$ or $f=f_{i}^-$ with $i\leq n-1$, the quiver $Q_{\Gamma'}$ does not coincide with the mutated quiver $\mu_f(Q_\Gamma)$ but instead with the quiver obtained by applying to the latter a quasi-permutation. Indeed, let's write $\{e_l''\}$ for the basis associated to $\mu_f(Q_\Gamma)$ and consider the quiver quasi-permutation 
$$\varsigma_f \colon \mu_f(Q_{\Gamma}) \to Q_{\Gamma'
}
$$
 defined by
\beq
\label{eq:sigma-f}
\varsigma_f(e''_\ell) =
\begin{cases}
e'_{\varphi(\ell)} - e'_{\varphi(f)}, &\text{if } \ell = (c,i) \text{ and} \;\; f\in\{ f_i^+(\Gamma'),f_{i-1}^-(\Gamma')\}, \\
%e'_\ell - \eps'_{\ell f} e'_f, &\text{if } f=f_{i}^+(\Gamma)\sim f_{i-1}^-(\Gamma') \text{ and} \;\; \ell = (c,i), \\
e'_{\varphi(\ell)}, &\text{otherwise,}
\end{cases}
\eeq
where $\varphi$ denotes the natural bijection $I_{\Gamma}\simeq I_{\Gamma'}$ associated to the square move at $f$.
Then we have
\begin{cor}
\label{cor:hat-mu-e}
The isometry $\hat \mu_f \colon \Lambda_\Gamma \longra \Lambda_{\Gamma'}$ factors as
\beq
\label{eq:mu-hat-factor}
\hat \mu_f = \varsigma_f \circ \mu_f,
\eeq
where $\mu_f \colon Q_\Gamma \to \mu_f(Q_{\Gamma})$ is the quiver mutation isometry defined by~\eqref{eq:mon-mut}.
\end{cor}
\begin{remark}
To avoid confusion let us note that in our setup there need not exist a bicolored graph $\Gamma''$ that $\mu_f(Q_\Gamma) = Q_{\Gamma''}$ in the notations of Corollary~\ref{cor:hat-mu-e}.
\end{remark}
The following Lemma is straightforward to check using the rule~\eqref{eq:ensemble-mat-change} for mutation of ensemble matrices:
\begin{lemma}
\label{prop:xi-mut-permut}
The quiver quasi-permutation $\varsigma_f$ admits the following expression in terms of the bases $\hc{\xi'}, \hc{\xi''}$ for $\Xi_{\Gamma'}$ and $\Xi_{\mu_f(Q)}$:
\beq
\label{eq:xi-atrans}
\varsigma_f(\xi''_\ell) =
\begin{cases}
\xi'_{\varphi(f)} + \xi'_{c,i}, &\text{if} \;\; \ell = f\in\{f_i^+(c,\Gamma'),f_{i-1}^-(c,\Gamma')\} \;\; \text{for some } 1\leq i\leq n , \\
%\xi'_f + \xi'_{c,i}, &\text{if} \;\; \ell = f=f_{i-1}^-(c,\Gamma') \;\; \text{for some } i<n , \\
\xi'_{\varphi(\ell)}, &\text{otherwise.}
\end{cases}
\eeq
\end{lemma}
%\begin{proof}
%%Proving the lemma amounts to checking that the map on $e$-lattices induced by the map above agrees with~\eqref{eq:sigma-f}. We do this in the case $f=f_i^+(c;\Gamma')$, with $f=f_{i}^-(c;\Gamma')$ being handled similarly. If $l$ is a face of $\Gamma'$ not appearing in the local picture~\eqref{fig:face-mutation}, 
%%
%%and for those 
%%With the help of formulas~\eqref{eq:ensemble-matrix} and~\eqref{eq:e-to-e''} we can express entries of the ensemble matrix $b''$ via those of $b$. Subsequently, using~\eqref{eq:ensemble-mat-change} we arrive at
%%$$
%%b''_{\ell_1,\ell_2} =
%%\begin{cases}
%%b'_{\ell_1,\ell_2} + b'_{\ell_1 f} [b'_{f \ell_2}]_+ + [b'_{f \ell_1}]_+ b'_{f \ell_2}, &\text{if} \;\; \ell_1 = (c_1,i_1), \; \ell_2 = (c_2,i_2), \\
%%b'_{\ell_1,\ell_2} + [b'_{\ell_1 f}]_+ b'_{f \ell_2}, &\text{if} \;\; \ell_1= (c_1,i_1), \; \ell_2 \notin\hc{(c_2,i_2), f}, \\
%%b'_{\ell_1,\ell_2} + b'_{\ell_1 f} [b'_{\ell_2 f}]_+, &\text{if} \;\; \ell_2 = (c_2,i_2), \; \ell_1 \notin \hc{(c_1,i_1),f}, \\
%%b'_{\ell_1,\ell_2}, &\text{otherwise,}
%%\end{cases}
%%$$
%%which together with~\eqref{eq:sigma-f} allows us to verify the desired statement.
%\end{proof}

%
%
%\begin{figure}[h]
%\subfile{fig-square-notch.tex}
%\caption{Amalgamation in the presence of a special diagonal.}
%\label{fig:square-notch}
%\end{figure}
%

Let $\Gamma' = \sigma_c^{\pm1}(\Gamma)$ be the bicolored graph obtained from $\Gamma$ by performing a shift $\sigma_c^{\pm1}$ at a tacked circle $c$. Using the natural identification $F(\Gamma) \simeq F(\Gamma')$ we can consider subsets $S_i^\pm(c) \subset F(\Gamma)$ defined as follows:
\beq
\label{eq:S-regions}
S_i^+(c) = \pi(\sigma_c(Z_i^+) \smallsetminus Z_i^+)
\qquad\text{and}\qquad
S_i^-(c) = \pi(Z_i^+ \smallsetminus \sigma_c^{-1}(Z_i^+)),
\eeq
where $Z_i^+ = Z_i^+(t)$ for some tack $t \in \pi^{-1}(c)$. The next result is an immediate consequence of the definition of vectors $e_{c,i}$ and the form $(\cdot,\cdot)_\Gamma$.
%In Figure~\ref{fig:square-notch}, the faces of $S_3^+$ correspond to the three shaded vertices of the quiver lying inside the triangle $ACD$. The faces of $S_3^-$ are in bijection with the shaded vertices inside the triangle $ABC$, excluding the one positioned on the diagonal $AC$.

\begin{lemma}
\label{lem:sigma-pm}
%There exist unique isometries $\sigma_c^\pm$ such that
The following formula defines a quiver quasi-permutation $\varsigma_c^{\pm1} \colon \Lambda_\Gamma \to \Lambda_{\Gamma'}$:
\beq
\label{eq:sigma-pm}
\varsigma_c^{\pm1}(e_\ell) =
\begin{cases}
e'_{\varphi(\ell)} \mp \sum_{f \in S_i^{\pm}(c)} \dot e'_{\varphi(f)}, &\text{if} \;\; \ell = (c,i), \\
e'_{\varphi(\ell)}, &\text{otherwise,}
\end{cases}
\eeq
where we write $\varphi$ for the natural bijection $I_{\Gamma}\simeq I_{\Gamma'}$ associated to the shift at $c$.
\end{lemma}

\begin{cor}
\label{cor:xi-shift-permut}
At the level of $\xi$-bases, the quiver quasi-permutation $\varsigma_c^{\pm1}$ reads
\beq
\label{eq:xi-shift-permut}
\varsigma_c^{\pm1}(\xi_\ell) =
\begin{cases}
\xi'_{\varphi(\ell)} \pm \xi'_{c,i}, &\text{if} \;\; \ell \in S_i^\pm(c), \\
\xi'_{\varphi(\ell)}, &\text{otherwise.}
\end{cases}
\eeq
\end{cor}

\subsection{Cluster $K_2$-coordinates}
\label{subsec:cluster-coord}

Let $T'S$ be the unit tangent bundle to a surface $S$. A \emph{twisted $SL_{n+1}$-local system} on $S$ is a local system on $T'S$ with monodromy $(-1)^{n}$ around any loop which rotates a vector $v \in T'S$ by $2\pi$. A choice of a spin structure on $S$ gives rise to an isomorphism between the spaces of local systems and twisted local systems, see~\cite[Section 2.3]{FG06b}. A \emph{twisted decorated local system} on a stratified surface $\decS$ is a twisted $G$-local system on $\decS_G$, a twisted $T$-local system on $\decS_T$, a $B$-reduction of the $G \times T$-local system along the walls, and a trivialization at each marked point. Now we describe cluster $K_2$-coordinates on the (classical, rather than quantum) moduli space of twisted decorated $SL_{n+1}$-local systems along the same lines as the procedure of~\cite{FG06b, Gon17}.

As before, let $\pi \colon \widetilde S \to S$ be the universal cover of a marked surface $S$, and $\Gamma$ an ideal bicolored graph of rank $n$ on $S$. For each marked point $m$ on $\widetilde S$, we define $R_j(m) \subset \widetilde S$ to be the region to the right of the zig-zag $z_j(m)$. Now let $t, t'$ be a pair of tacks on a tacked boundary component $c \subset \partial\widetilde S$, so that $t'$ directly follows $t$ as we traverse $c$ in the positive direction. In this case we define $R_j(t)$ to be the unique region bounded by $\check z(t)$, $\check z(t')$, and $z_j(c)$. Note that $t$ is the unique tack on $\widetilde S$ contained in $R_j(t)$. Given a face $f \in F(\widetilde\Gamma)$ and marked point or tacked circle $x$ on $\widetilde S$, we define the \emph{codistance} $d(f,x)$ to be the cardinality of the set  $\{1 \le j \le n \,|\, f \in R_j(x)\}$.
%\red{In plain language, the codistance from $f$ to $x$ is the number of regions $R_j(x)$ associated to $x$ which contain $f$. }
%$$
%d(f,x) =
%\begin{cases}
%n+1 - \min\limits_{f \in R_j(m)} j, &\text{if} \quad f \in R_n(m); \\
%0, &\text{otherwise.}
%\end{cases}
%$$
% \red{The intuitive meaning of this codistance is as follows: given a graph of rank $n$ on $\widetilde S$ and a special point $x \in M(\widetilde S) \cup T(\widetilde S)$, the region $R_n(x)$ is a disk containing $x$. Then if $f$ is a face whose interior is inside $R_n(x)$, the codistance of $f$ from $x$ measures the minimum number of zig-zags encircling $x$ encountered as one follows any path from a point in the interior of $f$ to a point outside of the disk $R_n(x)$.}
The following result is a trivial modification of a similar statement for $\widetilde S_\circ$, see~\cite{Gon17}.

\begin{prop}
For any $f \in F(\widetilde\Gamma)$, we have
$$
\sum_{x} d(f,x) = n+1,
$$
where the sum is taken over all special points $x$ in $\widetilde S$.
\end{prop}
We first construct the cluster $K_2$-coordinates in the case that the underlying surface $S$ is a disk. We define its `lifted boundary'' to be the section  $\partial S \rightarrow T'S$ given by taking the outward normal vector $v_z$ to each point $z\in \partial S$ on the boundary of $S$. Given a marked point $x\in\partial S$, let $x^+$ be the first point of $\decS_B\cap\partial S$ one sees when traversing $\partial S$ according to its orientation starting from $x$, and write $\hat x= (x^+,v_{x^+})$ for the corresponding point on the lifted boundary. As in the discussion from Example~\ref{ex:basic-disks}, from the decoration we produce an $N$-orbit in the fiber  $\mathcal{L}_{G;\hat x}$ of the $G$-local system  at $\hat x$. Now let $V=\mathbb{C}^{n+1}$ be an $n+1$ dimensional vector space with a volume form $\Omega \in \det V^*$, so that $G \simeq \mathrm{Stab}(\Omega)\subset GL(V)$.
If we choose a trivialization of the fiber  $\mathcal{L}_{G;\hat x}$, the data of an $N$-orbit in that fiber becomes that of a \emph{framed flag}: a pair $\bar F=(F,\nu)$ consisting of a complete flag $F=(F_i\subset V), ~\dim F_i =i$ in $V$, together with a choice of nonzero vectors $\nu_i(\bar F) \in \det F_i$ such that $\langle\Omega, \nu_{n+1}(\bar F)\rangle=1$.

%We write $\hat x$ for the image of a marked point $x\in \partial S$ on the lifted boundary.

%First, let us assume that $S$ is a disk with boundary marked points $x_1, \dots, x_k$, whose cyclic order is induced by the orientation of $S$. Choose a lift $\gamma \subset T'(\partial S)$ of the boundary $\partial S \simeq S^1$, such that the parallel transport along $\gamma$ induces a rotation by $2\pi$ in a fiber $T'_x(\partial S)$, and vectors $v_{x_j} \in \gamma \cap T'_{x_j}(\partial S)$. 

Now given a face $f \in F(\Gamma)$,  choose arbitrarily a marked point $x\in \partial S$ with $d(f,x)>0$. This choice linearizes the cyclic order on the set of all marked points (induced by the orientation of $\partial S$) into a total order on the set of boundary marked points so that $x=x_1$ is the minimal element. We write $\{x=x_1<x_2<\ldots<x_m\}$ for the set of all marked points with nonzero codistance from $f$, with $\hat x_k$ the corresponding points on the lifted boundary. Let $\hat y\in T'S$ be a point in the unit tangent bundle projecting to an interior point $y\in S$ of the face $f$, and whose tangent direction coincides with the outward normal $v_{x^+_1}$ to $\partial S$ at $x_1$. 
Then we can produce a collection of $m$ framed flags $\rho_f(\bar{F}_{\hat x_i})$ at $\hat y$ by parallel transporting the $N$-orbits $\bar{F}_{\hat x_i}$ from the fibers over $\hat x_i$ to  the fiber over $\hat y$ along the path in $T'S$ that rotates the tangent vectors $v_{x^+_i}$ clockwise to $v_{x^+_1}$.

%we pick a vector $v_f \in T'S|_f$ and a collection of paths $\gamma_j \subset T'S$ connecting $v_{x_j}$ to $v_f$. Consider an $(n+1)$-dimensional vector space $V$ with a volume form $\Omega \in \det V^*$. For $G = SL_{n+1}$ we can identify the quotient $G/N$ with the space of \emph{framed flags,} where a framed flag $\bar F$ is a pair $(F,\nu)$ consisting of a flag $F \in G/B$ with flag subspaces $F_i \subset V$ of dimension $i$ and non-zero forms $\nu_i(\bar F) \in \det F_i$. A decorated local system on $T'S$ induces a framed flag at each of the vectors $v_{x_j}$, and we denote their transports to $v_f$ along the paths $\gamma_j$ by $\bar F_j = \bar F_j(f) \in G/N$. 

Then if we choose a trivialization of the fiber over $\hat y$ and set
\beq
\label{eq:dot-A_f}
\dot A_f = \ha{\Omega, \nu_{d_1}(\rho_f(\bar F_{\hat x_1})) \wedge \dots \wedge \nu_{d_k}( \rho_f(\bar F_{\hat x_k}))},\qquad  d_j = d(f,x_j),
\eeq
the resulting function $\dot A_f$ depends on neither the choice of trivialization at $\hat y$ nor the arbitrary choice of point $x_1$, and so gives a well-defined function on the moduli space $\Pc_{G,S}^\diamond$ canonically associated to the face $f$. For brevity, we will often abuse notation by making the following abbreviation for this function:
\begin{align}
\label{eq:shorthand}
\ha{\Omega, \nu_{d_1}(\rho_f(\bar F_{\hat x_1})) \wedge \dots \wedge \nu_{d_k}( \rho_f(\bar F_{\hat x_k}))} \equiv [\nu_{d_1}({x_1}) \wedge \dots \wedge \nu_{d_k}({x_k})].
\end{align}

For general $S$ we reduce to the case above by lifting a twisted decorated local system on $S$ to a $\pi_1(S)$-equivariant twisted decorated local system on $T'\widetilde S$. Given a face $f \in F(\Gamma)$, we choose its lift $\tilde f \in \pi^{-1}(f)$ and any disk $D \subset \widetilde S$ on the universal cover $\widetilde S$, which contains all special points $x$ of $\widetilde S$ such that the codistance $d(\tilde f,x)$ is nonzero. We then set $\dot A_f = \dot A_{\tilde f}$, where the latter is defined by the formula~\eqref{eq:dot-A_f} with $\widetilde S$ replaced by the disk $D$. The function $\dot A_f$ is again well-defined in the sense that it does not depend on the lift $\tilde f$ of $f$ or the choice of the disk $D$; we refer the reader to~\cite{FG06b} for further details.

The functions $A_f$ are related to the $\dot{A}_f$ defined above as follows. As before, if $f \in F(\Gamma)$ is such that there are no faces $\tilde f_1, \tilde f_2 \in \pi^{-1}(f)$ sharing a bicolored edge, we define $A_f = \dot A_f$. Recall that if there exist such $\tilde f_1, \tilde f_2$ sharing an edge, then both faces belong to the region $Z_j(b)$, with $1 \le j \le n$ and $b$ a puncture or a tacked circle, which only contains faces in $\pi^{-1}(f)$. Moreover, there exists a maximal number $k$, such that $Z_i(b)$ only contains preimages of the same face $f_i$ for all $i \le k$, and there is a unique face $f_{k+1} \in F(\Gamma)$ such that each of its preimages shares bicolored edges with a pair of distinct faces $\tilde f_k, \tilde f'_k \in Z_k$. Then for all $1 \le j \le k+1$ we define $A_{f_j}$ via $\dot A_{f_j} = A_{f_j}\dot A_{f_{j-1}}$. In other words, we have
\begin{align}
\label{eq:adot-to-a}
\dot {A}_{f_j} = \prod_{s=1}^{j} A_{f_{s}}, \qquad A_{f_j} = \dot {A}_{f_j}/\dot {A}_{f_{j-1}}.
\end{align}
For more detail, see Example~\ref{eg:dot-A}.
Recall that a tacked circle $c \in \Cc(S)$ represents a boundary component of the decorated surface $\decS$, surrounded by a boundary $T$-annulus with a single boundary $T$-marked point at the tack $t$ of $c$. Given an ideal graph $\Gamma$ on the corresponding marked surface $S$, let us draw the $(G,T)$ domain wall on $S$ in such a way that it lies in a small neighborhood of $c$ and does not intersect (projections of) any zig-zags except those emanating from or terminating at $t$. Consider the point $p_t$ at which the domain wall meets (the projection of) the zig-zag $\check z(t)$. The $T$-local system at this point has a canonical trivialization, given by transporting the one at $t$ along $\check z(t)$. Further choosing a trivialization of the $G$-local system at $p_t$, we can identify its monodromy around the domain wall with an element $g_t$ of $G$, and the $B$-reduction of the $G\times T$-local system is identified with a framed flag $\bar F = (F,\nu)$. Then for each $1 \le i \le n$, we define a coordinate $A_{c,i}$ via
\begin{align}
\label{eq:Acdef}
\nu_i(g_t \bar F) = A_{c,n+1-i}^{-1} \nu_i(\bar F),
\end{align}
and note that $A_{c,i}$ is independent of the choice of $G$-trivialization at $p_t$. Since $g_t(F) = F$, the flag $F$ induces an order on the eigenvalues $\hc{\la_i}$ of the holonomy $g_t$ around the loop encircling $c$ by requiring that 
\begin{align}
\label{eq:A-frozen-eigenvalues}
%A_{c,i}^{-1} = \lambda_i \cdots \lambda_n
A_{c,i} = \lambda_1\cdots\lambda_{i}.
\end{align}

%\red{Proposal: replace preceding by equivalent formula $A_{c,i}=\lambda_0\cdots\lambda_{i-1}$.}

\begin{figure}[h]
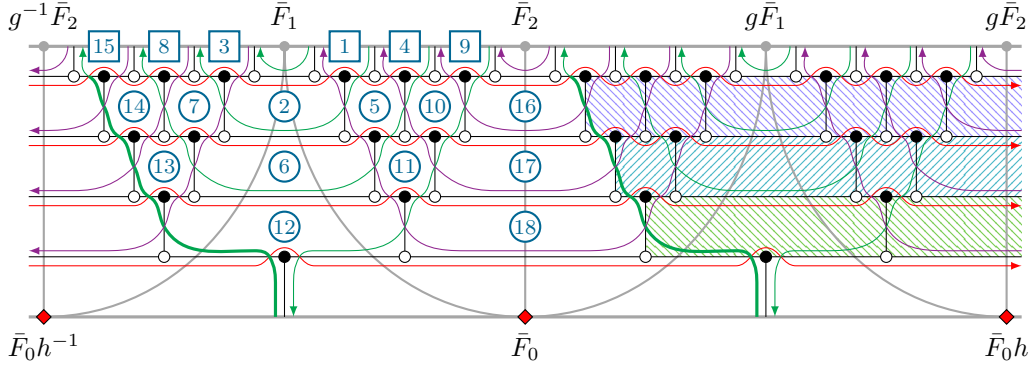

\subfile{notch-lift.tex}
\caption{Zig-zags and on an ideal bicolored graph $\widetilde\Gamma_{\tri}$ of rank 3.}
\label{fig:notch-lift}
\end{figure}

\begin{example}
Figure~\ref{fig:notch-lift} depicts the universal cover $\widetilde\Gamma$ of a bicolored graph shown in the left pane of Figure~\ref{fig:disk-web} with the quiver $Q_\Gamma$ appearing in Figure~\ref{fig:disk-quiver}. Framed flags at the tacks and marked points are labelled by $\bar F_j h^n$ and $g^n\bar F_j$, where $\bar F_j \in G/N$ and $(g,h)$ is the monodromy of the $G\times T$ local system around the tacked circle. Then the cluster $K_2$-coordinates are the functions, for example,
$$
A_{f_{11}} = \langle \Omega,\nu_2(\bar F_0) \wedge \nu_1(\bar F_2) \wedge \nu_1(\bar F_1)\rangle
$$
and
\begin{align*}
A_{f_{14}} &= \langle \Omega,\nu_1(\bar F_0) \wedge \nu_1(\bar F_1) \wedge \nu_2(g^{-1}\bar F_2)\rangle \\
&= \langle \Omega,\nu_1(g\bar F_0) \wedge \nu_1(g\bar F_1) \wedge \nu_2(\bar F_2)\rangle \\
&= \la_4 \langle \Omega,\nu_1(\bar F_0) \wedge \nu_1(g\bar F_1) \wedge \nu_2(\bar F_2)\rangle,
\end{align*}
where $\la_4$ is the `last' eigenvalue of the monodromy $g$.
\end{example}

\begin{example}
\label{eg:dot-A}
The following example illustrates how the relation~\eqref{eq:A-frozen-eigenvalues} between the frames $\{\dot{A}_{\ell}\}$ and $\{A_\ell\}$ reconciles the discrepancy between the Pl\"ucker relations for face functions $\dot{A}_\ell$ and the combinatorial cluster mutation formula for the $\Ac$-variables $A_\ell$.

For $G=SL_2$, consider the cluster coordinate systems associated to the two triangulations shown in Figure~\ref{fig:adot-examples1} which are related by the flip at the edge labelled 2 in the Figure.

\begin{figure}
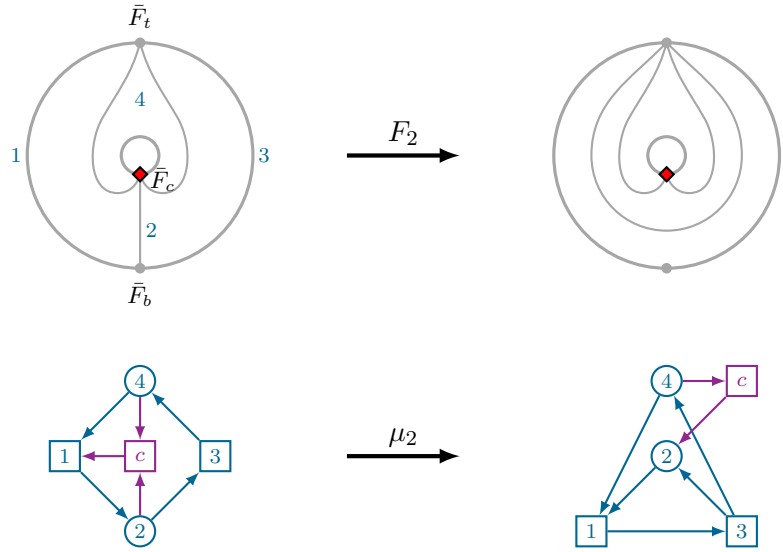

\subfile{fig-adot-examples1}
\caption{Triangulations and quiver for $G=SL_2$.}
\label{fig:adot-examples1}
\end{figure}

\begin{figure}[h]
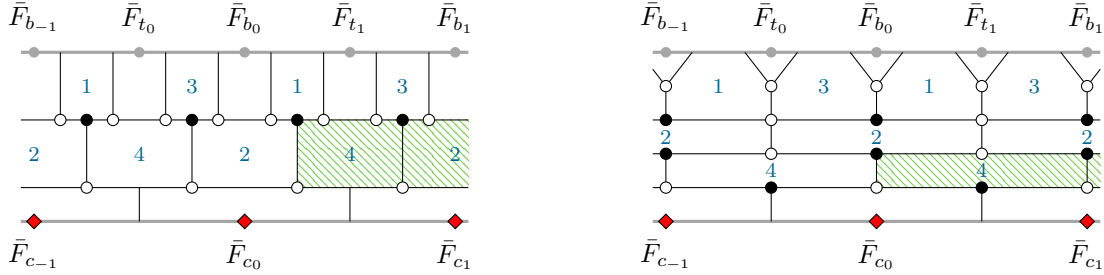

\subfile{fig-adot-examples1-unrolled}
\caption{Unrollings of the triangulations from Figure~\ref{fig:adot-examples1} to the universal cover.}
\label{fig:adot-examples1-unrolled}
\end{figure}

In Figure~\ref{fig:adot-examples1-unrolled}, we show their unrollings to the universal cover. In terms of the shorthand notation~\eqref{eq:shorthand}, the cluster coordinates associated to the edges of the triangulation in the left pane of Figure~\ref{fig:adot-examples1} are
\begin{align*}
&A_1 = [\nu_1(t_{1})\wedge\nu_1(b_1)], \quad A_2 = [\nu_1(b_0)\wedge\nu_1(c_0)],\\
&A_3 = [\nu_1(t_{0})\wedge\nu_1(b_{0})],\quad  A_4 = [\nu_1(t_{0})\wedge\nu_1(c_{0})],
\end{align*}
and following~\eqref{eq:Acdef} the frozen coordinate $A_c$ is defined by 
$$
\nu_1(c_0) = A_c\nu_1(c_1).
$$
For the triangulation in the right pane, there is a nontrivial relation between the $\dot A_\ell$ and $A_\ell$ frames owing to the fact that each face on the universal cover projecting to face $4$ shares a pair of bicolored edges with faces projecting to $2$ (and vice-versa).
Indeed, since 
$$
[\nu_1(t_1)\wedge\nu_1(c_0)] = A_c[\nu_1(t_1)\wedge\nu_1(c_1)] = A_cA_4,
$$
the 3-term Pl\"ucker relation
$$
[\nu_1(t_0)\wedge\nu_1(t_1)][\nu_1(b_0)\wedge\nu_1(c_0)] = [\nu_1(t_0)\wedge\nu_1(b_1)][\nu_1(t_1)\wedge\nu_1(c_0)] + [\nu_1(t_0)\wedge\nu_1(c_0)][\nu_1(b_0)\wedge\nu_1(t_1)]
$$
for face functions reads
$$
A_2\dot{A}'_2 = A_4(A_1 + A_3A_c).
$$
So in accordance with the recipe~\eqref{eq:adot-to-a}, it is not $\dot{A}'_2$ but rather the function 
$$
A_2' =\dot{A}'_2/A_4 = \frac{[\nu_1(t_0)\wedge\nu_1(t_1)]}{[\nu_1(t_{0})\wedge\nu_1(c_{0})]}
$$
that satisfies the combinatorial exchange relation
$$
A_2A_2' = A_1 +A_3A_c
$$
dictated by the quiver in Figure~\ref{fig:adot-examples1}.

%\begin{figure}
%\label{fig:adot-example3-i}
%\includegraphics[scale=.1]{tri-3-ii}
%\caption{Self-folded triangle unrolled to the universal cover for $G=SL_3$. }
%\end{figure}
%
\begin{figure}
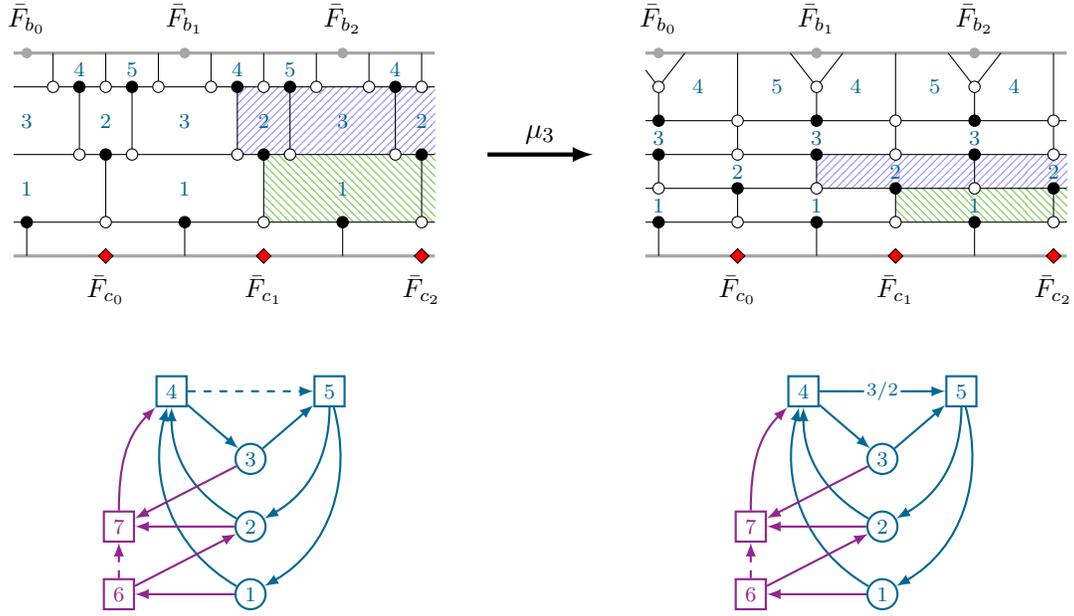

\subfile{fig-adot-example3}
\caption{Mutation at vertex 3 in a self-folded triangle for $G=SL_3$.}
\label{fig:adot-example3}
\end{figure}

Now let us give an example for $G=SL_3$, in which we begin with a single self-folded triangle and perform a square move at the face labelled 3, see Figure~\ref{fig:adot-example3}. In the self-folded triangulation we have
$$
\dot{A}_1 = A_1 = [\nu_1(b_1)\wedge \nu_2(c_1)], \quad \dot{A}_2 = [\nu_1(b_1)\wedge \nu_1(b_0)\wedge \nu_1(c_1)], \quad A_2 = \dot{A}_2/A_1,
$$
and
$$
A_3 = \dot{A}_3 = [\nu_1(c_1)\wedge\nu_2(b_1)].
$$
The square move at face 3 corresponds to a Pl\"ucker relation 
$$
A_3 \dot{A}'_3 = \dot{A}_2(A_{c,2}A_5+A_4),
$$
and so the variable satisfying the cluster exchange relation is
$$
A_3' = \dot{A}'_3/\dot{A}'_2.
$$
In the final cluster, we have
\begin{align}
\label{eq:bw-avar-ex}
\nonumber \dot{A}_1' &= [\nu_1(b_1)\wedge \nu_2(c_1)] \\
\nonumber \dot{A}_2' &= [\nu_1(b_2)\wedge \nu_1(b_1)\wedge \nu_1(c_2)]\\
\dot{A}_3' &= [\nu_1(b_3)\wedge \nu_1(b_2)\wedge \nu_1(b_1)].
\end{align}
\end{example}

The next two results describe the transformation of the coordinates $A_f$ under square moves and shifts at tacked circles. Just as in the case without tacked circles, the new coordinates are related to the old ones by quasi-cluster transformations. Let us illustrate precisely what we mean by this when performing a square move at face $f$ of $\Gamma$ to produce a new graph $\Gamma'$. In this case we have two collections of functions $\{A_\ell(\Gamma)\},~\{A'_\ell=A_\ell(\Gamma')\}$ on the moduli space $\Pc^\diamond_{SL_{n+1},S}$, obtained respectively from the combinatorial data of the graphs $\Gamma,\Gamma'$. If we identify the $\{A_\ell(\Gamma)\}$ with the generators of the coordinate ring $\Tca_{Q_{\Gamma}}$ of the $\Ac$-torus for $Q_\Gamma$ and the $\{A_\ell(\Gamma')\}$ with the corresponding generators of the coordinate ring $\Tca_{Q_{\Gamma'}}$, a quasicluster transformation~\eqref{eq:A-mut-def} from $Q_\Gamma$ to $Q_\Gamma'$ gives a map $\mathbb{C}(\{A_\ell(\Gamma)\})\rightarrow \mathbb{C}(\{A_\ell(\Gamma')\})$, i.e. a birational transformation of the moduli space $\Pc^\diamond_{SL_{n+1},S}$. The assertion of the next Proposition is that there exists a quasicluster transformation such that the corresponding birational map is the identity. In other words, the functions on the moduli space associated to $\{A_\ell(\Gamma)\},~\{A'_\ell=A_\ell(\Gamma')\}$ to the two graphs satisfy the algebraic relations dictated by a certain quasicluster transformation.

% Indeed, we may compare the mutated cluster $K_2$-variable $\mu_f(A_f)$ as defined in formula~\eqref{eq:def-a-mutation} 
%%of Section~\ref{subsec:a-vars}
%with the coordinate function $A_f'$ associated to the face $f$ in the graph $\Gamma$ obtained by applying the square move to $\Gamma$. 
%Given an admissible quadrilateral face $f$ and a tacked circle $c$, define 
%$$
%A_c(f) =
%\begin{cases}
%A_{c,i}, &\text{if} \; f = f^\pm_i(c), \\
%1, &\text{otherwise,}
%\end{cases}
%$$

\begin{prop}
\label{prop:A-move}
Let $\Gamma'$ be the bicolored graph obtained from $\Gamma$ by performing a square move at the face $f$, and let $Q_\Gamma,Q_{\Gamma'}$ be the corresponding quivers. Identify the functions $\{A_\ell(\Gamma)\}$ with the generators $\{A_\ell\}$ of $\Tca_{Q_\Gamma}$, and the $\{A_\ell(\Gamma')\}$ with the generators $\{A_\ell'\}$ of $\Tca_{Q_\Gamma}$. Write $\{A_\ell''\}$ for the cluster $K_2$ coordinate frame associated to the mutated quiver $\mu_f(Q_\Gamma)$. Then as functions on $\Pc^\diamond_{SL_{n+1},S}$, we have
$$
%A_f' =\mu_f(A_f) \prod_{c\in\Cc(S)}A_c(f)
A_f(\Gamma) = (\varsigma_f\circ\mu_f)(A_f(\Gamma)),
$$
where $\mu_f \colon \Tca_{Q_\Gamma}\rightarrow \Tca_{\mu_f(Q_\Gamma)}$ is the classical $K_2$-mutation~\eqref{eq:A-mut-def}, and $\varsigma_f \colon \Tca_{\mu_f(Q_\Gamma)}\rightarrow \Tca_{Q_{\Gamma'}}$ is the quasi-permutation on $\Ac$-tori induced by~\eqref{eq:xi-atrans}:
\beq
\varsigma_f(A''_\ell) =
\begin{cases}
A_f(\Gamma') A_{c,i}(\Gamma'), &\text{if} \;\; \ell = f\in\{f_i^+(c,\Gamma'),f_{i-1}^-(c,\Gamma')\} \;\; \text{for some } 1\leq i\leq n , \\
%\xi'_f + \xi'_{c,i}, &\text{if} \;\; \ell = f=f_{i-1}^-(c,\Gamma') \;\; \text{for some } i<n , \\
A_\ell(\Gamma'), &\text{otherwise}.
\end{cases}
\eeq
% $$
% A'_f =
% \begin{cases}
% \mu_f(A_f) A_{c,i}^{-1} &\text{if} \quad f \in \hc{f_i^-(c), f_{i+1}^+(c)}, \\
% \mu_f(A_f) &\text{otherwise,}
% \end{cases}
% $$

%Then we have $A_g = A'_g$ for any $g \ne f$,
%$$
%A_fA'_f =
%\begin{cases}
%\hr{\prod_{\ell \in I: (e_\ell,e_f)>0} A_\ell^{(e_\ell,e_f)} + \prod_{\ell \in I: (e_f,e_\ell)>0} A_\ell^{(e_f,e_\ell)}}A_{c,i}^{-1}, &\text{if} \quad f \in \hc{f_i^-(c), f_{i+1}^+(c)}, \\
%\prod_{\ell \in I: (e_\ell,e_f)>0} A_\ell^{(e_\ell,e_f)} + \prod_{\ell \in I: (e_f,e_\ell)>0} A_\ell^{(e_f,e_\ell)}, &\text{otherwise.}
%\end{cases}
%$$
\end{prop}
We have a similar result in the case of shifting a tack:
\begin{prop}
\label{prop:A-shift}
Let $\Gamma' = \sigma_c^{\pm1}(\Gamma)$ be the bicolored graph obtained from $\Gamma$ by performing a shift $\sigma_c^{\pm1}$ at a tacked circle $c$. Then we have
$$
A_\ell(\Gamma) = \varsigma_c^{\pm1}(A_\ell(\Gamma)),
%\begin{cases}
%A_\ell \cdot A_{c,i}^{\mp1} &\text{if} \quad \ell \in S_i^\pm(c), \\
%A_\ell &\text{otherwise,}
%\end{cases}
$$
where $\varsigma_c$ is the quasi-permutation on $\Ac$-tori induced by~\eqref{eq:xi-shift-permut}.
\end{prop}
% \blue{[More blue garbage: 
% Then we have
% $$
% A_f' = \frac{\mu_f(A_f)}{A_C(f)}.
% $$]}

Proposition~\ref{prop:A-move} follows from the classical 3-term Pl\"ucker relation in exactly the same way as the corresponding result in~\cite{FG06b} (see also Lemma 3.12 in~\cite{Gon17}), while Proposition~\ref{prop:A-shift} is immediate from the definition of the functions $A_\ell$.

This completes the construction of the cluster ensemble associated to the moduli spaces $(\Pc^\diamond_{SL_{n+1},S},~\Pc^\diamond_{PGL_{n+1},S})$ and the verification of its compatibility with the coordinate systems described at the beginning of this Section. 
%\begin{cor}
%The functions $\hc{A_\ell \,|\, \ell \in I_\Gamma}$ define the structure of a cluster $K_2$-variety on the moduli space $\Pc^\diamond_{SL_{n+1},S}$.
%\end{cor}
%
%\begin{proof}
%First note that the exchange rules for the functions $\hc{A_\ell}$ given in Propositions~\ref{prop:A-move} and~\ref{prop:A-shift} coincide with those for the classical $\Ac$-variables $\hc{Y_{\xi_\ell}}$ of the quiver $Q_\Gamma$ described by the formula~\eqref{eq:A-mut-def} and Propositions~\ref{prop:xi-mut-permut},~\ref{cor:xi-shift-permut}. It only remains to notice that the functions $\hc{A_\ell \,|\, \ell \in I_\Gamma}$ form a transcendence basis in the ring of functions on the moduli space $\Pc^\diamond_{G,S}$. 
%\end{proof}
%
%\begin{cor}
%The quiver $Q_\Gamma$ defines a non-degenerate cluster ensemble on the moduli spaces $\Pc^\diamond_{SL_{n+1},S}$ and $\Pc^\diamond_{PGL_{n+1},S}$.
%\end{cor}
%
%\begin{proof}
%The statement follows from an easily verifiable fact, that the functions $\hc{X_k \,|\, k \in I_\Gamma}$ defined by the ensemble map
%$$
%X_k = \prod_{i \in I_\Gamma} A_i^{b_{ik}},
%$$
%with matrix $b$ given by~\eqref{eq:ensemble-matrix}, form a transcendence basis in the ring of functions on the moduli space $\Pc^\diamond_{PGL_{n+1},S}$, see~\cite{FG06b}.
%\end{proof}
%
We mention one Corollary that we will use later:
\begin{cor}
\label{cor:transform-equal}
Let $\Gamma$ be an ideal bicolored graph on a marked surface $S$, and $\bs\mu_1, \bs\mu_2$ be compositions of square moves and shifts at tacked circles, such that $\bs\mu_1(\Gamma) =\Gamma'= \bs\mu_2(\Gamma)$ and the composite bijections $\varphi_1,\varphi_2:I_\Gamma\rightarrow I_{\Gamma'}$ are identical. Then we have an equality $\bs\mu_1^q = \bs\mu_2^q$ of the corresponding quantum quasi-cluster transformations.
\end{cor}

\begin{proof}
Since the functions $A_\ell$ associated to each face are completely determined by the bicolored graph, and moreover the identification between face sets of $\Gamma,\Gamma'$ induced by $\bs\mu_1,\bs\mu_2$  coincide, we have that $A_\ell(\bs\mu_1(\Gamma)) = A_\ell(\bs\mu_2(\Gamma))$ for all $\ell \in I_\Gamma$. Thus the classical quasi-cluster transformations $\bs\mu_1,\bs\mu_2$ are identical, which by Theorem~\ref{trop-criterion} implies that of the quantum quasi-cluster transformations.
\end{proof}

\begin{remark}
It is crucial here that we consider the bicolored graphs $\Gamma\hookrightarrow S$ up to isotopy and not diffeomorphism/combinatorial equivalence. If we replace the hypothesis that $\bs\mu_1(\Gamma)$ and $\bs\mu_2(\Gamma)$ be conjugate under an isotopy by the weaker hypothesis that they are conjugate under a diffeomorphism of $S$, the conclusion of Corollary~\ref{cor:transform-equal} is false.
\end{remark}

\subsection{Birational action of the braid group}
Suppose that $S$ is not a surface of genus $g$ with a single tacked circle and no other special points. Then associated to each tacked circle $c$ on $S$ there is an action of the braid group $B_{n+1}$ by Poisson birational transformations on the moduli space $\Pc^{\diamond}_{SL_{n+1},S}$ defined as follows. Consider the open locus in $\Pc^{\diamond}_{SL_{n+1},S}$ consisting of twisted decorated local systems $(\mathcal{L},\beta)$ with regular semisimple monodromy $g$ around $c$. Let us trivialize the fiber $(V,\Omega)$ of the associated rank $n$ vector bundle with volume form over the tack by choosing a unimodular basis $(u_i)$ of eigenvectors for $g$: 
$$
gu_i = \lambda_{n+2-i}u_i,\quad  1\leq i\leq n+1,
$$ such that the framed flag $\nu$ at the tack is given by
$$
\nu=(\nu_r)_{r=1}^{n+1} \colon \nu_r = u_1\wedge u_2\wedge\cdots\wedge u_{n+1},
$$
with $\langle \nu_{n+1},\Omega\rangle=1$.

Now we define an action of the generator $\sigma_i\in B_{n+1}$ on this open locus in $\Pc^{\diamond}_{SL_{n+1},S}$ which modifies only the framed flag $\nu$, replacing it by
$$
\sigma_i(\nu)=(\nu_i)_{r=1}^{n+1}\colon \nu_r = \sigma_i(u_1)\wedge \sigma_i(u_2)\wedge\cdots\wedge \sigma_i(u_r)
$$
where
\begin{align}
\label{eq:braids}
\sigma_i({u_j}) = \begin{cases}
u_{j+1}(1-\lambda_{i+1}/\lambda_{i}) \quad & i+j=n+1\\
-u_{j-1}(1-\lambda_{i+1}/\lambda_{i})^{-1} \quad & i+j=n+2\\
 u_j \quad & \text{else.}
\end{cases}
\end{align}
\begin{lemma}
\label{lem:braid}
Suppose that $S$ is not a surface of genus $g$ with a single tacked circle. Then formula~\eqref{eq:braids} defines an action of the braid group by Poisson birational automorphisms of $\Pc^{\diamond}_{SL_{n+1},S}$, which factors through the action of the cluster modular group. The projection $\Pc^{\diamond}_{SL_{n+1},S}\rightarrow \Pc_{PGL_{n+1},S}$ is $B_{n+1}$-equivariant, where the $B_{n+1}$ acts on the base factors through the quotient to the Weyl group $S_{n+1}$.
\end{lemma}
\begin{proof}
It is straightforward to check that the birational transformations~\eqref{eq:braids} satisfy the braid relations, and clear from the definition that the projection to $\Pc_{PGL_{n+1},S}$ is equivariant. We postpone the proof that the braid group acts by cluster automorphisms (and hence is Poisson) until the introduction of isolating coordinate systems for tacked circles in Definition~\ref{def:cpm-isol-quiver} of Section~\ref{sec:isolating-coords}.
\end{proof}

\subsection{Functor from Ptolemy to cluster groupoid}
\label{subsec:Pt-Cl}

Now we can construct the functor
\beq
\label{eq:Pt-to-Cl}
\Qc \colon \widehat\Pt(S) \longra \Cl_{\bs Q}.
\eeq
On the level of objects, for a triangulation $\tri$ we set
$$
\Qc(\tri) = Q_{\tri},
$$
where $Q_{\tri}$ is the quiver for the bipartite graph $\Gamma_{\tri}$ associated to $\tri$, defined in Section~\ref{subsec:quivers-from-graphs}.

Recall that arrows in $\widehat\Pt(S)$ are formal composites $A_\gamma\circ F_{\tri_1,\tri_2}$, where $\gamma\in\Gamma_S$ and $\tri_1,\tri_2$ are triangulations, while the morphisms in $\Cl_{\bs Q}$ are quasi-cluster transformations. The  quasi-cluster transformation assigned to the arrow $F_{\tri_1,\tri_2}$ labelled by a pair $(\tri_1,\tri_2)$ is defined as follows. 

First suppose that $\tri_1$ is obtained from $\tri_2$ by flipping a single diagonal $d$, and let us abbreviate the the corresponding arrow as
$$
F_d := F_{\tri_1,\tri_2}.
$$Since our triangulations are defined up to isotopy, the diagonal $d$ is uniquely defined. Let us present a quasi-cluster quiver transformation bringing $Q_{\tri_2}$ to $Q_{\tri_1}$ by transforming the bicolored graph $\Gamma_{\tri_2}$ into $\Gamma_{\tri_1}$ via a sequence of square moves and shifts at tacked circles.
%
%In order to define $\Qc$ on morphisms let us first describe quasi-cluster transformations corresponding to the flips $F_d$.

%First, let $d$ be a diagonal, shared by two non-special triangles $\Delta_1, \Delta_2$ in $\tri$, and $\boxslash \subset S$ be the subsurface, covered by them. Draw the triangles $\Delta_1,\Delta_2$ as shown on Figure~\ref{fig:flip}, omitting frozen vertices $v_{c,i}$ with $c \in C(S)$. The sides of $\boxslash$ are parallel to the $x$- and $y$-axes, and the diagonal $d$ is connecting the origin to the point $(n,n)$. The nodes of the quiver $Q_\boxslash$ are labelled by their coordinates and are placed at the integer points inside $\boxslash$ excluding the vertices of the latter. Given a vertex $(i,j)$, we consider a generalized permutation
%$$
%\vartheta_{(i,j)} \colon Y_\ell \longmapsto Y_{\sigma_{(i,j)}(\ell)},
%$$
%where $\sigma_f$ is defined in~\eqref{eq:sigma-f}.
If the diagonal $d$ is shared by two non-special triangles,  number the vertices of the subquiver corresponding to the quadriateral as shown on Figure~\ref{fig:flip}. Consider for each $1 \le k \le n$ the composition
\beq
\label{eq:mu-dk}
\mu_{d;k} = \prod_{(i,j) \in R(k)} \hat\mu_{(i,j)}^q,
\eeq
where $\hat\mu^q_{(i,j)}$ is the quasi-cluster transformation corresponding to the quiver transformation~\eqref{eq:mu-hat-factor}, and the set $R(k)$ consists of pairs $(i,j)$, such that
\begin{align*}
i-j &\in \hc{1-k, 3-k, \dots, k-1}, \\
i+j &\in \hc{k+1, k+3, \dots, 2n+1-k}.
\end{align*}
Note that all $k(n-k)$ transformations $\hat\mu^q_{(i,j)}$ in~\eqref{eq:mu-dk} commute, and therefore $\mu_{d;k}$ is well-defined. Moreover, each corresponds to performing a square move at the corresponding face of the bicolored graph, and the graph  $\mu_{d;n} \dots \mu_{d;1}(\Gamma_{\tri_2})$ coincides with $\Gamma_{\tri_1}$.
Hence we get a quasi-cluster transformation
$$
%\mu_d \colon \Tc_{\tri}^q \longra \Tc_{F_d(\tri)}^q, \qquad 
\mu_d = \mu_{d;n} \dots \mu_{d;1} \colon \Tca_{\tri_2} \dashrightarrow \Tca_{\tri_1},
$$
and we define
$$
\mathcal{Q}(F_{d}) = \mu_d.
$$
The transformation $\mu_d$ is a minor generalization of the one described in~\cite{FG06b}. The latter can be recovered from $\mu_d$ by replacing each factor $\hat \mu^q_{(i,j)}$ with $\mu^q_{(i,j)}$.

\begin{figure}[h]
\subfile{flip.tex}
\caption{A flip.}
\label{fig:flip}
\end{figure}

On the other hand, suppose $c$ is a tacked circle contained in the special triangle $\Delta_c$, and $d_\pm$ be the remaining two sides of $\Delta_c$, so that $d_+$ follows $d_-$ as we go around $c$ in the positive direction. To the flips $F_{d_\pm}$ we 
associate the quasi-permutations
$$
\mu_{d_\pm} = \varsigma_c^{\pm1} \colon \Tc_{\tri}^q \longra \Tc_{F_{d_\pm}(\tri)}^q
$$
defined by the quiver quasi-permutation~\eqref{eq:xi-shift-permut}. 
This completes the definition of $\mathcal{Q}(F_{\tri_1,\tri_2})$ for two triangulations related by a single flip. 

Now given a general pair $\tri_1,\tri_2$, choose any sequence $\bs d = (d_1, \dots d_k)$ of flips such that $\tri_1 = F_{d_k}\cdots F_{d_1}(\tri_2)$. Given any two such sequences $\bs d,\bs d'$, both cluster transformations
$$
\Qc(F_{\bs d}) = \Qc(F_{d_k})\circ\cdots\circ \Qc(F_{d_1}), \quad \Qc(F_{\bs d'}) = \Qc(F_{d'_k})\circ\cdots\circ \Qc(F_{d'_1})
$$
arise from composites of square moves and shifts which transform $\Gamma_{\tri_2}$ into $\Gamma_{\tri_1}$. Moreover, by Lemma~\ref{lem:no-automorphisms} the bijections on quiver label sets $I_{\tri_1}\rightarrow I_{\tri_2}$ induced by both cluster transformations must be identical. Hence we can apply Corollary~\ref{cor:transform-equal} to conclude that the cluster transformations $\Qc(F_{\bs d}),\Qc(F_{\bs d'}) $ are equal. This proves that the assignment
$$
\Qc(F_{\tri_1,\tri_2}) = \Qc(F_{d_k})\circ\cdots\circ \Qc(F_{d_1})
$$
is well-defined and functorial. 

Now given a general morphism $A_\gamma F_{\tri_1,\tri_2}$, we define
%Recall that morphisms in $\widehat\Pt(S)$ are formal compositions $A_\gamma \circ F_{\tri_2,\tri_1}$, where $\gamma$ is an element of the mapping class group $\Gamma_S$, and $F_{\tri_2,\tri_1}$ is the unique morphism in $\Pt(S)$ with source $\tri_1$ and target $\tri_2$. Consider a sequence of flips $F_{\bs d} = F_{d_k} \dots F_{d_1}$ such that $\Delta_2 = F_{\bs d}(\tri_1)$, and set
\beq
\label{eq:Qc-flip}
\Qc(A_\gamma F_{\tri_1,\tri_2}) = \varsigma_\gamma \circ \Qc(F_{\tri_1,\tri_2}),
\eeq
where $\varsigma_\gamma$ is the quasi-permutation morphism corresponding to the identification of graphs
$
\gamma:\Gamma_{\tri_1}\simeq \Gamma_{\gamma(\tri_1)}
$
induced by the diffeomorphism $\gamma$. This definition evidently respects the multiplication rule~\eqref{eq:enhanced-mult}, and thus we obtain a functor as promised.

In particular, if we fix a quiver $Q$ then by~\ref{lem:no-automorphisms} we get a homomorphism
$$
\Gamma_S \rightarrow \Aut_{\Cl_{\bs Q}}(Q)
$$
from the mapping class group of $S$ to the automorphism group of the object $Q$ in the cluster modular groupoid.

\begin{example}
Let $S$ be a torus with a single tacked circle $c$, so that $\Gamma_S\simeq B_3$ is the universal central extension of $SL(2,\mathbb{Z})$.  We compute the corresponding representation associated to the group $G=SL_2$. 

\begin{figure}
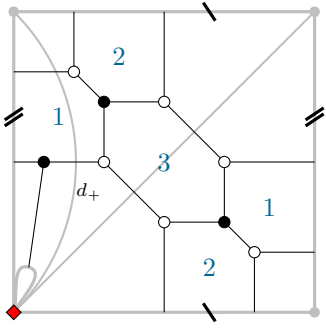

\subfile{fig-init-tri-tor}
\caption{Triangulation $\tri$ of a torus with a tacked circle.}
\label{fig:init-tri-tor}
\end{figure}

In Figure~\ref{fig:init-tri-tor} we illustrate an ideal triangulation $\tri$ of $S$ and corresponding black-white graph $\Gamma_{\tri}$. In the figure, opposite pairs of edges of the rhombus are identified. For the ease of writing formulas, let us label the proper faces of $\Gamma_{\tri}$ by the numbers $1,2,3$. We stress that this labelling is \emph{not} part of the data of the quiver $Q(\Gamma_{\tri})$, and the numbers are used only as typographically convenient shorthands for the elements of the actual label set of $Q(\Gamma_{\tri})$. We graphically represent the skew-pairing between the basis vectors of the quiver $Q(\Gamma_{\tri})$ in Figure~\ref{fig:quiv-tor}. If we perform the flip at the edge $d_+$ of the special triangle in $\tri$, we obtain the triangulation $\tri'$ and graph $\Gamma_{\tri'}$ shown in Figure~\ref{fig:delt-tri}.

\begin{figure}
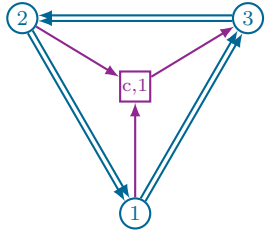

\subfile{quiv-tor.tex}
\caption{Graphical representation of skew-pairing for $Q_{\tri}$.}
\label{fig:quiv-tor}
\end{figure}

\begin{figure}
\subfile{delt-tri.tex}
\caption{Black-white graph $\Gamma_{\tri'}$.}
\label{fig:delt-tri}
\end{figure}

 We indicate the induced bijection $\varphi:I_{\Gamma_{\tri}}\rightarrow I_{\Gamma_{\tri'}}$ by again labelling the proper faces of $\Gamma_{\tri'}$ by the same alphabet $\{1,2,3\}$. Then in accordance with Lemma~\ref{eq:sigma-pm}, the corresponding classical cluster transformation reads
\begin{align*}
\varsigma_c^+\colon Y_{(c,1)}\mapsto Y'_{(c,1)}(Y_2')^{-1},\quad Y_{i}\mapsto Y_i',\quad i=1,2,3.
\end{align*}
But now observe that the element $\rho$ of the mapping class group of $S$ induced by the 120 degree counterclockwise rotation satisfies $\rho(\tri')=\tri$ and thus returns us to the original bicolored graph. Hence it induces a quasi-permutation morphism
$$
\varsigma_\rho\colon Y_1'\mapsto Y_3,\quad Y_3'\mapsto Y_2, \quad Y_2'\mapsto Y_1,\quad Y'_{c,1}\mapsto Y_{c,1}
$$
So the birational automorphism of $\Tc_{Q(\tri)}$ by which the mapping class $\rho$ acts is
$$
\rho_{Q_{\tri}} = \varsigma_\rho \circ \varsigma_c^+,
$$
$$
\rho_{Q_{\tri}}(Y_1)=Y_3,\quad \rho_{Q_{\tri}}(Y_2)=Y_1,\quad \rho_{Q_{\tri}}(Y_3)=Y_2, \quad \rho_{Q_{\tri}}(Y_{(c,1)}) = \frac{Y_{(c,1)}}{Y_1}.
$$
Note that we have
$$
\rho^3_{Q_{\tri}}(Y_{(c,1)}) = \frac{Y_{(c,1)}}{Y_1Y_2Y_3} = \frac{Y_{(c,1)}}{A_{(c,1)}}, \quad \rho^3_{Q_{\tri}}(Y_{i}) =Y_i,\quad i=1,2,3.
$$

On the other hand, consider the triangulation $\tri''$ obtained from the initial one $\tri$ by flipping the diagonal dual to the face of $\Gamma$ numbered 3 in Figure~\ref{fig:init-tri-tor}, which amounts to performing the square move at the face $f=f_2^+(\Gamma_{\tri})$. We illustrate the graph $\Gamma_{\tri''}$ in Figure~\ref{fig:tau-tri} where we again indicate the induced bijection $\varphi:I_{\Gamma_{\tri}}\rightarrow I_{\Gamma_{\tri''}}$ by labelling the proper faces of $\Gamma_{\tri''}$ by the same alphabet $\{1,2,3\}$. 

\begin{figure}
\subfile{fig-tau-tri.tex}
\caption{Black-white graph $\Gamma_{\tri''}$.}
\label{fig:tau-tri}
\end{figure}

The corresponding cluster transformation $\hat{\mu}_{f}$ reads
$$
\hat{\mu}_{f}\colon Y_1\mapsto Y'_1/(1+1/Y_3')^2,\quad Y_2\mapsto Y'_2(1+Y_3')^2,\quad Y_3\mapsto 1/Y'_3,\quad Y_{(c,1)}\mapsto Y_{(c,1)}'/(1+1/Y_3').
$$
This time we observe that the element $\tau$ of the mapping class group given by the Dehn twist along the $(0,1)$-curve on the torus satisfies $\tau(\tri'')=\tri$, and so induces a quasi-permutation morphism
$$
\varsigma_\tau\colon Y_2'\mapsto Y_3,\quad Y_3'\mapsto Y_2,\quad Y_1'\mapsto Y_1,\quad Y'_{(c,1)}\mapsto Y_{(c,1)}.
$$
So the birational automorphism of $\Tc_{Q(\tri)}$ by which the mapping class $\tau$ acts is
$$
\tau_{Q_{\tri}} = \varsigma_\tau \circ \hat{\mu}_f,\qquad 
$$
%$$
%\tau_{Q_{\tri}}(Y_1) = Y_1(1+Y_3)^2, \quad \tau_{Q_{\tri}}(Y_2) = 1/Y_3, \quad \tau_{Q_{\tri}}(Y_3) = Y_2(1+1/Y_3)^2,\quad Y_{(c,1)}\mapsto Y_{(c,1)}(1+Y_3).
%$$
$$
Y_1\mapsto Y_1/(1+1/Y_2)^2,\quad Y_2\mapsto Y_3(1+Y_2)^2,\quad Y_3\mapsto 1/Y_2,\quad Y_{(c,1)}\mapsto Y_{(c,1)}/(1+1/Y_2)
$$
The element 
$$
S = \rho_{Q_{\tri}}\circ\tau_{Q_{\tri}}^{-1}
$$
satisfies
$$
S^2 = \rho_{Q_{\tri}}^3,
$$
so these birational transformations indeed give a representation of the braid group $B_3$.
\end{example}

\section{Special cluster charts for isolating cylinders}

\label{sec:isolating-coords}

%Consider a marked surface $S$ with a closed simple curve $c$. Cutting $S$ along $c$ we obtain a surface $S'$ with a pair of tacked circles $c_\pm$ and a homeomorphism $\phi \colon c_+ \to c_-$, allowing us to recover $S$ by gluing the tacked circles of $S'$ along $\phi$. Figure~\ref{fig:local-gluing-surfaces} illustrates this setup for the case, when $S$ is a cylinder with a single marked point on each of its boundary components.

We will construct the algebra isomorphism $\eta_{c}$ in Theorem~\ref{thm:main-intro} using special cluster coordinates on the moduli spaces $\Pc_{G,S}$ and $\Pc^\diamond_{G,S'}$ (for $G$ of type $A_n$) which are adapted to a $c$-isolating triangulation $\tri$ and its image $\tri'$ under the cutting functor. When $n>1$, these special charts differ from those associated to the triangulations $\tri,\tri'$, but this difference is concentrated only at the cluster coordinates inside the isolating cylinders. The main advantage of the charts introduced here is that they allow for a ``separation of variables'', simplifying to the extent possible the action of the Dehn twist $\tau_c$ by cluster transformations.

\subsection{Cyclic Weyl words and quiver mutations}
\label{subsec:weyl-words}

We start by recalling a dictionary between (cyclic) double Weyl words and their transformations on the one hand and certain classes of quivers and quiver mutations on the other. We refer the reader to~\cite{BFZ05, FG06a} for further details and relation to the cluster coordinate charts on the double Bruhat cells. Our exposition follows that in~\cite{SS17}.

Consider two alphabets $\Agt_\pm = \hc{\pm 1, \dots, \pm n}$ and the union $\Agt = \Agt_+ \sqcup \Agt_-$. We introduce the notation
\beq
\label{eq:bar}
\overline a = a \mp (n + 1) \qquad\text{for}\qquad a \in \Agt_\pm,
\eeq
and note that $\overline{\Agt}_\pm = \Agt_\mp$. A \emph{double Weyl word} is an element of the free monoid generated by $\Agt$ modulo the \emph{shuffle relations}
\beq
\label{eq:ab}
\dots,a,b,\hdots = \dots,b,a,\dots \qquad \text{unless} \qquad a \in \hc{\pm b, b\pm1}.
\eeq
Here and in what follows we separate letters of a double Weyl word by commas. A \emph{cyclic double Weyl word} is a double Weyl word considered up to a cyclic shift of its letters. We will make use of the following transformations on (cyclic) double Weyl words: the \emph{braid move}
\beq
\label{eq:braid}
\dots,a,b,a,\hdots \longleftrightarrow \dots,b,a,b,\dots \qquad\text{for}\qquad a=b\pm1,
\eeq
the \emph{shuffle}
\beq
\label{eq:shuffle}
\dots,a,-a,\hdots \longleftrightarrow \dots,-a,a,\dots
\eeq
and the \emph{merge}
\beq
\label{eq:merge}
\dots,a,a,\hdots \longrightarrow \dots,a,\dots
\eeq
The \emph{length} $l(\i)$ of a (cyclic) double Weyl word $\i$ is the number of its letters. A (cyclic) word is \emph{(cyclically) irreducible} if its length cannot be reduced by applying relations~\eqref{eq:ab} and transformations~\eqref{eq:braid} -- \eqref{eq:merge}.

We now construct a quiver for each double Weyl word. Consider $n$ horizontal lines numbered 1 to $n$. For each letter $a$ in the alphabet $\Agt$ define an elementary quiver $Q_a$ as shown on Figure~\ref{fig-letters} (omitting all vertices on lines labelled 0 and $n+1$ along with all adjacent arrows, so that the quivers $Q_{\pm1}$ and $Q_{\pm n}$ have 3 vertices each). Given a word $\i$, we read it left to right and draw quivers $Q_a$ for each $a \in \Agt$. We then amalgamate adjacent vertices in each of the rows if they belong to different elementary quivers and denote the result by $Q_\bi$. It is easy to see that the quiver $Q_\bi$ is invariant under applying relation~\eqref{eq:ab} to the sequence of letters of $\bi$, and is hence well-defined.

\begin{figure}[h]
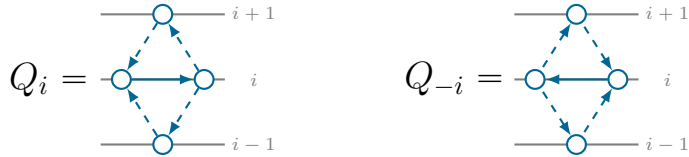

\subfile{fig-letters.tex}
\caption{Quivers corresponding to single letters.}
\label{fig-letters}
\end{figure}

Now, let $\bi$ be a cyclic word such that for each $1 \le i \le n$, there is at least two letters $a,b$ in $\bi$ with $\hm{a} = \hm{b} = i$. We construct a quiver for $\bi$ in a similar way to the one for the non-cyclic words, except now we replace the $n$ lines by $n$ circles on a cylinder. Equivalently, we first construct a quiver for any non-cyclic representative of $\bi$ and then amalgamate the left-most and the right-most vertex on each line. For example, the mutable subquiver in the left pane of Figure~\ref{fig-Qbox} is the quiver $Q_\i$, where $\i = (\i_-, \i_+)$ is a cyclic word and
\begin{align}
\label{iw0}
\i_+ &= (4,3,2,1,4,3,2,4,3,4), \\
\label{iw0bar}
\i_- &= (\overline 4, \overline 3, \overline 2, \overline 1, \overline 4, \overline 3, \overline 2, \overline 4, \overline 3, \overline 4).
\end{align}

The transformations~\eqref{eq:braid}, \eqref{eq:shuffle} now correspond to quiver mutations, see Figures~\ref{fig-braid} and~\ref{fig-pm} where mutation is performed at the pink vertex. Similarly, the transformation~\ref{eq:merge} corresponds to the quiver mutation on Figure~\ref{fig-21} after we have erased the pink vertex in the quiver on the right.

\begin{figure}[h]
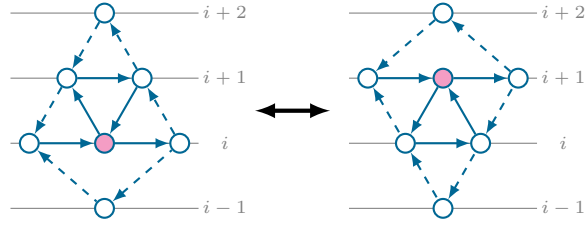

\subfile{fig-braid.tex}
\caption{Braid move $(i,i+1,i) \leftrightarrow (i+1,i,i+1)$.}
\label{fig-braid}
\end{figure}

\begin{figure}[h]
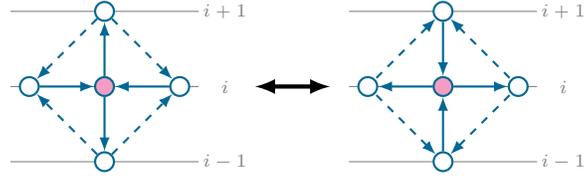

\subfile{fig-shuffle}
\caption{Shuffle $(i,-i) \leftrightarrow (-i,i)$.}
\label{fig-pm}
\end{figure}

\begin{figure}[h]
\subfile{fig-merge.tex}
\caption{Merge $(1,1) \to (1)$.}
\label{fig-21}
\end{figure}

\subsection{Reduction of cyclic Weyl words}
\label{subsec-Weyl-comb}

In this subsection we describe certain reductions of cyclic Weyl words, which were explained to us by M.~Gekhtman. In what follows these reductions yield sequences of quiver mutations which we use over and over again.

Given a letter $a \in \Agt$, let introduce a notation
$$
a^* = -\overline{a},
$$
and let
$$
-\bi = (-i_1, \dots, -i_p), \qquad \overline\bi = (\overline i_1, \dots, \overline i_p), \qquad \bi^* = (i_1^*, \dots, i_p^*), \qquad \bi^t = (i_p, \dots, i_1)
$$
for any word Weyl word $\bi = (i_1, \dots, i_p)$. We also define
$$
\i_{[i,j]} =
\begin{cases}
(i,i+1,\dots,j) &\text{if} \quad 0 < i \le j, \\
(i,i-1,\dots,j) &\text{if} \quad 0 < j \le i, \\
0 &\text{otherwise.}
\end{cases}
$$
so that, for example, $\overline\i_{[i,j]} = \i_{[\overline j, \overline i]}$. Now, let us consider words
$$
\i_c = \i_{[1,n]}, \qquad \i_{w_0} = \hr{\i_{[n,1]},\i_{[n,2]},\dots,\i_{[n,n]}},
$$
where notations $c$ and $w_0$ stand respectively for a Coxeter and the longest word in the Weyl group of $SL_{n+1}$. For $n=4$, the words $\i_{w_0}$ and $\overline \i_{w_0}$ are given by formulas~\eqref{iw0} and~\eqref{iw0bar} respectively. Given a collection $u_1, \dots, u_k$ of symbols in $\hc{c,w}$ we set
\beq
\label{eq:word-notation}
\bi_{u_1,\dots,u_k} = (\bi_{u_1}, \dots, \bi_{u_k}) \qquad\text{and}\qquad \bi_{-u} = -\bi_u, \qquad \bi_{\overline u} = \overline\bi_u, \qquad \bi_{u^*} = \bi_u^*.
\eeq
In particular, in our notations
$$
\isf = \hr{\overline\i_{w_0},\i_c^*,\i_c},
$$
and the following equalities hold:
$$
\bi^t_{w_0} = \bi_{w_0}, \qquad \bi^t_c = \bi_c^*.
$$
We also observe that the word $\icox$ is \emph{Coxeter,} that is it contains every letter from $\Agt$ exactly once. It is easy to see that every Coxeter word is cyclically irreducible, and that every two cyclic Coxeter words are related via shuffles and cyclic shifts. In the remainder of this subsection we describe three cyclic Weyl word reductions
\begin{align}
\label{phi1}
& \Phi_1 \colon \ifull \longmapsto \ihalf, \\
\label{phi2}
& \Phi_2 \colon \ihalf \longmapsto \icox, \\
\label{phi3}
& \Phi_3 \colon \isf \longmapsto \isym.
\end{align}

\subsubsection{Word transformation $\Phi_1$}

First, set
$$
\ifull^{(j,k)} = \hr{\overline \i_{w_0},\bi_{[1,j]},\i_{[n-k,j+1]},\i_{[n,j+2]},\dots,\i_{[n,n]} },
$$
so that
$$
\ifull^{(0,0)} = \ifull, \qquad \ifull^{(j,n-j-1)} = \ifull^{(j+1,0)}, \qquad \ifull^{(n-1,0)} = \ihalf.
$$
We factor the transformation $\Phi_1$ into a composite of $n-1$ ``waves'',
\beq
\label{eq:phi1-factor}
\Phi_1 = \Phi_1^{(n-1)} \dots \Phi_1^{(1)}
\qquad
\Phi_1^{(j)} = \Phi_1^{(j,n-j)} \dots \Phi_1^{(j,1)},
\eeq
so that
$$
\Phi_1^{(j+1,k+1)} \colon \ifull^{(j,k)} \longmapsto \ifull^{(j,k+1)},
$$
and refer to $\Phi_1^{(j)}$ as the $j$-th ``wave'' of $\Phi_1$, and to $\Phi_1^{(j,k)}$ as the $k$-th step of the $j$-th wave.

Let us now describe the transformation $\Phi_1^{(j+1,k+1)}$. Using that
$$
\ifull^{(j,k)} = \hr{\overline \i_{w_0},\bi_{[1,j]},n-k,\i_{[n-k-1,j+1]},\i_{[n,j+2]},\dots,\i_{[n,n]}},
$$
we first commute the letter $n-k$ all the way to the left through the word $\overline \i_{w_0}$, which consists of applying equivalences~\ref{eq:ab} together with exactly $k+1$ shuffles $(k-n,n-k) \mapsto (n-k,k-n)$. Then, we perform the cyclic shift that erases the letter $n-k$ on the very left and writes it on the far right, to obtain
\beq
\label{eq:temp1-words}
\hr{\overline \i_{w_0},\bi_{[1,j]},\i_{[n-k-1,j+1]},\i_{[n,j+2]},\dots,\i_{[n,n]},n-k}.
\eeq
The part of the latter word which reads $\hr{ \i_{[n,j+2]},\dots,\i_{[n,n]},n-k }$ can be rewritten as follows using only~\eqref{eq:ab}:
\begin{multline}
\hr{\i_{[n,j+2]},\dots,\i_{[n,n+1-k]},\i_{[n,n+2-k]},\dots,\i_{[n,n]},n-k} = \\
\label{eq:temp2-words}
\hr{\i_{[n,j+2]},\dots,\i_{[n,n+1-k]},n-k,\i_{[n,n+2-k]},\dots,\i_{[n,n]}}.
\end{multline}
We shall now restrict our attention to the subword
\beq
\label{eq:temp3-words}
\hr{\i_{[n,n-k]},\i_{[n,n+1-k]},n-k}.
\eeq
Using~\eqref{eq:ab} again we rewrite it as
$$
\hr{\i_{[n,n+1-k]},\i_{[n,n+2-k]},n-k,n+1-k,n-k},
$$
and then apply a braid move to get
$$
\hr{\i_{[n,n+1-k]},\i_{[n,n+2-k]},n+1-k,n-k,n+1-k}.
$$
The latter is a particular case of the word
$$
\hr{\i_{[n,s]},\i_{[n,s+1]},s,\i_{[s-1,t-1]},\i_{[s,t]}}
$$
for $s=t=n+1-k$. We can once again rewrite it as
$$
\hr{\i_{[n,s+1]},\i_{[n,s+2]},s,s+1,s,\i_{[s-1,t-1]},\i_{[s,t]}}
$$
and use a braid move to get
$$
\hr{\i_{[n,s+1]},\i_{[n,s+2]},s+1,s,s+1,\i_{[s-1,t-1]},\i_{[s,t]}} = \hr{\i_{[n,s+1]},\i_{[n,s+2]},s+1,\i_{[s,t-1]},\i_{[s+1,t]}}.
$$
Proceeding by induction on $s$ we arrive at
$$
\hr{n,n-1,n,n-1,\i_{[n-2,t-1]},\i_{[n-1,t]}},
$$
use a braid move once again, to get
$$
\hr{n,n,n-1,n,\i_{[n-2,t-1]},\i_{[n-1,t]}} = \hr{n,n,\i_{[n-1,t-1]},\i_{[n,t]}},
$$
and finally, apply the merge $(n,n) \mapsto n$ to obtain
$$
\hr{n,\i_{[n-1,t-1]},\i_{[n,t]}} = \hr{\i_{[n,t-1]},\i_{[n,t]}}.
$$
Recall that $t = n+1-k$ and thus the word~\eqref{eq:temp3-words} has been reduced to $\hr{\i_{[n,n-k]},\i_{[n,n+1-k]}}$. Re-inserting it into the word~\eqref{eq:temp2-words} and then~\eqref{eq:temp1-words} we arrive at
$$
\ifull^{(j,k+1)} = \hr{\overline \i_{w_0},\bi_{[1,j]},\i_{[n-k-1,j+1]},\i_{[n,j+2]},\dots,\i_{[n,n]} }.
$$
To summarize, the transformation $\Phi_1^{(j+1,k+1)}$ consists of
\begin{itemize}
\item $k+1$ shuffles $(k-n,n-k) \mapsto (n-k,k-n)$;
%\item a cyclic shift $(n-k,\i) \mapsto (\i,n-k)$;
\item $k$ braid moves $((s+1)^*,s^*,(s+1)^*) \mapsto (s^*,(s+1)^*,s^*)$ for $s=k,\dots,2,1$;
\item a merge $(n,n) \mapsto n$.
\end{itemize}

\begin{example}
We spell out the transformation $\Phi_1$ in case $n=4$. Below we write $\bi \stackrel{t}\mapsto \bi'$ if, up to equivalences~\eqref{eq:ab}, the transformation from $\bi$ to $\bi'$ is $t$, where $t=b$ denotes a braid move~\eqref{eq:braid}, $t=s^k$ stands for $k$ shuffles~\eqref{eq:shuffle}, and $t=m$ indicates a merge~\eqref{eq:merge}. Then we have
\begin{align*}
\Phi_1^{(1,1)} \colon \quad
\ifull^{(0,0)} &= (\overline \i_{w_0},4,3,2,1,4,3,2,4,3,4) \stackrel{s}\longmapsto (4,\overline \i_{w_0},3,2,1,4,3,2,4,3,4) \\
&= (\overline \i_{w_0},3,2,1,4,3,2,4,3,4,4) \stackrel{m}\longmapsto (\overline \i_{w_0},3,2,1,4,3,2,4,3,4) = \ifull^{(0,1)}, \\
\Phi_1^{(1,2)} \colon \quad
\ifull^{(0,1)} &= (\overline \i_{w_0},3,2,1,4,3,2,4,3,4) \stackrel{s^2}\longmapsto (3,\overline \i_{w_0},2,1,4,3,2,4,3,4) \\
&= (\overline \i_{w_0},2,1,4,3,2,4,3,4,3) \stackrel{b}\longmapsto (\overline \i_{w_0},2,1,4,3,2,4,4,3,4) \\
&\stackrel{m}\longmapsto (\overline \i_{w_0},2,1,4,3,2,4,3,4) = \ifull^{(0,2)}, \\
\Phi_1^{(1,3)} \colon \quad
\ifull^{(0,2)} &=  (\overline \i_{w_0},2,1,4,3,2,4,3,4) \stackrel{s^3}\longmapsto (2,\overline \i_{w_0},1,4,3,2,4,3,4) \\
&= (\overline \i_{w_0},1,4,3,2,4,3,4,2) \stackrel{b}\longmapsto (\overline \i_{w_0},1,4,3,4,3,2,3,4) \\
& \stackrel{b}\longmapsto (\overline \i_{w_0},1,4,4,3,4,2,3,4) \stackrel{m}\longmapsto (\overline \i_{w_0},1,4,3,2,4,3,4) = \ifull^{(0,3)} = \ifull^{(1,0)}.
\end{align*}
The second wave $\Phi_1^{(2)}$ reads
\begin{align*}
\Phi_1^{(2,1)} \colon \quad
\ifull^{(1,0)} &= (\overline \i_{w_0},1,4,3,2,4,3,4) \stackrel{s}\longmapsto (4,\overline \i_{w_0},1,3,2,4,3,4) \\
&= (\overline \i_{w_0},1,3,2,4,3,4,4) \stackrel{m}\longmapsto (\overline \i_{w_0},1,3,2,4,3,4) = \ifull^{(1,1)}, \\
\Phi_1^{(2,2)} \colon \quad
\ifull^{(1,1)} &= (\overline \i_{w_0},1,3,2,4,3,4) \stackrel{s^2}\longmapsto (3,\overline \i_{w_0},1,2,4,3,4) \\
&= (\overline \i_{w_0},1,2,4,3,4,3) \stackrel{b}\longmapsto (\overline \i_{w_0},1,2,4,4,3,4) \\
&\stackrel{m}\longmapsto (\overline \i_{w_0},1,2,4,3,4) = \ifull^{(1,2)} = \ifull^{(2,0)}.
\end{align*}
And finally, the wave $\Phi_1^{(3)}$ becomes
\begin{align*}
\Phi_1^{(3,1)} \colon \quad
\ifull^{(2,0)} &= (\overline \i_{w_0},1,2,4,3,4) \stackrel{s}\longmapsto (4,\overline \i_{w_0},1,2,3,4) \\
&= (\overline \i_{w_0},1,2,3,4,4) \stackrel{m}\longmapsto (\overline \i_{w_0},1,2,3,4) = \ifull^{(2,1)} = \ifull^{(3,0)}.
\end{align*}
\end{example}

\subsubsection{Word transformation $\Phi_2$}
\label{subsec:phi2}

For $0 \le j \le n-1$ and $0 \le k \le n-1-j$ we now set
%$$
%\ihalf^{(j,k)} = \hr{\i_c,\overline\bi_{[1,j]}, \overline \i_{[n-k,j+1]}, \overline \i_{[n,j+2]},\dots,\overline \i_{[n,n]}},
%$$
$$
\ihalf^{(j,k)} = \hr{\overline\bi_{[n,n]}, \ldots, \overline \i_{[j+2,n]}, \overline \i_{[j+1,n-k]}, \overline \i_{[j,1]},\i_c},
$$
so that $\ihalf^{(j,n-1-j)} = \ihalf^{(j+1,0)}$ and $\ihalf^{(n-1,0)} = \icox$. We then factor
\beq
\label{eq:phi2-factor}
\Phi_2 = \Phi_2^{(n-1)} \dots \Phi_2^{(1)},
\qquad
\Phi_2^{(j)} = \Phi_2^{(j,n-j)} \dots \Phi_2^{(j,1)},
\eeq
so that
$$
\Phi_2^{(j+1,k+1)} \colon \ihalf^{(j,k)} \longmapsto \ihalf^{(j,k+1)}.
$$
%The steps of $\Phi_2$ are similar to those of $\Phi_1$ with the following minor alterations: the subwords $\bi_+$ in the alphabet $\Agt_+$ are replaced with $\overline\bi_+$, and the subword $\overline\bi_{w_0}$ is replaced with $\bi_c$. The latter implies that it takes just one shuffle to move a letter $\overline{a} \in \Agt_-$ through $\bi_c$. Thus, the transformation $\Phi_2^{(j+1,k+1)}$ consists of
The steps of $\Phi_2$ are similar to those of $\Phi_1$, namely, the transformation $\Phi_2^{(j+1,k+1)}$ consists of
\begin{itemize}
\item a shuffle $(-k-1,k+1) \mapsto (k+1,-k-1)$;
\item $k$ braid moves $(-s-1,-s,-s-1) \mapsto (-s,-s-1,-s)$ for $s=k,\dots,2,1$;
\item a merge $(-1,-1) \mapsto -1$.
\end{itemize}
%\begin{itemize}
%\item a shuffle $(k,-k) \mapsto (-k,k)$;
%%\item a cyclic shift $(-k,\i) \mapsto (\i,-k)$;
%\item $k$ braid moves $(-s-1,-s,-s-1) \mapsto (-s,-s-1,-s)$ for $s=k,\dots,2,1$;
%\item a merge $(-1,-1) \mapsto -1$.
%\end{itemize}

\subsubsection{Word transformation $\Phi_3$}

Let us set
$$
\isf^{(j,k)} = \hr{\overline\i_{w_0},\i_c^*,\bi_{[1,n-j-1]},n-j+k},
$$
so that $\isf^{(0,0)} = \isf$ and $\isf^{(n,0)} = \bi_{\overline{w}_0,c^*}$, and factor
\beq
\label{eq:phi3-factor}
\Phi_3 = \Phi_3^{(n)} \dots \Phi_3^{(1)}, \qquad \Phi_3^{(j)} = \Phi_3^{(j,j)} \dots \Phi_3^{(j,1)}.
\eeq
The transformation $\Phi_3^{(j+1,k+1)}$ reads as follows. Setting $s = n-j+k$, we first apply $n+1-s$ shuffles to move the letter $s$ all the way to the right through $\overline\bi_{w_0}$. Then, if $k \ne j$ we apply the braid move $(s,s+1,s) \mapsto (s+1,s,s+1)$ and arrive at $\isf^{(j,k+1)}$, whereas if $k = j$ we merge $(n,n) \mapsto n$ and arrive at $\isf^{(j+1,0)}$.

\begin{example}
Let us spell out the transformation $\Phi_3$ for $n=2$. In that case we have
%$$
%\isf = (\overline\i_{w_0},2,1,1,2) \qquad\text{and}\qquad \overline\i_{w_0} = (-1,-2,-1).
%$$
%Then
$$
\Phi_3^{(1,1)} \colon \quad
\isf^{(0,0)} = (\overline\i_{w_0},2,1,1,2) \stackrel{s}\longmapsto (\overline\i_{w_0},2,2,1,1) \stackrel{m}\longmapsto (\overline\i_{w_0},2,1,1) = \isf^{(1,0)},
$$
and the second wave $\Phi_3^{(2)}$ reads
\begin{align*}
\Phi_3^{(2,1)} \colon \quad
\isf^{(1,0)} &= (\overline\i_{w_0},2,1,1) \stackrel{s^2}\longmapsto (\overline\i_{w_0},1,2,1) \stackrel{b}\longmapsto (\overline\i_{w_0},2,1,2) = \isf^{(1,1)}, \\
\Phi_3^{(2,2)} \colon \quad
\isf^{(1,1)} &= (\overline\i_{w_0},2,1,2) \stackrel{s}\longmapsto (\overline\i_{w_0},2,2,1) \stackrel{m}\longmapsto (\overline\i_{w_0},2,1) = \isf^{(2,0)} = \bi_{\overline{w}_0,c^*}.
\end{align*}
\end{example}

\subsection{Quiver transformations}
\label{subsec-quiver-mut}

We now use the dictionary from Section~\ref{subsec:weyl-words} between cyclic double Weyl words and quivers on a cylinder to turn reductions~\eqref{phi1}~--~\eqref{phi3} of the former into transformations of the latter, which we denote by the same symbols.

We write $Q_\bi$ for the quiver defined by the cyclic Weyl word $\bi$. If $\bi = \bi_{u_1,\dots,u_k}$ in the sense of notation~\eqref{eq:word-notation}, we denote the corresponding quiver by $Q_{u_1,\dots,u_k}^n$, where $n$ stands for the number of letters in the alphabet $\Agt_+$ and will be omitted when is clear from the context. For example, the unfrozen part of the quiver in the left pane of Figure~\ref{fig-Qbox} is $Q_{\overline{w}_0, w_0}^4$. Note that the quivers $Q_{-\bi^t}$, $Q_{\bi^*}$, $Q_{\overline{\bi}^t}$ are all isomorphic to $Q_\bi$, as they can be obtained from the latter by reflecting it respectively across the vertical axes, across the horizontal axes, and across the origin. 

\begin{figure}[h]
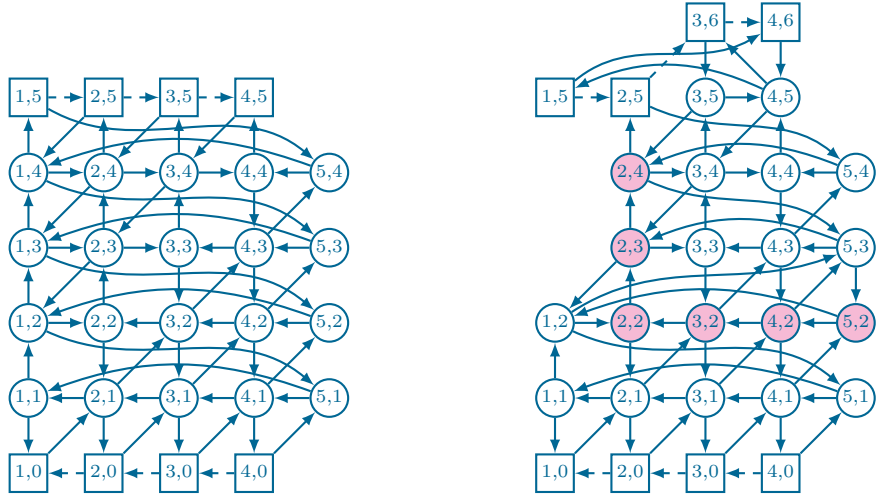

\subfile{fig-Q00.tex}
\caption{Quiver $Q_{\mathrm{cyl}}^4$ (on the left) and its image under $\Phi_1^{(1,2)} \Phi_1^{(1,1)}$ (on the right).}
\label{fig-Qbox}
\end{figure}

\subsubsection{Quiver transformation $\Phi_1$}
\label{subsec-muts-double1}

Consider the quiver $\Qcyl \simeq Q_{\tri}$ for a triangulation $\tri$ of a cylinder with a single marked point on each boundary component, see the left side of Figure~\ref{fig:local-gluing-surfaces}. Let us further fix an oriented simple closed curve $c$ on the cylinder generating its fundamental group.

We arrange the nodes of $ \Qcyl$ into a square grid with $n+2$ rows and $n+1$ columns and label the nodes by their coordinates, so the vector pointing in the first coordinate direction corresponds to the chosen orientation of $c$, and two coordinate directions form a positive frame on $S$. The setup is illustrated in the left pane of Figure~\ref{fig-Qbox} for $Q_{\mathrm{cyl}}^4$. Note that the mutable part of $\Qcyl$ is $Q_{\overline{w}_0, w_0}^n$.
%{We write  $\{\upsilon_c^r| 1\leq c\leq n+1,1\leq r\leq n\}$ for its initial set of tropical variables.}
%We arrange its nodes into a square grid with $n$ rows and $n+1$ columns, and label the nodes by their coordinates.
We define the quiver transformation $\Phi_1$ via factorization~\eqref{eq:phi1-factor} into individual steps. After each step of $\Phi_1$ we will rearrange the nodes and relabel them according to their new position, thereby establishing the bijection between the indexing sets of the two quivers. To indicate the embedding of the quiver as a planar graph on the cylinder, we draw with bent lines arrows going around its back-side. Then the quiver transformation $\Phi_1^{(j,k+1)}$ reads as follows:
\begin{itemize}
\item mutate at nodes $(x,n-k)$ as $x$ runs from $n+1$ to $n+1-k$;
\item mutate at nodes $(n-k,y)$ as $y$ runs from $n-k$ to $n$.
\end{itemize}
The dependence on the number $j$ of the wave appears in the following repositioning and relabelling of the nodes:
\begin{itemize}
%\item shift the vertex $(n-k,n)$ to the position $(n-k,2n-j)$;
\item move the node $(n-k,y)$ to position $(n-k,y+1)$ for all $y \ge n-k$;
\item move the node $(x,n-k)$ to position $(x-1,n-k)$ for all $n+1-k \le x \le n+1$;
\item move the node $(j,n-k)$ to position $(n+1,n-k)$;
%\item if $k=n-j$ move the node $(1,n+j)$ to position $(1,n+j+1)$;
\item rename nodes according to their new position.
\end{itemize}
In particular, in the process of applying the wave $\Phi_1^{(j)}$ we are losing nodes positioned at $(j,y)$ as $y$ runs from $n$ down to $j+1$, and acquiring those positioned at $(x,n+j+1)$ as $x$ runs from $n$ to $j+1$.
%Finally at the end of the wave we shift the vertex $(1,n+j)$ up by 1.

\begin{example}
The step $\Phi_1^{(1,3)}$ consists of consecutive mutations at the shaded nodes on the right pane of Figure~\ref{fig-Qbox}, starting with $(5,2)$ and ending at $(2,4)$. We then move nodes $(2,y)$ upward by 1 for all $2 \le y \le 5$, move nodes $(x,2)$ left by 1 for all $3 \le x \le 5$, move the node $(1,2)$ to position $(5,2)$, and finish by relabeling nodes according to their new position.
\end{example}

We denote by $\Qcone$ the result of applying $\Phi_1$ to (the mutable part of) the quiver $\Qcyl$. Figure~\ref{fig-Qcone} shows $Q_{\mathrm{cone}}^4$, where positions of vertices have been changed without relabelling. Following the quiver through the steps of the transformation $\Phi_1$, one can see that the bottom $n$ mutable rows of $\Qcone$ form a subquiver $Q_{\overline w_0, c}^n$, while the top $n-2$ mutable rows of $\Qcone$ form $Q_{\overline c,w_0}^{n-2}$. 

%{Using Lemma~\ref{tropical-chain}, it is simple to check that the tropical variables sitting at the nodes $X_c^r$ at the end of the $j$-th wave of $\Phi_1$ are given by
%%%%%%%% old formulas below where N=n+1; 
%% \begin{align}
%% \label{eq:trop-formula-phi1}
%% \nonumber    X_N^r &\mapsto \prod_{a=0}^{\min(j,r-1)}\upsilon_a^r \\
%% \nonumber X_{r+1}^{N+j-r} &\mapsto \left(\prod_{a\neq r}^{N}\upsilon_a^{r+1}\prod_{b=r+2}^{N-1}\upsilon_{r+1}^b\right)^{-1} ,\quad 1\leq r\leq j-1 , \\
%% X_{r+k}^{N+j-r} &\mapsto \upsilon_r^{r+k} ,\quad 1\leq r\leq j-1 ,\quad 2\leq k\leq N-r-1, \\
%% \nonumber X_{k}^{N} &\mapsto \left(\prod_{a=k}^{N}\upsilon_a^{k}\prod_{a=1}^{j-1}\upsilon_a^{k}\prod_{b=k+1}^{N-1}\upsilon_{k}^b\right)^{-1} ,\quad j+1\leq k\leq N-1 , \\
%% \nonumber X_c^r&\mapsto \nu_c^r \quad \text{otherwise}
%% \end{align}
%\begin{align}
%\label{eq:trop-formula-phi1}
%\nonumber    X_{n+1}^r &\mapsto \prod_{a=0}^{\min(j,r-1)}\upsilon_a^r \\
%\nonumber X_{r+1}^{n+1+j-r} &\mapsto \left(\prod_{a\neq r}^{n+1}\upsilon_a^{r+1}\prod_{b=r+2}^{n}\upsilon_{r+1}^b\right)^{-1} ,\quad 1\leq r\leq j-1 , \\
%X_{r+k}^{n+1+j-r} &\mapsto \upsilon_r^{r+k} ,\quad 1\leq r\leq j-1 ,\quad 2\leq k\leq n-r, \\
%\nonumber X_{k}^{n+1} &\mapsto \left(\prod_{a=k}^{n+1}\upsilon_a^{k}\prod_{a=1}^{j-1}\upsilon_a^{k}\prod_{b=k+1}^{n}\upsilon_{k}^b\right)^{-1} ,\quad j+1\leq k\leq n , \\
%\nonumber X_c^r&\mapsto \upsilon_c^r \quad \text{otherwise}
%\end{align}
%
%where we understand lower indices modulo $n+1$, so that $X_0^r=X_{n+1}^r$.}

\begin{figure}[h]
\subfile{fig-Qcone.tex}
\caption{Quiver $Q_{\mathrm{cone}}^4$.}
%and its tropical variables evolved from $Q_4^{\mathrm{box}}$ by applying the cluster transformation $\Phi_1$. To avoid cluttering the figure, each node $X^r_c$ without a green label is understood to simply carry tropical variable $\upsilon^r_c$.}
\label{fig-Qcone}
\end{figure}

%After applying all $n-1$ waves, we arrive at the quiver $Q_n^{\mathrm{cone}}$. We show the quiver $Q_4^{\mathrm{cone}}$ in Figure~\ref{fig-Qcone}, in which we have shifted all vertices in the $r$-th row to the left by $\frac{r-1}{2}$ if $1 \le r \le n$ and by $\frac{2n-1-r}{2}$ if $n \le r \le 2n-1$. 

\subsubsection{Quiver transformation $\Phi_2$}
\label{subsec-muts-double2}

Let us reposition the nodes of the subquiver, formed by the bottom $n+1$ rows of~$\Qcone$. Move the node $(n+1,y)$ to position $(y,y)$ for all $1 \le y \le n$, the node $(x,y)$ to position $(x+1,y)$ if $0 \le y \le x \le n$, and relabel all nodes according to their new positions. The resulting subquiver $Q_{\mathrm{frust}}^n$ is shown for $n=4$ on the left pane of Figure~\ref{fig-Q0c}. We define the quiver transformation $\Phi_2$ on $Q_{\mathrm{frust}}^n$ via factorization~\eqref{eq:phi2-factor}, where the step $\Phi_2^{(j,k)}$ reads
\begin{itemize}
\item mutate at vertex $(k,k)$;
\item mutate at vertices $(k+1,y)$ as $y$ runs from $k$ to 1.
\end{itemize}
As before, the dependence on $j$ appears in the relabelling of the nodes, which follows $\Phi_2^{(j,k)}$:
\begin{itemize}
\item shift vertices $(k+1,y)$ to the position $(k+1,y-1)$ for all $y \le k$;
\item shift the vertex $(k,k)$ to the position $(k+1,k)$;
\item shift the vertex $(n+2-j,k)$ to the position $(k,k)$;
\item rename vertices according to their new position.
\end{itemize}
In particular, in the process of applying the wave $\Phi_1^{(j)}$ we are losing nodes positioned at $(n+2-j,y)$ as $y$ runs from $1$ up to $n-j$, and acquiring nodes positioned at $(x,-j)$ as $x$ runs from $2$ to $n+1-j$.

\begin{figure}[h]
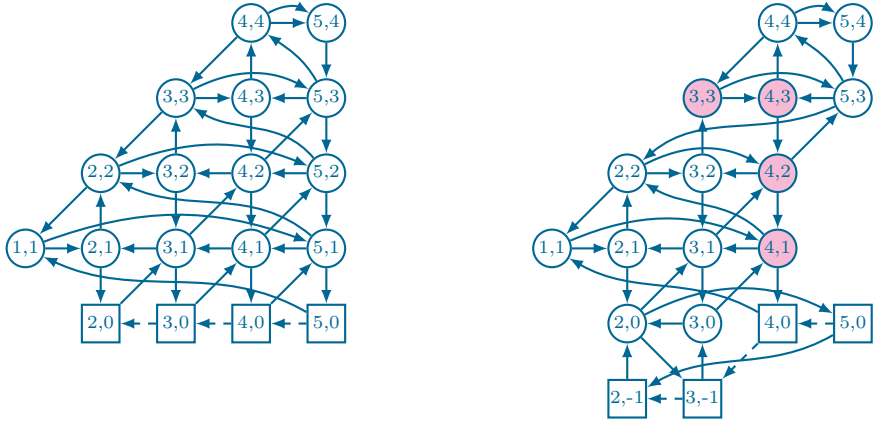

\subfile{fig-Q0c.tex}
\caption{Quiver $Q_{\mathrm{frust}}^4$ (on the left) and its image under $\Phi_2^{(1,2)} \Phi_2^{(1,1)}$ (on the right).}
\label{fig-Q0c}
\end{figure}

\begin{example}
The step $\Phi_2^{(1,3)}$ consists of consecutive mutations at the shaded nodes on the right pane of Figure~\ref{fig-Q0c}, starting with $(3,3)$ and ending at $(4,1)$. We then move the nodes $(4,y)$ down by 1 for all $0 \le y \le 3$, move the node $(3,3)$ to position $(4,3)$, move the node $(5,3)$ to the position $(3,3)$, and finish by relabeling nodes according to their new position.
\end{example}

%{One can again use Lemma~\ref{tropical-chain} to track the effect of each wave in $\Phi_2$ at the level of tropical cluster variables. This time let us simply record the variables attached to each vertex of the final quiver $\mathcal{Q}_{n+1}^{\mathrm{candy}}$ obtained by applying the entire sequence $\Phi_2\circ\Phi_1$ to $\mathcal{Q}_{n+1}^{\mathrm{box}}$. The nodes of $\mathcal{Q}_{n+1}^{\mathrm{candy}}$ of course still carry the variables described in formulas~\eqref{eq:trop-formula-phi1} for the quiver $\mathcal{Q}_{n+1}^{\mathrm{cone}}$ (with the relabelling $(X^r_c)_{\mathrm{candy}}=(X^{n+1+r}_{c+n-r})_{\mathrm{cone}} $), while the $Y$- and $C$-labelled nodes correspond to tropical coordinates
%%%%%%% old formulas with N=n+1:
% \begin{align}
%     \label{eq:trop-formula-phi2}
%     C_1^r&\mapsto \prod_{b\neq r}\upsilon^r_b,\\
%     C_2^r&\mapsto \upsilon_r^r,\\
% Y^r_c&\mapsto \upsilon_{N-r}^c, \qquad 1\leq r\leq N-1,~1\leq c<N-2\\
% Y^r_{N-2}&\mapsto \left(\prod_{a\neq N-r}\upsilon_a^{N-r-1}\prod_{b=1}^{N-r-2}\upsilon^b_{N-r-1}\right)^{-1}
% \end{align}}
%\begin{align}
%    \label{eq:trop-formula-phi2}
%    C_1^r&\mapsto \prod_{b\neq r}\upsilon^r_b,\\
%    C_2^r&\mapsto \upsilon_r^r,\\
%Y^r_c&\mapsto \upsilon_{n+1-r}^c, \qquad 1\leq r\leq n,~1\leq c<n-1\\
%Y^r_{N-2}&\mapsto \left(\prod_{a\neq n+1-r}\upsilon_a^{n-r}\prod_{b=1}^{n-r-1}\upsilon^b_{n-r}\right)^{-1}
%\end{align}}
Now we have all the necessary ingredients to define the isolating cluster subordinate to a $c$-isolating ideal triangulation.
\begin{defn}
\label{def:c-isol-quiver}
Suppose that $\tri$ is an isolating triangulation for an {oriented} simple closed curve $c$ on a surface $S$, so that $\tri$ has the local form shown in the left pane of Figure~\ref{fig:local-gluing-surfaces}. We define the \emph{$c$-isolating quiver} $Q_{\tri;c}=\Phi_c(Q_{\tri})$ subordinate to $\tri$ to be the result of applying to $Q_{\tri}$ the composition of mutations
\beq
\label{eq:Phi-c}
\Phi_c = \Phi_2 \circ \Phi_1
\eeq
within the subquiver $Q_{\mathrm{cyl}}$ of $Q_{\tri}$. 
\end{defn}
The adjacency matrix of the subquiver $\Phi_c(\Qcyl)$ is represented graphically in Figure~\ref{fig-Qcandy} for the $n=4$ case.
%Figure~\ref{fig-Qcandy} shows $Q_{\tri;c}^4$, where the nodes have been repositioned by their labels were preserved.
Following the quiver $Q_{\mathrm{frust}}^n$ through the steps of $\Phi_2$, one can see that the bottom $n-2$ rows of $Q_{\tri;c}^n$ form  the quiver $Q_{\overline{w}_0, c^*}^{n-2}$, while the rows 1 to $n$ form the quiver $Q_{\overline c,c}^n$.

\begin{figure}[h]
\subfile{fig-candy.tex}
\caption{Quiver $Q_{\tri;c}^4$.}
% and the tropical variables evolved from $\Qc_4^{\mathrm{box}}$ after $\Phi_2\circ\Phi_1$.}
\label{fig-Qcandy}
\end{figure}

\begin{notation}
\label{not:LambdaS-split}
Let $Q_{\tri;c}$ be the $c$-isolating quiver subordinate to an ideal triangulation $\tri$. We introduce convenient aliases for certain vertices of $Q_{\tri;c}$ as follows:
$$
e_{s_i} = e_{(i,i)}, \qquad e_{t_i} = e_{(i+1,i)}, \qquad 1\leq i\leq n
$$
$$
e_{h_+} = e_{(n,n+1)},\qquad e_{h_-} = e_{(2,0)}.
$$
Note that we have 
$$
(e_{s_i},e_{t_j}) = \mathfrak{A}_{ij}
$$
where $\mathfrak{A}$ is the Cartan matrix of type $A_n$.
We write $\Lambda_{S_{\geq0}}$ for the rank $2(n+1)$-sublattice of $\Lambda_{\tri;c}$ spanned by the $e_{h_\pm}$ together with the $e_{s_i},e_{t_i}$, and $\Lambda_{S_{<0}}$ for the sublattice spanned by all other $e$-basis vectors for $\Lambda_S$. Hence we have a non-orthogonal direct sum decomposition
\begin{align}
\label{eq:proto-split}
\Lambda_{S}(Q_{\tri;c}) = \Lambda_{S_{\geq0}}(Q_{\tri;c})\bigoplus \Lambda_{S_{<0}}(Q_{\tri;c}).
\end{align}
We also write $\Lambda_{S_{\leq0}}$ for the sublattice spanned by $\Lambda_{S_{<0}}$ together with the $e_{h_{\pm}}$, and $\Lambda_{S_{>0}}$ for the span of the complementary set of $e$-basis vectors for $\Lambda_S$. We set
$$
\mathcal{T}_{\leq0}(\tri;c) = \mathcal{T}(\Lambda_{S_{\geq0}}(Q_{\tri;c})), \quad \mathcal{T}_{<0}(\tri;c) = \mathcal{T}(\Lambda_{S_{<0}}(Q_{\tri;c}))
$$
and define analogously the quantum tori $\mathcal{T}_{\geq0}(\tri;c),\mathcal{T}_{>0}(\tri;c)$.
\end{notation}
The notation above is illustrated in the following picture:
\begin{figure}[h]
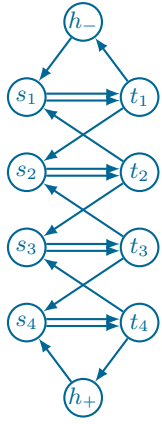

\subfile{fig-Q-Toda}
\caption{Quiver $Q_{S_{\geq0}}^4$.}
\label{fig:Q-Toda}
\end{figure}
\begin{remark}
It is easy to check that in the compatible pair associated to $Q_{\tri;c}$ we have the relations
\begin{align}
\label{eq:xi-sn-tn}
\xi_{s_n} - \xi_{t_n} &= \sum_{j=1}^n \frac{j}{n+1}(e_{s_j} + e_{t_j}), \\
\xi_{s_1} - \xi_{t_1} &= \sum_{j=1}^n \hr{1-\frac{j}{n+1}}(e_{s_j} + e_{t_j}),
\end{align}
and that the vectors
$$
z_+ = e_{v_+} + \xi_{s_n} - \xi_{t_n} \qquad\text{and}\qquad z_- = e_{v_-} + \xi_{s_1} - \xi_{t_1}
$$
are skew orthogonal to $\Lambda_{S_{\geq0}}$.
\end{remark}

\subsubsection{Quiver transformation $\Phi_{c_+}$}
\label{subsec-muts-single}
Now let $S'$ be a surface with a pair of tacked circles $c_\pm$.  Suppose that $\tri'$ is an ideal triangulation of $S'$ containing isolating cylinders for both $c_\pm$, so that $\tri'$ is locally of the form shown in the right pane of
%Consider a cylinder $S_+$, so that one of its boundary components is a tacked circle $c_+$, and the other contains a single marked point, as shown on the right side of
 Figure~\ref{fig:local-gluing-surfaces}. Hence $Q_{\tri'}$ contains subquivers $\Qsf({\pm})$ associated to the isolating cylinders for $c_\pm$, whose vertices we arrange as shown on the left pane of Figure~\ref{fig:Qsf} for $n=4$; see also the right pane of Figure~\ref{fig:self-folded-special} for the $n=3$ case.

\begin{figure}[h]
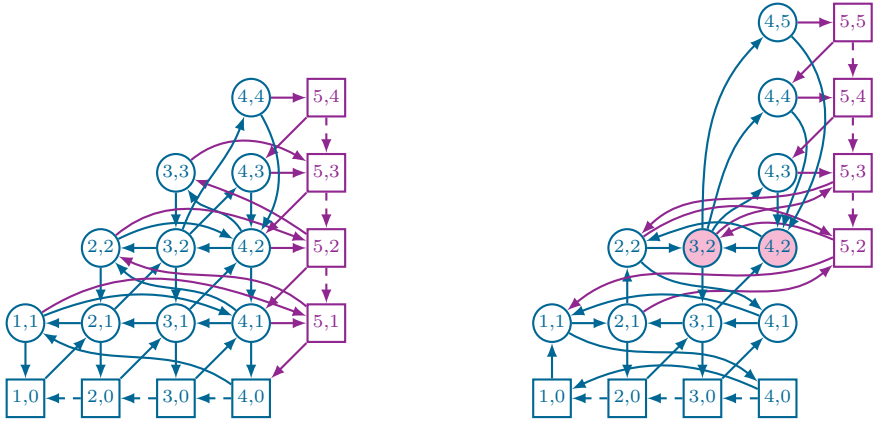

\subfile{fig-Qsf.tex}
\caption{Quiver $Q_{\mathrm{sf}}^4$ (on the left) and its image under $\Phi_{c_+}^{(0)}$ (on the right).}
\label{fig:Qsf}
\end{figure}

We now define a quiver transformation
\beq
\label{eq:Phi-c_+}
\Phi_{c_+} = \Phi_3 \circ \Phi_{c_+}^{(0)}
\eeq
of $Q_{\tri'}$ which modifies only the subquiver $\Qsf(+)$. The preliminary wave $\Phi_{c_+}^{(0)}$ consists of
\begin{itemize}
\item apply transformation~\eqref{eq:e-to-e''} at nodes $(i,i)$ for $1 \le i \le n-1$;
\item move the node $(n+1,y)$ to position $(n+1,y+1)$ for all $1 \le y \le n$;
\item move the node $(n,y)$ to position $(n,y+1)$ for $y=n-1$ and $y=n$;
\item move the node $(n-1,n-1)$ to position $(n,n-1)$;
\item relabel nodes according to their new position.
\end{itemize}
Observing that the bottom $n-2$ mutable rows of $\Phi_{c_+}^{(0)}(\Qsf)$ form the subquiver $Q_{\overline{w}_0,c^*,c}^n$, we now apply the transformation $\Phi_3$, which admits factorization~\eqref{eq:phi3-factor}. The step $\Phi_3^{(j,k)}$ reads as follows:
\begin{itemize}
\item apply~\eqref{eq:e-to-e''} at nodes $(x,n-2-j+k)$ as $x$ runs from $n-1-j+k$ to $n$;
\item if $k \ne j$ move the node $(n,n-2-j+k)$ to position $(n-1-j+k,n-1-j+k)$;
\item if $k = j$ move the nodes $(n+1,y)$ and $(n,y)$ with $y \ge n-2$ up by 1;
\item move the node $(x,n-2-j+k)$ right by 1 for all $n-2-j+k \le x \le n-1$;
\item relabel nodes according to their new position.
\end{itemize}
The same procedure applied to the isolating cylinder for $c_-$ defines a cluster transformation $\Phi_{c_-}$.

\begin{defn}
\label{def:cpm-isol-quiver}
We define the \emph{$c_\pm$-isolating quiver} 
$$
Q_{\tri';c_\pm} = \Phi_{c_+}\circ\Phi_{c_-}(Q_{\tri'})
$$ 
subordinate to a $c_\pm$-isolating ideal triangulation $\tri'$ to be the result of applying the commuting cluster transformations $\Phi_{c_+}$ and $\Phi_{c_-}$ to the quiver $Q_{\tri'}$. If $c_+$ is a single tacked circle on $S$ and $\tri'$ a $c_+$-isolating triangulation, we refer to 
$$
Q_{\tri';c_+} = \Phi_{c_+}(Q_{\tri'})
$$
as a $c_+$-isolating quiver subordinate to $\tri'$.
%We show the quiver $Q_{\tri_+;c_+}^4$ on Figure~\ref{fig-Qmitre}, where the nodes have been repositioned without relabelling.
\end{defn}
The local adjacency matrix of $Q_{\tri';c_\pm}$ near $c_+$ is represented graphically in the left pane of Figure~\ref{fig-Qmitre} for the $n=4$ case.

\begin{remark}
The transformation $\hat\mu_f$ defined by~\eqref{eq:e-to-e''} differs from the mutation $\mu_f$ when $f \in \hc{f_{i-1}^-(c), f_i^+(c)}$ for some tacked circle $c$ and $1 \le i \le n$. Let $Q_\Gamma$ be the quiver defined by the ideal bicolored graph $\Gamma$, and $k$ be the largest such number that $\dot e_{c,i} \ne e_{c,i}$. Then for each $1 \le i \le k$ we have $f_i^+ = f_i^-$ and the node $(c,i)$ (where $c$ is still a tacked circle rather than the coordinate of the node) is 2-valent with arrows $f_i^- \to (c,i) \to f_{i+1}^+$. The node $(c,k+1)$ is 3-valent with an arrow $f_{k+1}^+ \to (c,k+1)$ and an oriented triangle formed by vertices $f_{k+1}^-$, $(c,k+1)$, $f_{k+2}^+$ listed in that order. Finally, the nodes $(c,j)$ with $k+2 \le j \le n$ are 4-valent with an oriented triangle formed by vertices $f_j^-$, $(c,j)$, $f_{j+1}^+$ and a pair of arrows $f_j^+ \to (c,j) \to f_{j-1}^-$. Moreover, in the latter case there is an arrow $f_j^+ \to f_{j-1}^-$. This analysis allows one to identify the nodes $v$ of the quiver $Q_\Gamma$ such that $\mu_v \ne \hat \mu_v$ without having to draw the bicolored graph $\Gamma$. For example, in the quiver $\Phi_3^{(0)}(Q_{\mathrm{sf}}^4)$ on the right pane of Figure~\ref{fig:Qsf} we have $k=2$, for $2 \le y \le 5$ nodes $(5,y) = (c,6-y)$ where $c$ is the unique tacked circle, and
\begin{align*}
(4,5) &= f_1^+ = f_1^-, & (4,3) &= f_3^+, & (3,2) &= f_3^-, & (1,1) &= f_5^+, \\
(4,4) &= f_2^+ = f_2^-, & (2,2) &= f_4^+, & (2,1) &= f_4^-.
\end{align*}
\end{remark}

\begin{example}
In order to perform step $\Phi_3^{(1,1)}$ on the quiver $\Phi_{c_+}^{(0)}(Q_{\mathrm{sf}}^4)$, see the right pane of Figure~\ref{fig:Qsf}, one applies the composition $\hat\mu^q_{(4,2)} \circ \hat\mu^q_{(3,2)}$ remembering that both nodes are $f_3^-$ in the respective quivers, moves the nodes $(4,r)$ and $(5,r)$ with $2 \le r \le 5$ upward by 1, moves the nodes $(2,2)$ and $(3,2)$ right by 1, and finishes by relabelling all nodes according to their new position.
\end{example}

\begin{figure}[h]
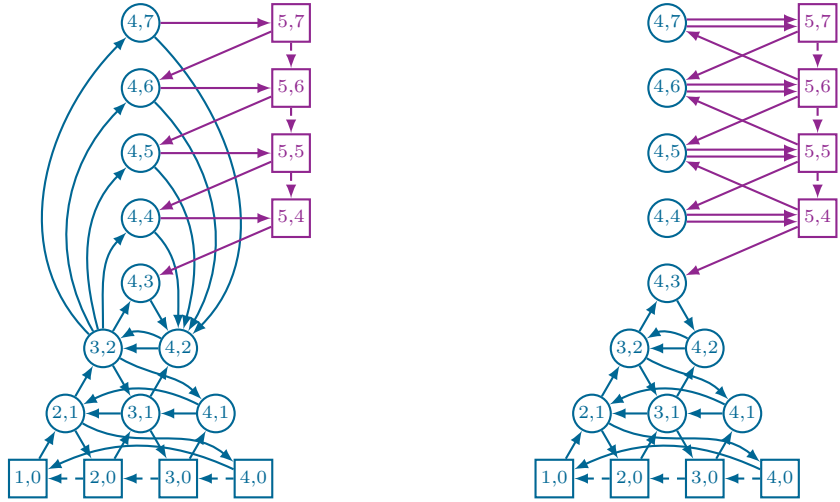

\subfile{fig-mitre.tex}
\caption{Graphical representation of pairings between basis vectors for the basis $\{e_\ell\}$ associated to the subquiver of $Q_{\tri';c_\pm}^4$ in the vicinity of tacked circle $c_+$ (on the left), and for the corresponding dotted basis $\{\dot{e}_\ell\}$  (on the right).}
\label{fig-Qmitre}
\end{figure}

\begin{notation}
\label{not:LambdaSprime-split}
Given a pair of tacked circles $c_\pm$ on $S'$, let $Q_{\tri';c_\pm}$ be the $c_\pm$-isolating quiver subordinate to an ideal triangulation $\tri$. We introduce convenient aliases for certain vertices of $Q_{\tri';c_\pm}$ as follows:
$$
e_{v^{\pm}_i} \equiv e_{(n,2n-i)^\pm}, \qquad 1\leq i \leq n+1 
$$
$$
e_{(c_\pm,i)} \equiv e_{(n+1,2n-i)^\pm}, \qquad  1\leq i \leq n. 
$$
In plain language, in order to match with the notations of Section~\ref{sec:cluster-coord}, the new labeling of the mutable vertices in the top $n+1$ rows of $Q_{\tri';c_\pm}$ by $\hc{v_i^\pm}$ =has the property that the index $i$ increases as we move further away from the boundary circles $c_\pm$.
\end{notation}

\begin{remark}
\label{rmk:dotted-bases}
Recall the bases $\hc{e_\ell}$ and $\hc{\dot e_\ell}$ introduced in Section~\ref{subsec:quivers-from-graphs}. The pairings between the  vectors for each of these bases are represented graphically in Figure~\ref{fig-Qmitre}. Inspecting the bicolored graph $\Gamma_{\tri;c_\pm}$, one observes that $\dot e_{v_i^\pm} =  \alpha_{c_\pm,i}$, where the vectors $\alpha_{c_\pm,i}$ are defined by~\eqref{eq:alpha-c}. Thus we have $(\dot e_{v_i^\pm},e_{c,j}) = \mathfrak{A}_{ij}$, where $\mathfrak{A}$ is the Cartan matrix of type $A_n$. Moreover, for all $1 \le i \le n$ we have $e_{v_i^\pm} = \dot e_{v_i^\pm} + \ldots + \dot e_{v_{n+1}^\pm}$ and $e_{c_\pm,i} = \dot e_{c_\pm,i}$. 
\end{remark}
%
%We finish this subsection with a brief discussion of the quiver $Q_{\tri';c_\pm}$. To match the notations of Section~\ref{sec:cluster-coord}, we relabel its frozen nodes associated to the tacked circles $c_\pm$ by $(c_\pm,i)$, where $i$ runs from 1 to $n$ counting top to bottom. We also relabel the mutable nodes in the top $n+1$ rows of $Q_{\tri';c_\pm}$ by $\hc{v_i^\pm}$, where $1 \le i \le n+1$. For example, in the notations of Figure~\ref{fig-Qmitre} we now have $v_i^+ = (n,2n-i)$ and $(c_+,i) = (n+1,2n-i)$. Recall the bases $\hc{e_\ell}$ and $\hc{\dot e_\ell}$ introduced in Section~\ref{subsec:quivers-from-graphs}, which yield the quivers $Q_{\tri';c_\pm}$ and $\dot Q_{\tri';c_\pm}$ respectively, see Figure~\ref{fig-Qmitre}. Inspecting the bicolored graph $\Gamma_{\tri_+;c_+}$ dual to the quiver $\dot Q_{\tri_+;c_+}$, see Figure~\ref{fig:graph-mitre}, one observes that $\dot e_{v_i^\pm} = - \alpha_{c_\pm,i}$, where the vectors $\alpha_{c_\pm,i}$ are defined by~\eqref{eq:alpha-c}. Thus we have $(\dot e_{v_i^\pm},e_{c,j}) = C_{ij}$, with $C$ being the Cartan matrix of type $A_n$. Moreover, for all $1 \le i \le n$ we have $e_{v_i^\pm} = \dot e_{v_i^\pm} + \ldots + \dot e_{v_{n+1}^\pm}$ and $e_{c_\pm,i} = \dot e_{c_\pm,i}$.

\begin{figure}[h]
\subfile{fig-graph-mitre}
\caption{Bicolored graph $\Gamma_{\tri_+;c_+}$.}
\label{fig:graph-mitre}
\end{figure}

\subsection{Action of the braid group in an isolating cluster.}
Now that we have introduced isolating clusters for tacked circles we can finish the proof of Proposition~\ref{lem:braid} that the braid group acts by cluster Poisson transformations. Indeed, under the hypothesis of the Lemma, given a tacked circle $c$ there exists a $c$-isolating cluster $Q_{\tri',c}$. Let $b$ be the special point which is connected to $c$ in the corresponding triangulation $\tri'$, and let $c_i,b_i$ be their preimages on the universal cover so that the special triangles are $(c_i,c_{i+1},b_i)$.
Then we have (cf. Example~\ref{eg:dot-A})
\begin{align}
\dot{A}_{v_i} = [\nu_{1}(b_i)\wedge\nu_{1}(b_{i-1})\wedge\cdots\wedge\nu_{1}(b_1)\wedge\nu_{n+1-i}(c_{i})],
\end{align}
and so an easy computation shows that
\begin{align}
\label{eq:braid-xi}
\sigma_i(A_{v_j}) = \begin{cases}
\frac{ A_{v_{i+1}}A_{c,i+1}}{ A_{c,i}} \quad & j=i\\
  \frac{A_{v_{i}} A_{c,i}}{A_{c,i+1}} \quad & j=i+1\\
A_{v_j} \quad & \text{else}.
\end{cases}
\end{align}
Hence in a $c$-isolating cluster the braid group acts by a quasicluster automorphism, which proves the action factors through the cluster modular group.

%The action on $\xi$-basis:
%$$
%\sigma_i(\xi_{v_j}) = \begin{cases}
% \xi_{v_{i+1}} - \xi_{c,i}+\xi_{c,i+1} \quad & j=i\\
%  \xi_{v_{i}} + \xi_{c,i}-\xi_{c,i+1} \quad & j=i+1\\
%  \xi_{v_j} \quad & \text{else}
%\end{cases}
%$$
%where we understand $\xi_{c,n+1}=0$. 

In a $c$-isolating cluster, the quasi-permutation implementing the action of the braid $\sigma_i$ acts on the $e$-basis by
\begin{align}
\label{eq:braid-e}
\sigma_i(e_{c,j}) = \begin{cases}
e_{c,i-1}+e_{c,i} - \dot e_{v_i} \quad & j=i-1\\
-e_{c,i}+ \dot e_{v_i} \quad & j=i\\
e_{c,i}+e_{c,i+1} \quad & j=i+1\\
  e_{c,j} \quad & \text{else.}
\end{cases}
\end{align}
Note that the inverse transformation is given by
$$
\sigma^{-1}_i(e_{c,j}) = \begin{cases}
e_{c,i-1}+e_{c,i}  \quad & j=i-1\\
-e_{c,i}- \dot e_{v_i} \quad & j=i\\
e_{c,i}+e_{c,i+1} + \dot e_{v_i} \quad & j=i+1\\
  e_{c,j} \quad & \text{else,}
\end{cases}
$$
and we have
$$
\sigma^{2}_i(e_{c,j}) = e_{c,j} - (\alpha_i,\alpha_j)\dot{e}_{v_i}.
$$
Similarly, for the `oppositely oriented' frozen basis vectors $\{\overline{e}_{c,i}\}$ from~\eqref{eq:bar-e-ci} we have
\begin{align}
\label{eq:braid-ebar}
\sigma_i(\overline{e}_{c,j}) = \begin{cases}
\overline{e}_{c,i-1}+\overline{e}_{c,i}  \quad & j=i-1\\
-\overline{e}_{c,i}- \dot e_{v_i} \quad & j=i\\
\overline{e}_{c,i}+\overline{e}_{c,i+1}+ \dot e_{v_i} \quad & j=i+1\\
  \overline{e}_{c,j} \quad & \text{else.}
\end{cases}
\end{align}

\subsection{Dehn twists.}

Let $c$ be a simple closed curve on a marked surface $S$. As usual we write $S'$ for the result of cutting $S$ along $c$, and $\phi \colon c_+ \to c_-$ for the homeomorphism defining the gluing $S'$ of back into $S$.

Now suppose that there exists a $c$-isolating ideal triangulation of $S$, and let $Q_{\tri;c}$ be the corresponding $c$-isolating cluster as defiend in the previous section. We write $Q_{\tri';c}$ for the the $c_\pm$-isolating cluster subordinate to the ideal triangulation $\tri'$ obtained by applying the cutting functor $\mathcal{C}_c$ to $\tri$ .

%both tacked circles $c_\pm$, see Figure~\ref{fig:local-gluing-surfaces}. Then we have
%$$
%\Phi_c(Q_{\tri}) = Q_{\tri;c} \qquad\text{and}\qquad \Phi_\phi(Q_{\tri'}) = Q_{\tri';c_\pm} = Q_{\tri_+;c_+} \sqcup Q_{\tri_-;c_-},
%$$
%where $\Phi_\phi = \Phi_{c_-} \circ \Phi_{c_+}$ is a composition of commuting quiver transformations $\Phi_{c_\pm}$.
%

%We write $\Lambda_{S_{\geq0}}$ for the sublattice in $\Lambda_{\tri;c}$ spanned by the basis vectors $\hc{e_{s_j}, e_{t_j} \,|\, 1 \le j \le n}$ together with the 
%
%\blue{/////////////}
%
%
%Consider the \emph{Toda quiver} $Q_{\mathrm{Toda}}^n \simeq Q_{\overline{c}, c}^n$ with $2n$ vertices $\hc{s_j, t_j \,|\, 1 \le j \le n}$, and arrows representing the pairing
%$$
%(e_{s_i},e_{t_j}) = C_{ij},
%$$
%where $C$ is the Cartan matrix of type $A_n$. We will also use its augmented version $\widehat Q_{\mathrm{Toda}}^n$, where we add 2 additional vertices $v_\pm$ along with arrows $t_1 \to v_- \to s_1$ and $t_n \to v_+ \to s_n$, see Figure~\ref{fig:Q-Toda} for the $n=4$ case. Note that $\widehat Q_{\mathrm{Toda}}^n$ is isomorphic to the subquiver formed by the middle rows of $Q_{\tri;c}^n$.
%

\begin{prop}
The following diagram in the cluster modular groupoid is commutative:
$$
\begin{tikzcd}
Q_{\tri} \arrow{r}{\Qc(\tau_c)} \arrow[swap]{d}{\Phi_c} & Q_{\tri} \arrow{d}{\Phi_c} \\
Q_{\tri;c} \arrow[swap]{r}{\bs\mu_c} & Q_{\tri;c}
\end{tikzcd}
$$
Here $\Qc(\tau_c)$ denotes the image of the Dehn twist $\tau_c$ under the functor~\eqref{eq:Pt-to-Cl}, $\Phi_c$ is the cluster transformation denoted by the same symbol as that of the corresponding quiver, and the cluster transformation
\beq
\label{eq:short-Dehn}
\bs\mu_c = \prod_{j=1}^{n}(s_j,t_j)\circ\prod_{j=1}^{n}\mu_{s_j},
\eeq
consists of $n$ commuting mutations at the vertices $s_1,\ldots,s_n$ of $Q_{\tri,c}$, followed by $n$ commuting transpositions $(s_j,t_j)$.
\end{prop}

\begin{proof}
Using Theorem~\ref{trop-criterion}, the proposition can be proved directly by a straightforward comparison of tropical cluster transformations. Alternatively, we can give a less computational proof using Corollary~\ref{cor:transform-equal}. Indeed, we have $Q_{\tri} = Q_{\Gamma_{\tri}}$, and it is simple to check that $Q_{\tri;c} = Q_{\Gamma_{\tri;c}}$, where for $n=4$ the bicolored graph $\Gamma_{\tri;c}$ is shown on Figure~\ref{fig:graph-candy}. Since all mutations in the transformation $\Phi_c$ are 4-valent and can be interpreted as admissible face mutations, by Corollary~\ref{cor:transform-equal} the Dehn twist $\tau_c$ is realized in $Q_{\tri;c}$ by a sequence of face mutations turning the graph $\Gamma_{\tri;c}$ into $\tau_c^{-1}(\Gamma_{\tri;c})$, followed by a permutation of nodes turning $\tau_c^{-1}(\Gamma_{\tri;c})$ back into $\Gamma_{\tri;c}$. The latter coincides with the cluster transformation $\bs\mu_c$, so the proposition follows.
\end{proof}

\begin{figure}[h]
\subfile{fig-graph-candy}
\caption{Bicolored graph $\Gamma_{\tri;c}$.}
\label{fig:graph-candy}
\end{figure}

Let us now look at the corresponding picture on the cut surface $S'$. As can be easily seen from the bicolored graph $\Gamma_{\tri_+;c_+}$ on Figure~\ref{fig:graph-mitre}, the Dehn twist $\tau_{c_\pm}$ is realized in the quiver $Q_{\tri';c_\pm}$ via a single shift at $c_\pm$ in the sense of Definition~\ref{defn:shift}: we have 
$$
\tau_{c_\pm} = \sigma_{c_\pm}.
$$
The cluster transformation for the shift $\sigma_{c_\pm}$ is expressed via the quasi-permutation $\varsigma_{c_\pm}$ described in Lemma~\ref{eq:sigma-pm}, which in $Q_{\tri';c_\pm}$ simplifies to
\begin{align}
\label{eq:cutdehn}
\tau_{c_\pm}=\varsigma_{c_\pm}(e_\ell) =
\begin{cases}
e'_{\iota(\ell)} - \dot e'_{\iota(v_i^\pm)}, &\text{if} \;\; \ell = (c_\pm,i), \\
e'_{\iota(\ell)}, &\text{otherwise,}
\end{cases}
\end{align}
where $\iota$ is the natural bijection between the set of faces of the bicolored graph
$$
\Gamma_{\tri';c_\pm} = \Gamma_{\tri_+;c_+} \sqcup \Gamma_{\tri_-;c_-}
$$
and that of its image under the shift $\sigma_{c_\pm}$.

\subsection{Fundamental Hamiltonians of an isolatable simple closed curve.}
% Comparing formulas~\eqref{eq:gluedehn} and~\eqref{eq:cutdehn},  we see that the desired gluing isomorphism $\eta:\mathcal{H}_{S;\mu_\pm}\rightarrow\mathcal{H}_{S',c_\pm;\mu_\pm}$ must diagonalize the operator $\tau$ in~\eqref{eq:tau-def}. 
% We are now ready to construct the gluing isomorphism
% $\eta:\mathcal{H}_{S}\rightarrow\mathcal{H}_{S'_c}$ 
% in the case that the simple closed curve $c$ is isolatable.
% The isomorphism $\eta$ will be obtained by {diagonalizing} the Dehn twist $D_c$, which in a special cluster coordinate system $\bi_{S;(\tri;c)}$ is realized by the operator $\tau$ in~\eqref{eq:tau-def}. In \cite{SS18} it was shown that such a diagonalization is achieved by the \emph{Whittaker transform} for the modular double $U_{q,\tilde q}(\mathfrak{sl}_{n+1},\mathbb{R})$, which we will now review. 

We now use the quiver $Q_{\tri;c}$ to describe a commutative subalgebra of the universal Laurent ring algebra  $\Lbb^{\Ac}_{\tri;c}$, invariant under the action of the Dehn twist $\tau_c$. This subalgebra is generated by certain elements $H_1(c), \ldots, H_{n}(c)$, which we call the \emph{fundamental Hamiltonians} associated to an orientation of the simple closed curve $c$.

% , which act trivially on the factor $\mathcal{H}_{\Lambda_{\geq0;\mu_\pm}}$ in~\eqref{eq:S-amalg}. 
% To construct these elements, let us first choose an ideal triangulation $\tri$ containing an isolating cylinder for $c$, and then pass from the cluster associated to this triangulation to the special cluster coordinate system $\bi_{S;(\tri;c)}=\Phi_2\circ\Phi_1(\bi_{S,\tri})$ described in the previous section.

%For the remainder of this section we abbreviate $Q_{\tri;c}^n$ to $Q_n$. 
To construct the fundamental Hamiltonians we consider two augmentations of the quiver $Q_{\tri;c}$. The first of these, which we denote by $\widehat{Q}_{\tri;c}$, is the quiver obtained from $Q_{\tri;c}$ by adding a mutable vertex $a$ together with a frozen vertex $A$, and arrows $t_1 \to a \to s_1$ and $a \to A$.

%, where $s_1,t_1$ are nodes of the subquiver $\widehat Q_{\mathrm{Toda}}^n \subset Q_n$, an arrow $a \to A$, and denote the result by $\widehat Q_n$. 

For the second augmentation, which denote by $\widehat{Q}_{\tri;c}'$, we add a mutable vertex $b$ and a frozen vertex $B$, with arrows $t_n \to b \to s_n$, and an arrow $B \to b$.
%, and denote the result by  

We illustrate the subquivers of $\widehat{Q}_{\tri;c}$ and $\widehat{Q}_{\tri;c}'$ containing the new nodes along with those of $Q_{S_{\geq0}}$ on the left and the right panes of Figure~\ref{fig:aug-quiv} respectively. We have a non-orthogonal direct sum decomposition
\beq
\label{eq:Baxter-lattice-decomp}
\La_{\widehat{Q}_{\tri;c}'} \simeq \La_{Q_{\tri;c}} \oplus \Z\ha{e_b,e_B}.
\eeq
Although we do not extend $\widehat{Q}_{\tri;c}'$ to a compatible pair, the compatible pair $(\Lambda_{\tri;c},\Xi_{\tri;c})$ induces a skew form on the lattice $\Xi_{{Q}_{\tri;c}} \oplus \Z\ha{e_b,e_B}$.
There is a non-isometric isomorphism of lattices
\beq
\label{eq:Baxter-lattice-iso}
\Xi_{{Q}_{\tri;c}} \oplus \Z\ha{e_b,e_B} \simeq \Xi_{{Q}_{\tri;c}} \oplus \Z\ha{e_u,e_B},
\eeq
where
$$
e_u = e_b + \xi_{s_n} - \xi_{t_n}.
$$
Using~\eqref{eq:xi-sn-tn} we see that the direct summands in the right hand side are orthogonal with respect to the form $(\cdot,\cdot)_{\widehat{Q}_{\tri;c}'}$ and that $(e_B,e_u)=1$.

\begin{figure}[h]
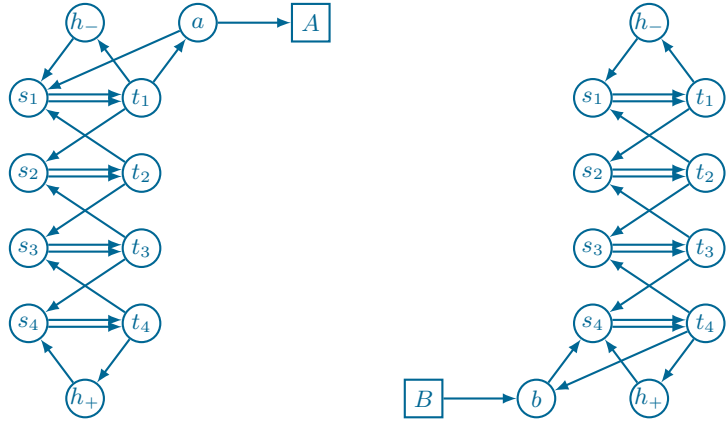

\subfile{fig-Baxter.tex}
\caption{Subquiver of $\widehat{Q}_{\tri;c}$ (on the left) and that of $\widehat{Q}_{\tri;c}'$ (on the right) for $n=4$.}
\label{fig:aug-quiv}
\end{figure}

Now consider the cluster transformation $\mu_{\mathrm{Baxter}}$ defined as the composition of the following $2n+1$ mutations starting from the quiver $\widehat{Q}_{\tri;c}$: 
\beq
\label{eq:baxseq}
\mu_{\mathrm{Baxter}} = \mu_{s_n}\mu_{t_n}\ldots \mu_{s_2}\mu_{t_2}\mu_{s_1}\mu_{t_1}\mu_a.
\eeq
If $I, I'$ are the label sets of the quivers $\widehat{Q}_{\tri;c}, \widehat{Q}_{\tri;c}'$, then there is a permutation morphism  $\iota \colon I \to I'$ between  the  quivers $\mu_{\mathrm{Baxter}}(\widehat{Q}_{\tri;c})$ and $\widehat{Q}_{\tri;c}'$ given by
\begin{align}
\label{eq:longseq}
a \longmapsto t_1 \longmapsto s_1 \longmapsto t_2 \longmapsto s_2 \longmapsto \ldots \longmapsto t_n \longmapsto s_n \longmapsto b,
\qquad
A \longmapsto B,
\end{align}
and identity on all other labels. We set $\widetilde\mu_{\mathrm{Baxter}} = \iota \circ \mu_{\mathrm{Baxter}}$, and denote the corresponding quasi-cluster transformation by the same symbol.

As is clear from inspecting the quiver, we have $(e_A,e_\ell)\leq0 $ for all mutable directions $\ell$ in $\widehat{Q}_{\tri;c}$. It follows that the quantum cluster $\Xc$-variable $Y_A = Y_{e_A}$ remains Laurent under all 1-step mutations and thus lies in the universal Laurent ring $\Lbb_{\widehat{Q}_{\tri;c}}$. Hence $\widetilde\mu_{\mathrm{Baxter}}(Y_A) \in \Lbb_{\widehat{Q}_{\tri;c}'}$. On the other hand, by the decomposition~\eqref{eq:Baxter-lattice-decomp} and the isomorphism~\eqref{eq:Baxter-lattice-iso}, we have an embedding of algebras
$$
\Lbb_{\widehat{Q}_{\tri;c}'} \subset \Lbb_{{Q}_{\tri;c}}^\Ac \otimes \Z[q^{\pm1}]\ha{Y_u^{\pm1},Y_B^{\pm1}},
$$
where $Y_uY_B = q^2Y_BY_u$.
%The tensor product in the right hand side is naturally a $\Lbb_{Q_n}^\Ac$-bimodule under left and right multiplication. Since any sequence of mutations in the cluster algebra associated to $\mathcal{Q}$ is given by conjugation by an element of (a completion of) $\mathcal{T}_{\mathcal{Q}}$, we conclude that the  `coefficients' of $\mu_{\mathrm{Baxter}}(X_A)\in \Lbb_{\widehat{\mathcal{Q}'}}$ are elements of the universal Laurent ring $\Lbb_{\mathcal{Q}}$. 
Since the two tensor factors commute, the coefficients of $\widetilde\mu_{\mathrm{Baxter}}(Y_A)$, considered as a Laurent polynomial in $\hc{Y_u, Y_B}$, are elements of the upper cluster algebra $\Lbb_{\tri;c}$. 
% \red{[State as Proposition, along with fact that they quantize traces of holonomy?]}

The mutation sequence $\mu_{\mathrm{Baxter}}$ and its generalizations are studied in~\cite{SS18,SS19b} under the name of the \emph{Baxter operator}. In particular, it follows from Proposition~\ref{prop:comm-bax} recalled in Section~\ref{sec:alg-whit} that  $\widetilde\mu_{\mathrm{Baxter}}(Y_A)$ can be written as
$$
\widetilde\mu_{\mathrm{Baxter}}(Y_A) = \sum_{k=0}^{n+1} H_k(\tri;c) Y_{e_B+ke_u} \in \Lbb_{Q_{\tri;c}}^\Ac \otimes \Z[q^{\pm1}]\ha{Y_u^{\pm1},Y_B^{\pm1}},
$$
with $H_k(c) = H_k(\tri;c) \in \Lbb_{Q_{\tri;c}}^\Ac$.

\begin{defn}
\label{def:fund-hamiltonians}
The \emph{fundamental quantum Hamiltonians} associated to the simple closed curve $c$ are defined to be the coefficients $H_k(c)\in\Lbb^\Ac_{\tri;c}$ of $Y_{e_B+ke_u}$ of the polynomial $\widetilde\mu_{\mathrm{Baxter}}(Y_A)$.
%\begin{align}
%\label{eq:Bax-toda-def}
%\mu_{\mathrm{Baxter}}(X_A) =: \sum_{k=0}^{n+1} H_k(c) Y_{e_B+ke_u} \in \Lbb^q_{SL_{n+1},S}\otimes \mathbb{Z}[X_u^{\pm1}]\otimes \mathbb{Z}[X^{\pm1}_{A'}].
%\end{align}
\end{defn}

\begin{example}
For $G=SL_{2}$, the only nontrivial fundamental Hamiltonian is given by
\begin{align}
\label{eq:sl2-fund-ham}
\nonumber H_1(c) &= Y_{\xi_{t_1}-\xi_{s_1}} + Y_{\xi_{t_1}-\xi_{s_1}+e_{s_1}} + Y_{\xi_{t_1}-\xi_{s_1}+e_{s_1}+e_{t_1}}\\
&=Y_{\xi_{t_1}-\xi_{s_1}} + Y_{\xi_{h_+}+\xi_{h_-}-\xi_{s_1}-\xi_{t_1}} + Y_{\xi_{s_1}-\xi_{t_1}}.
\end{align}
\end{example}

The next Lemma is a consequence of Proposition~\ref{prop:comm-bax} in Section~\ref{sec:alg-whit}.

\begin{lemma}
The fundamental Hamiltonians $H_k(c)$ generate a commutative subalgebra in $\Lbb^\Ac_{\tri;c}$.
\end{lemma}

In Section~\ref{sec:3-to-3} we will prove the following theorem justifying the claim, implicit in the notation, that the $H_k(c)$ depend only on the curve $c$ and not on the choice of ideal triangulation used in the definition above.

\begin{theorem}
\label{thm:monodromy-well-defined}
The elements $H_k(c)$ are independent of the choice of isolating triangulation for $c$: if $\tri,\tri'$ are two triangulations containing an isolating cylinder for $c$ and 
$$
\mu \colon Q_{\tri;c} \longra Q_{\tri';c}
$$
is the cluster transformation between the corresponding special coordinate systems, then we have
$$
\mu(H_k(\tri;c))=H_k(\tri';c).
$$
\end{theorem}

\begin{remark}
In a companion paper~\cite{SS25} we will give an alternative construction of the elements $H_k(c)$ which identifies them with $q$-deformations of the traces of the monodromy around $c$ in the $k$-th fundamental representation of $SL_{n+1}(\mathbb{C})$. This construction gives a separate proof of Theorem~\ref{thm:monodromy-well-defined}, also explained in~\cite{SS25}.
\end{remark}

\section{Algebraic Whittaker transform}
\label{sec:AlgWhit}
% We are now ready to construct the gluing isomorphism
% $\eta:\mathcal{H}_{S}\rightarrow\mathcal{H}_{S'_c}$ 
% in the case that the simple closed curve $c$ is isolatable.
% Recall from~\eqref{eq:S-amalg} that an isolating seed $\bi_{S;(\tri,c)}$ induces a factorization of Hilbert spaces $\mathcal{H}_{S} \simeq \mathcal{H}_{\Lambda_{\geq0}}\otimes \mathcal{H}_{\Lambda_{\leq0};0}$.
The topic of this section is the local version of the gluing isomorphism $\eta_c$ in Theorem~\ref{thm:main-intro}. 
This local picture amounts to an algebraic version of the spectral transform for the type-$A_n$ $q$-difference open Toda chain, which we will use to identify the commutative subalgebra of $\Lbb_{G,S}$ generated by the fundamental Hamiltonians $H_k(c)$ for an isolatable simple closed curve $c$ with the character ring of $G$. For convenience, we will prove the main results for the group $G=GL_{n+1}$, and then deduce from these all the needed statements for the semisimple groups $G=SL_{n+1}$ and $PGL_{n+1}$.
%as well as the natural transformation $\eta$ of Theorem~\ref{thm:long}.
%Both are based on the identification of the fundamental Hamiltonians $H_k(c)$ for an isolatable simple closed curve $c$ with the commuting Hamiltonians of the $U_{q,\tilde q}(\mathfrak{sl}_{n+1},\mathbb{R})$ Toda chain.  In this picture,   $\mathbb{W}_c$ and $\eta$ are respectively algebraic and analytic versions of a spectral transform which simultaneously diagonalizes the Toda Hamiltonanians and the Dehn twist $D_c$.
% The $U_{q,\tilde q}(\mathfrak{sl}_{n+1},\mathbb{R})$-Whittaker functions are joint eigenfunctions for
% The isomorphism $\eta$ will be obtained by simultaneously diagonalizing the Dehn twist $D_c$ and the fundamental Hamiltonians $H_k(c)$.
% % which in a special cluster coordinate system $\bi_{S;(\tri,c)}$ is realized by the operator $\tau$ in~\eqref{eq:tau-def}. 
% In \cite{SS18} it was shown that such a diagonalization is achieved by the \emph{Whittaker transform} for the modular double $U_{q,\tilde q}(\mathfrak{sl}_{n+1},\mathbb{R})$, which we will shortly review. We also discuss the algebraic avatar of this Whittaker transform, which identifies the universal Laurent ring of the Toda chain cluster algebra with the spherical part of the nil-double affine Hecke algebra of type $A_n$.

\subsection{The Toda chain cluster algebra}
\label{subsec:ltoda}
In this section we recall some known results about the $Q$-system cluster algebra associated to the group $GL_{n+1}$, which is also known as the $GL_{n+1}$ Toda chain cluster algebra. Let $\Prm \simeq \mathbb{Z}^{n+1}$ be the weight lattice for $GL_{n+1}$, with standard basis $\{\eps_i\}_{i=0}^{n}$ and basis of fundamental weights
$$
\omega_i = \sum_{r\leq i}\eps_i, \quad 1\leq i\leq n+1.
$$ 
We write $\langle\cdot,\cdot\rangle$ for the standard Euclidean form on $\Prm$ for which $\{\eps_i\}_{i=1}^{n+1}$ is an orthonormal basis, and
consider the double lattice 
$$
%\Lambda = \Lambda_{GL_{n+1}} =  \mathrm{span}\left\{p_\lambda,x_\lambda \right\}_{\lambda\in P} \simeq P\oplus P,
\La = \La_{GL_{n+1}} \simeq \Prm \oplus \Prm
$$
equipped with the skew-form
\begin{align}
\label{eq:double-form}
((\lambda_1,\lambda_2),(\mu_1,\mu_2))_{\Lambda} = \langle\lambda_1,\mu_2\rangle - \langle\lambda_2,\mu_1\rangle.
\end{align}
For $\la \in \Prm$ we set
$$
p_\la = (\la,0), \qquad x_\la = (0,\la)
$$
and write $P_\lambda$, $X_\lambda$ for the corresponding generators of the quantum torus $\mathcal{T}_{GL_{n+1}}$ associated to the double lattice. Then the nontrivial commutation relations in $\mathcal{T}_{GL_{n+1}}$ take the form
$$
P_\lambda X_{\mu} = q^{-2\langle\lambda,\mu\rangle}X_{\mu}P_\lambda, \quad \lambda,\mu\in \Prm.
$$
We will also make the abbreviations
$$
P_i = P_{\eps_i}, \quad X_i = X_{\eps_i} \quad\text{for}\quad i=1,\ldots, n+1.
$$
Consider the following pair of bases in $\Lambda$:
\begin{align}
\label{eq:toda-xi-basis}
% Old polarization for reference:
%\underline\xi = \{\xi_{s_i},\xi_{t_i}\}_{i=1}^{n+1}, \quad \xi_{s_{i}} = p_{\omega_i} + x_{\omega_i}, \quad \xi_{t_{i}} =  
\bs\xi &= \{\xi_{s_i},\xi_{t_i}\}_{i=1}^{n+1}, \qquad \xi_{s_{i}} = p_{\omega_i}+x_{\omega_i}, \qquad \xi_{t_{i}} =  x_{\omega_i}, \\
\label{eq:toda-e-basis}
% Old polarization for reference:
%\underline e = \{e_{s_i},e_{t_i}\}_{i=1}^{n+1}, \quad e_{s_{i}} = x_{\eps_{i}-\eps_{i-1}},\quad e_{t_i}= p_{\eps_{i-1}-\eps_{i}}+x_{\eps_{i-1}-\eps_{i}} \quad (\eps_{i+1}=0).
%%%% second oldest polarization for reference:
\bs e &= \{e_{s_i},e_{t_i}\}_{i=1}^{n+1}, \qquad e_{s_{i}} =  x_{-\alpha_i}, \quad e_{t_i} = p_{\alpha_i} + x_{\alpha_i},
\end{align}
where $\alpha_i = \eps_i-\eps_{i+1}$ is the $i$-th simple root, and $\alpha_{n+1} = x_{n+1}$. Then $(\bs e,\bs \xi)$ defines a compatible pair with $\Xi= \Lambda$, where the frozen directions are $I_0 = \{s_{n+1},t_{n+1}\}$.

\begin{remark}
\label{rmk:toda-embed}
%<<<<<<< Updated upstream
Note that the map sending the basis vectors $\{e_{s_i},e_{t_i}\}_{i=1}^n$ from the quiver $Q_{\tri;c}$ defined in Notation~\ref{not:LambdaS-split} to the correspondingly labelled ones in the lattice $\Lambda_{GL_{n+1}}$ is an isometry. In particular, we can use this embedding to regard $\mathcal{T}_{S_{>0}}$ as a subalgebra in the quantum torus for the $GL_{n+1}$ Toda chain.
%=======
%Note that the map sending the basis vectors $\{e_{s_i},e_{t_i}\}_{i=1}^n$ from the quiver $Q_{\tri;c}$ defined in Notation~\ref{not:LambdaS-split} to the correspondingly labelled ones in the lattice $\Lambda_{GL_{n+1}}$ is an isometry. In particular, we can use this embedding to regard $\mathcal{T}_{S>0}$ as a subalgebra in the quantum torus for the $GL_{n+1}$ Toda chain.
%>>>>>>> Stashed changes
\end{remark}

As in the case of $Q_{\tri;c}$, there is an element $\tau$ of the cluster modular group given by mutating in all directions $\{s_j\}_{j=1}^n$ and then applying the generalized permutation 
$$
t_{n+1}\mapsto s_{n+1}, \quad s_{n+1} \mapsto t_{n+1} +2s_{n+1}, \quad\text{and}\quad s_j \leftrightarrow t_j \quad\text{for all}\quad 1\le j\le n.
$$
By~\eqref{eq:Xmut-negative}, the quantum cluster transformation implementing its action on the universal Laurent ring can be expressed as
\begin{align}
\label{eq:tau}
\tau = \gamma\circ \mathrm{Ad}\left(\prod_{j=1}^n\Psi(X_{-\alpha_j})^{-1}\right), 
\end{align}
where $\gamma$ is the automorphism of $\mathcal{T}_{GL_{n+1}}$ defined on the generators by
$$
\gamma(P_i)=P_i, \quad \gamma(X_i) = qP_iX_i,\quad i=1,\ldots, n+1.
$$
%Later we will use the completed version $\widehat{\mathcal T}_{GL_{n+1}}$ of the quantum torus consisting of formal infinite sums $f = \sum_{\lambda,\mu \in \mathrm{supp}(f)} c_{\lambda,\mu}X_\lambda P_\mu$ where $c_{\lambda,\mu}\in\mathbb{Z}((q))$ which are bounded below in the sense that \red{need to fix this}
%$$
%\min\left\{\langle \rho,\lambda \rangle,-\langle \rho,\mu \rangle \right\}_{\lambda,\mu\in \mathrm{supp}(f)} >-\infty, \quad \rho = \sum_i\omega_i.
%$$

We write $\Lbb_{Toda}(GL_{n+1})$ for the universal Laurent ring corresponding to this compatible pair, and $\mathcal{A}_{Toda}(GL_{n+1})$ for the corresponding quantum cluster algebra. By the results of Goodearl and Yakimov~(\cite{GY21}, Theorem B), it is known that in fact 
$$
\mathcal{A}_{Toda}(GL_{n+1})= \Lbb_{Toda}(GL_{n+1}).
$$
Moreover, by the same Theorem both are isomorphic as $\mathbb{Z}[q^{\pm1}]$-algebras to the localization of the ring $A_{q^2}(\mathfrak{n}\hr{(s_0s_1)^{n+1}}) \otimes_{\mathbb{Z}[q^{\pm2}]} \mathbb{Z}[q^{\pm1}]$ at a pair of generalized minors, where $A_{q^2}(\mathfrak{n}(w))$ is the dual canonical form of the quantum coordinate ring of the unipotent cell associated to the affine Weyl group element $w$ in the loop group $LSL_2$. Under this isomorphism, the two inverted generalized minors correspond to the frozen variable $A_{s+1},A_{t+1}$, while the
$\tau^m(A_{t_1})$ with $-n\leq m\leq n+1$ correspond to the dual PBW generators of the quantum coordinate ring. Hence $\Lbb_{Toda}(GL_{n+1})$ is generated as a $\mathbb{Z}[q^{\pm1}]$-algebra by $\{\tau^m(A_{t_1})\}_{m\in\mathbb{Z}}$ together with the inverses of the two frozen variables.

\subsection{Residue algebra}
\label{sec:nildaha}

In this section we define an explicit subalgebra $\Dres$ of rational $q$-difference operators on the maximal torus $T$ of $GL_{n+1}$ which will serve as the ``spectral'' description of the cluster algebra $\Lbb_{Toda}(GL_{n+1})$.

%The algebra $\Dres$ will be shown in Theorem~\ref{thm:toda-sdiff-iso} to be isomorphic to the universal Laurent ring $\Lbb_{Toda}$, and as a Corollary, to the spherical nil-DAHA $\SH_{q,t=0}$.
%This isomorphism will identify the Hamiltonians $H_k\in\Lbb_{Toda}$ with multiplication operators, and so can be thought of as an algebraic version of the spectral transform for the Toda system.

We define the algebra of $q$-difference operators on $T$ to be the quantum torus
\begin{align}
\label{eq:DqT}
\Dc_q(T) = \Z[q^{\pm1}]\ha{w_i^{\pm1}, D_i^{\pm} \,\big|\, 1 \le i \le n+1}
\end{align}
with defining relations
$$
D_iw_j = q^{2\delta_{ij}}w_jD_i, \qquad w_iw_j=w_jw_i, \qquad D_iD_j = D_jD_i.
$$
The $w_i$ generate a commutative subalgebra which we identify with the group ring $\Oc(T)$ of the weight lattice $\Prm$ of $GL_{n+1}$. The action of the invertible generators $w_i$ by conjugation defines an \emph{internal grading} on $\Dc_q(T)$ by $\Prm$. We write $\Dc^\mu_q(T)$ for the $\mu$-graded piece with respect to this grading, with the sign convention that
$$
A \in \Dc^\mu_q(T) \iff A w_\lambda  = q^{2\langle \la,\mu\rangle}w_\lambda A \quad \text{for all } \la \in P.
$$
In concrete terms, $\Dc^\mu_q(T)$ consists of all elements of the form $f(\bs w) D_1^{\mu_1}\cdots D_{n+1}^{\mu_{n+1}}$ where $f$ is an arbitrary Laurent polynomial in $\bs w$. For a general element $A\in \Dc_q(T)$ we write $A_\mu$ for its degree $\mu$ component, so that
$$
A = \sum_\mu A_\mu, \quad A_\mu \in \Dc^\mu_q(T) .
$$
The algebra $\Dc_q(T)$ has a standard representation on $\Oc(T)$ in which the subalgebra $\Oc(T)$ acts on itself by multiplication operators, and the generators $D_i$ act by 
\begin{align}
\label{eq:Di}
(D_i\circ f)(\bs w) = f(w_1,\ldots, q^2w_i,\ldots, w_{n+1}).
\end{align}
Hence as a left $\Dc_q(T)$ module, this representation is identified with the quotient of $\Dc_q(T)$ by the left ideal $\langle D_i-1\rangle_{i=1}^{n+1}$. We will sometimes abuse notation by writing $Af$ instead of $A \circ f$ whenever it is clear from the context that we mean the action of $\Dc_q(T)$ on $\Oc(T)$ rather than the multiplication in $\Dc_q(T)$.

Now consider the Ore localization
\begin{align}
\label{eq:ore1}
\Dc_q(T)_\loc = \Dc_q(T)\big[(1-q^{2k}w_\alpha)^{-1} \,\big|\, \alpha \in \Delta_+, \; k\in\Z\big],
\end{align}
where for a positive root $\alpha =\eps_i-\eps_j$ we set $w_{\alpha} = w_i/w_j$. Since the elements we invert commute with all of $\Oc(T)$, the localized ring also inherits an internal grading whose pieces we denote by $\Dc^\mu_q(T)_\loc $. Given $\alpha \in \Delta_+$ and $k \in \Z$ we introduce the following notation for the divisors at which we localize in~\eqref{eq:ore1}:
\begin{align}
\label{eq:divisors}
d_{\alpha,k} = \hc{w_{\alpha} = q^{-2k}}.
\end{align}
We write $\Oc(T)_\loc$ for the subring of $\Dc_q(T)_\loc$ given by its internal-degree zero piece. This subring is commutative and we can identify it with the ring of rational functions on $T$ regular away from the union of the $d_{\alpha,k}$. We again have a standard difference operator representation of the algebra $\Dc_q(T)_\loc$ on $\Oc(T)_\loc$.

Associated to each divisor $d_{\alpha,k}$ we have an increasing \emph{order} filtration
\beq
\label{eq:filtration}
\ldots \subset \Dc_q(T)_\loc^{\alpha,k;-1} \subset \Dc_q(T)_\loc^{\alpha,k;0} \subset \Dc_q(T)_\loc^{\alpha,k;1} \subset \ldots
\eeq
of $\Dc_q(T)_\loc$ by left $\Oc(T)_\loc$-modules
$$
\Dc_q(T)_\loc^{\alpha,k;m} = (1-q^{2k}w_\alpha)^{m}\Dc_q(T)\big[(1-q^{2\ell}w_\beta)^{-1} \,\big|\, \beta \ne \alpha, \; \ell \ne k].
$$
So associated to each divisor $d_{\alpha,k}$ is a graded vector space
$$
\gr_{\alpha,k}(\Dc_q(T)_\loc) = \bigoplus_{m \in \Z} \gr_{\alpha,k}^m(\Dc_q(T)_\loc)
$$
where
$$
\gr_{\alpha,k}^m(\Dc_q(T)_\loc) = \Dc_q(T)_\loc^{\alpha,k;m}/\Dc_q(T)_\loc^{\alpha,k;m+1},
$$
along with natural projections
\beq
\label{eq:pi-alpha-k}
\pi_{\alpha,k;m} \colon \Dc_q(T)_\loc^{\alpha,k;m} \longra \gr_{\alpha,k}^m(\Dc_q(T)_\loc).
\eeq
Note that the order filtration associated to any $d_{\alpha,k}$ is compatible with and completely independent from the internal grading defined earlier, so each order-filtered piece can be further decomposed by internal degree:
$$
\Dc^\mu_q(T)_\loc^{\alpha,k;m} = \Dc_q(T)_\loc^{\alpha,k;m} \cap \Dc_q^\mu(T)_\loc,
$$
%The projections~\eqref{eq:pi-alpha-k} respect the internal grading, that is
$$
\pi_{\alpha,k;m} \colon \Dc_q^\mu(T)_\loc^{\alpha,k;m} \longra \gr_{\alpha,k}^m(\Dc_q^\mu(T)_\loc).
$$
\begin{defn}
\label{def:eval-res}
We say that an element $A \in \Dc_q(T)_\loc$ has a \emph{pole of order $m$} at the divisor $d_{\alpha,k}$ if $A \in \Dc_q(T)_\loc^{\alpha,k;-m}$ and $\pi_{\alpha,k;-m}(A) \ne 0$. For $A \in \Dc_q(T)_\loc^{\alpha,k;-1}$ we use the shorthand
$$
\pi_{\alpha,k;-1}(A) = \res_{\alpha,k}(A),
$$
and refer to $\res_{\alpha,k}(A)$ as the \emph{residue} of $A$ at $d_{\alpha,k}$. Similarly, for $A \in \Dc_q(T)_\loc^{\alpha,k;0}$ we abbreviate
$$
\pi_{\alpha,k;0}(A) = A|_{d_{\alpha,k}},
$$
and say that $A|_{d_{\alpha,k}}$ is the \emph{restriction} of $A$ to $d_{\alpha,k}$.
\end{defn}

 Given a weight $\mu \in \Prm$ consider the automorphism $\Ad_{D^\mu}$ of $\Dc_q(T)_\loc$ given by
\begin{align}
\label{eq:shifts}
\Ad_{D^\mu}\colon A \longmapsto A^{[\mu]} = D^\mu A D^{-\mu}, \quad D^\mu = D_1^{\mu_1}\cdots D_{n+1}^{\mu_{n+1}}.
\end{align}
Note that unlike $\Ad_{w_\mu}$, the automorphism $\Ad_{D^\mu}$ does not respect the order filtrations on $\Dc_q(T)_\loc$ and is not even diagonalizable. Instead, the isomorphisms~\eqref{eq:shifts} satisfy
\beq
\label{eq:pi-ad}
\Ad_{D^\mu} \circ \pi_{\alpha,k;n} = \pi_{\alpha,k+\ha{\alpha,\mu};n} \circ \Ad_{D^\mu}.
\eeq
and therefore induce isomorphisms of $\mathbb{Z}[q^{\pm1}]$-modules between associated graded factors for order filtrations associated to different divisors:
\beq
\label{eq:ad-D-iso}
\Ad_{D^\mu} \colon \gr_{\alpha,k}^m(\Dc_q(T)_\loc) \longra \gr_{\alpha,k+\ha{\alpha,\mu}}^m(\Dc_q(T)_\loc),
\eeq
which are compatible with the internal grading.
%$$
%\Ad_{D^\mu} \colon \gr_{\alpha,k}^m(\Dc_q^\mu(T)_\loc) \longra \gr_{\alpha,k+\ha{\alpha,\mu}}^m(\Dc_q^\mu(T)_\loc).
%$$

There is a natural action of the Weyl group $W=S_{n+1}$ by automorphisms of the algebra $\Dc_q(T)_\loc$ induced by the reflection representation on the weight lattice, or equivalently by the permutation action on the coordinates $\bs w = (w_1, \ldots, w_{n+1})$ and $\bs D = (D_1, \ldots, D_{n+1})$. We write $\Dc_q(T)^W_\loc$ for the subalgebra of $W$-invariants with respect to this action. Note that $A \in \Dc_q(T)_\loc^W$ if and only if $s_\alpha(A_\mu) = A_{s_\alpha(\mu)}$ for every $\mu \in \Prm$ and $\alpha \in \Delta_+$. The Weyl group action induces isomorphisms of $\mathbb{Z}[q^{\pm1}]$-modules
$$
s_\alpha \colon \gr_{\alpha,k}^m(\Dc_q^\mu(T)_\loc) \longra \gr_{\alpha,-k}^m(\Dc_q^{s_\alpha(\mu)}(T)_\loc),
$$
which satisfy
\beq
\label{eq:pi-s}
s_\alpha \circ \pi_{\alpha,k;m} = \pi_{\alpha,-k;m} \circ s_\alpha.
\eeq

\begin{lemma}
\label{lem:symmetry-condition}
For any Weyl-invariant element $A \in \Dc_q(T)_\loc^W$, any divisor $d_{\alpha,k}$, any  choice of internal degree $\mu \in \Prm$, and any choice of $d_{\alpha,k}$-order $m\in \Z$, we have
\beq
\label{eq:symmetry-condition}
\pi_{\alpha,k;m}\hr{A_\mu + (-1)^{m+1}A_{s_{\alpha(\mu)}}^{[k\alpha]} D^{\ha{\alpha,\mu}\alpha}}=0.
\eeq
\end{lemma}

\begin{proof}
Write
$$
\pi_{\alpha,k;m}(A_\mu) = (1-q^{2k}w_\alpha)^m a_{\mu;m}(\bs w) D^\mu + \Dc_q^\mu(T)_\loc^{\alpha,k;m+1}
$$
where $a_{\mu;m} \in \Dc_q^0(T)_\loc^{\alpha,k;0}$ is a rational function on $T$ regular at $d_{\alpha,k}$. Then we have
\begin{align*}
\pi_{\alpha,k;m}\hr{(-1)^{m+1}A_{s_{\alpha(\mu)}}^{[k\alpha]} D^{\ha{\alpha,\mu}\alpha}}
&= (-1)^{m+1}\pi_{\alpha,k;n}\hr{\Ad_{D^{k\alpha}}(s_\alpha(A_\mu))} D^{\ha{\alpha,\mu}\alpha} \\
&= (-1)^{m+1}\Ad_{D^{k\alpha}}(s_\alpha(\pi_{\alpha,k;m}\hr{A_\mu})) D^{\ha{\alpha,\mu}\alpha} \\
&= (1-q^{-2k}w_\alpha^{-1})^m \Ad_{D^{k\alpha}}(s_\alpha(a_{\mu;m}(\bs w))) D^\mu + \Dc_q^\mu(T)_\loc^{\alpha,k;m+1} \\
&= (-1)^m \pi_{\alpha,k;m}(A_\mu),
\end{align*}
The second equality follows from~\eqref{eq:pi-ad} and~\eqref{eq:pi-s}, and the last one from the relation
$$
\Ad_{D^{k\alpha}}(s_\alpha(f(\bs w))) \big|_{d_{\alpha,k}} = f(\bs w) \big|_{d_{\alpha,k}},
$$
which holds for any $f \in \Dc_q^0(T)_\loc^{\alpha,k;0}$.
\end{proof}

% and we form the semidirect/smash product algebra
% $
% \mathcal{D}_q[T]^{loc}\rtimes \mathbb{Z}S_{n+1}.
% $

%The algebra $\Dc_q(T)_\loc^W$ acts on the field of symmetric rational functions $\Q(q)(w_0,\ldots, w_n)^W$, and preserves the subring of the latter consisting of functions regular outside of the union of divisors
%\beq
%\label{eq:divisors-def}
%\bs d = \bigcup_{\substack{\alpha \in \Delta_+ \\ k\in\Z}} d_{i,j,k}
%\qquad\text{where}\qquad
%d_{\alpha,k} = \hc{w_\alpha = q^{-2k}}.
%\eeq

The following definition and lemma are modelled on the constructions of~\cite{GKV97}, in particular their Definition 1.3 and Theorem 1.4.
%where each $A_\mu(\bs w)$ is regular outside of the union of divisors
%$$
%\bs d = \bigcup_{\substack{\alpha \in \Delta_+ \\ k\in\Z}} d_{\alpha,k}
%\qquad\text{where}\qquad
%d_{\alpha,k} = \hc{w_\alpha = q^{-2k}}.
%$$

\begin{defn}
\label{def:Dres}
We define $\Dres$ to be the subspace of $\Dc_q(T)_\loc^W$ consisting of elements $A$ such that 
\begin{enumerate}
\item for each $\alpha \in \Delta_+$ and $k \in \Z$ the element $A$ has at most a simple pole at $d_{\alpha,k}$;
\item for all $\mu \in \Prm$, $\alpha \in \Delta_+$, and $k \in \Z$ the corresponding residues satisfy
\beq
\label{eq:residue-condition}
\res_{\alpha,k}\hr{A_{\mu}+A_{s_\alpha(\mu)+k\alpha}D^{(\ha{\alpha,\mu}-k)\alpha}}=0.
\eeq

\end{enumerate}
\end{defn}
% \begin{lemma}
% \label{lem:residues}
%     The action of an element $A = \sum_{\mu\in\mathbb{Z}^{n+1}} A_\mu(\bs w)D^\mu$ of $\mathcal{D}_q(T)^W_{loc}$ preserves the subspace $\mathbb{Q}[w^\pm_0,\ldots, w^\pm_n]^{W}\subset \mathbb{Q}(w_0,\ldots, w_n)^W$ of symmetric Laurent polynomials if and only if
%     \begin{enumerate}
%         \item each coefficient $a_\mu(\bs w)$ has at worst a simple pole at each of the divisors $\delta_{i,j,k}$, and
%         \item the corresponding residues
%         $$
%         \Res_{\delta_{i,j,k}}(A_\mu) := \left((1-q^{2k}w_i/w_j)A_\mu(\bs w)\right)\bigg|_{\delta_{i,j,k}}
%         $$
%         satisfy
%         \begin{align}
% \label{eq:residue-condition}
%     \mathrm{Res}_{\delta_{i,j,k}}(A_{\mu}) + \mathrm{Res}_{\delta_{i,j,k}}(A_{s(\mu)+k(\epsilon_i-\epsilon_j)})=0,
% \end{align}
%     \end{enumerate}
% where $(\epsilon_i)$ is the standard basis in $\mathbb{Z}^{n+1}$ and $s = (ij)$ is the reflection in the root hyperplane $\{w_i=w_j\}$.
% \end{lemma}

\begin{remark}
Because $A$ is required to be invariant under the action of the Weyl group, it would suffice to impose the condition~\eqref{eq:residue-condition} for a single positive root $\alpha \in \Delta_+$.
\end{remark}

Note that for $k=0$ the condition~\eqref{eq:residue-condition} is a particular case of~\eqref{eq:symmetry-condition}, and hence is automatically satisfied. Moreover, combining~\eqref{eq:symmetry-condition} and~\eqref{eq:residue-condition} we obtain
\beq
\label{eq:sym-res}
\res_{\alpha,k}\hr{D^{k\alpha}A_\nu-A_{\nu+k\alpha}}=0
\eeq
for any $A \in \Dres$, $\nu \in \Prm$, $\alpha \in \Delta_+$, and $k \in \Z$.

%Note that the reflection $s_{ij}$ sends the divisor $\delta_{i,j,k}$ to $\delta_{i,j,-k}$, and for any function $f$ we have 
%\begin{align}
%\label{eq:jacobian-res}
%s_{ij}\hr{    \mathrm{Res}_{\delta_{i,j,k}}f } = -\mathrm{Res}_{\delta_{i,j,-k}}s_{ij}(f) .
%\end{align}
%Combining this with the $W$-invariance it follows that for any element $A$ of $\mathcal{D}_{res}$ we have
%\begin{align}
%\label{eq:sym-res}
%s_{ij}\hr{    \mathrm{Res}_{\delta_{i,j,k}}A_\mu } = -\mathrm{Res}_{\delta_{i,j,-k}}A_{s_{ij}(\mu)} .
%\end{align}

\begin{lemma}
The subspace $\mathcal{D}_{res}$ is closed under multiplication, and hence forms a subalgebra in  $\mathcal{D}_q(T)_{loc}^W$.
\end{lemma}

\begin{proof}
First we show that the product of two elements of $\Dres$ has at most simple poles at each divisor $d_{\alpha,k}$.
For this we need to show that for any pair of elements $A,B \in \Dres$, any $\alpha \in \Delta_+$, and any $k \in \Z$ we have
\beq
\label{eq:pi-2-vanishes}
\pi_{\alpha,k;-2}(AB)=0.
\eeq
%Since the projection maps $\pi_{\alpha,k;n}$ preserve the internal grading of $\Dc_q(T)_\loc$ by $P$, the condition~\eqref{eq:pi-2} is equivalent to
%$$
%\pi_{\alpha,k;-2}\hr{(AB)_\nu}=0 \qquad\text{for all}\qquad \nu \in P.
%$$
%for all $\nu \in P$.
%
We can express the component of weight $\nu \in \Prm$ of the product $AB$ as
$$
(AB)_\nu = \frac12 \sum_{\mu\in P} K_{\alpha,k;\mu}
\qquad\text{where}\qquad
K_{\alpha,k;\mu} = A_\mu B_{\nu-\mu} + A_{s_\alpha(\mu)+k\alpha} B_{\nu-s_\alpha(\mu)+k\alpha}.
$$
Observing that
$$
\pi_{\alpha,k;-2}\hr{A_\mu B_{\nu-\mu}} = \res_{\alpha,k}(A_\mu) \res_{\alpha,k-\ha{\alpha,\mu}}(B_{\nu-\mu})
$$
we obtain 
\begin{multline*}
\pi_{\alpha,k;-2}\hr{K_{\alpha,k;\mu}}
= \res_{\alpha,k}(A_\mu) \res_{\alpha,k-\ha{\alpha,\mu}}(B_{\nu-\mu}) \\
+ \res_{\alpha,k}(A_{s_\alpha(\mu)+k\alpha}) \res_{\alpha,\ha{\alpha,\mu}-k}(B_{\nu-s_\alpha(\mu)-k\alpha}).
\end{multline*}
Adding and subtracting the product
\begin{multline*}
\res_{\alpha,k}(A_{s_\alpha(\mu)+k\alpha}) D^{(\ha{\alpha,\mu}-k)\alpha} \res_{\alpha,k-\ha{\alpha,\mu}}(B_{\nu-\mu}) \\
= \res_{\alpha,k}(A_{s_\alpha(\mu)+k\alpha}) \res_{\alpha,\ha{\alpha,\mu}-k}\hr{D^{(\ha{\alpha,\mu}-k)\alpha} B_{\nu-\mu}}
\end{multline*}
we arrive at the sum
\begin{multline*}
\res_{\alpha,k}\hr{A_\mu+A_{s_\alpha(\mu)+k\alpha} D^{(\ha{\alpha,\mu}-k)\alpha}} \res_{\alpha,k-\ha{\alpha,\mu}}\hr{B_{\nu-\mu}} \\ + \res_{\alpha,k}\Big(A_{s_\alpha(\mu)+k\alpha}\Big) \res_{\alpha,\ha{\alpha,\mu}-k} \hr{B_{\nu-s_\alpha(\mu)-k\alpha} - D^{(\ha{\alpha,\mu}-k)\alpha} B_{\nu-\mu}}.
\end{multline*}
But the first summand vanishes by~\eqref{eq:residue-condition} and the second one by~\eqref{eq:sym-res}, and therefore~\eqref{eq:pi-2-vanishes} holds.

It now remains to show that the condition~\eqref{eq:residue-condition} holds for the product $AB$, that is
$$
\res_{\alpha,k}\hr{(AB)_\nu + (AB)_{s_\alpha(\nu)+k\alpha} D^{(\ha{\alpha,\nu}-k)\alpha}}=0
$$
for all $\nu \in \Prm$, $\alpha\in\Delta_+$, and $k \in \Z$. As before, we write
$$
(AB)_\nu + (AB)_{s_\alpha(\nu)+k\alpha} D^{(\ha{\alpha,\nu}-k)\alpha} = \frac12 \sum_{\mu \in P} L_{\alpha,k;\mu},
$$
where
\begin{multline*}
L_{\alpha,k;\mu} = A_\mu\hr{B_{\nu-\mu}+B_{s_\alpha(\nu)-\mu+k\alpha}D^{(\ha{\alpha,\nu}-k)\alpha}} \\
+ A_{s_\alpha(\mu)+k\alpha}\hr{B_{\nu-s_\alpha(\mu)-k\alpha}+B_{s_\alpha(\nu)-s_\alpha(\mu)}D^{(\ha{\alpha,\nu}-k)\alpha}}
\end{multline*}
By~\eqref{eq:residue-condition} we have
\begin{align*}
&\res_{\alpha,k-\ha{\alpha,\mu}}\hr{B_{\nu-\mu}+B_{s_\alpha(\nu)-\mu+k\alpha}D^{(\ha{\alpha,\nu}-k)\alpha}} = 0, \\
&\res_{\alpha,\ha{\alpha,\mu}-k}\hr{B_{\nu-s_\alpha(\mu)-k\alpha}+B_{s_\alpha(\nu)-s_\alpha(\mu)}D^{(\ha{\alpha,\nu}-k)\alpha}} = 0,
\end{align*}
and therefore
\begin{multline*}
\res_{\alpha,k}(L_{\alpha,k;\mu}) = \res_{\alpha,k}\hr{A_\mu} \hr{B_{\nu-\mu}+B_{s_\alpha(\nu)-\mu+k\alpha}D^{(\ha{\alpha,\nu}-k)\alpha}} \Big|_{d_{\alpha,k-\ha{\alpha,\mu}}} \\
+ \res_{\alpha,k}\hr{A_{s_\alpha(\mu)+k\alpha}} \hr{B_{\nu-s_\alpha(\mu)-k\alpha}+B_{s_\alpha(\nu-\mu)}D^{(\ha{\alpha,\nu}-k)\alpha}} \Big|_{d_{\alpha,\ha{\alpha,\mu}-k}}.
\end{multline*}
It follows from~\eqref{eq:symmetry-condition} that
\begin{multline*}
D^{(\ha{\alpha,\mu}-k)\alpha} \hr{B_{\nu-\mu}+B_{s_\alpha(\nu)-\mu+k\alpha}D^{(\ha{\alpha,\nu}-k)\alpha}} \Big|_{d_{\alpha,k-\ha{\alpha,\mu}}} \\
= \hr{B_{\nu-s_\alpha(\mu)-k\alpha}+B_{s_\alpha(\nu-\mu)}D^{(\ha{\alpha,\nu}-k)\alpha}} \Big|_{d_{\alpha,\ha{\alpha,\mu}-k}},
\end{multline*}
and thus $\res_{\alpha,k}(L_{\alpha,k;\mu})$ is equal to the product
$$
\res_{\alpha,k} \hr{A_\mu + A_{s_\alpha(\mu)+k\alpha} D^{(\ha{\alpha,\mu}-k)\alpha}} \hr{B_{\nu-\mu}+B_{s_\alpha(\nu)-\mu+k\alpha}D^{(\ha{\alpha,\nu}-k)\alpha}} \Big|_{d_{\alpha,k-\ha{\alpha,\mu}}},
$$
whose first factor vanishes by~\eqref{eq:residue-condition}. This concludes the proof.
\end{proof}
\begin{remark}
Because of the Weyl invariance condition in Definition~\ref{def:Dres}, the algebra $\Dres$ no longer carries an internal grading. Nonetheless, the internal grading on $\Dc_q(T)_{loc}$ induces a filtration
%$$
%\Dc_q^{\eta,\mu}(T)_{loc} = \bigcup_{\eta \le \nu \le \mu} \Dc_q^\nu(T)_{loc},
%$$
%where $\eta \le \nu$ if $\nu-\eta \in P_+$. Then we also have an 
$$
\Dres=\bigcup_{\mu \in P^+}\Dres^{\leq\mu}, \qquad \Dres^{\leq\mu} = \Dres \cap \Dc_q^{\leq\mu}(T)_{loc},
$$
where 
$$
\Dc_q^{\leq\mu}(T)_{loc} = \bigoplus_{\nu\leq \mu}\Dc_q^{\nu}(T)_{loc}
$$
and the dominance order relation $\nu\leq \mu$ means that $\mu-\nu$ is a linear combination of simple roots with non-negative coefficients.
\end{remark}

%The following lemma explains the point of the conditions in Definition~\ref{def:Dres}.

%Let $I_D$ be the left ideal in $\Dc_q(T)_\loc$ generated by $\ha{D_i-1}_{i=0}^n$. The natural isomorphism of $\Dc_q^0(T)_\loc$-modules
%$$
%\Dc_q^0(T)_\loc \simeq \Dc_q(T)_\loc/I_D
%$$
%makes $\Dc_q^0(T)_\loc$ a $\Dc_q(T)_\loc$-module via
%$$
%A \circ f = Af+I_D.
%$$
The main point of the conditions in Definition~\ref{def:Dres} transpires from the following lemma:
%Note that the subspace $\Dc_q^0(T)^W \subset \Dc^0_q(T)_\loc$ coincides with the space $\Z[q^{\pm1}][\bs w^{\pm1}]^W$ of symmetric Laurent polynomials in $\bs w$.
\begin{lemma}
\label{lem:res-preserved}
The action of the $\Dres$ on $\Oc(T)_\loc$ preserves the subspace $\Oc(T)^W \subset \Oc(T)_\loc$ of symmetric Laurent polynomials.
\end{lemma}
\begin{proof}
For any $A \in \Dres$ and $f \in \Oc(T)^W $ it is clear that $A \circ f $ can have at worst simple poles at each divisor $d_{\alpha,k}$, so it suffices to check that $\res_{\alpha,k}(A \circ f)=0$ for all $\alpha \in\Delta_+$ and $k\in\Z$. Writing
$$
A \circ f = \frac12 \sum_{\mu \in P} \hr{A_{\mu}+A_{s_\alpha(\mu)+k\alpha}} \circ f,
$$
and observing by~\eqref{eq:symmetry-condition} that
$$
\big(D^\mu \circ f\big) \big|_{d_{\alpha,k}} = \big(D^{s_\alpha(\mu)+k\alpha} \circ f \big) \big|_{d_{\alpha,k}},
$$
we arrive at
$$
\res_{\alpha,k}(A \circ f) = \frac12 \sum_{\mu \in P} \res_{\alpha,k} \big(A_{\mu}D^{-\mu}+A_{s_\alpha(\mu)+k\alpha} D^{(k-\ha{\alpha,\mu})\alpha}\big) \big(D^\mu \circ f\big) \big|_{d_{\alpha,k}}.
$$
The latter expression vanishes by~\eqref{eq:residue-condition}, which concludes the proof.
\end{proof}
%\red{I changed this to get rid of inverse in $\Psi$, any objection?}
%\blue{Not at all, but why the $q^{-1}$?}
%Consider a filtration of $\Dc_q(T)_{loc}$ by subspaces
%$$
%\Dc_q^{\eta,\mu}(T)_{loc} = \bigcup_{\eta \le \nu \le \mu} \Dc_q^\nu(T)_{loc},
%$$
%where $\eta \le \nu$ if $\nu-\eta \in P_+$. Then we also have an induced filtration
%$$
%\Dres=\bigcup_{\eta, \mu \in P}\Dres^{\eta,\mu}, \qquad \Dres^{\eta,\mu} = \Dres \cap \Dc_q^{\eta,\mu}(T)_{loc}.
%$$

Let us introduce 
$$
\Delta(\bs w) = \prod_{\alpha\in\Delta_+}\Psi_q(-q^{-1}w_\alpha).
$$
% Like the usual Vandermonde, it has simple zeros when $w_i=w_j$, although in addition to these it has zeros and poles at infinitely many other divisors.
Given a dominant weight $\mu \in \Prm^+$, we define an element $\widetilde{D}^\mu \in \Dc_q(T)_{loc}$ by
%$$
%\widetilde{D}^\mu = \Delta_{alg}(\bs w)^{-1}D^\mu \Delta_{alg}(\bs w) = \tilde{A}_\mu D^\mu,
%$$
$$
%\label{eq:def-tildeD}
\widetilde{D}^\mu = \Delta(\bs w)D^\mu \Delta(\bs w)^{-1} = d_\mu(\bs w) D^\mu,
$$
where
$$
d_\mu(\bs w) =\prod_{\alpha \in \Delta_+} \prod_{k=1}^{\ha{\alpha,\mu}} \hr{1-q^{2(k-1)}w_\alpha}^{-1}.
$$
Recall that a weight $\mu \in \Prm^+$ is \emph{minuscule} if $\ha{\alpha,\mu} \le 1$ for any $\alpha \in \Delta_+$, so that there is no dominant weight strictly less than $\mu$ in the dominance order. For $GL_{n+1}$, minuscule weights are those of the form
$$
\omega = \omega_k + l\omega_{n+1}, \quad 0\leq k\leq n, ~l\in\mathbb{Z}
$$
where $\omega_{j}=\sum_{i=1}^{j}\eps_i$ are the fundamental weights (and so $\omega_0=0$ and $\omega_{n+1}$ is the determinant character).
%
%For $GL_{n+1}$, the minuscule weights are those of the form $\mu = \omega_{l,m}$ for some $1 \le l \le n+1$ and $m \in \Z$, where
%$$
%\omega_{l,m} =
%\begin{cases}
%m\omega_{n+1} + \omega_l &\text{if} \quad 1 \le l \le n, \\
%m\omega_{n+1} &\text{if} \quad l=0.
%\end{cases}
%$$
In particular, for a minuscule weight $\mu$ we have
$$
d_{\mu}(\bs w) =\prod_{\ha{\alpha,\mu}=1} \hr{1-w_\alpha}^{-1}.
$$
%\beq
%\label{eq:Dtilde-residues}
%\res_{\alpha,k}(\widetilde{D}_\mu) =
%\begin{cases}
%1 \quad & 0\leq k< \ha{\alpha,\mu}, \\
%0 \quad & \text{otherwise,}
%\end{cases}
%\eeq
%and
%$$
%\widetilde{D}^{\omega_r} = \prod_{\substack{0\le i<r \\ r \le j\le n}} \frac{1}{1-w_i/w_j} \cdot D^{\omega_r}.
%$$
%, and its leading coefficient with respect to the dominance order on $P_+$ is $f(\bs w) \widetilde{D}^{\mu}$.

% It follows from Lemma~\ref{lem:residues} that the linear subspace of $\mathcal{D}_q(T)_{loc}^W$ satisfying conditions (1-2) forms an algebra, which we denote by $\mathcal{D}_{res}$. 

\begin{defn}
An element  of $\Dres$ is said to be \emph{minuscule} if it is contained in $\Dres^{\leq \omega}$ for some minuscule weight $\omega$.
\end{defn}
Any minuscule element has the form
$$
\Rc_{\omega}[f] = \sum_{[\sigma] \in W/\mathrm{Stab}(\omega)} \sigma\hr{f(\bs w) \widetilde D^{\omega}}, \quad f \in \Oc(T)^{\mathrm{Stab}(\omega)},
$$
where we sum over coset representatives for the stabilizer $\mathrm{Stab}(\omega)\simeq S_k\times S_{n+1-k}$ in $W\simeq S_{n+1}$ of the minuscule weight $\omega=\omega_k + l\omega_{n+1}$. We abbreviate
$$
\Rc_{\omega}= \Rc_{\omega}[1],
$$
and say that the minuscule element $\Rc_{\omega}[f]$ is \emph{dressed} by the function $f$.

\begin{example}
Let us identify
\beq
\label{eq:T-Stab}
\Oc(T)^{\mathrm{Stab}(\omega)} \simeq \Z\hs{w_1^{\pm1}, \ldots, w_k^{\pm1}}^{S_k} \otimes \Z\hs{w_{k+1}^{\pm1}, \ldots, w_{n+1}^{\pm1}}^{S_{n+1-k}}.
\eeq
Then for $f = e_k \otimes 1$, where $e_j$ denotes the $j$-th elementary symmetric functions, we have
$$
\Rc_{\omega_k}[e_k \otimes 1] = \sum_{|I|=k}\prod_{\substack{i \in I \\ j \notin I}}\frac{1}{1-w_j/w_i} \prod_{i\in I}w_iD_i.
$$
\end{example}

\begin{remark}
It is not hard to check directly that the elements $\Rc_{\omega}$ generate a commutative subalgebra in $\Dres$ as $\omega$ runs over the set of all minuscule weights. In Section~\ref{sec:alg-whit} we will construct a basis in $\mathbb{Z}[q^{\pm1}]$ consisting of joint eigenvectors for this subalgebra.
\end{remark}
%\begin{remark}
%\label{rem:partial-symmetry}
%Note that in the formula~\eqref{eq:sym-elements} we may assume that the function $f$ is invariant under the action of the subgroup $S_l \times S_{n+1-l} \subset S_{n+1}$, where $0 \le l \le n$ is such that $\mu = \omega_l + m\omega_{n+1}$ for some $m \in \Z$.
%\end{remark}

    %     We extend the definition above to general dominant weights $\mu = (\mu_0\geq \mu_1\geq\ldots \geq \mu_n),~ \mu = \sum_{a=1}^{n+1}k_a\omega_a$ of $GL_n$ by writing
    %     \begin{align*}
    %         \widetilde{D}_\mu = \prod_{a=1}^{n+1}\widetilde{D}_{\omega_a}^{k_a},
    %     \end{align*}
    %     and then to arbitrary elements of the weight lattice using the Weyl group action:
    %     $$
    %     \widetilde{D}_\mu = \sigma \cdot \widetilde{D}_{\mu^+}, \quad \sigma(\mu)=\mu^+\in P^+.
    %     $$

% Consider the $\mathbb{Z}^{n+1}$-grading on the ring $\mathcal{D}_q(T)_{loc}$ given by $\deg(D^\mu)=\mu,~ \deg(w_i)=0$. It induces a filtration on the subring $\mathcal{D}_q(T)^W_{loc}$

\begin{lemma}
\label{lem:minuscule-generation}
The algebra $\Dres$ is generated over $\Z[q^{\pm1}]$ by the minuscule elements.
\end{lemma}

\begin{proof}
Given a general element $A$ of $\Dres$, let us consider the finite set of summands $A_\mu$ of $A$ having maximal (in dominance order) internal degree as elements of $\Dc_q(T)_\loc$. Since $A$ is Weyl-invariant, this gives us a finite set of dominant weights $\mu_i\in \Prm^+$. Define the total \emph{height} of $A$ to be the sum of the heights $\langle \rho,\mu_i\rangle$ of each of these weights. %Consider a filtration of $\Dc_q(T)_{loc}$ by subspaces
%$$
%\Dc_q^{\eta,\mu}(T)_{loc} = \bigcup_{\eta \le \nu \le \mu} \Dc_q^\nu(T)_{loc},
%$$
%where $\eta \le \nu$ if $\nu-\eta \in P_+$. Then we also have an induced filtration
%$$
%\Dres=\bigcup_{\eta, \mu \in P}\Dres^{\eta,\mu}, \qquad \Dres^{\eta,\mu} = \Dres \cap \Dc_q^{\eta,\mu}(T)_{loc}.
%$$
%We shall argue by induction on $\mu-\eta$. For the base of induction, note that $\Dres^{\mu,\mu}=0$ unless $\mu_0=\mu_n$, in which case $\Dres^{\mu,\mu} = \Dc_q^0(T)^W D^\mu = \Dc_q^0(T)^W \widetilde D^\mu$. 
Now fix one such weight $\mu$.
%, and let $A_\mu$ Consider an element
%$$
%A = A_\mu + \sum_{\eta \le \nu<\mu}A_\nu \in \Dres^{\eta,\mu}.
%$$ 
The residue conditions~\eqref{eq:residue-condition} imply that the summand $A_\mu$ is regular except possibly at the divisors $d_{\alpha,k}$ with $0 \le k < \ha{\alpha,\mu}$, at which $A$ may have simple poles. Indeed, let us first assume that $k\ge\ha{\alpha,\mu}$, which implies $s_\alpha(\mu)+k\alpha \ge \mu$. Then the condition~\eqref{eq:residue-condition} together with the dominance-maximality of $A_\mu$ among the summands of $A$ implies that $\res_{\alpha,k}(A_\mu)=0$. On the other hand, if $k<0$ then $s_\alpha(\mu)+k\alpha=s_\alpha(\mu-k\alpha)$ belongs to the $W$-orbit of the element $\mu-k\alpha > \mu$. By the $W$-invariance of $A$ we have $A_{s_\alpha(\mu)+k\alpha}=s_\alpha(A_{\mu-k\alpha})=0$, and hence once again $\res_{\alpha,k}(A_\mu)=0$ by~\eqref{eq:residue-condition}. This shows the summand $A_\mu$ can indeed only have (simple) poles at the divisors $d_{\alpha,k}$ with $0 \le k<\ha{\alpha,\mu}$.

It follows that $A_\mu$ has the form $f(\bs w) \widetilde D^\mu$ for some $f \in \Oc(T)$. But since $\Oc(T)$ is generated as an algebra by the subrings of $S_j\times S_{n+1-j}$-symmetric functions, there exists a linear combination of products of minuscule elements of the form
%We then factor
%$$
%f(\bs w) \widetilde D^\mu = f(\bs w) \widetilde D^{\mu_n \omega_{n+1}} \prod_{j=1}^{n} \widetilde D^{\ha{\alpha_j,\mu}\omega_j},
%$$
%where $\omega_j = \epsilon_0 + \ldots + \epsilon_{j-1}$, and
$$
{\Rc_{\omega_{n+1}}^{|\mu|-\sum_j\ha{j\alpha_j,\mu}}[f_{n+1}] \prod_{j=1}^n \Rc_{\omega_j}^{\ha{\alpha_j,\mu}}[f_j] }
%\Sym\hr{f(\bs w) \widetilde D^{\mu_n \omega_{n+1}}} \prod_{j=1}^n \Sym\hr{\widetilde D^{\ha{\alpha_j,\mu}\omega_j}} \in \Dres^{\eta,\mu}
$$
whose leading term coincides with $A_\mu$, and with all other summands having degree strictly less than $\mu$ in dominance order. So subtracting this sum of products of minuscules from $A$ we obtain a new element $A'\in \Dres$ with strictly smaller total height. Hence by induction it follows that $\Dres$ is generated by the minuscule elements. 
\end{proof}

\subsection{The algebra $\Dres$ as a Coulomb branch}
\label{sec:coulomb}
Using Lemma~\ref{lem:minuscule-generation} we can identify the algebra $\mathcal{D}_{\mathrm{res}}$ with the quantized BFN Coulomb branch ring $K_{G[\![z]\!]\rtimes{\mathbb{C}}^\times_{2:1}}(\mathrm{Gr}_{GL_{n+1}})$ associated to the pair $(G,N)=(GL_{n+1},0)$, following the setup described in~\cite[Appendix A]{BFN16b}. Here $\mathrm{Gr}_{GL_{n+1}}=G(\!(z)\!)/G[\![z]\!]$ refers to the affine Grassmannian of $GL_{n+1}$, and the ring structure is defined by convolution, see~\cite{BFN16a,CW18} for details.

As in~\cite{FT17}, we work equivariantly with respect to the double cover $\mathbb{C}^\times_{2:1}$ of the loop rotation torus, and identify the corresponding equivariant parameter with the parameter $q$ from the previous section.
\begin{remark}
\label{rmk:scalars}
In this paper, we use the Fock--Goncharov ``integral'' convention for quantization parameters, so that the quantum torus associated to the mutable sub-lattice $\Lambda_{\mathrm{mut}}$ is an algebra over $\mathbb{Z}[q^{\pm1}]$, rather than $\mathbb{Z}[q^{\pm\frac{1}{2}}]$ as would be the case in the convention of Berenstein--Zelevinsky~\cite{BZ05}. This means that our $q$ in the definition of $\mathcal{D}_{\mathrm{res}}$ (and later in the definition of $\Lbb_{\mathrm{Toda}}(GL_{n+1})$) corresponds to the parameter $\nu$  of~\cite{FT17} and~\cite{GY21}, and to the parameter $q^{-1}$ in~\cite{CW18}. The authors of the latter paper work equivariantly with respect to a 4-fold cover $\mathbb{C}^\times_{4:1}$ of the loop rotation group, and denote by $q^{\frac{1}{2}}$ the equivariant parameter for that cover. Recall from Convention~\ref{rmk:tori-roots} that we lift to a when we define the quantum tori for the moduli space $\Pc^{\diamond}_{SL_{n+1},S}$.
\end{remark} 
Recall the $k$-th fundamental coweight  $\omega_k=\sum_{r=1}^k\epsilon_r$ for $G=GL_{n+1}$, and $\mathrm{Gr}^{\omega_k}\subset\mathrm{Gr}_{GL_{n+1}}$ the corresponding closed $G[\![z]\!]$-orbit. These orbits are isomorphic to finite dimensional Grassmannians of $k$-dimensional quotients of $\mathbb{C}^{n+1}$, and each has an equivariant line bundle $\mathcal{L}_{\omega_k}$ given by the top exterior power of the tautological vector bundle $\mathcal{V}$ of quotients $\mathbb{C}^{n+1}/V$.  There is a similar tautological vector bundle $\widetilde{\mathcal{V}}$ whose fibers the subspaces of $\mathbb{C}^{n+1}$ given by the kernel of the quotient map. For any function of $n+1$ variables $f$ symmetric under $S_k\times S_{n+1-k}$, there is a  class $f(\mathcal{V},\widetilde{\mathcal{V}})$ defined in the usual way, with the $i$-th elementary symmetric polynomial corresponding to  $\Lambda^i\mathcal{V}$.

As shown in ~\cite{BFN16a,BFN16b,FT17}, the localization theorem in equivariant $K$-theory provides an embedding of the convolution algebra $K_{G[\![z]\!] \rtimes \mathbb{C}^\times_{2:1}}(\mathrm{Gr}_{GL_{n+1}})$ into a localization of the Coulomb branch ring associated to the pair $(T,0)$, where $T$ is  the maximal torus $T$ of $GL_{n+1}$. This latter ring coincides  with the algebra $\mathcal{D}_q(T)_{\mathrm{loc}}$ defined in the previous section.

\begin{lemma}
\label{lem:small-genset}
\begin{enumerate}
\item Localization in equivariant $K$-theory defines an isomorphism of algebras $K_{G[\![z]\!] \rtimes \mathbb{C}^\times_{2:1}}(\mathrm{Gr}_{GL_{n+1}})\simeq\mathcal{D}_{res}$. 
\item The algebra $\mathcal{D}_{res}$ is generated over $\mathbb{Z}[q^{\pm1}]$ by the minuscule elements $D_{\omega_{n+1}}^{\pm1}$ and $\{\mathcal{R}_{1}[w_1^m] \,\big|\, m\in\mathbb{Z}\}$.
\end{enumerate}
\end{lemma}
\begin{proof}
In Lemma 2.18 and Corollary 2.21 of~\cite{CW18}, it is proven that the classes $\{[\mathcal{L}_{\omega_1}^{\otimes m}]\}_{m \in \mathbb{Z}}$, together with the inverse of the class $[\mathcal{O}_{\mathrm{Gr}_{\omega_{n+1}}}]$, generate $K_{G[\![z]\!] \rtimes\mathbb{C}^\times_{2:1}}(\mathrm{Gr}_{GL_{n+1}})$ as an algebra over $\mathbb{Z}[q^{\pm1}]$. Next, the image under the localization map of any class supported on a closed orbit $\mathrm{Gr}_{\omega_k}$ can be computed explicitly as in Section 8(i) of~\cite{FT17}. In particular, the image of the class $[\mathcal{L}_{\omega_k}^{\otimes m}]$ coincides with the minuscule element $\mathcal{R}_{k}[w_1^m]$, while the class $[\mathcal{O}_{\mathrm{Gr}_{\omega_{n+1}}}]$ is mapped to $D_{\omega_{n+1}}\in \mathcal{D}_{res}$. Thus it follows that the image of the Coulomb branch ring under the localization map is contained in $\mathcal{D}_{res}$. On the other hand, letting $f$ range over all $S_k\times S_{n+1-k}$-symmetric polynomials in the bundles $\mathcal{V},\widetilde{\mathcal{V}}$ it follows that the image of $K_{G[\![z]\!] \rtimes \mathbb{C}^\times_{2:1}}(\mathrm{Gr}_{GL_{n+1}})$ contains all minuscule elements of $\mathcal{D}_{res}$, which generate the latter as an algebra by Lemma~\ref{lem:minuscule-generation}. This shows that the image of the Coulomb branch ring under localization is precisely $\mathcal{D}_{res}$, and hence that the latter is generated by the elements listed in point (2).
\end{proof}
A more conceptual proof of part (1) of the previous lemma for any reductive group $G$ will be given in a forthcoming paper by D. Klyuev.

\subsection{Algebraic Whittaker transform}
\label{sec:alg-whit}
In this section we formulate an algebraic version of the Whittaker transform for the $GL_{n+1}$ q-difference open Toda chain, which will identify the algebras $\mathbb{L}_{Toda}$ and $\mathcal{D}_{res}$. It can be read as an algebraic counterpart to the analytic construction from~\cite{SS18}.

Let $\Dcr^{n+1}$ and $\Dcr^{n+1;\mathrm{int}}$ be respectively the spaces of all $\Q(q)$- and $\Z[q^{\pm1}]$-valued functions on the $GL_{n+1}$ weight lattice $\Prm$. When the underlying weight lattice is fixed, we shall omit the lower index and write $\Dcr = \Dcr^{n+1}$. Denote by $\Fcr$ and $\Fcr^{\mathrm{int}}$ their subspaces of finitely-supported functions. The following formulas define faithful representation of the quantum torus $\Tc_{GL_{n+1}}$ of the initial seed of the Toda chain cluster algebra on $\Fcr$:
\beq
\label{eq:toda-torus-F-action}
(P_{\epsilon_i} \phi)(\mu) = -q^{-1}\phi(\mu-\eps_i), \qquad (X_{\epsilon_i} \phi)(\mu) = (-q)^{2\mu_i+\rho_i} \phi(\mu),
\eeq
where
$$
\rho = (0,-1,\ldots,-n).
$$
%and
%$$
% (X_{\eps_1}\phi)(\mu) = q^{2\mu_1}\phi(\mu).
%$$
In particular, for a simple root $\alpha_i = \eps_i - \eps_{i+1}$ we have
\beq
\label{eq:Xalpha-action}
(X_{\alpha_i} \phi)(\mu) = - q^{2(\mu_i-\mu_{i+1})+1} \phi(\mu).
\eeq
%In terms of the translation $T_\lambda(f)(\mu) = f(\mu+\la)$ and multiplication operators $\{q^{2\lambda_i}\}_{i=1}^{n+1}$ on $M$, the action of $X_{\alpha_i}$ reads
%\begin{align}
%\label{eq:tta2}
%P_{\eps_i} \mapsto -q^{-1}T_{-\eps_i},\qquad X_{\alpha_i} \mapsto -q^{2(\lambda_{i}-\lambda_{i+1})+1}, \quad i=1,\ldots, n.
%\end{align}
Let 
\begin{align}
\label{eq:Vdef}
\Vcr^{\mathrm{int}}=\left\{ \phi\in \Fcr^{\mathrm{int}} \,\big|\, \mu\notin \Prm^+ \implies \phi(\mu)=0\right\}
\end{align}
be the $\mathbb{Z}[q^{\pm1}]$-submodule of functions vanishing outside the dominant cone $\Prm^+\subset \Prm$. We identify elements of the cone $\Prm^+=\{\lambda_1\geq\lambda_2\geq\cdots\geq \lambda_{n+1}\}$ with integer partitions in the usual way. Clearly the space $\Vcr^{\mathrm{int}}$ is not preserved by the $\mathcal{T}_{GL_{n+1}}$-action on $\Fcr$, since the action involves all translation operators $P_\lambda$, not just those corresponding to dominant $\lambda$.

\begin{lemma}
\label{lem:pres}
The $\mathbb{Z}[q^{\pm1}]$-submodule $\Vcr^{\mathrm{int}}\subset\Fcr^{\mathrm{int}}$ is preserved by the action of the Toda universal Laurent ring $\mathbb{L}_{Toda}(GL_{n+1})$, and defines a faithful representation of $\mathbb{L}_{Toda}(GL_{n+1})$.
\end{lemma}
\begin{proof}
The condition for a function $\phi\in \Fcr$ to lie in $\Vcr^{\mathrm{int}}$ has the form
$$
\mu_{i}-\mu_{i+1} <0 \implies \phi(\mu)=0.
$$
Given $A\in\Lbb_{Toda}$, we expand it as \begin{align}
\label{eq:a-decomp}
A = \sum_{\nu\in \Prm}f_\nu(X)P_\nu,
\end{align}
so that 
$$
(A\cdot \phi)(\mu) = \sum_{\nu\in \Prm}(-q)^{|\nu|}f_\nu\hr{(-q)^{2\mu+\rho}} \phi(\mu-\nu).
$$
The preservation of $\Vcr^{\mathrm{int}}$ by $\Lbb_{Toda}$ will follow if we can show that for all $\mu,\nu\in \Prm$ such that $\mu\notin \Prm^+$ and $\mu-\nu\in \Prm^+$ we have $f_\nu\hr{(-q)^{2\mu+\rho}}=0$.
The automorphism part of mutation in direction $s_i$ consists of conjugation by $\Psi(Y_{-s_i})=\Psi(X_{\alpha_i})$. Hence if $\langle \alpha_i,\nu\rangle=\nu_{i}-\nu_{i+1}\geq0$ we see that $\Ad_{\Psi(X_{\alpha_i})}(P_\nu)$ remains a Laurent polynomial, while 
\beq
%\label{eq:divisable}
\nu_{i}<\nu_{i+1}\implies \Psi(X_{\alpha_i}) P_\nu \Psi(X_{\alpha_i})^{-1} = \left(\prod_{k=0}^{\nu_{i+1}-\nu_{i}-1}\frac{1}{1+q^{2k+1}X_{\alpha_i}}\right) P_\nu.
\eeq
 Hence $A$ remains Laurent under mutation at $s_i$ if and only if for each $\nu$ with $\nu_i<\nu_{i+1}$, the coefficient $f_\nu$ in~\eqref{eq:a-decomp} has the form
\beq
\label{eq:fnu-factor}
f_\nu(X) = g_\nu(X)\prod_{k=0}^{\nu_{i+1}-\nu_{i}-1}\left(1+q^{2k+1}X_{\alpha_i}\right) \qquad\text{with}\qquad g_\nu(X)\in \mathbb{Z}[q^{\pm1}][X^{\pm1}].
\eeq
Now let $\mu,\nu\in \Prm$ be such that $\mu\notin \Prm^+$ and $\mu-\nu\in \Prm^+$. Then there is some $i$ such that $\mu_{i}-\mu_{i+1}<0$, and the condition $\mu-\nu\in \Prm^+$ says that $\mu_{i}-\mu_{i+1}+\nu_{i+1}-\nu_{i}\geq0$, so we have
$$
\nu_{i}-\nu_{i+1}\leq \mu_{i}-\mu_{i+1}<0
$$
and by~\eqref{eq:fnu-factor} we see that $\prod_{k=0}^{\mu_{i+1}-\mu_{i}-1}\left({1+q^{2k+1}X_{\alpha_i}}\right)$ divides $f_\nu(X)$. So  $f_\nu(X)$ corresponds to a multiplication operator by a Laurent polynomial in the functions $(-q)^{2\mu_i+\rho_i}$ divisible by
$$
\prod_{k=1}^{\mu_{i+1}-\mu_{i}}\left({1-q^{2(\mu_{i}-\mu_{i+1}+k)}}\right)
$$
which implies $f_\nu\hr{(-q)^{2\mu+\rho}}=0$ as claimed. The faithfulness follows by observing that no nonzero element of $\mathcal{T}_{GL_{n+1}}$ can annihilate all of $V$, which is proved by a standard argument used, for example, in~\cite[Lemma 4.1]{AC+24}.
\end{proof}
The module $\Vcr^{\mathrm{int}}$ has a basis consisting of delta-functions $\{\delta_\lambda\}_{\lambda\in\Prm^+}$ supported at a single partition $\lambda\in \Prm^+$:
$$
\delta_\lambda(\mu) =
\begin{cases} 1 &\text{if} \;\; \lambda=\mu, \\
0 & \text{otherwise.}
\end{cases}
$$
We are going to identify $\Vcr^{\mathrm{int}}$ with the $\mathbb{Z}[q^{\pm1}]$-module of symmetric Laurent polynomials in $w_1,\ldots, w_{n+1}$ using the basis $\{W_\lambda\}_{\lambda\in\Prm^+}$ of \emph{$q$-Whittaker polynomials}, which can be constructed recursively as follows. 

First we set up some notations for various rings of iterated formal Taylor and Laurent series. Given an ordered list of commuting indeterminates $\bs w = (w_1,\ldots, w_l)$, we consider the iterated formal Laurent series rings defined inductively by
$$
\Kc_{w_1} = \Q(q)((w_1)),\quad \Kc_{w_1,\ldots, w_l} = \Kc_{w_1,\ldots, w_{l-1}}((w_l)).
$$
Equivalently, elements of $\Kc_{\bs w}$ are formal infinite $\Q(q)$-linear combinations 
$$
\psi = \sum_{I=(i_1,\ldots,i_l)\in\mathbb{Z}^l}c_Iw_1^{i_1}\cdots w_l^{i_l}, \quad c_I\in \mathbb{Q}(q).
$$
where the set $I$ is well-ordered with respect to the reverse-lexicographic order on $\mathbb{Z}^l$. Examples of elements of $\mathcal{K}_{\bs w}$ are the series expansions of rational functions in the regime $|w_l|<\ldots<|w_2| < |w_1|<1$. We also consider $\mathcal{K}^{\mathrm{int}}_{\bs w}$ where we replace the initial choice of coefficient field $\Q(q)$ by the ring $\Z[q^{\pm1}]$. Finally, we write 
$$
\Kc_{\bs w;z} = \Kc_{\bs w}[[z]].
$$

Consider the space
$$
\Dcr_{\bs w} = \Dcr \otimes_{\Q(q)} \Kc_{\bs w}
$$
of all, not necessarily finitely supported, $\Kc_{\bs w}$-valued functions on $\Prm$ along with its $\Z[q^{\pm1}]$-submodule of all $\Kc_{\bs w}^{\mathrm{int}}$-valued functions
$$
\Dcr_{\bs w}^{\mathrm{int}} \subset \Dcr_{\bs w}.
$$
%If $\bs v = \bs w = (w_1,\ldots, w_{n+1})$, an example of such a function is
%$$
%\Omega(\lambda) = \prod_{k=1}^{n+1} w_k^{\lambda_k}.
%$$
Note that the action~\eqref{eq:toda-torus-F-action} of the quantum torus $\Tc_{GL_{n+1}}$ on $\Fcr$ extends naturally to its action on $\Dcr_{\bs w}$. 

If $z$ is a formal indeterminate, write $\Tc_z$ for the ring of power series in $z$ with coefficients in $\Tc_{GL_{n+1}}\otimes_{\Z[q^{\pm1}]}\Q(q)$.  When $z$ is a distinct indeterminate from the $\bs w$, elements of $\Tc_z$ give well-defined operators on $\Dcr_{\bs w;z}$. 
%The same is true of elements of ${\mathcal T}^{\mathbb{Q}(q)}_{GL_{n+1}}[[zw^I]]$ for any Laurent monomial $w^I$.
 On the other hand, because elements of $\Dcr_{\bs w}$ do not have finite support, there is no well-defined action of $\Tc_{w_k}$ on $\Dcr_{\bs w}$. For example, one cannot make sense of the action of $\sum_{m\ge0} w_k^m P_j^m$ on the distribution $\psi(\la) = w_k^{\la_j}$. Let us set $\bs w = (w_1, \ldots, w_{n+1})$ and consider the subspace $\Dcr^{(n,1)}_{\bs w} \subset \Dcr^{n+1}_{\bs w}$ consisting of distributions $\psi$ for which there exists a function $M_\psi \colon \Z \rightarrow \Z$ such that
$$
\psi(\la) = \sum_{m>M_\psi(\la_{n+1})} w_{n+1}^m g_{\la,m}(w_1,\ldots,w_n).
$$
In other words, we require that for each integer $d$, the set of all exponents of $w_{n+1}$ in $\psi(\la)$ has a lower bound, as $\la$ ranges over all partitions with $\la_{n+1}=d$. In particular, $\Dcr^{(n,1)}_{\bs w}$ contains functions of the form
$$
\psi(\la_1, \ldots, \la_{n+1}) = f(\la_{n+1})g(\la_1, \ldots, \la_n) \quad\text{where}\quad f\in \Dcr^1_{w_{n+1}} \quad\text{where}\quad g\in\Dcr^n_{w_1, \ldots, w_n}.
$$
Note that the subalgebra $\Tc^{(n,1)}_{w_{n+1}} \subset \Tc_{w_{n+1}}$ of elements commuting with $X_{n+1}$ has a well-defined action on $\Dcr^{(n,1)}_{\bs w}$. Similarly, we consider the subspace $\Dcr^{(n,1)}_{\bs w;z} \subset \Dcr_{\bs w;z}$ consisting of functions $\psi$ such that 
$$
\psi(\la) = \sum_{k\ge0}\sum_{m>M_\psi(k,\la_{n+1})} z^k w_{n+1}^m g_{\la,m,k}(w_1,\ldots,w_n).
$$
%where the lower bound $M(k,\la_{n+1})$ depends only on $k$ and $\la_{n+1}$.
Then $\Dcr^{(n,1)}_{\bs w;z}$ has a well-defined action of the algebra $\Tc^{(n,1)}_{w_{n+1};z}$, which consists of expressions
$$
A(w_{n+1},z)=\sum_{k\ge0} z^k A_k(w_{n+1}) F_k(P_{n+1})
$$
where $A_k(w_{n+1})\in \Tc^{(n,1)}_{w_{n+1}}$ and $F_k(P_{n+1}) \in \Q(q)[P_{n+1}^{\pm1}]$.

\begin{defn}
%Let $z$ be a formal indeterminate and write ${\mathcal T}^{\mathbb{Q}(q)}_{GL_{n+1}}[[z]]$ for the ring of power series in $z$ with coefficients in ${\mathcal T}_{GL_{n+1}}\otimes_{\mathbb{Z}[q^{\pm1}]}\mathbb{Q}(q)$. We write ${\mathcal T}^{\mathbb{Z}((q))}_{GL_{n+1}}[[z]]$ for the corresponding ring with $\mathbb{Z}((q))$ rather than $\mathbb{Q}(q)$-coefficients. 
Define the element $R_n(z) \in \Tc^{(n,1)}_z$ to be 
\begin{align}
\label{def:recursion-op}
R_n(z) &= \prod_{j=1}^{\substack{n \\ \longrightarrow}}\Psi(zP_{j})\Psi(zq^{-1}X_{\alpha_j}P_{j})\\
&=\Psi(zP_{1})\Psi(zq^{-1}X_{\alpha_1}P_{1})\cdots \Psi(zP_{n})\Psi(zq^{-1}X_{\alpha_n}P_{n}).
\end{align}
The $(GL_n,GL_{n+1})$-\emph{recursion operator} is defined by
\beq
\begin{aligned}
&\mathscr{R}_{n+1\leftarrow n}\colon \Dcr^n_{w_1,\ldots, w_n}\longrightarrow \Dcr^{n+1}_{w_1,\ldots, w_{n+1}},\\
%& \mathscr{R}_{n+1\leftarrow n}(\psi)(\la_1,\ldots, \la_{n+1})= w_{n+1}^{\lambda_{n+1}} (R_n(w_{n+1})\cdot \psi)(\la_1,\ldots, \la_n).
&\mathscr{R}_{n+1\leftarrow n}(\psi) =  R_n(w_{n+1})\cdot( w_{n+1}^{\lambda_{n+1}}\psi).
\end{aligned}
\eeq
\end{defn}
We can compute the Taylor coefficients of $R_n(z)$ using the following lemma:

\begin{lemma}
\label{lem:pentasum1}
Let $U,V$ generate a quantum torus $UV=q^2VU$, and let  $z_1,z_2$ be formal invertible indeterminates commuting with $U$, $V$, and each other. Then
\begin{align}
\label{eq:pentasum1}
(z_1V;q^2)^{-1}_\infty(-z_2UV;q^2)^{-1}_\infty 
&=\sum_{m\geq0}(-1)^m{q^{-m(1+m)}}\frac{(z_1^{-1}z_2U;q^{-2})_m}{(q^{-2};q^{-2})_m}z_1^mV^m. 
\end{align}
\end{lemma}
\begin{proof}
 Taylor expanding both $q$-Pochhammer symbols and using that $(UV)^k = q^{k(1-k)}U^kV^k$, the finite sum computing the coefficient of $V^m$ in their product is evaluated using the $q$-binomial theorem to yield the formula in the Lemma.
\end{proof}
%\begin{lemma}
%\label{lem:pentasum1}
%Let $U,V$ generate a quantum torus $VU=q^2UV$, and let  $z_1,z_2$ be formal invertible indeterminates commuting with $U,V$ and each other. Then
%\begin{align}
%\label{eq:pentasum1}
%(-z_1V;q^2)_\infty(z_2UV;q^2)_\infty 
%&=\sum_{m\geq0}{q^{m(m-1)}}\frac{(z_1^{-1}z_2U;q^2)_m}{(q^2;q^2)_m}z_1^mV^m. 
%\end{align}
%\end{lemma}
%
%In particular, if $U=q^{2l}$ for $l\geq0$ then the sum~\eqref{eq:pentasum1} truncates at $m=l$, becoming
%$$
%(z_1V;q^2)^{-1}_\infty(z_2UV;q^2)^{-1}_\infty 
%&=\sum_{m\geq0}{q^{m(1-m)}}\frac{(z_1^{-1}z_2U;q^{-2})_m}{(q^{-2};q^{-2})_m}z_1^mV^m. 
%$$
%
%$$
%\frac{(q^{2p};q^2)_m}{(q^2;q^2)_m} = \binom{m+p-1}{m}_{q^2}, %= (-1)^mq^{m(m-1-p)}\binom{p}{m}_{q^2}.
%$$
Recalling the $q$-binomial coefficients
$$
\binom{m}{k}_{q^2} = \frac{(q^2;q^2)_{m}}{(q^2;q^2)_{k}(q^2;q^2)_{m-k}},
$$
it follows from the Lemma that for $\psi\in\Dcr_{n;(w_1,\ldots,w_n)}$ we have
%\begin{align}
%%\label{eq:tcoeffs}
%\nonumber(R_n(z)\cdot \psi)(\la) = \sum_{r_1,\ldots, r_n\geq0}(-z)^{\sum_jr_j}&q^{\sum_jr_j(r_j+1)}\prod_{j=1}^n\binom{\lambda_{j+1}-\la_j+r_j-1}{r_j}_{q^2}\\
%&\times  \psi(\la_1-r_1,\ldots, \la_n-r_n) \in\mathscr{D}^{w_1,\ldots,w_n,z}.
%\end{align}
%Using the reciprocity 
%$$
%\binom{m-k-1}{m}_{q^2} = (-1)^mq^{m(m-1)-2mk}\binom{k}{m}_{q^2}
%$$
%we can further rewrite its action as
\beq
\label{eq:tcoeffs}
(R_n(z)\cdot \psi)(\la) = \sum_{r_1,\ldots, r_n\geq0}z^{\sum_jr_j}\prod_{j=1}^n\binom{\lambda_{j}-\la_{j+1}}{r_j}_{q^2}\psi(\la_1-r_1,\ldots, \la_n-r_n).
\eeq

\begin{remark}
%\begin{enumerate}
Since the $q$-binomial coefficients lie in $\Z[q^{\pm1}]$, it follows from the formula~\eqref{eq:tcoeffs} above that $\mathscr{R}_{n+1\leftarrow n}$ maps $\Dcr_{w_1,\ldots, w_n}^{n;\mathrm{int}}$ to $\Dcr_{w_1,\ldots, w_{n+1}}^{n+1;\mathrm{int}}$.
%
%\item Because of the multiplication by $w_{n+1}^{\la_{n+1}}$, the image of $\mathscr{D}^{(w_1,\ldots, w_n)}$ under $\mathscr{R}_{n+1\leftarrow n}$ does not lie in $\mathscr{D}^{(w_1,\ldots, w_n;w_{n+1})}$.
%\end{enumerate}
\end{remark}

\begin{defn}
The $GL_{n+1}$ Whittaker kernel is the element of $\Dcr_{w_1,\ldots,w_{n+1}}^{(n,1);\mathrm{int}}$ given by
$$
W^{(n+1)} = \mathscr{R}_{n+1\leftarrow n}\circ\mathscr{R}_{n\leftarrow n-1}\circ\cdots \circ\mathscr{R}_{2\leftarrow 1} (w_1^{\lambda_1}).
$$
\end{defn}

\begin{lemma}
If $\la\in\Prm^+$ the series $W^{(n+1)}(\la)$ is computed by the following subtraction-free finite sum
$$
W^{(n+1)}(\la) = \sum_{\nabla \in GZ_{n+1}} \prod_{\substack{1\leq i\leq n  \\ 1\leq j\leq i}}\binom{\la^{(i+1)}_{j}-\la^{(i+1)}_{j+1}}{\la^{(i+1)}_{j}-\la^{(i)}_{j}}_{q^2} \prod_{k=1}^{n+1}w_k^{|\la^{(k)}|- |\la^{(k-1)}| }
$$
over the set 
%$$
%GZ(\la) = \left\{(\la_{j}^{(a)})_{1\leq a\leq n+1}^{1\leq j\leq a}~\big|~ \la_{j}^{(a)}\geq \la_{j+1}^{(a)}, ~ \la_{j}^{(a)}\geq \la_{j}^{(a-1)}, \la_j^{(n+1)}=\la_j\right\}
%$$
$$
GZ(\la) = \left\{(\la_{j}^{(a)})_{1\leq a\leq n+1}^{1\leq j\leq a} \,\Big|\, \la_{j}^{(a)}\geq \la_{j}^{(a-1)}\geq \la_{j+1}^{(a)}, \; \la_j^{(n+1)}=\la_j\right\}
$$
of all Gelfand-Zeitlin arrays with top row given by the partition $\la$. In particular, for dominant $\lambda$ the element $W(\la)$ lies in the semiring $\Z_{\ge0}[q^{\pm1}][\bs w]$.

\end{lemma}
\begin{proof}
The second-to-top row in any array $\nabla\in GZ(\la)$ gives a partition $\mu$ which interlaces with $\lambda$, so that
$$
GZ(\la) = \bigcup_{\mu ~\text{interlacing }\la} GZ(\mu) .
$$
We can parametrize the set of such interlacing partitions as  $\mu = (\la_1-r_1,\la_2-r_2,\ldots, \la_n-r_n)$ where $0\leq r_j\leq \lambda_j-\lambda_{j+1}$, and these parts are exactly the arguments of $\psi$ in~\eqref{eq:tcoeffs}. Since $\lambda$ is dominant each $\lambda_{j}-\lambda_{j+1}\geq0$ and so the Gaussian binomial coefficients in~\eqref{eq:tcoeffs} vanish whenever $r_j>\la_j-\la_{j+1}$, and hence the Lemma follows by induction on $n$.
\end{proof}
\begin{remark}
When $\la\notin \Prm^+$ is not dominant, the series computing $W^{(n+1)}(\la)$ does not truncate and hence is a genuine formal Laurent series in $w_2/w_1,\ldots, w_{n+1}/w_n$. 
\end{remark}
%We define the $q$-Whittaker polynomial to be
%\begin{align}
%W_\la = \overline{W(\la)}\in \mathbb{Z}[q^{\pm1}][\bs w].
%\end{align}

Consider the involution on $\Kc_{\bs w}$ defined by
\beq
\label{eq:invv}
\overline{q} = q^{-1}, \quad \overline{w}_j = w_j, \qquad 1\leq j\leq n+1.
\eeq
The natural pairing between $\Fcr$ and $\Dcr$ extends to a $\Kc_{\bs w}$-valued pairing
\beq
\label{eq:Mpair}
(\cdot, \cdot)_{\Fcr} \colon \Fcr \otimes_{\Q(q)} \Dcr_{\bs w}\longrightarrow \Kc_{\bs w}, \qquad (f,\psi)_{\Fcr} = \sum_{\la\in\Prm}f(\la)\overline{\psi(\la)},
\eeq
which is $\mathbb{Q}(q)$-linear in the first factor and anti-linear with respect to~\eqref{eq:invv} in the second. Note that the sum in~\eqref{eq:Mpair} is finite since $f$ is compactly supported, and that we have
\beq
\label{eq:adjts}
(T_i f,\psi) = (f,T_i^{-1}\psi), \qquad (q^{2\lambda_i}f,\psi) = (f,q^{-2\lambda_i}\psi),
\eeq
where $T_if(\la) = f(\la+\epsilon_i)$ is the translation operator. 

Given a new formal indeterminate $z$, we extend the involution~\eqref{eq:invv} to $\Kc_{\bs w;z}$ by declaring $\overline{z}=z$. Write $\Fcr_z$ for the space of $\Q(q)[[z]]$-valued functions on $\Prm$ such that the the coefficient of each power of $z$ is a $\Q(q)$-valued function with finite support.
%\red{Here, $\Fcr_z$ is in fact larger than $\Fcr \otimes_{\Q(q)} \Q(q)[[z]]$. Also, I replaced all of $\Q(q)((\bs w ^{\pm1}))$ with $\Kc_{\bs w;z}$, was that right?}
Then~\eqref{eq:Mpair} extends naturally to a well-defined $\Kc_{\bs w;z}$-valued pairing $(\cdot,\cdot)_{\Fcr_z}$ between $\Fcr_z$ and $\Dcr_{\bs w;z}$.
%\begin{remark}
%\label{rmk:q-adjts}
Using~\eqref{eq:adjts} and the relation
$$
\Psi_q(z)^{-1} = \Psi_{q^{-1}}(z) \in \mathbb{Q}(q)[[z]]
$$
it is easy to check that for any $g\in\mathscr{D}_{\bs w;z}$ and $f\in \Fcr_z$ we have
\begin{align*}
\hr{f,\Psi_q(zP_j)^{-1}g}_{\Fcr_z} &= \hr{\Psi_q(zP^{-1}_j)f,g}_{\Fcr_z}, \\
\hr{f,\Psi_q(zq^{-1}X_{\alpha_j}P_{j+1}^{-1})g}_{\Fcr_z} &= \hr{\Psi_q(zq^{-1}X_{-\alpha_j}P_{j+1})^{-1}f,g}_{\Fcr_z}.
\end{align*}

%\end{remark}

Consider the lattice $\Lambda_{GL_{n+1};\zeta}$ obtained by extending $\Lambda_{GL_{n+1}}$ by a rank 1 lattice $\mathbb{Z}\zeta$. We equip it with a basis~\eqref{eq:toda-e-basis} of $\Lambda_{GL_{n+1}}$ to which we adjoin the vector $e_{h_-}=-\zeta -p_{\eps_{1}}$, and identify the quantum torus element $Y_{\zeta}$ with the indeterminate $z$. Now recall from~\eqref{eq:baxseq} the mutation sequence $\mu_{\mathrm{Baxter}}$ used to define the fundamental Hamiltonians. Then the inverse of the automorphism part of the cluster transformation obtained by applying $\mu_{\mathrm{Baxter}}$ to the extended quiver, with $e_h$ playing the role of $a$, yields the \emph{Baxter operator} $\Qf_{n+1}(z) = \Qf^+_{n+1}(z)$ given by
%\red{we've got the $\pm$-mismatch}
\begin{align}
\label{eq:bax-qdl}
\Qf^+_{n+1}(z) =  \Psi(zP_{1})^{-1}\prod_{j=1}^{\substack{n \\ \longrightarrow}}\Psi(zq^{-1}X_{-\alpha_j}P_{j+1})^{-1}\Psi(zP_{j+1})^{-1}\in \Tc_z.
\end{align}
%\red{Explain how we extend $\tau$}
The analogous construction where we instead adjoin $e_{h_+}=-\zeta +p_{\eps_{n+1}}$ yields the \emph{opposite} Baxter operator
\begin{align}
\label{eq:opbax-qdl}
\Qf^-_{n+1}(z) =  \Psi(zP^{-1}_{n+1})^{-1}\prod_{j=1}^{\substack{n \\ \longleftarrow}}\Psi(q^{-1}zX_{-\alpha_j}P^{-1}_{j})^{-1}\Psi(zP^{-1}_{j})^{-1}\in \Tc_z.
\end{align}
The inductive proof of the following Proposition is the same as that of the corresponding result in~\cite{SS18}, except that the computations with compact quantum dilogarithms are performed in the completed quantum torus for the extended quiver.
\begin{prop}
\label{prop:comm-bax}
Abbreviating $\Qf^\pm(z) = \Qf^\pm_{n+1}(z)$, the Baxter operators have the following properties:
\begin{enumerate}
\item For commuting indeterminates $z,w$ and any $\epsilon_1, \epsilon_2 \in \hc{\pm}$ we have
$$
\Qf^{\epsilon_1}(z)\Qf^{\epsilon_2}(w) = \Qf^{\epsilon_2}(w)\Qf^{\epsilon_1}(z).
$$
\item The series $\Qf^\pm(z)$ are preserved by the automorphism $\tau$ from~\eqref{eq:tau}:
$$
\tau\left(\Qf(z)^\pm\right) = \Qf^\pm(z).
$$
\item The series $\Qf(z)$ satisfies the difference equation
\begin{align}
\label{eq:alg-bax-qde}
\Qf(q^{-1}z)  = \Hc(z)\Qf(qz), \qquad {\mathcal{H}}(z) = \sum_{k=0}^{n+1}z^k {H}_k
\end{align}
where the formal adjoints of the fundamental Hamiltonians $H_{k}$ with respect to~\eqref{eq:Mpair} act by the Hamiltonians of the $q$-difference open Toda chain:
\begin{align}
\label{eq:explicit-ham}
 H^\star_k = (-q)^{k}\sum_{|I|=k} \prod_{m\in G(I)}^k(1-q^{2(\la_{m}-\la_{m+1})})T_{\sum_{i\in I}\eps_i}, \qquad 
\end{align}
where for $k$-subset $I\subseteq\{1,\ldots n+1\}$ we write 
$$
G(I) = \left\{ 1\leq m\leq n~|~ m\notin I~\text{and } m+1 \in I \right\}.
$$

\item In $\Tc_{GL_{n+1}} \otimes_{\Z[q^{\pm1}]} \Z((q))[[z]]$ we have the infinite product formula
\begin{align}
\label{eq:productformula}
\Qf(z) = \prod_{m\geq0}{\mathcal{H}}(q^{2m+1}z),
\end{align}
where the coefficient of $z^k$ is a polynomial in  $H_1,\ldots, H_{n+1}$ times the series expansion of a rational function in $q$.
\end{enumerate}
\end{prop}
It follows from Proposition~\ref{prop:comm-bax} that the action of $\Qf^\pm(z)$ on $\Fcr_z$ preserves the subspace $\Vcr_z$ of functions vanishing outside the dominant cone of partitions. Moreover, since the automorphism $\tau$ preserves $\Qf(z)$, it follows that it also fixes each fundamental Hamiltonian. Using this observation it is straightforward to derive the following lemma, which is a special case of a more general result by Di Francesco and Kedem (see~\cite{DFK16}, Theorem 3.8):
%Moreover, by Theorem 3.8 of~\cite{DFK16} the elements $\tau^m(A_{t_1})$ and the first fundamental Hamiltonian $H_1$ satisfy \red{(check sign!)}
\begin{lemma}
\label{lem:time-translation}
For all $m\in\mathbb{Z}$, we have
$$
\hs{\tau^m(A_{t_1}),H_{\pm1}} = \pm(q-q^{-1}) \tau^{m\pm1}(A_{t_1}), \\
$$
where
$$
H_{-k} = H_{n+1-k}H_{n+1}^{-1}.
$$
\end{lemma}
Using the formulas~\eqref{eq:adjts} we see that
$$
(\Qf_{n+1}(z)f,g)_{\Fcr_z} = (f,\Qf^\star_{n+1}(z)g)_{\Fcr_z},
$$
where the {conjugate} $GL_{n+1}$ Baxter operator $\Qf^\star_{n+1}(z)$ is given by
%related to the recursion operator via
%\begin{align}
%\label{eq:bax}
%\widetilde{Q}_{n+1}(z) = R_{n}(z)\Psi(zP_{n+1}).
%\end{align}
\begin{align}
\label{eq:bax}
\Qf^\star_{n+1}(z) =\Psi(zP_{n+1}^{-1})\prod_{j=1}^{\substack{n \\ \longleftarrow}}\Psi(zq^{-1}X_{\alpha_j}P_{j+1}^{-1})\Psi(zP_{j}^{-1}). 
\end{align}
Note that we have a recursive description
$$
\Qf^\star_{n+1}(z) = T_{n+1}(z)\Qf^\star_n(z)
$$
where
$$
T_{n+1}(z)=\Psi(zP_{n+1}^{-1})\Psi(zq^{-1}X_{\alpha_n}P_{n+1}^{-1}),
$$
and the recursion operator $R_n(z)$ can be expressed as
$$
R_n(z) = \left(\Qf_n^-(z)\right)^\star\Psi(qzX_{\alpha_n}P_n).
$$
%For a monomial $M$ in the quantum torus we abbreviate
%$$
%[M]  = \Psi(:M:),
%$$
%where $:M:$ denotes the Weyl ordering~\eqref{eq:weyl-order}.

The following Lemma is straightforward to derive using the pentagon identity and the recursive descriptions of $\Qf^\star(z)$ and $R(w)$:
\begin{lemma}
%In $\mathcal{A}^{w_{n+1};z}_{n,n+1}$ we have
In $\Tc^{(n,1)}_{w_{n+1};z}$ we have
\label{lem:tlem1}
$$
\Qf_{n+1}^\star(z)R_n(w_{n+1}) = R_n(w_{n+1})T_{n+1}(z)\Psi(zw_{n+1}X_{\alpha_n})\Qf_n^\star(z).
$$
\end{lemma}
%\begin{proof}
%By the recursive formulas for $\widetilde{Q}$ and $R$ and Proposition~\ref{prop:comm-bax}, we have
%\begin{align*}
%{Q}^\star_{n+1}(z)R_n(w_{n+1}) &= T_{n+1}(z){Q}_{n}^\star(z)({Q^o_n})^\star(w_{n+1})[w_{n+1}X_{n}/X_{n+1} P_n]\\
%&=T_{n+1}(z)({Q^o_n})^\star(w_{n+1}){Q}_{n}^\star(z)[w_{n+1}X_{n}/X_{n+1} P_n].
%\end{align*}
%Using the pentagon identity once, we get that
%\begin{align*}
%{Q}_{n}^\star(z)[w_{n+1}X_{n}/X_{n+1} P_n] &= T_n(z)[w_{n+1}X_{n}/X_{n+1} P_n]{Q}_{n-1}^\star(z)\\
%&=[w_{n+1}X_{n}/X_{n+1} P_n][zw_{n+1}X_n/X_{n+1}]T_n(z){Q}_{n-1}^\star(z)\\
%&=[w_{n+1}X_{n}/X_{n+1} P_n][zw_{n+1}X_n/X_{n+1}]{Q}_{n}^\star(z).
%\end{align*}
%On the other hand. another application of the pentagon identity gives
%\begin{align*}
%T_{n+1}(z) [w_{n+1}X_{n-1}/X_nP_{n-1}][w_{n+1}P_n]&=[w_{n+1}X_{n-1}X_n^{-1}P_{n-1}][w_{n+1}P_n][zP_{n+1}^{-1}]\\
%&\times[w_{n+1}zX_nX_{n+1}^{-1}P_nP_{n+1}^{-1}][zX_nX_{n+1}^{-1}P_{n+1}^{-1}].
%\end{align*}
%Since $[zX_nX_{n+1}^{-1}P_{n+1}^{-1}]$ and $[w_{n+1}X_{n}/X_{n+1} P_n]$ commute, we can apply the pentagon identity in the form
%$$
%[zP_{n+1}^{-1}][zw_{n+1}X_n/X_{n+1}P_{n}/P_{n+1}][w_{n+1}X_n/X_{n+1}^{-1}P_{n}] = [w_{n+1}X_n/X_{n+1}P_{n}][zP_{n+1}^{-1}]
%$$
%to give
%\begin{align*}
%&T_{n+1}(z) [w_{n+1}X_{n-1}/X_nP_{n-1}][w_{n+1}P_n] [w_{n+1}X_{n}/X_{n+1} P_n]=\\
%&[w_{n+1}X_{n-1}X_n^{-1}P_{n-1}][w_{n+1}P_n][w_{n+1}X_n/X_{n+1}P_{n}]T_{n+1}(z),
%\end{align*}
%and recombining the factors we get the Lemma.
%\end{proof}
Another important lemma is the following:
\begin{lemma}
\label{lem:tlem2}
For any $g^{(n)}\in\Dcr^n_{w_1,\ldots, w_n;z}$, we have %the following identity in $\mathscr{D}_{n+1,n}^{w_1,\ldots, w_n,z}$:
$$
\Psi(zq^{-1}X_{\alpha_n}P_{n+1}^{-1})^{-1}\cdot w_{n+1}^{\la_{n+1}}g^{(n)} =  \Psi(zw_{n+1}X_{\alpha_n}) w_{n+1}^{\la_{n+1}}g^{(n)}.
$$
\end{lemma}
\begin{proof}
Before starting the proof, note that both sides are well-defined since each operator factor is clearly in $\Tc^{(n,1)}_{w_{n+1};z}$ and $w_{n+1}^{\la_{n+1}}g^{(n)}$ is in $\Dcr^{(n,1)}_{w_1,\ldots,w_{n+1};z}$ as we can take bounding function $M(k,\la_{n+1})=\la_{n+1}$. Now as in the proof of Lemma~\ref{lem:pentasum1}, we use
$$
UV=q^{2}VU \implies(UV)^k = q^{k(k-1)}U^kV^k
$$
to bring all $P_{n+1}^{-1}$ factors to the right, where they act on $w_{n+1}^{\la_{n+1}}g^{(n)}$ by $w_{n+1}$. So comparing the formulas~\eqref{eq:qpochinv} and~\eqref{eq:qpoch} for the series expansions of the $q$-Pochhammer and its reciprocal, the Lemma follows.
\end{proof}
Now recall that the Whittaker kernel $W^{(n+1)}$ is an element of $\Dcr_{w_1,\ldots,w_{n+1}}^{(n,1)}$ and hence of $\Dcr_{w_1,\ldots,w_{n+1};z}^{(n,1)}$.

\begin{prop}
\label{lem:whit-bax}
The Whittaker kernel diagonalizes the conjugate Baxter operator: we have
\begin{align}
\label{eq:whit-bax}
\Qf^\star_{n+1}(z) W^{(n+1)} = W^{(n+1)}\prod_{j=1}^{n+1}\Psi(-qzw_j).
\end{align}
\end{prop}
\begin{proof}
We prove the proposition by induction over the rank, combining Lemmas~\ref{lem:tlem1} and~\ref{lem:tlem2}. We have
\begin{align*}
\Qf^\star_{n+1}(z)W^{(n+1)} &= \Qf^\star_{n+1}(z)R_n(w_{n+1})\cdot w_{n+1}^{\la_{n+1}}W^{(n)}\\
&=R_n(w_{n+1})T_{n+1}(z)\Psi(zw_{n+1}X_{\alpha_n})\Qf_n^\star(z)\cdot w_{n+1}^{\la_{n+1}}W^{(n)}.
\end{align*}
By the induction hypothesis, $\Qf_n^\star(z)$ acts on $W^{(n)}$ with eigenvalue $\prod_{j=1}^n\Psi(-qzw_j)$. On the other hand, using Lemma~\ref{lem:tlem2} we have
\begin{align*}
R_n(w_{n+1})T_{n+1}(z)\Psi(zw_{n+1}X_{\alpha_n+})\cdot w_{n+1}^{\la_{n+1}}W^{(n)}
&= R_n(w_{n+1})\Psi(zP_{n+1}^{-1})\cdot w_{n+1}^{\la_{n+1}}W^{(n)}\\
&=\Psi(-qzw_{n+1}) R_n(w_{n+1})\cdot w_{n+1}^{\la_{n+1}}W^{(n)}\\
&=\Psi(-qzw_{n+1}) W^{(n+1)},
\end{align*}
and the Proposition follows.
\end{proof}

Recall the $z$-expansion~\eqref{eq:alg-bax-qde} of the Baxter operator in the ${H}_k$. From it we deduce that their adjoints act on the Whittaker kernel by the multiplication operators
\begin{align}
\label{eq:pieri}
H^\star_k \cdot W= (-q)^{k}e_k(\bs w) W, \qquad 1\leq k\leq n+1,
\end{align}
where $e_k(\bs w)$ is the $k$-th elementary symmetric function in the alphabet $\bs w$. In other words, we see that $H_k^\star$ is the operator defining the $k$-th Pieri rule for the Whittaker kernel. 
%\begin{remark}
%% obsolete remark since we changed the conventions:
%In our conventions, the $H^*_k$ act to lower total $\bs w$ degree. But since the top Hamiltonian $H^*_{n+1}$ acts as multiplication by $(-q)^{-n-1}w_1^{-1}\cdots w_{n+1}^{-1}$, we can simply multiply each $H_k^*$ by the inverse to reformulate the Pieri rules in terms of degree-raising operators.
%\end{remark}

\begin{example}
Writing $W_{\bs\la} = W(\bs\la)$, the two Pieri rules for $GL_2$ read
\begin{align*}
(w_1+w_2)W_{\la_1,\la_2} &=\left(T_{\la_1} + \big(1-q^{2(\la_1-\la_2)}\big)T_{\la_2}\right)W_{\la_1,\la_2}\\
&= W_{\la_1+1,\la_2} + \big(1-q^{2(\la_1-\la_2)}\big) W_{\la_1,\la_2+1},
\end{align*}
and
$$
w_1w_2W_{\la_1,\la_2} = W_{\la_1+1,\la_2+1}.
$$
\end{example}
%\begin{cor}
%\label{cor:pieri}
%The $W_\la$ are eigenfunctions for the fundamental Toda Hamiltonians: we have the Pieri rules
%\begin{align}
%\label{eq:pieri}
%{H_k}W_\la &= e_k(\bs w^{-1}) W_\la,\quad 1\leq k\leq n+1.
%\end{align}
%\end{cor}
\begin{cor}
\label{cor:basis}
For any dominant weight $\lambda\in \Prm^+$, the Laurent polynomial $W^{(n+1)}(\la)$
%\begin{align}
%W^{std}_\la = \overline{}
%\end{align} 
coincides with the standard definition of the $q$-Whittaker polynomial in $(n+1)$-variables, i.e. 
$$
W^{(n+1)}(\la) = P^{(n+1)}_\la(q,0,\bs w)%=\widetilde{H}_\la(q,t,\bs w)|_{t^{\eta(\la)}}.
$$
where $P_\la$ is Macdonald's $P$-polynomial.
%\red{$P_\la$ is a disaster...}
In particular, for $\la$ dominant $W(\la)$ is symmetric in $\bs w$ and the set $\{{W(\la)}\}_{\la\in\Prm^+}$ forms a $\mathbb{Z}[q^{\pm1}]$-basis for the ring of symmetric Laurent polynomials in $(w_1,\ldots, w_{n+1})$ with coefficients in $\mathbb{Z}[q^{\pm1}]$.
\end{cor}
\begin{proof}
The identification follows since both collections of polynomials are uniquely characterized by the Pieri rules~\eqref{eq:pieri} with initial condition $W_\emptyset=1$; the remaining claims are standard.
\end{proof}

\begin{defn}
We define the \emph{algebraic Whittaker transform} by
\begin{align}
\label{eq:alg-whit}
\Wc \colon \Vcr^{\mathrm{int}} \longra \Z[q^{\pm1}][\bs w^{\pm1}]^{sym}, \qquad f \longmapsto \sum_{\lambda \in \Prm} f(\la)\overline{W(\la)}.
\end{align}
\end{defn}

Using the embeddings $\Vcr^{\mathrm{int}} \hookrightarrow \Fcr^{\mathrm{int}} \hookrightarrow \Fcr$ and $\mathbb{Z}[q^{\pm1}][\bs w^{\pm1}]^{sym}\hookrightarrow \Dcr_{w_1,\ldots, w_{n+1}}$, we can write it equivalently as
\begin{align}
\label{eq:sc-w}
\mathcal{W}(f) = (f,W)_\Fcr.
\end{align}
By Corollary~\ref{cor:basis}, $\mathcal{W}$ is an isomorphism of $\mathbb{Z}[q^{\pm1}]$-modules as it takes the basis $\{\delta_\la\}_{\la\in P^+}$ in $\Vcr^{\mathrm{int}}$ to the basis $\{\overline{W(\la)}\}_{\la\in\Prm^+}$ in $\mathbb{Z}[q^{\pm1}][\bs w^{\pm1}]^{sym}$. The Pieri rules~\eqref{eq:pieri} become the following intertwining relation for the Toda Hamiltonians under algebraic Whittaker transform:
\begin{align}
\label{eq:toda-intertwining}
\mathcal{W}\circ H_k = (-q)^{-k}e_k(\bs w)\circ \mathcal{W}, \quad \quad 1\leq k\leq n+1.
\end{align}
Moreover, by Lemma~\ref{lem:whit-bax} it follows that the natural extension of $\mathcal{W}$ to an isomorphism of $\mathbb{Q}(q)[[z]]$-modules 
$$
\Wc_z \colon \Vcr_z \longrightarrow \Q(q)[\bs w^{\pm1}]^{sym}[[z]]
$$
satisfies the intertwining relation
\beq
\label{eq:Q-W-inter}
\Wc_z \circ \Qf(z) = \prod_{j=1}^{n+1}\Psi(-qzw^{-1}_j)^{-1}\circ \Wc_z.
\eeq

%\begin{align*}
%(R_n(w_{n+1})\cdot \psi)(\la) = \sum_{\mu_1\geq\la_2,\ldots, \mu_{n}\geq\la_{n+1}}&(-w_{n+1})^{\sum_jr_j}q^{\sum_jr_j(r_j+1)}\prod_{j=1}^n\binom{\mu_{j}-\la_j-1}{\mu_{j}-\la_{j+1}}_{q^2}\\
%& \times \psi(\mu_1-\lambda_2,\mu_1-\lambda_3,\ldots,\mu_n-\lambda_{n+1}).
%\end{align*}

Since the $q$-Whittaker polynomials are $t=0$ specializations of Macdonald polynomials, they also satisfy a dual set of difference equations with respect to the variables $\bs w$: we have
\begin{align}
\label{eq:qwhit-dual-eigen}
\mathcal{R}_{\omega_k}\cdot \overline{W}_\la = q^{2\sum_{r=1}^k\la_r}\overline{W_\la},\quad 1\leq k\leq n+1,
\end{align}
where we recall $\mathcal{R}_{\omega_k}=\mathcal{R}_{\omega_k}[1]$ is the following minuscule element of $\mathcal{D}_{res}$:
$$
\mathcal{R}_{\omega_k} = \sum_{|I|=k}\prod_{\substack{i \in I \\ j \notin I}}\frac{1}{1-w_j/w_i} \prod_{i\in I}D_i.
$$
They also satisfy the ``raising operator'' identities describing the action of the excited operators
\begin{align*}
\mathcal{R}_{\omega_k}[e_k \otimes 1]  = \sum_{|I|=k}\prod_{\substack{i \in I \\ j \notin I}}\frac{1}{1-w_j/w_i} \prod_{i\in I}w_iD_i,
\end{align*}
which read
$$
\mathcal{R}_{\omega_k}[e_k \otimes 1] \cdot \overline{W}_\la= q^{2\sum_{r=1}^k\la_r}\overline{W}_{\la+\omega_k},\quad 1\leq k\leq n+1.
$$
Hence we have intertwining relations
\begin{align}
\label{eq:A-inter}
\mathcal{W}\circ A_{t_k} &= (-q)^{k(1-k)/2}\mathcal{R}_{\omega_k}\circ\mathcal{W},\\
\mathcal{W}\circ A_{s_k} &= (-1)^k(-q)^{k(1-k)/2}\mathcal{R}_{\omega_k}[e_k \otimes 1]\circ\mathcal{W}.
\end{align}

%\red{This is exactly image of the class $[\mathcal{O}_{\mathrm{Gr}_{\omega_k}}]$ under localization, see formula (8.1) of Finkelberg-Tsymbaliuk~\cite{FT17}.}
Consider the algebra automorphism $\widetilde\gamma $ of $D_q(T)$ defined on generators by
$$
\widetilde\gamma(w_i)=w_i,\quad \widetilde\gamma(D_i) = {-}w_iD_i, \quad i=1,\ldots, n+1.
$$
%\red{Can you check the sign and $q$-powers here?}
It is easy to see that $\widetilde\gamma$ preserves $\Dres$.

\begin{prop}
\label{prop:alg-whit}
There is a unique algebra isomorphism 
\begin{align}
\label{eq:alg-whit-def}
\mathbb{W}\colon \mathbb{L}_{Toda}\longrightarrow \mathcal{D}_{res}
\end{align}
such that
\begin{align}
\label{eq:el-inter}
\mathbb{W}(H_k) &= (-q)^{-k}e_k(\bs w) \\
\mathbb{W}(Y_{\xi_{t_k}}) &= (-q)^{k(1-k)/2}\mathcal{R}_{\omega_k},\\
\mathbb{W}(Y_{\xi_{s_k}}) &= (-1)^k(-q)^{k(1-k)/2}\mathcal{R}_{\omega_k}[e_k \otimes 1]
%\mathbb{W}\circ \tau &= \gamma\\
%\mathbb{W}\circ Q(z) &= \prod_{j=1}^{n+1}\Psi(zw_j)\circ \mathbb{W},
\end{align}
%\red{Does this literally match Cautis-Williams though? Issues with $q^{1/2}$? Probably have to make $X_1$ act by $(-q)^{n/2}$ to get correct Grassmannian dimensions for perversity?}
It satisfies the intertwining relations
\begin{align}
\label{eq:aut-inter}
\mathbb{W}\circ \tau &= \widetilde\gamma \circ \mathbb{W}\\
\mathbb{W}\circ \mathrm{Ad}(\Qf^+(z)) &= \prod_{j=1}^{n+1}\mathrm{Ad}(\Psi(-q^{-1}zw_j)^{-1})\circ \mathbb{W},\\
\mathbb{W}\circ \mathrm{Ad}(\Qf^-(z)) &= \prod_{j=1}^{n+1}\mathrm{Ad}(\Psi(-qzw^{-1}_j)^{-1})\circ \mathbb{W}.
\end{align}
%where $\widetilde{\gamma}$ is the restriction to $\mathcal{D}_{res}$ of the automorphism of the localized quantum torus 
%$$
%\gamma(
%$$
where in the last two identities we abuse notation and use the same symbol for $\mathbb{W}$ and its extension to the $z$-completions.
\end{prop}
\begin{proof}
Using the isomorphism $\mathcal{W} \colon V\simeq \mathbb{Z}[q^{\pm1}][\bs w^{\pm1}]^{sym}$ we can identify both $\mathcal{D}_{res}$ and $\Lbb_{Toda}$ as subrings of the endomorphism ring $\mathrm{End}_{\mathbb{Z}[q^{\pm1}]}(\mathbb{Z}[q^{\pm1}][\bs w^{\pm1}]^{sym})$. We claim that these subrings are identical. Indeed, comparing the relations from Lemma~\ref{lem:time-translation} and the evident commutation relations between $e_1(\bs w^{\pm1})$ and $\mathcal{R}_1$, it follows that the action of $\tau^m(A_{t_1})$ coincides up to multiplication by a power of $q$ with that of the minuscule operator $\mathcal{L}_{\omega_k}^{\otimes m}$. Similarly, the frozen $\Ac$-variables $X_{\omega_{n+1}}$ and $P_{\omega_{n+1}}$ are intertwined respectively with powers of $q$ times $\mathcal{R}_{{n+1}}=D_1\cdots D_{n+1}$ and $e_{n+1}(\bs w)$. Since the $\tau^m(A_{t_1}),X_{\pm\omega_{n+1}},P_{\pm\omega_{n+1}}$ generate $\Lbb_{Toda}$ by the results of Goodearl and Yakimov from Section~\ref{subsec:ltoda}, while the $\mathcal{L}_{\omega_k}^{\otimes m}$ together with $\mathcal{R}_{{n+1}}$ and its inverse generate $\mathcal{D}_{res}$ by Lemma~\ref{lem:small-genset}, the claim follows and the construction above defines an isomorphism $\mathbb{W}\colon \mathbb{L}_{Toda}\simeq \mathcal{D}_{res}$. With~\eqref{eq:el-inter} already having been established, it remains to prove that the relations~\eqref{eq:aut-inter} hold. Since 
the $\tau^m(A_{s_1})$ generate $\Lbb_{Toda}$ (which is a torsion-free $\mathbb{Z}[q^{\pm1}]$-module), the relation involving $\tau$ follows from the commutation relations in Lemma~\ref{lem:time-translation}. Finally, the relations involving the Baxter operators follow from the intertwining relation~\eqref{eq:Q-W-inter} and its analog for $\Qf^-(z)$.
\end{proof}

\section{Changing an isolating cylinder}
\label{sec:3-to-3}

\subsection{Bi-fundamental Baxter operators}
\label{subsec:bi-baxter}

%\red{Here I write $\bs\mu_{n,0}$ for what was $Q_{n+1}$ above. Need to reconcile that. Also, we use notation $Q_{n,k}$ for quivers, and can't go on denoting Baxter operators the same way. At the very least, we need to change the font, e.g. $\Qf_{n,k}$}

In this section we define \emph{bi-fundamental Baxter operators} via sequences of mutations, presented as transformations of cyclic Weyl words, as in Section~\ref{sec:isolating-coords}. For $0 \le k \le n$ and $2 \le j \le 2(n+k)$, let us set
$$
\bi_{n,k;j} = 
\begin{cases}
\bi_{[\max(1,\lfloor j/2 \rfloor-k), \min(\lfloor j/2 \rfloor,n)]} &\text{if $j$ is even,} \\
\bi_{[\min(\lfloor j/2 \rfloor,n),\max(1,\lfloor j/2 \rfloor+1-k)]} &\text{if $j$ is odd.}
\end{cases}
$$
We also set $\bi_{n,k;j} = \varnothing$ for $j<2$ or $j>2(n+k)$.
%For $0 \le k \le n$ and $j \in \Z$, let us set
%\begin{align*}
%\bi_{n,k;j} &=
%\begin{cases}
%\bi_{[\max(1,j-k), \min(j,n)]} &\text{if} \;\, 1 \le j \le n+k, \\
%\varnothing &\text{otherwise,}
%\end{cases} \\
%\hat \bi_{n,k;j} &=
%\begin{cases}
%\bi_{n,k;\lfloor j/2 \rfloor} &\text{if $j$ is even,} \\
%\bi_{n,k-1;\lfloor j/2 \rfloor}^t &\text{if $j$ is odd.}
%\end{cases}
%\end{align*}
Now, for $0 \le j \le 2(n+k)+1$ consider cyclic words
$$
\bi_{n,k}^{(j)} = \bar \bi_c \bi_{n,k;j} \bi_c^t \bi_{n,k;j+1},
$$
where $\bi_c = (1, \ldots, n)$. For example, for $n=3$, $k=2$ we obtain the following sequence of cyclic Weyl words
\begin{align*}
&\bi_{n,k}^{(0)} = \bi_{\bar c,c^t}, &
&\bi_{n,k}^{(1)} = \bi_{\bar c,c^t,1}, &
&\bi_{n,k}^{(2)} = \bi_{\bar c,1,c^t,1}, &
&\bi_{n,k}^{(3)} = \bi_{\bar c,1,c^t,1,2}, \\
&\bi_{n,k}^{(4)} = \bi_{\bar c,1,2,c^t,2,1}, &
&\bi_{n,k}^{(5)} = \bi_{\bar c,2,1,c^t,1,2,3}, &
&\bi_{n,k}^{(6)} = \bi_{\bar c,1,2,3,c^t,3,2}, &
&\bi_{n,k}^{(7)} = \bi_{\bar c,3,2,c^t,2,3}, \\
&\bi_{n,k}^{(8)} = \bi_{\bar c,2,3,c^t,3}, &
&\bi_{n,k}^{(9)} = \bi_{\bar c,3,c^t,3}, &
&\bi_{n,k}^{(10)} = \bi_{\bar c,3,c^t}, &
&\bi_{n,k}^{(11)} = \bi_{\bar c,c^t}.
\end{align*}
Note that the length of words $\bi_{n,k}^{(j)}$ increases as $j$ runs from 1 to $2k+1$, stays constant for $2k+1 \le j \le 2n$, and decreases as $j$ runs from $2n+1$ to $2(n+k)+1$. The maximal length is $2(n+k)+1$ if $k<n$ and is $2(n+k)$ for $k=n$.

We now define word transformations $\Phi_{n,k}^{(j+1)} \colon \bi_{n,k}^{(j)} \to \bi_{n,k}^{(j+1)}$ for all $0 \le j \le 2(n+k)$. For $2k+1 \le j \le 2n-1$ the transformation $\Phi_{n,k}^{(j+1)}$ consists of two commuting steps. The first step transforms the word $\bi_{n,k;j+1}\bi_{\bar c}$ into $\bi_{\bar c} \bi_{n,k;j+1}$, and consists of $\len(\bi_{n,k;j+1})$ shuffles. The second one transforms $\bi_{n,k;j} \bi_c^t$ into $\bi_c^t \bi_{n,k;j+2}$, and consists of $\len(\bi_{n,k;j})$ braid moves. It is easy to see that when $j$ is odd, the mutations within each step commute as well. Now, for $j \ge 2n$ we define $\Phi_{n,k}^{(j+1)}$ is a similar way, except for one of the braid moves is replaced with a $(n,n) \to n$ merge. For $j \le 2k$ the definition is similar again, but one of the braid moves is replaced with an \emph{insertion} $1 \to (1,1)$, inverse to the merge $(1,1) \to 1$. Finally, in the case $n=k$, the transformation $\Phi_{n,k}^{(2n+1)}$ contains both an insertion $1 \to (1,1)$ and a merge $(n,n) \to n$. Finally, we set
$$
\Phi_{n,k} = \Phi_{n,k}^{2(n+k)+1} \circ \ldots \circ \Phi_{n,k}^{(1)}.
$$

\begin{figure}[h]
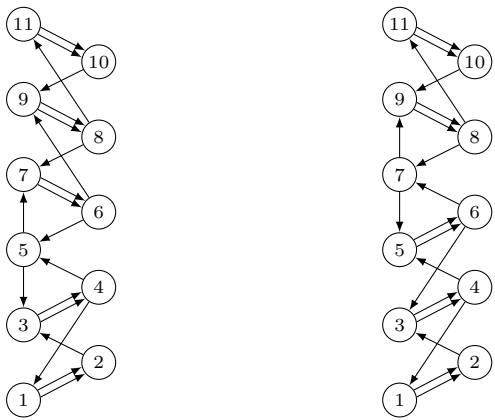

\subfile{fig-Qmn.tex}
\caption{Quiver $Q_{3,2}$ (on the left) and its image under $\bs\mu_{3,2}$ (on the right).}
\label{fig:Qmn}
\end{figure}

Using the dictionary between Weyl words and quivers described in Section~\ref{sec:isolating-coords},we can now present the transformation $\Phi_{n,k}$ as a sequence of mutations $\bs\mu_{n,k}$. Let $Q_{n,k}$ be the quiver with $2n+2k+1$ vertices, which contains $Q^n_{\mathrm{Toda}}$ and $Q^k_{\mathrm{Toda}}$ as subquivers, and such that its remaining vertex is connected to the bottom two vertices of $Q^n_{\mathrm{Toda}}$ and the top two vertices of $Q^k_{\mathrm{Toda}}$ as shown on Figure~\ref{fig:Qmn}. Then $\bs\mu_{n,k}$ mutates the subquiver $Q^k_{\mathrm{Toda}}$ from the bottom of the $Q^n_{\mathrm{Toda}}$ to its top, producing a quiver isomorphic to $Q_{n,k}$.

\begin{defn}
We call $\bs\mu_{n,k}$ the \emph{bi-fundamental Baxter mutation sequence.}
\end{defn}

\begin{remark}
\label{rem:bi-Baxter}
The mutation sequence~\eqref{eq:baxseq} is recovered as $\bs\mu_{n,0}$.
\end{remark}

To help us visualize the bi-fundamental Baxter operator, consider Figure~\ref{fig:Baxter-move}. We start by positioning the nodes of $Q_{n,k}$, as shown on the left pane of the Figure. After we have applied the bi-fundamental Baxter operator $\bs\mu_{n,k}$ the nodes will reach their final positions indicated on the right pane. After applying the subsequence $\bs\mu_{n,k}^{(j)}$, corresponding to the word transformation $\Phi_{n,k}^{(j)}$, we move the node with the smallest label, which has not yet reached its final position, one step along the path formed by yellow arrows. If this position its occupied, we move the node occupying it by one step as well, and so on. For example, after the very first mutation at node $5$ we shall be moving nodes 1 up to 6. This way, after we have applied $\bs\mu_{n,k}^{(j)} \circ \ldots \circ \bs\mu_{n,k}^{(1)}$, the Weyl word spelled by the arrows in the $n$ horizontal lines containing 4 nodes each is $\bi_{n,k}^{(j)}$.

\begin{figure}[h]
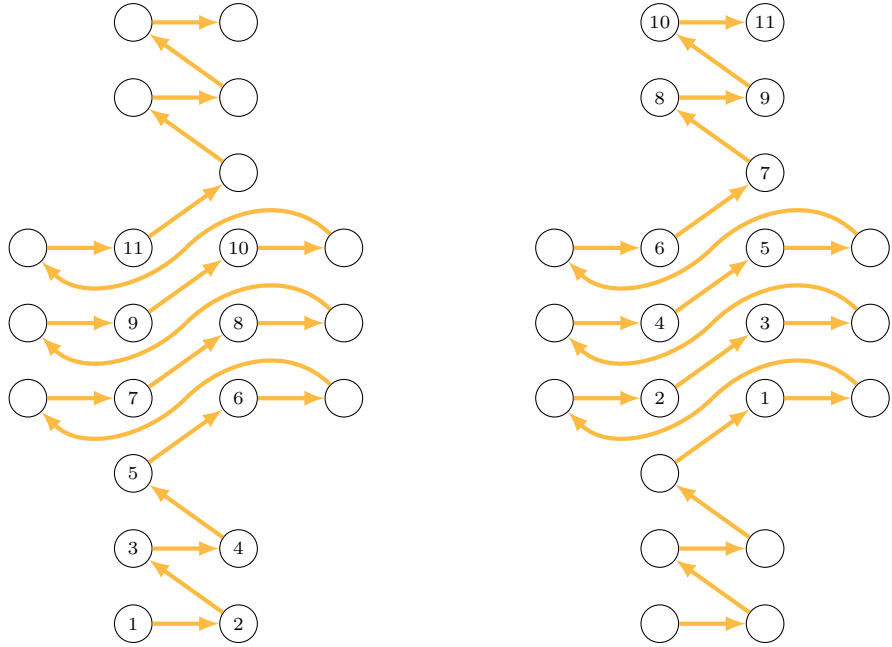

\subfile{fig-Baxter-shifts.tex}
\caption{Shifts of vertices during $\bs\mu_{3,2}$.}
\label{fig:Baxter-move}
\end{figure}

\begin{remark}
\label{rem:Baxter-arrows}
Tracking the mutation sequence $\bs\mu_{n,k}$, one notices that when we are about to mutate at a node labelled $i$, the only two nodes with arrows pointing to it are those with labels $i\pm1$.
\end{remark}

The sequence $\mu_{n,k}$ can also be encoded in the following neat combinatorial way. Consider the parallelogram on Figure~\ref{fig:Baxter-rhombus} made of copies of vertices of $Q_{n,k}$. Then the sequences of mutations in $\mu_{n,k}$ can be obtained by reading the columns of Figure~\ref{fig:Baxter-rhombus} left to right. We number the columns, so that the left most one has humber 1. Mutations within odd columns commute. Within each even column, highlighted by dashed rectangles, we read circled vertices bottom to top, and the non-circled ones top to bottom. Sequences of mutations at circled and non-circled vertices at a given even column commute as well. Mutations at circled vertices of column $j$ correspond to insertions, braid moves, and merges in the word transformation $\Phi_{n,k}^{(j)}$, while those at the non-circled ones correspond to shuffles. For example, we can factor the bi-fundamental Baxter operator $\bs\mu_{3,2}$ into 11 \emph{columns}
$$
\bs\mu_{3,2} = \bs\mu_{3,2}^{(11)} \circ \ldots \circ \bs\mu_{3,2}^{(1)},
$$
each itself factored as
$$
\mu_{3,2}^{(j)} = \mu_{3,2}^{(j), \bullet} \circ \mu_{3,2}^{(j), \circ}
$$
where
\begin{align*}
\bs\mu_{3,2}^{(1)} &= \mu_5, &
\bs\mu_{3,2}^{(2)} &= \mu_6 \circ \mu_4, &
\bs\mu_{3,2}^{(3)} &= \mu_5 \circ \mu_7\mu_3, \\
\bs\mu_{3,2}^{(4)} &= \mu_4\mu_8 \circ \mu_6\mu_2, &
\bs\mu_{3,2}^{(5)} &= \mu_3\mu_7 \circ \mu_9\mu_5\mu_1, &
\bs\mu_{3,2}^{(6)} &= \mu_2\mu_6\mu_{10} \circ \mu_8\mu_4, \\
\bs\mu_{3,2}^{(7)} &= \mu_5\mu_9 \circ \mu_{11}\mu_7\mu_3, &
\bs\mu_{3,2}^{(8)} &= \mu_4\mu_8 \circ \mu_{10}\mu_6, &
\bs\mu_{3,2}^{(9)} &= \mu_7 \circ \mu_9\mu_5, \\
\bs\mu_{3,2}^{(10)} &= \mu_6 \circ \mu_8, &
\bs\mu_{3,2}^{(11)} &= \mu_7.
\end{align*}

\begin{figure}[h]
\subfile{fig-Baxter-rhombus.tex}
\caption{Mutation sequence $\bs\mu_{3,2}$.}
\label{fig:Baxter-rhombus}
\end{figure}

\begin{remark}
\label{rem:Baxter-refactor}
Using above observations on the commutativity of certain mutations in the sequence $\bs\mu_{n,k}$, it is easy to see that we can refactor $\bs\mu_{n,k}$ as
$$
\bs\mu_{n,k} = \tilde{\bs\mu}_{n,k}^{2n+2k+1} \circ \ldots \circ \tilde{\bs\mu}_{n,k}^{(1)},
$$
where
$$
\tilde{\bs\mu}_{n,k}^{(j)} = \bs\mu_{n,k}^{(j+1),\bullet} \bs\mu_{n,k}^{(j),\circ},
$$
and $\bs\mu_{n,k}^{(2n+2k+2),\bullet}=1$. Let us also notice that by reflecting the parallelogram on Figure~\ref{fig:Baxter-rhombus} across the horizontal axis, see Figure~\ref{fig:Baxter-rhombus-2}, and applying the same recipe, we obtain equivalent factorization of the same cluster transformation $\bs\mu_{n,k}$.
\end{remark}

\begin{figure}[h]
\subfile{fig-Baxter-rhombus-2.tex}
\caption{Mutation sequence $\bs\mu_{3,2}$.}
\label{fig:Baxter-rhombus-2}
\end{figure}

%Finally, let us describe what is going on with the tropical variables as we apply the sequence of mutations $\mu_{n,k}^{-1}$. Here we always assume that $n \ge k$. We write $\mathrm y_i^{(j)}$ for the $i$-th tropical variable after we applied the $j$-th column of $\mu_{n,k}^{-1}$. For example, $\mathrm y_i^{(0)} = \mathrm y_i$ for all $1 \le i \le 2(n+k)+1$. In the notations of Figure~\ref{fig:Baxter-move} we have $\mathrm y_i^{(1)} = \mathrm y_i$ for $i \notin \hc{6,7,8}$, $\mathrm y_6^{(1)} = \mathrm y_6 \mathrm y_7$, $\mathrm y_7^{(1)} = \mathrm y_7^{-1}$, $\mathrm y_8^{(1)} = \mathrm y_7 \mathrm y_8$.

In the remainder of this section, we make extensive use of tropical cluster transformations. We start with the following useful result.

\begin{lemma}
\label{lem:Baxter-trop-var}
Let $\yrm_i$ denote the tropical cluster variables of the quiver $Q_{n,k}$, and $\yrm'_i$ be the image of $\yrm_i$  under the tropicalization of $\bs\mu_{n,k}^{-1}$. Then we have
$$
\mathrm y'_i =
\begin{cases}
\mathrm y_{i+2n+1} &\text{if} \;\; i < 2k+1, \\
\yrm_{[1,2(n+k)+1]}^{-1} &\text{if} \;\; i = 2k+1, \\
\mathrm y_{i-2k-1} &\text{if} \;\; i > 2k+1,
\end{cases}
$$
where $\yrm_{[j,k]} = \prod_{r=j}^k \yrm_r$.
\end{lemma}

\begin{proof}
Denote by $\yrm_i^{(j)}$ the image of $\yrm_i$ under the tropicalization of the first $j$ columns of $\bs\mu_{n,k}^{-1}$, so that $\yrm^{(0)} = \yrm$ and $\yrm^{(2n+2k+1)} = \yrm'$.
%that is $\hr{\bs\mu_{n,k}^{(2n+2k+2-j)}}^{-1} \ldots \hr{\bs\mu_{n,k}^{(2n+2k+1)}}^{-1}$.
Chasing the tropical variables through the above description of the bi-fundamntal Baxter operator we find that
$$
\yrm_i^{(j)} =
\begin{cases}
\yrm_i &\text{if $j < |2n+1-i|$,} \\
\yrm_{\hs{n+\frac{i-j+1}{2},n+\frac{i+j+1}{2}}} &\text{if $i-j$ is odd, $|2n+1-i| \le j \le 2(n+k)-|2k+1-i|$,} \\
\yrm_{\hs{n+1+\frac{i-j}{2},n+\frac{i+j}{2}}}^{-1} &\text{if $i-j$ is even, $|2n+1-i| \le j-1 \le 2(n+k)-|2k+1-i|$,} \\
\yrm_{2n+1+i} &\text{if $i \le 2k$, $j \ge 2n+1+i$,} \\
\yrm_{i-2k-1} &\text{if $i \ge 2k+2$, $j \ge 2n+4k+3-i$,}
\end{cases}
$$
which yields the desired result.
\end{proof}

%\begin{figure}[h]
%\subfile{fig-trop-evolution.tex}
%\caption{Tropical variables along $\mu_{3,2}^{-1}$ (reading right to left).}
%\label{fig:Baxter-trop}
%\end{figure}
%

\begin{defn}
Let $\sigma_{n,k}$ be the unique, non-trivial for $n=k$, permutation identifying quivers $Q_{n,k}$ and $\bs\mu_{n,k}(Q_{n,k})$. We define the \emph{Baxter twist} to be the following cluster transformation
$$
\beta_{n,k} = \sigma_{n,k} \circ \bs\mu_{n,k}.
$$
\end{defn}

\begin{cor}
The Baxter twist $\beta_{n,k}$ commutes with the Dehn twists $\tau_n$, $\tau_k$ from formula~\eqref{eq:short-Dehn}. Moreover, we have
$$
\beta_{n,k}^2 = \tau_n^{-k-1} \tau_k^{-n-1}.
$$
\end{cor}

\begin{proof}
The statements follow from a computation with tropical variables as in the proof of Lemma~\ref{lem:Baxter-trop-var} and Theorem~\ref{trop-criterion}.
\end{proof}

Recall that a cluster transformation $\bs\mu$ is called a \emph{DT-transformation} if its tropicalization $\bs\mu^t$ satisfies $\bs\mu^t(\yrm_i) = \yrm_1^{-1}$ for all tropical variables $\yrm_i$. The following corollary is proved in the same as the one above.

\begin{cor}
The cluster transformation
$$
\tau_n^{-n-1} \tau_k^{-k-1} \circ \beta_{n,k}^{-1} = \tau_n^{k-n} \tau_k^{n-k} \circ \beta_{n,k}
$$
is the DT-transformation for the quiver $Q^{n,k}_{\mathrm{Toda}}$.
\end{cor}

Now consider the lattice
$$
\La_{n,k} = \La_{GL_{n+1}} \oplus \La_{GL_{k+1}} \oplus \Z[\zeta],
$$
together with a skew-symmetric bilinear form, such that the three direct summands are orthogonal to each other and the embeddings $\La_{GL_{n+1}} \hookrightarrow \La_{n,k}$ and $\La_{GL_{k+1}} \hookrightarrow \La_{n,k}$ are isometries. Let
$$
\Prm_{n,k} \simeq \Prm_{GL_{n+1}} \oplus \Prm_{GL_{k+1}}
$$
be the weight lattice for the group $GL_{n+1} \times GL_{k+1}$. For a pair $(\la,\mu) \in P_{n,k}$ we define the following elements of $\La_{n,k}$:
\beq
\label{eq:pnk-xnk}
p_{(\la,\mu)} = (p_\la,p_\mu,0) \qquad x_{(\la,\mu)} = (x_\la,x_\mu,0).
\eeq
We equip $\La_{n,k}$ with the basis, obtained as the union of bases~\eqref{eq:toda-e-basis} for $\La_{GL_{n+1}}$ and $\La_{GL_{k+1}}$ and the vector $e_h = p_{(-\epsilon_1,\epsilon_{k+1})}-\zeta$. The resulting quiver is isomorphic to $Q_{n,k}$.

As in the previous section, let $\Dcr^{n,k}$ be the space of all $\Q(q)$-valued functions on $\Prm_{n,k}$, $\Fcr^{n,k} \subset \Dcr^{n,k}$ be its subspace of compactly supported functions, and $\Vcr^{n,k} \subset \Fcr^{n,k}$ be the subspace of those vanishing outside of the positive cone, that is
$$
\Vcr^{n,k} = \big\{\phi \in \Fcr^{n,k} \,\big|\, \phi(\la,\mu)=0 \;\;\text{for any}\;\; (\la,\mu) \notin \Prm^+_{n,k}\big\}.
$$
As before, the integral versions of the above spaces are obtained by replacing the $\Q(q)$-valued functions with the $\Z[q^{\pm1}]$-valued ones. We then set
$$
\mathscr A^{n,k}_{\bs w, \bs u} = \mathscr A^{n,k} \otimes_{\Q(q)} \Kc_{\bs w, \bs u}
\qquad\text{and}\qquad
\mathscr A^{n,k}_{\bs w, \bs u;z} = \mathscr A^{n,k} \otimes_{\Q(q)} \Kc_{\bs w, \bs u;z}
$$
where $\mathscr A = \Dcr, \Fcr, \Vcr$.

Write $\Tc^{n,k}$ for the quantum torus defined by $\La_{n,k}$. We denote by $P_{(\la,\mu)}$ and $X_{(\la,\mu)}$ its elements corresponding to the vectors~\eqref{eq:pnk-xnk}. We also write $\Tc^{n,k}_z$ for the space of Taylor series in $z$ with coefficients in $\Tc^{n,k}$. Evidently, the action of $\Tc^{n,k}$ on $\Dcr^{n,k}$ preserves the subspace $\Fcr^{n,k}$, and the same is true for the action of $\Tc^{n,k}_z$ on $\Dcr^{n,k}_z$.

\begin{defn}
The \emph{bi-fundamental Baxter operator} $\Qf_{n,k}\in \Tc^{n,k}_z$ is the inverse of the product of quantum dilogarithms that defines the automorphism part $\mu_{n,k}^\sharp$ of the cluster transformation $\mu_{n,k}$.
\end{defn}

Identifying $z$ with $Y_\zeta$ we have $\Qf_{n,k} = \Qf_{n,k}(z) \in \Tc^{n,k}_z$. We now write an explicit formula for its conjugate $\Qf^\star_{n,k}(z)$. First, we define
\begin{align*}
u_{(0,s)} &= p_{(0,\epsilon_s)} + x_{(0,\alpha_s)} & &\hspace{-2.6cm}\text{for}\qquad 1 \le s \le k, \\
u_{(s,0)} &= p_{(\epsilon_s,0)} - x_{(\alpha_{s-1},0)} & &\hspace{-2.6cm}\text{for}\qquad 2 \le s \le n+1,
\end{align*}
and set
$$
A^+(a,b \,|\, k) = \prod_j^{\longrightarrow} \Psi\hr{z Y_{a_{(0,j)} - b_{(j-k,0)}}},
\qquad
A^-(a,b \,|\, k) = \prod_j^{\longleftarrow} \Psi\hr{z Y_{a_{(0,j)} - b_{(j-k,0)}}},
$$
where $a,b \in \hc{u,p}$, we abbreviate $p_k = p_{\epsilon_k}$, and both products are taken over all $j$ such that both indices $j$ and $j-k$ are within the range. For example, we have
$$
A^+(u,u \,|\, k-1) = 1 \qquad\text{and}\qquad A^-(p,p \,|\, k) = \Psi\hr{zP_{(-\epsilon_1,\epsilon_{k+1})}}.
$$
%\red{Need Lemma for commuting bi-fundamental Baxter through Whittaker transform. What follows is a draft.}
%\blue{If you're going to work on this section tomorrow, can you add a discussion of the representation $\mathbb{V}$ we consider below? I vote for functions taking values in $\mathbb{Q}(q)((z))$ vanishing outside the dominant cone in the weight lattice of $GL_{n+1}\times GL_{k+1}$, and letting the connecting vertex act by $z^{-1}P_{k+1;k+1}P_{1;n+1}^{-1}$ so that the product of dilogarithms with our usual sign convention will act on $\mathbb{V}$ (right?).   }
%
%Similar to the previous section, let us polarize the quiver $Q_{n,k}$ using canonical pairs
%\begin{align*}
%\bs x^{(1)} &= \hr{x_1^{(1)}, \ldots, x_{k+1}^{(1)}}, &
%\bs p^{(1)} &= \hr{p_1^{(1)}, \ldots, p_{k+1}^{(1)}}, \\
%\bs x^{(2)} &= \hr{x_1^{(2)}, \ldots, x_{n+1}^{(2)}}, &
%\bs p^{(2)} &= \hr{p_1^{(2)}, \ldots, p_{n+1}^{(2)}}.
%\end{align*}
%Let us also introduce notation
%\begin{align*}
%u_s^{(1)} &= p_s^{(1)} + x_s^{(1)} - x_{s+1}^{(1)} \qquad\text{for}\qquad 1 \le s \le k, \\
%u_s^{(2)} &= p_s^{(2)} + x_s^{(2)} - x_{s-1}^{(2)} \qquad\text{for}\qquad 2 \le s \le n+1,
%\end{align*}
%and set
%$$
%A^+(\bs a, \bs b | k) = \prod_j^{\longrightarrow} \psi(a_j - b_{j-k}),
%\qquad
%A^-(\bs a, \bs b | k) = \prod_j^{\longleftarrow} \psi(a_j - b_{j-k}),
%$$
%where $\psi(a_j-b_{j+k})=1$ if any of the indices is out of range. \red{Here we write $\psi(x+y)$ for $\Psi^q(:\!\!XY\!\!:)$}.
Then using Remark~\ref{rem:Baxter-arrows}, we can write the adjoint to the bi-fundamental Baxter operator as
$$
\Qf^\star_{n,k}(z) = \Qf_{n,k}^{\star, (2n+2k+1)}(z) \circ \ldots \circ \Qf_{n,k}^{\star, (1)}(z),
$$
where
$$
\Qf_{n,k}^{\star, (j)}(z) = 
\begin{cases}
A^+\hr{u,u \,|\, k-(j+1)/2} A^-\hr{p,p \,|\, k - (j-1)/2} &\text{if $j$ is odd,} \\
A^+\hr{p,u \,|\, k-j/2} A^-\hr{u,p \,|\, k - j/2} &\text{if $j$ is even.}
\end{cases}
$$
Using Remark~\ref{rem:Baxter-refactor} we can also rewrite the operator $\Qf^\star_{n,k}(z)$ as
\beq
\label{eq:Baxter-tilde-factor}
\Qf^\star_{n,k}(z) = \widetilde\Qf_{n,k}^{\star, (2n+2k+1)}(z) \circ \ldots \circ \widetilde\Qf_{n,k}^{\star, (1)}(z),
\eeq
where
$$
\widetilde\Qf_{n,k}^{\star, (j)}(z) = 
\begin{cases}
A^+\hr{p,u \,|\, k-(j+1)/2} A^-\hr{p,p \,|\, k - (j-1)/2} &\text{if $j$ is odd,} \\
A^+\hr{u,u \,|\, k-(j+2)/2} A^-\hr{u,p \,|\, k - j/2} &\text{if $j$ is even.}
\end{cases}
$$
%In particular, we have
%\begin{align*}
%\Qf_{n+1}^+(z) &= \Qf_{n,0}\hr{zP_{(0,\epsilon_1)}}, \\
%\Qf_{k+1}^+(z) &= \Qf_{0,k}\hr{zP_{(-\epsilon_1,0)}}.
%\end{align*}

Consider the product
$$
W^{(n+1,k+1)}(\bs w, \bs u) = W^{(n+1)}(\bs w) W^{(k+1)}(\bs u) \in \Dcr^{n,k}_{\bs w,\bs u}
$$
of the two Whittaker kernels. We are now ready to prove the following generalization of Proposition~\ref{lem:whit-bax}. 

\begin{prop}
\label{prop:whit-bibax}
The Whittaker kernel $W^{(n,k)}(\bs w, \bs u)$ diagonalizes the conjugate bi-fundamental Baxter operator:
\begin{align}
\label{eq:whit-bibax}
\Qf^\star_{n,k}(z) W^{(n+1,k+1)}(\bs w, \bs u) = W^{(n+1,k+1)}(\bs w, \bs u) \prod_{r=1}^{n+1} \prod_{s=1}^{k+1} \Psi(zw_r/u_s).
\end{align}
\end{prop}

\begin{proof}
We shall prove the equality~\eqref{eq:whit-bibax} by induction on $n$. For $n=0$ the desired result is equivalent to that of Proposition~\ref{lem:whit-bax}. Now recall the equality
$$
W^{(n+1)}(\bs w) = R_n(w_{n+1}) w_{n+1}^{\la_{n+1}}W^{(n)}(\bs w'),
$$
where $\bs w' = (w_1, \ldots, w_n)$, as well as factorization~\eqref{eq:Baxter-tilde-factor}. Applying pentagon relation over and over again one commutes each factor of $\Qf^\star_{n,k}$ to the right past $R_n(w_{n+1})$ arriving at the relation
$$
\Qf^\star_{n,k}(z) R_n(w_{n+1}) = R_n(w_{n+1}) \Qf^\star_{0,k}(zP_{(-\epsilon_{n+1},0)}) F_{n,k} \Qf^\star_{n-1,k}(z),
$$
where the factor $F_{n,k}$ preserves the product $w_{n+1}^{\la_{n+1}}W^{(n,k+1)}(\bs w',\bs u)$ thanks to Lemma~\ref{lem:tlem2}. Thus, we have
\begin{align*}
\Qf^\star_{n,k}(z) W^{(n+1,k+1)}(\bs w,\bs u) &= \Qf^\star_{n,k}(z) R_n(w_{n+1}) w_{n+1}^{\la_{n+1}}W^{(n,k+1)}(\bs w',\bs u) \\
&= R_n(w_{n+1}) \Qf^\star_{0,k}(zP_{(-\epsilon_{n+1},0)}) F_{n,k} \Qf^\star_{n-1,k}(z) w_{n+1}^{\la_{n+1}}W^{(n,k+1)}(\bs w',\bs u).
\end{align*}
Applying the induction hypothesis, we arrive at
\begin{align*}
\Qf^\star_{n,k}(z) W^{(n+1,k+1)}(\bs w,\bs u)
&= R_n(w_{n+1}) w_{n+1}^{\la_{n+1}}W^{(n,k+1)}(\bs w',\bs u) \prod_{r=1}^{n+1} \prod_{s=1}^{k+1} \Psi(zw_r/u_s) \\
&= W^{(n+1,k+1)}(\bs w,\bs u) \prod_{r=1}^{n+1} \prod_{s=1}^{k+1} \Psi(zw_r/u_s).
\end{align*}
\end{proof}

\begin{cor}
We have the following equalitites:
\begin{align*}
\Qf^\star_{n,k}(z) W^{(n+1,k+1)}(\bs w,\bs u)
&= \prod_{j=1}^{n+1} \Qf^\star_{0,k}(zP_{(-\epsilon_j,0)}) W^{(n+1,k+1)}(\bs w,\bs u) \\
&= \prod_{j=1}^{k+1} \Qf^\star_{n,0}(zP_{(0,\epsilon_j)}) W^{(n+1,k+1)}(\bs w,\bs u).
\end{align*}
\end{cor}

\begin{cor}
\label{cor:biToda-coeff}
The coefficient of $z^k$ in the power series $\Qf^\star_{n,k}(z)$ can be expressed as a polynomial in the $GL_{n+1}\times GL_{k+1}$ Toda Hamiltonians $H_{j}^{(n)}H_l^{(k)}$ with coefficients in $\mathbb{Q}(q)$.
\end{cor}

\subsection{Refined umbral moves}
\label{subsec:refined-umbral}
Let $c$ be a closed simple curve on a marked surface $S$, $\tri$ a $c$-isolating triangulation of $S$, $e \in \tri$ a non-boundary shadow of $c$, and $C \subset \tri$ the $c$-isolating cylinder. Denote by $S'$ the result of cutting $S$ along $c$. As usual, we let $c_\pm \subset S'$ be the new pair of tacked circles, consider the triangulation $\Delta' = \Cc_c(\Delta)$ of $S'$, and denote by $C_\pm$ the isolating cylinders of $c_\pm$. Furthermore, we assume that the arc $e_+ \in \tri'$ corresponding to $e \in \tri$ is a shadow of $c_+$, and thus belongs to $C_+$. The goal of this section is to prepare ground for the proof of Theorem~\ref{thm:alg-MF-coherence}, which states that Whittaker transforms $\Wbb_c$ intertwine cluster transformations, which relate different $c$- or $c_\pm$-isolating clusters on $S$ and $S'$. To that end we introduce \emph{refined umbral moves} and factor them in a way, which will allow us to commute them through the Whittaker transform.

\begin{defn}
The \emph{refined umbral moves} are the compositions
%\begin{align*}
%F_{c;e} &= \Phi_{F_{C;e}(\tri);c} \circ F_{C;e} \circ \Phi_{\tri;c}^{-1}, & 
%F_{c_+;e_+} &= \Phi_{F_{C_+;e_+}(\tri');c_+} \circ F_{C_+;e_+} \circ \Phi_{\tri';c_+}^{-1}, \\
%F_{c;e}^{op} &= \Phi_{F_{C;e}(\tri);c} \circ F^{op}_{C;e} \circ \Phi_{\tri;c}^{-1}, & 
%F_{c_+;e_+}^{op} &= \Phi_{F_{C_+;e_+}(\tri');c_+} \circ F^{op}_{C_+;e_+} \circ \Phi_{\tri';c_+}^{-1}.
%\end{align*}
$$
F_{c;e} = \Phi_{F_{C;e}(\tri);c} \circ F_{C;e} \circ \Phi_{\tri;c}^{-1},
\qquad\qquad
F_{c_+;e_+} = \Phi_{F_{C_+;e_+}(\tri');c_+} \circ F_{C_+;e_+} \circ \Phi_{\tri';c_+}^{-1},
$$
where $F_{C;e}$, $F_{C_+;e_+}$ are the umbral moves, see Definition~\ref{def:umbral}, $\Phi_{\tri;c}$ is the transformation defined in~\eqref{eq:Phi-c}, and $\Phi_{\tri';c_+}$ is the one defined in~\eqref{eq:Phi-c_+}. \emph{Opposite refined umbral moves} are defined in a similar fashion.
\end{defn}

First, let us recall that each umbral move is a composition of 3 flips. Consider edges $e,e^1,e^2$ of $\tri$, and $e_+, e_+^1, e_+^2$ of $\tri'$ such that
$$
F_{C;e} = F_{e^2} F_{e^1} F_e
\qquad\text{and}\qquad
F_{C;e_+} = F_{e_+^2} F_{e_+^1} F_{e_+}.
$$
For $1 \le k \le 6$ define quivers
\begin{align*}
Q_{c;e;1} &= Q_{F_{C;e}(\tri);c}, &
Q_{c;e;2} &= Q_{F_{C;e}(\tri)}, &
Q_{c;e;3} &= Q_{F_{e^1} F_e(\tri)}, \\
Q_{c;e;4} &= Q_{F_e(\tri)}, &
Q_{c;e;5} &= Q_{\tri}, &
Q_{c;e;6} &= Q_{\tri;c}.
\end{align*}
We denote the cluster transformations sending $Q_{c;e;k}$ to $Q_{c;e;j}$ by $\bs\mu_{c;e;j\,\circlearrowleft\,k}$ if $j<k$ and by $\bs\mu_{c;e;j\,\circlearrowright\,k}$ otherwise, so that $
\bs\mu_{c;e;j\,\circlearrowleft\,k} = \bs\mu_{c;e;k\,\circlearrowright\,j}^{-1}$. For example we have $\bs\mu_{c;e;2\,\circlearrowleft\,5} = F_{C;e}$ and $\bs\mu_{c;e;1\,\circlearrowleft\,6} = F_{c;e}$. We define quivers $Q_{c_+;e_+;k}$ and transformations $\bs\mu_{c_+;e_+;j\,\circlearrowleft\,k}$ in a similar fashion. For the case when $S'_+$ is a cylinder with a tack on one boundary component and 2 marked points on the other, we show quivers $Q^5_{c_+;e_+;k}$, $1 \le k \le 6$, on Figures~\ref{fig:Q-1}--\ref{fig:Q-6}. Here and in what follows the upper index $n$ in the notation for the quiver is the rank of the group $G$, defining the character variety.

We now define quivers $Q_{c;e;k}$ and $Q_{c_+;e_+;k}$ for $7 \le k \le 13$, and place the 13 quivers on the 13-hour clock face. Quivers $Q^5_{c_+;e_+;k}$ are shown on Figures~\ref{fig:Q-7}--\ref{fig:Q-13}. Quivers $Q^n_{c;e;k}$ may be obtained in the following way. First, define quivers $Q^n_{c;k}$ by setting $Q^n_{c;k} \simeq \widehat Q^n_{\mathrm{Toda}}$ for $k = 7,8,9$ or $k=13$ and $Q^n_{c;k} \simeq \mu_{n,0}(\widehat Q^n_{\mathrm{Toda}})$ for $k=10,11,12$, see Figure~\ref{fig:Qk-aux}. Next, recall quiver $Q^n_{\tri;c_-}$, where $\tri$ consists of a single self-folded triangle, see the left pane of Figure~\ref{fig-Qmitre}. Finally, amalgamate vertices $c_n$ of $Q^n_{c_+;e_+;k}$ and $(n,n-1)$ of $Q^n_{\tri;c_-}$ with vertices $h_\pm$ of $Q^n_{c;k}$, and erase vertices labelled $h_k$ and $d_k$ in $Q^n_{c_+;e_+;k}$ and the corresponding vertices in $Q^n_{\tri;c_-}$ labelled respectively $(n,n-1+k)$ and $(n+1,n-1+k)$.

\begin{figure}[h]
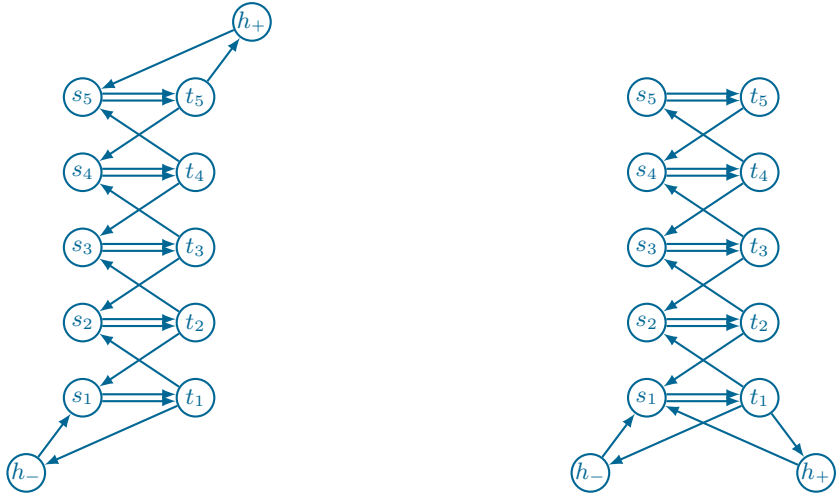

\subfile{fig-Qk-aux.tex}
\caption{Quivers $Q^5_{c;k}$ for $k = 7,8,9,13$ (on the left) and $k=10,11,12$ (on the right).}
\label{fig:Qk-aux}
\end{figure}

For the remainder of this section we abbreviate $Q_{c_+;e_+;k}$ to $Q_k$. As above we let $\bs\mu_{j\,\circlearrowleft\,k}$ and $\bs\mu_{j\,\circlearrowright\,k}$ denote the cluster transformations relating the $Q_k$ and $Q_j$, depending on whether $Q_k$ precedes or follows $Q_j$ as we traverse the face of the clock in the clock-wise fashion. Each time, having defined $\bs\mu_{k+1\,\circlearrowright\,k}$ we shall also explain how to modify it to obtain $\bs\mu_{c;e;k+1\,\circlearrowright\,k}$.

Recall the quiver $Q_\Delta$ corresponding to the standard bicolored graph on a single triangle $\Delta$, see Figure~\ref{fig:net-quiver}. Let us label vertices of $Q_\Delta$ by a triple of numbers $(i,j,k)$ reflecting the co-distance of the corresponding face of the bicolored graph to the three vertices of $\Delta$. In particular, we have $i+j+k=n$ for each vertex. Let $A$ be the vertex of $\Delta$, co-distance to which is measured by the first number in the triple $(i,j,k)$. Then we define the transformation $\Phi_{A;*}$ as follows:
$$
\Phi_{\Delta;A} = \prod_{r=1}^{\substack{n-1 \\ \longrightarrow}} \prod_{s=1}^{\substack{r \\ \longrightarrow}} \prod_{j=1}^{r-s+1} \mu_{(n-r+s,r-s-j,j)}.
$$
We remark that mutations under the third product commute. Applying $\Phi_{\Delta;A}$ to the quiver $Q_\Delta$ and swapping positions of the nodes $(j,n-j,0)$ and $(n-j,0,j)$ for all $1 \le j \le n$ we recover the quiver $Q_\Delta$ with all arrows inverted. The quiver $Q_6$ contains a subquiver obtained by amalgamating $Q_\Delta$ to $\Phi_c(Q_\mathrm{sf})$ along $e$. We define the quiver $Q_7 = \Phi_{\Delta;A}(Q_6)$, where $A$ is the vertex of $\Delta$ opposite the side $e$, that is $\Phi_{\Delta;A}$ starts and ends with a mutation at the node $c_1$ in the notations of Figure~\ref{fig:Q-7}. We set $\bs\mu_{7\,\circlearrowright\,6} = \Phi_{\Delta;A}$, and define $\bs\mu_{c;e;7\,\circlearrowright\,6}$ in an identical way, replacing $\Phi_c(Q_\mathrm{sf})$ with $\Phi_c(Q_\mathrm{cyl})$, see Figure~\ref{fig-Qcandy}.

We now recall the transformation $\Phi_2$ from Section~\ref{subsec:phi2} and consider its close cousin $\overline \Phi_2 \colon \bi_{\overline w_0,c^*} \to \bi_{-c,c}$. The $(j+1,k+1)$-st step of the latter only differs from that of $\Phi_2$ in starting with the shuffle $(k+1,-k-1) \mapsto (-k-1,k+1)$ instead of its inverse. Note that $Q_7$ contains a subquiver $Q^{n-1}_{\overline w_0,c^*}$, where the upper index shows the number of the letters in the alphabets $\Agt_\pm$. On Figure~\ref{fig-Qcandy} this is the subquiver formed by vertices in rows $n+1$ to $2n-1$ counting from the bottom. We now set $Q_{7;0} = Q_7$ and define $Q_{7;1} = \overline\Phi_2(Q_{7;0})$. Then the mutable vertices of $Q_{7;1}$ positioned in the rows $n$ to $2n-3$ form the subquiver $Q^{n-2}_{\overline w_0,c}$, which allows us to apply $\Phi_2$. The resulting quiver $Q_{7;2} = \Phi_2(Q_{7;1})$ again contains a subquiver $Q^{n-3}_{\overline w_0,c^*}$ in the rows $n-1$ to $2n-5$ and we may proceed by induction. We then define
$$
Q^n_8 = Q^n_{7;n-2} = \bs\mu_{8\,\circlearrowright\,7}(Q^n_7),
$$
where $\bs\mu_{8\,\circlearrowright\,7}$ is an alternating composition of transformations $\overline\Phi_2$ and $\Phi_2$, starting with $\overline\Phi_2$ on $n-1$ letters and dropping the number of letters by 1 with each factor. The quiver $Q^5_8$ is shown on Figure~\ref{fig:Q-8}. The transformation $\bs\mu_{c;e;8\,\circlearrowright\,7}$ is identical to $\bs\mu_{8\,\circlearrowright\,7}$.

Quiver $Q^n_8$ contains subquivers $Q^k_{\mathrm{Toda}}$ for all $1 \le k \le n-1$. We define $\bs\mu_{9\,\circlearrowright\,8}$ as the following composition of commuting Dehn twists
$$
\bs\mu_{c;e;9\,\circlearrowright\,8} = \bs\mu_{9\,\circlearrowright\,8} = \prod_{k=1}^{n-1} \tau_k^{-k-1}.
$$
Let us clarify the shape of quiver $Q^n_9$: in the notations of Figure~\ref{fig:Q-9} we have
\begin{align*}
(e_{t_{k,1}},e_{r_k}) &= a_n(k)+1, & (e_{t_{k,k-1}},e_{l_k}) &= a_n(k), \\
(e_{s_{k,1}},e_{r_k}) &= -a_n(k), & (e_{s_{k,k-1}},e_{l_k}) &= 1-a_n(k),
\end{align*}
where
$$
a_n(k) =
\begin{cases}
k-1 &\text{if $n-k$ is even,} \\
1 &\text{if $n-k$ is odd.} \\
\end{cases}
$$

Note that $Q^n_9$ still contains subquivers $Q^k_{\mathrm{Toda}}$ for $1 \le k \le n-1$ and recall that bi-fundamental Baxter operators $\bs\mu_{n,k}$. We then set
$$
%\bs\mu_{10\,\circlearrowright\,9} = \bs\mu_{n-1,0}^{-n-1} \circ \prod_{k=1}^{n-1} \bs\mu_{k,k-1}^{-1}
\bs\mu_{10\,\circlearrowright\,9} = \bs m_{10\,\circlearrowright\,9} \circ \prod_{j=1}^{n+1}\bs\mu^{-1}_{n-1,0}(h_j) \circ \prod_{k=1}^{n-1} \bs\mu_{k,k-1}^{-1}
$$
where the Baxter mutation sequence $\bs\mu^{-1}_{n-1,0}(h_j)$ is applied to the subquiver formed by $Q^{n-1}_{\mathrm{Toda}}$ the node $h_j$, and we understand $h_{n+1} = c_n$. The quasi-permutation $\bs m_{10\,\circlearrowright\,9}$ is characterized by being identity on all cluster variables but the $Y_{d_j}$ and the condition
$$
\bs m_{10\,\circlearrowright\,9}(Y_{d_j}^{(10)}) = Y_{d_j}^{(9)},
$$
where $Y_\ell^{(k)}$ denotes the cluster variable labelled by the node $\ell$ in $Q_k$. In turn, the quiver $Q^n_{c;e;9}$ contains $Q^k_{\mathrm{Toda}}$ for $1 \le k \le n$ and we set
$$
\bs\mu_{c;e;10\,\circlearrowright\,9} = \prod_{k=1}^n \bs\mu_{k,k-1}^{-1}.
$$

As is demonstrated on Figures~\ref{fig:Q-8} and~\ref{fig:Q-10}, the quivers $Q_8$ and $Q_{10}$ contain isomorphic mutable subquivers, same applies to $Q_{c;e;8}$ and $Q_{c;e;10}$. Thus, we may set $\bs\mu_{11\,\circlearrowright\,10} = \bs\mu_{7\,\circlearrowleft\,8} = \bs\mu_{8\,\circlearrowright\,7}^{-1}$ and $\bs\mu_{c;e;11\,\circlearrowright\,10} = \bs\mu_{c;e;7\,\circlearrowleft\,8}$.

The transformation $\bs\mu_{12\,\circlearrowright\,11}$ is in fact very similarly to $\bs\mu_{8\,\circlearrowright\,7}$. Namely, we have already noticed that $Q^n_{11}$ contains a subquiver, isomorphic to $Q^{n-1}_{\overline w_0,c^*}$. Skipping the bottom row of the latter, that is considering only rows $n+2$ to $2n-1$ as shown on Figure~\ref{fig:Q-11}, we obtain the subquiver $Q^{n-2}_{\overline w_0,c^*}$. Thus, we may set $Q^n_{11;0} = Q^n_{11}$ and define $Q^n_{11;1} = \overline\Phi_2^{(n-2)}(Q^n_{11;0})$. The latter quiver in turn contains a subquiver $Q^{n-4}_{\overline w_0,c}$ and we proceed as in the definition of $\bs\mu_{8\,\circlearrowright\,7}$. We then define
%$$
%Q^n_{11;k+1} = 
%\begin{cases}
%\overline\Phi_2^{(n-2k-2)}(Q^n_{11;k}) &\text{if $k$ is even,} \\
%\Phi_2^{(n-2k-2)}(Q^n_{11;k}) &\text{if $k$ is odd,}
%\end{cases}
%$$
$$
Q^n_{12} = Q^n_{11;\left\lfloor \frac{n-2}{2} \right\rfloor} = \bs\mu_{12\,\circlearrowright\,11}(Q^n_{11}),
$$
where $\bs\mu_{12\,\circlearrowright\,11}$ is an alternating composition of transformations $\overline\Phi_2$ and $\Phi_2$, starting with $\overline\Phi_2$ on $n-1$ letters and dropping the number of letters by 1 with each factor. The quiver $Q^5_{12}$ is shown on Figure~\ref{fig:Q-12}.

Observe that the quiver $Q^n_{12}$ contains subquivers $Q^{n-2k}_{\mathrm{Toda}}$ for $1 \le k \le \left\lfloor \frac{n}{2} \right\rfloor$. We define
$$
\bs\mu_{13\,\circlearrowright\,12} = \bs m_{13\,\circlearrowright\,12} \circ  \prod_{j=1}^{n+1}\bs\mu^{-1}_{n-2,0}(h_j) \circ \prod_{k=1}^{\left\lfloor \frac{n-2}{2} \right\rfloor} \bs\mu_{n-2k,n-2k-2}^{-1} \circ \prod_{k=1}^{\left\lfloor \frac{n}{2} \right\rfloor} \tau_{n-2k}^{2k-n-1},
$$
where the Baxter mutation sequence $\bs\mu^{-1}_{n-1,0}(h_j)$ is applied to the subquiver formed by $Q^{n-1}_{\mathrm{Toda}}$ the node $h_j$, and the quasi-permutation $\bs m_{13\,\circlearrowright\,12}$ is defined similarly to $\bs m_{10\,\circlearrowright\,9}$.
The transformation $\bs\mu_{c;e;13\,\circlearrowright\,12}$ is defined by a simpler formula:
$$
\bs\mu_{c;e;13\,\circlearrowright\,12} = \prod_{k=1}^{\left\lfloor \frac{n}{2} \right\rfloor} \hr{\bs\mu_{n+2-2k,n-2k}^{-1} \circ \tau_{n-2k}^{2k-n-1}}.
$$

Finally, we observe that quivers $Q_1$ and $Q_{11}$ have isomorphic mutable subquivers, see Figures~\ref{fig:Q-1} and~\ref{fig:Q-11}, which allows us to complete the clock face by setting
$$
\bs\mu_{13\,\circlearrowleft\,1} = \bs\mu_{c;e;12\,\circlearrowright\,11}.
$$
This way we have
$$
\bs\mu_{13\,\circlearrowleft\,1}(Q_{1}) = Q_{13}.
$$
Using similar considerations, we have
$$
\bs\mu_{c;e;13\,\circlearrowleft\,1}(Q_{c;e;1}) = Q_{c;e;13}
\qquad\text{with}\qquad
\bs\mu_{c;e;13\,\circlearrowleft\,1} = \bs\mu_{c;e;12\,\circlearrowright\,11}.
$$

\begin{prop}
\label{prop:umbral-factorization}
We have equalities of cluster transformations
$$
F_{c;e} = \bs\mu_{c;e;1\,\circlearrowright\,6}
\qquad\text{and}\qquad
F_{c_+;e_+} = \bs\mu_{c_+;e_+;1\,\circlearrowright\,6}.
$$
\end{prop}

\begin{proof}
We prove only the second equality, as the proof of the first one is similar. By the definition of the refined umbral move and that of transformations $\bs\mu_{k\,\circlearrowleft\,j}$, we have $F_{c_+;e_+} = \bs\mu_{1\,\circlearrowleft\,6}$. Thus it suffices to show that
$$
\bs\mu_{k\,\circlearrowleft\,j} = \bs\mu_{k\,\circlearrowright\,j}
$$
for all, equivalently any, $1 \le j \ne k \le 13$. Let us chose $j=9$ and $k=13$. By Theorem~\ref{trop-criterion}, it suffices to chase tropical variables through transformations $\bs\mu_{13\,\circlearrowleft\,9}$ and $\bs\mu_{13\,\circlearrowright\,9}$, which is done in Lemmas~\ref{lem:8-left-9}--\ref{lem:13-left-9} and Lemmas~\ref{lem:10-right-9}--\ref{lem:13-right-9} respectively. Proofs of these Lemmas are combinatorial exercises, similar in spirit to that of Lemma~\ref{lem:Baxter-trop-var}, and we omit them for brevity. Comparing the tropical variables of quivers $\bs\mu_{13\,\circlearrowright\,9}(Q_9)$ and $\bs\mu_{13\,\circlearrowleft\,9}(Q_9)$ completes the proof.
\end{proof}

%\begin{theorem}
%\label{thm:refined-umbral}
%Let $\tri_1, \tri_2$ be the $c$-isolating triangulations on $S$, and $\tri'_1, \tri'_2$ be their images on $S'$ under the cutting functor $\Cc_c$. Denote by $\mu_S \colon \Theta_{\tri_1;c} \to \Theta_{\tri_2;c}$ and $\mu_{S'} \colon \Theta_{\tri'_1;c_\pm} \to \Theta_{\tri'_2;c_\pm}$ the cluster transformations, relating the corresponding refined $c$- or $c_\pm$-isolating seeds. Then we have the following commutative diagram: \red{(not sure about notations for algebras... Also, why don't we use $\Lbb'$ instead of $\Lbb^{res}$?}
%$$
%\begin{tikzcd}
%\Lbb_{\tri_1;c} \arrow{r}{\Wbb_{\tri_1;c}} \arrow{d}{\mu_S} & \Lbb^{res}_{\tri'_1;c_\pm} \arrow{d}{\mu_{S'}} \\
%\Lbb_{\tri_2;c} \arrow{r}{\Wbb_{\tri_2;c}} & \Lbb^{res}_{\tri'_2;c_\pm}
%\end{tikzcd}
%$$
%\end{theorem}
%
%\begin{proof}
%By Lemma~\ref{lem:re-iso}, it suffices to treat the case when $\mu_S$ is a refined umbral move. Now the desired statement follows from Proposition~\ref{prop:umbral-factorization} and the intertwining relation~\ref{eq:alg-q-inter}.
%\red{Shall we be a bit more careful and discuss the opposite umbral moves, as well as the case when we have a Dehn twist, but don't have space for an umbral move?}
%\end{proof}
%
%
%
%

\clearpage

\null
\vfill
\begin{figure}[h]
\subfile{fig-Q-3.tex}
\caption{Quiver $Q_1^5$.}
\label{fig:Q-1}
\end{figure}
\null
\vfill

\clearpage

\null
\vfill
\begin{figure}[h]
\subfile{fig-Q-2-sf.tex}
\caption{Quiver $Q_2^5$.}
\label{fig:Q-2}
\end{figure}
\null
\vfill

\clearpage

\null
\vfill
\begin{figure}[h]
\subfile{fig-Q-14-back-qg-2.tex}
\caption{Quiver $Q_3^5$.}
\label{fig:Q-3}
\end{figure}
\null
\vfill

\clearpage

\null
\vfill
\begin{figure}[h]
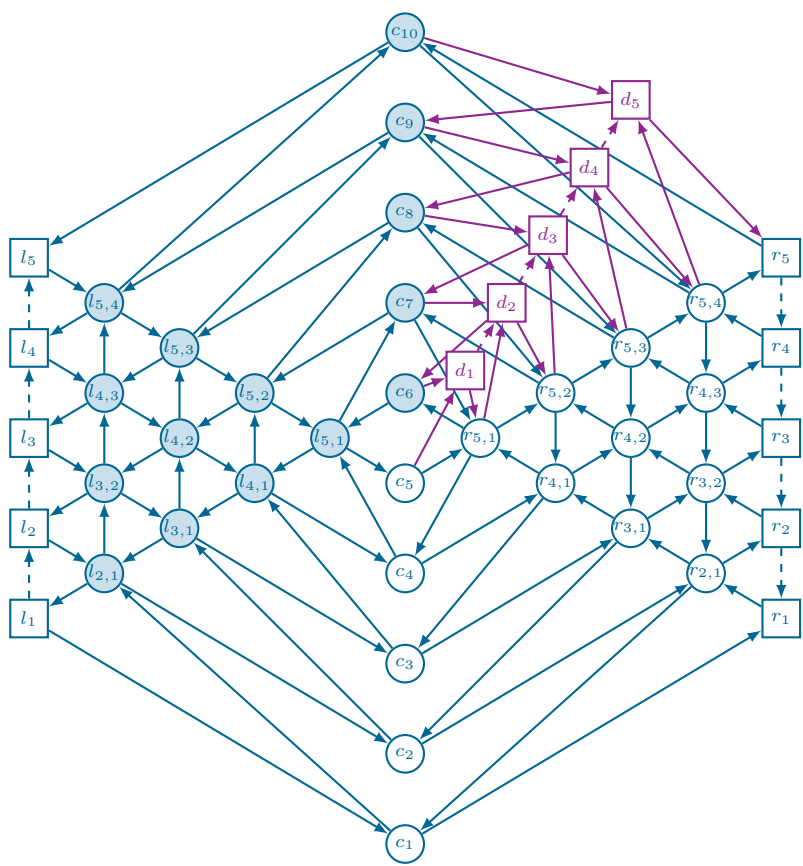

\subfile{fig-Q-14-back-qg.tex}
\caption{Quiver $Q_4^5$.}
\label{fig:Q-4}
\end{figure}
\null
\vfill

\clearpage

\null
\vfill
\begin{figure}[h]
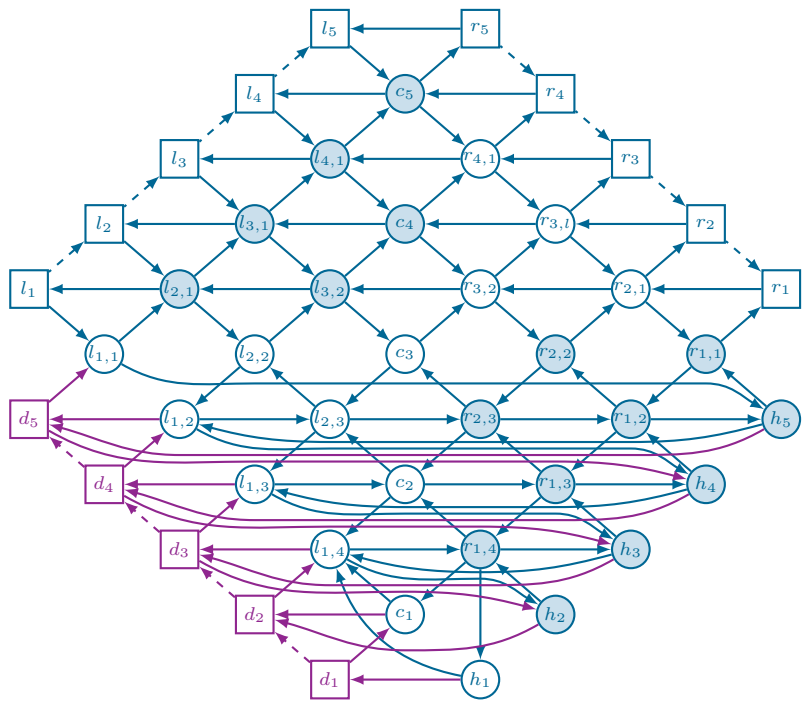

\subfile{fig-Q-13-back-sf.tex}
\caption{Quiver $Q_5^5$.}
\label{fig:Q-5}
\end{figure}
\null
\vfill

\clearpage

\null
\vfill
\begin{figure}[h]
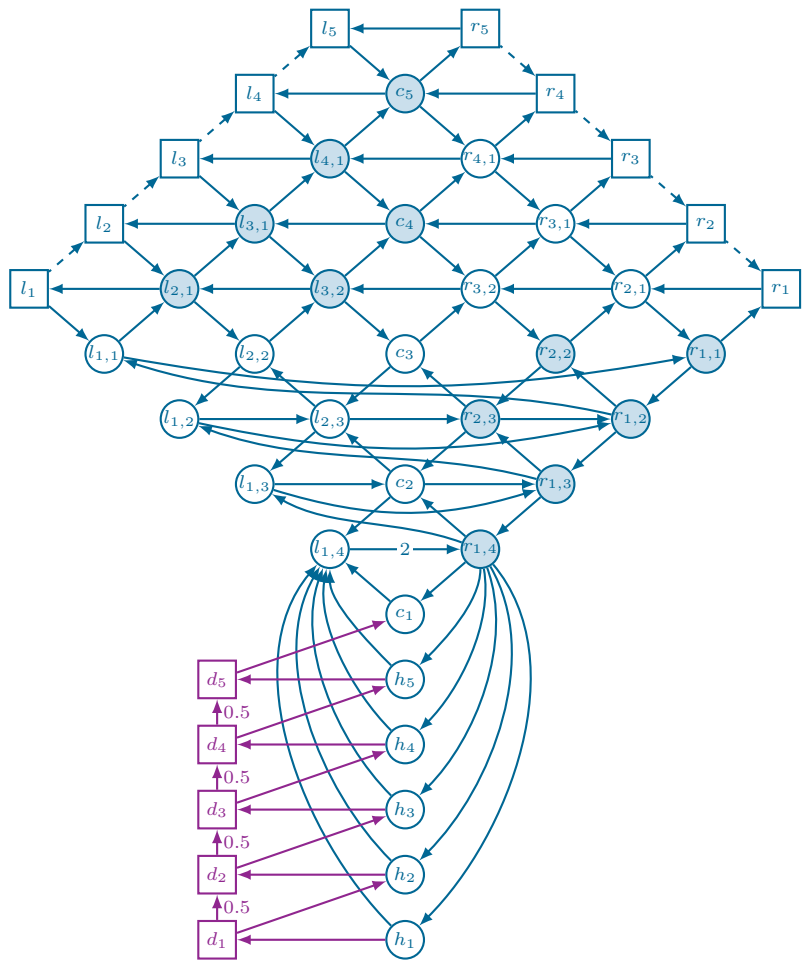

\subfile{fig-Q-12-dual.tex}
\caption{Quiver $Q_6^5$.}
\label{fig:Q-6}
\end{figure}
\null
\vfill

\clearpage

\null
\vfill
\begin{figure}[h]
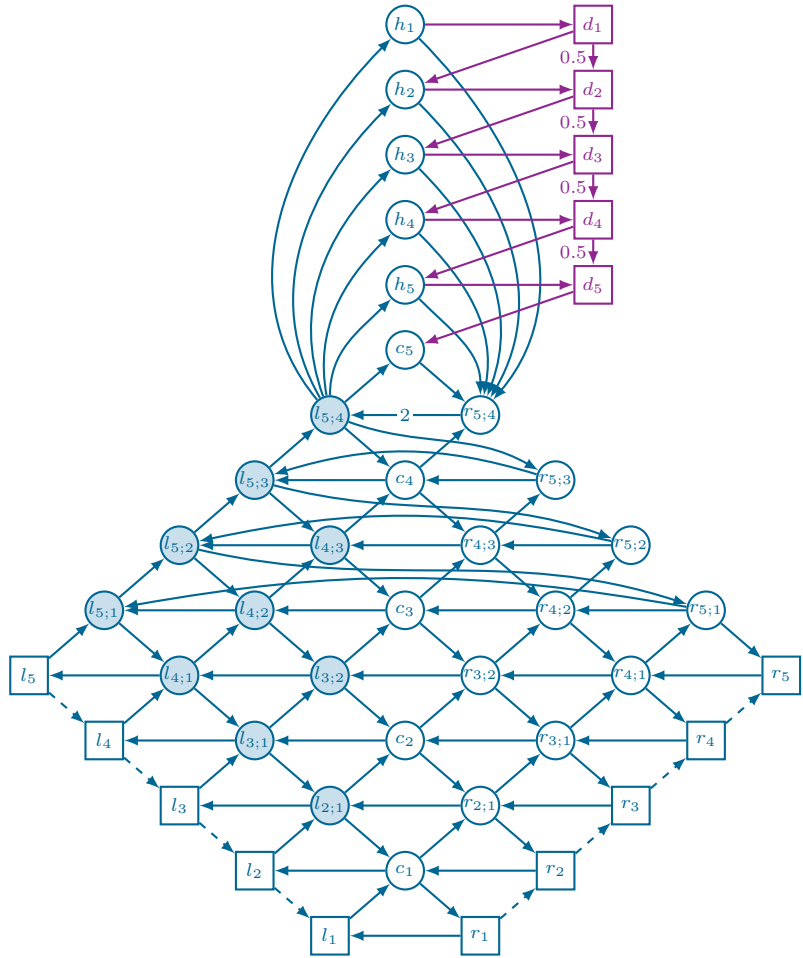

\subfile{fig-Q-11-back.tex}
\caption{Quiver $Q_7^5$.}
\label{fig:Q-7}
\end{figure}
\null
\vfill

\clearpage

\null
\vfill
\begin{figure}[h]
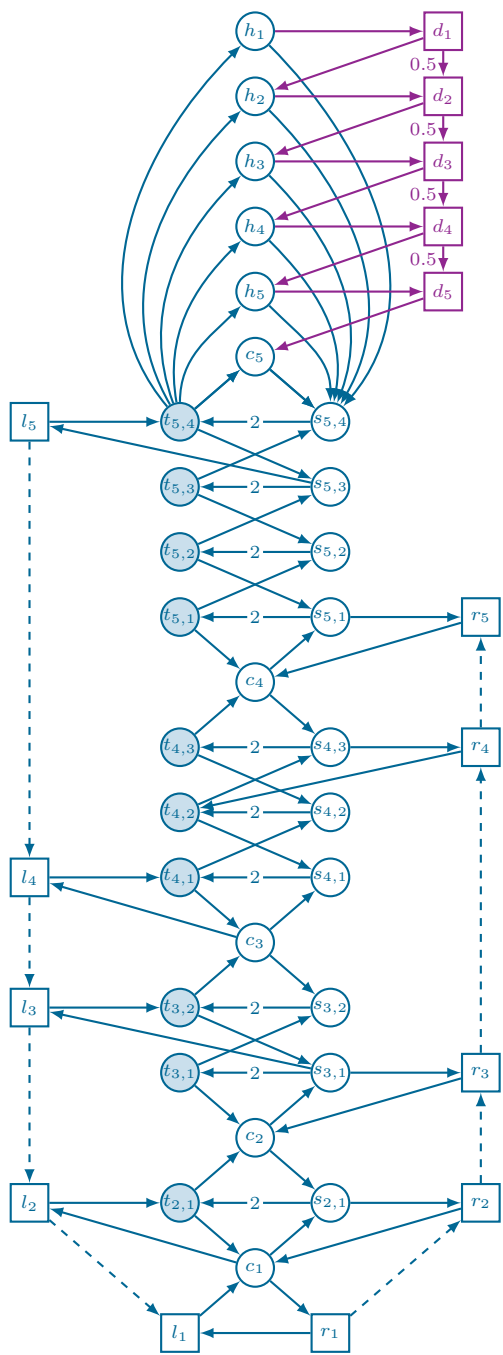

\subfile{fig-Q-10-dehn.tex}
\caption{Quiver $Q_8^5$.}
\label{fig:Q-8}
\end{figure}
\null
\vfill

\clearpage

\null
\vfill
\begin{figure}[h]
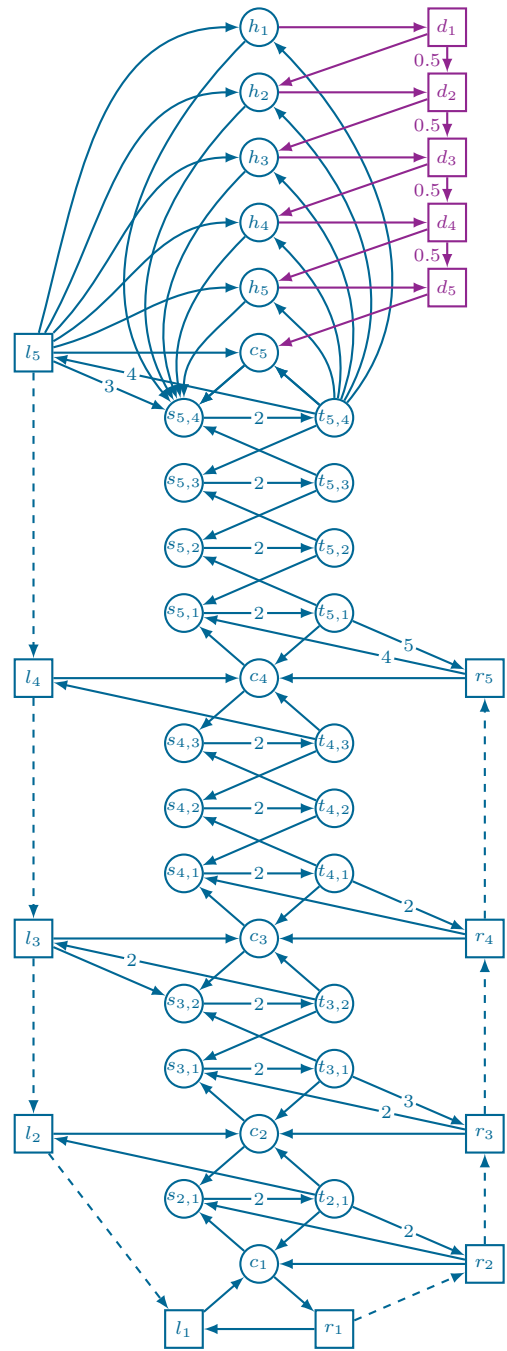

\subfile{fig-Q-9-baxter.tex}
\caption{Quiver $Q_9^5$.}
\label{fig:Q-9}
\end{figure}
\null
\vfill

\clearpage

\null
\vfill
\begin{figure}[h]
\subfile{fig-Q-8-chain.tex}
\caption{Quiver $Q_{10}^5$.}
\label{fig:Q-10}
\end{figure}
\null
\vfill

\clearpage

\null
\vfill
\begin{figure}[h]
\subfile{fig-Q-7-half-back.tex}
\caption{Quiver $Q_{11}^5$.}
\label{fig:Q-11}
\end{figure}
\null
\vfill

\clearpage

\null
\vfill
\begin{figure}[h]
\subfile{fig-Q-6-half-dehn.tex}
\caption{Quiver $Q_{12}^5$.}
\label{fig:Q-12}
\end{figure}
\null
\vfill

\clearpage

%\null
%\vfill
%\begin{figure}[h]
%\subfile{fig-Q-4-half-chain-extra.tex}
%\caption{Extra}
%\label{fig:Q-extra}
%\end{figure}
%\null
%\vfill
%
%\clearpage

\null
\vfill
\begin{figure}[h]
\subfile{fig-Q-4-half-chain.tex}
\caption{Quiver $Q_{13}^5$.}
\label{fig:Q-13}
\end{figure}
\null
\vfill

\clearpage

%And this, my friends, is one of those situations where we will suffer no mater what course of action we choose. Our goal is to commute an umbral move through the (algebraic) Whittaker transform. To that end we factor the umbral move into a very long sequence of transformations, which we can commute through. Let us start by inspecting quivers $Q^{n}_1$ to $Q^{n}_{15}$.
%Note that an umbral move can be realized via the quasi-cluster transformation \red{(numbering is wrong here)}
%$\Phi_{6 \nwarrow 1} \colon Q^{n}_1 \longra Q^{n}_6$.
%First, we will show that
%$$
%\Phi_{6 \nwarrow 1}= \Phi_{6 \swarrow 1}.
%$$
%To that end we show the following equality of tropical quasi-cluster transformation:
%$$
%\Phi_{15 \swarrow 9} = \Phi_{15 \nwarrow 9}.
%$$

%Let
%$$
%a_n(k) =
%\begin{cases}
%k-1 &\text{if $n-k$ is even,} \\
%1 &\text{if $n-k$ is odd.} \\
%\end{cases}
%$$
%Then the quiver $Q_9^n$ can be described as follows. It contains nodes $l_k, r_k, h_k, c_k, t_{k;j}, s_{k;j}$, where $1 \le k \le n$ and $1 \le j < k$. Nodes $s_{k;j},t_{k;j}$ form a $\sl_k$-Toda quiver, with $s_{k;j}$ being the sources of double arrows and $t_{k,j}$ being their targets. \blue{Blah-blah and the important bit.} We have
%\begin{align*}
%(e_{t_{k;1}},e_{r_k}) &= a_n(k)+1, & (e_{t_{k;k-1}},e_{l_k}) &= a_n(k), \\
%(e_{s_{k;1}},e_{r_k}) &= -a_n(k), & (e_{s_{k;k-1}},e_{l_k}) &= 1-a_n(k).
%\end{align*}

For the remainder of this section we write $\yrm_f$ for the tropical variable, corresponding to the node labelled $f$ in $Q_9$. Let us also introduce the following notations: for $j < k$ we set
$$
v_{k,j}^{(a,b)} = \yrm_{s_{k;j}}^a \yrm_{t_{k;j}}^b.
\qquad\text{and}\qquad
v_{k,j}^a = v_{k,j}^{(a,a)}
$$
We also declare
$$
v_{k+1,0} = v_{k,k} = \yrm_{c_k},
\qquad\text{and}\qquad
v_{0,0} = v_{1,0} =  t_{1;1} = s_{1;1} = 1.
$$
Then for $k_1 \le k_2$, $1 \le j_1<k_1$, $1 \le j_2<k_2$ we set
$$
v_{[k_1,j_1|k_2,j_2]} = \prod_{r=j_1}^{k_1} v_{k_1,r} \cdot \prod_{s=k_1+1}^{k_2-1} \prod_{r=1}^s v_{s,r} \cdot \prod_{r=1}^{j_2} v_{k_2,r}.
$$
Furthermore, we fix
$$
h_{n+1} = c_n \qquad\text{and}\qquad v_{n+1,j} = h_{n+1-j}/h_{n+2-j},
$$
so that
$$
h_{n+1-j} = v_{[n,n|n+1,j]}.
$$

\begin{lemma}
\label{lem:8-left-9}
Images $\yrm_\ell^{(8)} = \bs\mu^t_{8\,\circlearrowleft\,9}(\yrm_\ell)$ of the tropical variables $\yrm_\ell$ of $Q^n_9$ under the tropical cluster transformation $\bs\mu^t_{8\,\circlearrowleft\,9}$ may be written as follows. For all $k$ we have
$$
\yrm^{(8)}_{c_k} = \yrm_{c_k},
\qquad \yrm^{(8)}_{h_k} = \yrm_{h_k},
\qquad \yrm^{(8)}_{d_k} = \yrm_{d_k},
\qquad \yrm^{(8)}_{t_{k,j}} = v_{k,j}^{(-k,1-k)}, \\
\qquad \yrm^{(8)}_{s_{k,j}} = v_{k,j}^{(k+1,k)}.
$$
Then
$$
\yrm^{(8)}_{r_k} = \yrm_{r_k}
\qquad\text{and}\qquad
\yrm^{(8)}_{l_k} = \yrm_{l_k} \cdot
\begin{cases}
v_{k,k-1}^{(k,k-1)} &\text{if $n-k$ is even,} \\
\prod_{j=1}^{k-1} v_{k,j}^{(k+1-j,k-j)} &\text{if $n-k$ is odd.}
\end{cases}
$$
\end{lemma}

\begin{lemma}
\label{lem:7-left-9}
Images $\yrm_\ell^{(7)} = \bs\mu^t_{7\,\circlearrowleft\,9}(\yrm_\ell)$ of the tropical variables $\yrm_\ell$ of $Q^n_9$ under the tropical cluster transformation $\bs\mu^t_{7\,\circlearrowleft\,9}$ may be written as follows. For $ \le k \le n+1$ we have
$$
\yrm^{(7)}_{h_k} = \yrm_{h_k},
\qquad \yrm^{(7)}_{d_k} = \yrm_{d_k}.
$$
For $1 \le k \le n$ we have
$$
\yrm^{(7)}_{r_k} = \yrm_{r_k}
\qquad\text{and}\qquad
\yrm^{(7)}_{l_k} = \yrm_{l_k} \cdot \prod_{j=1}^{k-1} \hr{v_{k-1,j}^{j-1} v_{k,j}^{(k+a_n(k)-j,k+a_n(k)-1-j)}}.
$$
For $1 \le k \le n-1$ we have
$$
\yrm^{(7)}_{c_k} = v_{[k,k|k+1,k-1]},
\qquad \yrm^{(7)}_{l_{k+1,k}} = v_{k+1,k}^{(-a_n(k+1)-1,-a_n(k+1))},
\qquad \yrm^{(7)}_{r_{n,k}} = v_{n,k}^{(n+1,n)}.
$$
For $2 \le k \le n-1$, $1 \le j \le k-1$ we have
\begin{align*}
\yrm^{(7)}_{l_{k+1,j}} &= v_{[k,j+1|k+1,j-1]}^{-1} v_{k+1,j}^{(-a_n(k)-2,-a_n(k)-1)}, \\
\yrm^{(7)}_{r_{k,j}} &= v_{k,j}^{(a_n(k)+2,a_n(k)+1)} v_{[k,j+1|k+1,j-1]}^{-1}.
\end{align*}
\end{lemma}

\begin{lemma}
\label{lem:6-left-9}
Images $\yrm_\ell^{(6)} = \bs\mu^t_{6\,\circlearrowleft\,9}(\yrm_\ell)$ of the tropical variables $\yrm_\ell$ of $Q^n_9$ under the tropical cluster transformation $\bs\mu^t_{6\,\circlearrowleft\,9}$ may be written as follows. For $1 \le k \le n$ we have
$$
\yrm^{(6)}_{h_k} = \yrm_{h_k},
\qquad \yrm^{(6)}_{d_k} = \yrm_{d_k}.
$$
For $1 \le k \le n$ we have
$$
\yrm^{(6)}_{r_k} = \yrm_{r_k} \cdot v_{[k,1|k,k]}
\qquad\text{and}\qquad
\yrm^{(6)}_{l_k} = \yrm_{l_k} \cdot \prod_{j=1}^{k-1} \hr{v_{k-1,j}^j v_{k,j}^{(k+a_n(k)-j,k+a_n(k)-1-j)}}.
$$
Now,
$$
\yrm^{(6)}_{c_k} = 
\begin{cases}
\yrm_{c_n} &\text{if $k=1$,} \\
v_{[n+2-k,0|n+2-k,n-k]} &\text{if $1 < k \le (n+1)/2$,} \\
v_{[k-1,n-k|k,n-k-1]}^{-1} &\text{if $k > (n+1)/2$.}
\end{cases}
$$
Then
$$
\yrm^{(6)}_{l_{k,j}} = 
\begin{cases}
v_{[k-1,j|k,j-1]}^{-1} v_{k,j}^{(-a_n(k)-1,-a_n(k))} &\text{if $1 \le j < k$,} \\
v_{[1,1|n,1]} \prod_{r=2}^n v_{r,1}^{(a_n(r)+2,a_n(r)+1)} &\text{if $j=k=1$,} \\
v_{[k,k|n+2-k,k-1]} \prod_{r=k+1}^{n+1-k} v_{r,k}^{(a_n(r)+2,a_n(r)+1)} &\text{if $j=k > 1$,} \\
v_{n,k}^{(n+1,n)} &\text{if $j > k = 1$,} \\
v_{n+1-k,j}^{(a_n(n+1-k)+2,a_n(n+1-k)+1)} v_{[n+1-k,j+1|n+2-k,j-1]} &\text{if $2 \le k < j$.}
\end{cases}
$$
Finally,
$$
\yrm^{(6)}_{r_{k,j}} = 
\begin{cases}
v_{k,j}^{(a_n(k)+1,a_n(k))}v_{[k,j+1|k+1,j]} &\text{if $n-k < j < k$,} \\
v_{k,j}^{(a_n(k)+1,a_n(k))} &\text{if $j = n-k < k$,} \\
v_{[k+1,k+1|n+1-k,k-1]}^{-1} \prod_{r=k+1}^{n+1-k} v_{r,k}^{(-a_n(k)-2,a_n(k)-1)} &\text{if $j=k$,} \\
v_{n+1-k,j}^{(-a_n(n+1-k)-1,-a_n(n+1-k))} &\text{if $j = n-k > k$,} \\
v_{[n-k,j+1|n+1-k,j-1]}^{-1} v_{n+1-k,j}^{(-a_n(n+1-k)-2,-a_n(n+1-k)-1)} &\text{if $k < j < n-k$.} \\
\end{cases}
$$
\end{lemma}

\begin{lemma}
\label{lem:5-left-9}
Images $\yrm_\ell^{(5)} = \bs\mu^t_{5\,\circlearrowleft\,9}(\yrm_\ell)$ of the tropical variables $\yrm_\ell$ of $Q^n_9$ under the tropical cluster transformation $\bs\mu^t_{5\,\circlearrowleft\,9}$ may be written as follows. 
%For all $1 \le k \le n$ we have $\yrm'_{l_k}$, $\yrm'_{r_k}$ as in theprevious Lemma. Same applies to $\yrm'_{l_{k,j}}$, $\yrm'_{r_{k,j}}$ for $j<k$ and to $\yrm'_{c_k}$ for $k > (n+1)/2$.
We have
$$
\yrm^{(5)}_{l_{k,k}} = v_{[k,k|n+2-k,k-1]} \prod_{r=k+1}^{n+1-k} v_{r,k}^{(a_n(r)+2,a_n(r)+1)},
$$
$$
\yrm^{(5)}_{c_1} = v_{[n+1,0|n+1,n-1]},
$$
and
$$
\yrm^{(5)}_{h_k} = 
\begin{cases}
\yrm_{h_n} &\text{if $k=n$,} \\
v_{[n,n+2-k|n+1,n-k]}^{-1} &\text{if $k<n$.} \\
\end{cases}
$$
For all other nodes $f$ of $Q^n_5$ the tropical variables are expressed by the same exact formulas as in the previous Lemma: $\yrm^{(5)}_f = \yrm^{(6)}_f$.
\end{lemma}

\begin{lemma}
\label{lem:4-left-9}
Images $\yrm_\ell^{(4)} = \bs\mu^t_{4\,\circlearrowleft\,9}(\yrm_\ell)$ of the tropical variables $\yrm_\ell$ of $Q^n_9$ under the tropical cluster transformation $\bs\mu^t_{4\,\circlearrowleft\,9}$ may be written as follows. For $1 \le k \le n$ we have
$$
\yrm^{(4)}_{r_k} = \yrm_{r_k} \cdot v_{[k,1|k,k]},
\qquad
\yrm^{(4)}_{l_k} = \yrm_{l_k} \cdot \prod_{j=1}^{k-1}\hr{v_{k-1,j}^j v_{k,j}^{(n+a_n(k)-j,n+a_n(k)-1-j)}}.
$$
Then
$$
\yrm^{(4)}_{c_k} =
\begin{cases}
v_{[k,k|k+1,k]} &\text{if $k \le n$,} \\
v_{[n,2n+1-k|n+1,2n-k]}^{-1} &\text{if $k > n$}
\end{cases}
$$
and
$$
\yrm^{(4)}_{d_k} = \yrm_{d_k}.
$$
Finally for all $1 \le j <k \le n$ we have
\begin{align*}
\yrm^{(4)}_{l_{k,j}} &= v_{[k-1,j|k,j-1]}^{-1} v_{k,j}^{(-a_n(k)-1,-a_n(k))}, \\
\yrm^{(4)}_{r_{k,j}} &= v_{k,j}^{(a_n(k)+1,a_n(k))} v_{[k,j+1|k+1,j]}^{-1}.
\end{align*}
\end{lemma}

\begin{lemma}
\label{lem:3-left-9}
Images $\yrm_\ell^{(3)} = \bs\mu^t_{3\,\circlearrowleft\,9}(\yrm_\ell)$ of the tropical variables $\yrm_\ell$ of $Q^n_9$ under the tropical cluster transformation $\bs\mu^t_{3\,\circlearrowleft\,9}$ may be written as follows. For $1 \le k \le n$ we have
$$
\yrm^{(3)}_{d_{n+1-k}} = \yrm_{d_{n+1-k}} \cdot v_{[k,k|n+1,k]} \prod_{j=k+1}^{n} v_{j,k}^{(-a_n(j)-1,-a_n(j))}.
$$
The remainder of the formulas coincide with those in the previous Lemma.
\end{lemma}

\begin{lemma}
\label{lem:2-left-9}
Images $\yrm_\ell^{(2)} = \bs\mu^t_{2\,\circlearrowleft\,9}(\yrm_\ell)$ of the tropical variables $\yrm_\ell$ of $Q^n_9$ under the tropical cluster transformation $\bs\mu^t_{2\,\circlearrowleft\,9}$ may be written as follows. For $1 \le k \le n$ we have
$$
\yrm^{(2)}_{d_k} = \yrm_{d_k},
\qquad
\yrm^{(2)}_{l_k} = \yrm_{l_k} \cdot v_{[k,0|k,k-1]},
\qquad
\yrm^{(2)}_{r_k} = \yrm_{r_k} \cdot \prod_{j=1}^{k-1} v_{k,j}^{(a_n(k)+2,a_n(k)+1)} \prod_{j=0}^{k-1} v_{k+1,j}^{k+1-j}.
$$
Then
$$
\yrm^{(2)}_{c_k} =
\begin{cases}
v_{[k,1|k+1,1]}^{-1} &\text{if $k \le (n+1)/2$,} \\
v_{[k,2k-n-1|k+1,2k-n-1]} &\text{if $k > (n+1)/2$.}
\end{cases}
$$
and
$$
\yrm^{(2)}_{h_k} =
\begin{cases}
v_{[n,k|n+1,k-1]}^{-1} &\text{if $k<n$,} \\
v_{[n,n-1|n+1,n]}^{-1} &\text{if $k=n$.}
\end{cases}
$$
Finally, we have
$$
\yrm^{(2)}_{l_{k,j}} =
\begin{cases}
v_{[k-1,k-j-1|k,k-j-1]} \cdot v_{k,k-j}^{(a_n(k),a_n(k)-1)} &\text{if} \;\; j<\min(k-1,n+1-k), \\
v_{k,1}^{(a_n(k),a_n(k)-1)} &\text{if} \;\; j = k-1 \le n+1-k, \\
\prod_{r=1}^{k-j-1} v_{[k-1,r|k,r]} \cdot v_{[k,1|k,k-j]}^{(a_n(k),a_n(k)-1)} &\text{if} \;\; j = n+1-k < k-1, \\
v_{k,j+k-n-1}^{(-a_n(k),1-a_n(k))} &\text{if} \;\; j = k-1 > n+1-k, \\
v_{[k-1,j+k-n-1|k,j+k-n-2]}^{-1} \cdot v_{k,k+j-n-1}^{(-a_n(k)-1,-a_n(k))} &\text{if} \;\; n+1-k < j < k-1,
\end{cases}
$$
and
$$
\yrm^{(2)}_{r_{k,j}} =
\begin{cases}
v_{k,k-j}^{(-a_n(k),1-a_n(k))} \cdot v_{[k,k-j+1|k+1,k-j+1]}^{-1} &\text{if} \;\; j < n+1-k, \\
v_{[k,1|k,k-j]}^{(-a_n(k),1-a_n(k))} \cdot \prod_{r=1}^{k-j+1} v_{[k,r|k+1,r]}^{-1} &\text{if} \;\; j = n+1-k, \\
v_{k,k+j-n-1}^{(a_n(k)+1,a_n(k))} \cdot v_{[k,k+j-n|k+1,k+j-n-1]} &\text{if} \;\; j > n+1-k.
\end{cases}
$$
\end{lemma}

\begin{lemma}
\label{lem:1-left-9}
Images $\yrm_\ell^{(1)} = \bs\mu^t_{1\,\circlearrowleft\,9}(\yrm_\ell)$ of the tropical variables $\yrm_\ell$ of $Q^n_9$ under the tropical cluster transformation $\bs\mu^t_{1\,\circlearrowleft\,9}$ may be written as follows. For $1 \le k \le n$ we have
$$
\yrm^{(1)}_{h_k} = v_{n,n-1} \cdot \yrm_{h_k}.
$$
We also have
$$
\yrm^{(1)}_{c_n} = v_{[n,n-1|n,n]}
\qquad\text{and}\qquad
\yrm^{(1)}_{r_{n,1}} = v_{[n,1|n,n-1]}^{(-n,1-n)} v_{n,n}^{-1} \cdot \prod_{r=1}^{n} v_{[n,r|n+1,r]}^{-1}.
$$
Finally, for $2 \le k \le n-1$ we have
$$
\yrm^{(1)}_{r_{n,k}} = v_{n,k-1}^{(n,n-1)}.
$$
The remainder of the formulas coincide with those in the previous Lemma.
\end{lemma}

\begin{lemma}
\label{lem:13-left-9}
Images $\yrm_\ell^{(13)} = \bs\mu^t_{13\,\circlearrowleft\,9}(\yrm_\ell)$ of the tropical variables $\yrm_\ell$ of $Q^n_9$ under the tropical cluster transformation $\bs\mu^t_{13\,\circlearrowleft\,9}$ may be written as follows. Let
$$
u_{n,k} = v_{[k,2k-n|k,k]}^{n-2k-1} \cdot \prod_{r=1}^{2k-n} v_{k+1,r}^{r+n-2k-1}.
$$
Then for $k > (n+2)/2$ we have
\begin{align*}
\yrm^{(13)}_{l_{k,n+1-k}} &= 
\begin{cases}
v_{[k-1,2k-n-2|k,2k-n-2]} \cdot v_{[k,2k-n-2|k,2k-n-1]}^{(k-1,k-2)}, &\text{if $n-k$ is even,} \\
v_{[k-1,2k-n-2|k,0]} \cdot \prod_{r=1}^{k-2} v_{k,r}^{(k+a_n(k)-r-2,k+a_n(k)-r-3)}, &\text{if $n-k$ is odd.}
\end{cases} \\
\yrm^{(13)}_{r_{k,n+1-k}} &= 
\begin{cases}
\prod_{r=1}^{2k-n-1} v_{k,r}^{(-r-k,-r-k+1)} \cdot u_{n,k}, &\text{if $n-k$ is even,} \\
\prod_{r=2k-n-2}^{2k-n-1} v_{k,r}^{(-r-2,-r-1)} \cdot u_{n,k}, &\text{if $n-k$ is odd,}
\end{cases}
\end{align*}
while for $k = (n+2)/2$
\begin{align*}
\yrm^{(13)}_{r_{k,n+1-k}} &= v_{k,1}^{(-a_n(k)-2,-a_n(k)-1)} \cdot u_{n,k}, \\
\yrm^{(13)}_{l_{k,n+1-k}} &=  v_{k,1}^{(a_n(k),a_n(k)-1)}.
\end{align*}
Then for $k > (n+1)/2$ we have
$$
\yrm^{(13)}_{c_k} = v_{[k,2k-n-1|k,k]}.
$$
Finally, for $1 \le k < n/2$ we have
\begin{align*}
\yrm^{(13)}_{s_{n+1-2k,j}} &=
\begin{cases}
v_{n+1-k,j}^{(n+1-k,n-k)} &\text{if $k$ is odd,} \\
v_{n+1-k,j}^{(n-k,n-k-1)} &\text{if $k$ is even,}
\end{cases} \\
\yrm^{(13)}_{t_{n+1-2k,j}} &=
\begin{cases}
v_{n+1-k,j}^{(k-n,k-n+1)} &\text{if $k$ is odd,} \\
v_{n+1-k,j}^{(k-n+1,k-n+2)} &\text{if $k$ is even.}
\end{cases} \\
\end{align*}
The remainder of the formulas coincide with those in the previous Lemma.
\end{lemma}

\begin{lemma}
\label{lem:10-right-9}
Images $\yrm_\ell^{(10)} = \bs\mu^t_{10\,\circlearrowright\,9}(\yrm_\ell)$ of the tropical variables $\yrm_\ell$ of $Q^n_9$ under the tropical cluster transformation $\bs\mu^t_{10\,\circlearrowleft\,9}$ may be written as follows. For $1 \le k \le n$ we have
$$
\yrm^{(10)}_{s_{k;j}} = \yrm_{s_{k;k-j}},
\qquad \yrm^{(10)}_{t_{k;j}} = \yrm_{t_{k;k-j}},
\qquad \yrm^{(10)}_{c_k} = v_{[k,1|k+1,k]}^{-1},
\qquad \yrm^{(10)}_{h_k} = v_{[n,1|n+1,k-1]}^{-1},
$$
and
$$
\yrm^{(10)}_{d_k} = \yrm_{d_k}.
$$
Then
\begin{align*}
\yrm^{(10)}_{l_k} &= 
\yrm_{l_k} \cdot \prod_{j=1}^{k-1} v_{[k-1,j|k,1]} \cdot
\begin{cases}
v_{[k,2|k,k-1]}^{k-1} &\text{if $n-k$ is even,} \\
\prod_{j=2}^{k-1} v_{k,j}^{(k+1-j,k-j)} &\text{if $n-k$ is odd,} \\
\end{cases} \\
\yrm^{(10)}_{r_k} &= 
\yrm_{r_k} \cdot \prod_{j=0}^{k} v_{k+1,j}^{k+1-j} \cdot
\begin{cases}
\yrm_{t_{k;1}}^{-1} \cdot v_{[k,1|k,k-1]}^{k+1} &\text{if $n-k$ is even,} \\
\prod_{j=1}^{k-1} v_{k,j}^{(j+2,j+1)} &\text{if $n-k$ is odd.} \\
\end{cases} \\
\end{align*}
\end{lemma}

\begin{lemma}
\label{lem:11-right-9}
Images $\yrm_\ell^{(11)} = \bs\mu^t_{11\,\circlearrowright\,9}(\yrm_\ell)$ of the tropical variables $\yrm_\ell$ of $Q^n_9$ under the tropical cluster transformation $\bs\mu^t_{11\,\circlearrowleft\,9}$ may be written as follows. First, we have
$$
\yrm^{(11)}_{l_k} = \yrm_{l_k} \cdot \yrm_{[k,0|k,k-1]}, \qquad
\yrm^{(11)}_{r_k} = \yrm_{r_k} \cdot \prod_{j=1}^{k-1} \hr{\yrm_{s_{k;j}}^{a_n(k)+j+1} \yrm_{t_{k;j}}^{a_n(k)+j}} \cdot \prod_{j=0}^k v_{k+1,j}^{k+1-j}.
$$
Then,
$$
\yrm^{(11)}_{c_k} = v_{[k,1|k+1,1]}^{-1} \qquad\text{for}\qquad 1 \le k \le n-1
$$
and
$$
\yrm^{(11)}_{h_k} = v_{[n,1|n+1,k-1]}^{-1} \qquad\text{for}\qquad 1 \le k \le n+1.
$$
For $2 \le k \le n$ and $1 \le j \le k-1$ let us set
\begin{align*}
a_{k,j} &= 
\begin{cases}
v_{[k-1,k-j-1|k,k-j-1]} &\text{if} \quad j<k-1, \\
1 &\text{if} \quad j = k-1.
\end{cases} \\
b_{k,j} &= 
\begin{cases}
\yrm_{s_{k,k-j}}^{a_n(k)} \yrm_{t_{k,k-j}}^{a_n(k)-1} &\text{if} \quad k<n, \\
\yrm_{s_{n,n-j}}^{n-j} \yrm_{t_{n,n-j}}^{n-j-1} &\text{if} \quad k = n.
\end{cases}
\end{align*}
Then
$$
\yrm^{(11)}_{l_{k,j}} = a_{k,j} \cdot b_{k,j}, \qquad
\yrm^{(11)}_{r_{k,j}} = b^{-1}_{k,j} \cdot
\begin{cases}
v_{[k,k-j+1|k+1,k-j+1]}^{-1} &\text{if} \quad k < n, \\
v_{n,n-j}, &\text{if} \quad k = n.
\end{cases}
$$
Finally. for $1 \le k \le n$ we have
$$
\yrm^{(11)}_{d_k} = \yrm_{d_k}.
$$
\end{lemma}

\begin{lemma}
\label{lem:12-right-9}
Images $\yrm_\ell^{(12)} = \bs\mu^t_{12\,\circlearrowright\,9}(\yrm_\ell)$ of the tropical variables $\yrm_\ell$ of $Q^n_9$ under the tropical cluster transformation $\bs\mu^t_{12\,\circlearrowleft\,9}$ may be written as follows. Setting $l_{k;0} = l_k$, $r_{k;0} = r_k$ we have that variables $\yrm^{(12)}_{l_{k,j}}$, $\yrm^{(12)}_{r_{k,j}}$ are the same as in the previous Lemma if $k+j \le n$. Same applies to $\yrm^{(12)}_{c_k}$ if $1 \le k \le (n+1)/2$ and to $\yrm^{(12)}_{h_k}$ for all $k$. Then for $k$ such that $k > (n+1)/2$ we have
$$
\yrm^{(12)}_{c_k} = v_{[k,1|k+1,2k-n]}^{-1}.
$$
Furthermore, for $n/2+1<k \le n$ we have
$$
\yrm^{(12)}_{l_{k,n+1-k}} = 
\begin{cases}
\prod_{j=1}^{2k-2-n} v_{[k-1,j|k,2k-2-n]} \cdot v_{k,2k-1-n}^{(k-1,k-2)} &\text{if $n-k$ is even,} \\
\prod_{j=1}^{2k-2-n} v_{[k-1,j|k,j]} \cdot v_{[k,2|k,2k-1-n]}^{(1,0)} &\text{if $n-k$ is odd.} \\
\end{cases}
$$
and for $k=n/2+1$ we have
$$
\yrm^{(12)}_{l_{k,n+1-k}} = v_{k,1}^{(a_n(k),a_n(k)-1)}.
$$
Then
$$
\yrm^{(12)}_{r_{k,n+1-k}} = 
\begin{cases}
\prod_{j=1}^{2k-2-n} v_{k,j+1}^{(k-n-j,k+1-n-j)} &\text{if $n-k$ is even,} \\
v_{k,2k-n-1}^{(0,1)} &\text{if $n-k$ is odd.} \\
\end{cases}
$$
%Finally, for $(n+3)/2 \le k \le n$ and $1 \le j \le 2k-n-2$ we have
%\begin{align*}
%\yrm'_{s_{k,j}} &=
%\begin{cases}
%v_{k,2k-n-1-j}^{(n-k+1,n-k)} &\text{if $n-k$ is even,} \\
%v_{k,2k-n-1-j}^{(1,0)} &\text{if $n-k$ is odd.}
%\end{cases} \\
%\yrm'_{t_{k,j}} &=
%\begin{cases}
%v_{k,2k-n-1-j}^{(k-n,k-n+1)} &\text{if $n-k$ is even,} \\
%v_{k,2k-n-1-j}^{(0,1)} &\text{if $n-k$ is odd.}
%\end{cases}
%\end{align*}
For $1 \le k \le (n-1)/2$ and $1 \le j \le n-2k$ we have
\begin{align*}
\yrm^{(12)}_{s_{n+1-2k,j}} &=
\begin{cases}
v_{n+1-k,n+1-2k-j}^{(k,k-1)} &\text{if $k$ is odd,} \\
v_{n+1-k,n+1-2k-j}^{(1,0)} &\text{if $k$ is even.}
\end{cases} \\
\yrm^{(12)}_{t_{n+1-2k,j}} &=
\begin{cases}
v_{n+1-k,n+1-2k-j}^{(1-k,2-k)} &\text{if $k$ is odd,} \\
v_{n+1-k,n+1-2k-j}^{(0,1)} &\text{if $k$ is even.}
\end{cases} \\
\end{align*}
Finally. for $1 \le k \le n$ we have
$$
\yrm^{(12)}_{d_k} = \yrm_{d_k}.
$$
\end{lemma}

\begin{lemma}
\label{lem:13-right-9}
We have the equality $\bs\mu^t_{13\,\circlearrowright\,9}(\yrm_\ell) = \bs\mu^t_{13\,\circlearrowleft\,9}(\yrm_\ell)$ for all tropical variables $\yrm_\ell$ of $Q^n_9$.
\end{lemma}

\section{Gluing for Laurent rings}
\label{sec:alg-gluing}
In this section we describe how the quantum universal Laurent rings  associated to $\Pc^\diamond_{G,S}$ behave under the gluing of surfaces along tacked circles. In order to do this, we need to define the algebra $\Lbb_{S';\phi}$ associated to the gluing data $\phi \colon c_+\simeq c_-$. This we do in several steps as follows. We first associate to each cluster torus for $\Pc^\diamond_{G,S'}$ a symplectically reduced quantum torus and universal Laurent ring defined in Section~\ref{subsec:symplectic-reduction}. Next, we localize each of these tori at the Ore denominator set described in~\eqref{eq:denom-set}, which commutes with all non-frozen cluster variables and is thus mutation-invariant. Finally, we impose conditions analogous to \eqref{eq:residue-condition} on the poles and residues of elements of these localized quantum tori, and take invariants for an action of the Weyl group associated to the pair of tacked circles $c_\pm$. The resulting residue universal Laurent ring is described precisely in Definition~\ref{def:Lres}, and in Theorem~\ref{thm:alg-MF} we prove that it is $\Gamma_{S;c}$-equivariantly isomorphic to $\Lbb_S$.

\subsection{Symplectic reduction at a pair of tacked circles}
\label{subsec:symplectic-reduction}
Suppose that $S'$ is a marked surface with a pair of tacked circles $c_\pm$, and denote by $S$ the surface obtained by gluing these two tacked circles together as is described in Section~\ref{subsec:cut-n-glue}. 
% \blue{Should we phrase the story here in terms of the basis elements $e_{c_\pm,i}$ and $f_{c_\pm,i}$ from the definition of the compatible pair?}
Recall that for each cluster chart $Q$ for $\Pc^\diamond_{SL_{n+1},S'}$  we have lattices $\Lambda_{S',Q}\subset\Xi_{S',Q} $ with $2n$ frozen $e$-basis elements $e_{c_\pm,i}\in \Lambda_{S',Q},~1\leq i\leq n$, as well as the $2n$ frozen $\xi$-basis elements $\xi_{c_\pm,i}\in\Xi_{S',Q} ,~ 1\leq i\leq n$. 

We now form a pair of lattices $\Lambda_{Q}^c\subset\Xi_Q^c$ 
%which we will use to understand the surface $S$ obtained by gluing these two tacked circles together as explained in Section~\ref{subsec:cut-n-glue}. The lattice $\Lambda_{\bi}^c$ and its skew form are obtained from $\Lambda_{\bi}$ 
as ``symplectic reductions'' of $\Lambda_{S',Q}$ and $\Xi_{S',Q}$. We first consider the sublattice $\Xi^\omega_{Q}\subset\Xi_{Q}$ given by the orthogonal complement to the span of the elements $\xi_{c_+,i}+\xi_{c_-,i}$,
\begin{align}
\label{eq:quotient-gens}
\Xi^\omega_{Q} =\left\{\xi\in \Xi_{S',Q}~\big|~ (\xi,\xi_{c_+,i}+\xi_{c_-,i})=0,\quad 1\leq i \leq n\right\},
\end{align}
and then take the quotient 
\begin{align}
\label{eq:reduced-lattice}
\Xi_{Q}^c:= \frac{\Xi^\omega_{Q}}{\ha{\xi_{c_+,i}+\xi_{c_-,i}}_{i=1}^{n}}.
\end{align}
% \red{[Confusion about signs here; probably coming from use of outdated conventions for frozens at punctures in network section.]}\blue{Actually it is probably accurate; }
We write $\xi_{c,i}$ for the image of $\xi_{c_+,i}$ in $\Xi_{Q}^{c}$. The reduced lattice $\Xi_{Q}^c$ inherits a well-defined skew form from that of $\Xi_Q$, the  kernel of which is canonically identified with that of $\Xi_{Q}$.  We note that the elements~\eqref{eq:quotient-gens} are skew-orthogonal to all mutable basis vectors $e_k$: indeed, the sublattice in $\Lambda_Q$ spanned by all $e_l$ except for the frozen basis vectors $e_{c_\pm,i}$ associated to the two tacked circles coincides with the based lattice $\Lambda_{S_\circ'}$ associated by Fock and Goncharov~\cite{FG09a} to the surface $S_\circ'$ where each $c_\pm$ is replaced by a puncture, while the kernel of the skew form on $\Lambda_{S_\circ'}$ contains all $2n$ elements  $\xi_{c_\pm,i}$. The lattices $\Lambda_{Q}^{\omega}$ and $\Lambda_{Q}^{c}$ are constructed similarly from $\Lambda_Q$, replacing the $\xi_{c_\pm,i}$ by the elements $\alpha_{c_\pm,i}$ from~\eqref{rmk:root-cas} corresponding to the simple roots of $G$.

Write $\mathcal{T}_{\Xi_Q}^\omega$ for the quantum torus associated to $\Xi_{Q}^\omega$ and $\mathcal{T}^c_{\Xi_Q}$ for that associated to $\Xi_{Q}^c$, so that $\mathcal{T}_{Q}^c$ is isomorphic to the quotient of $\mathcal{T}_{Q}^\omega$ by the 2-sided ideal $\mathcal{I}_Q\subseteq \mathcal{T}_{\Xi_Q}^\omega$ generated by the $n$ elements $1-Y_{\xi_{c_+,i}+\xi_{c_-,i}}$, which lie in the center of $\mathcal{T}_{\Xi_Q}^\omega$. Note that we can also identify $\mathcal{T}_{\Xi_Q}^\omega$ with the centralizer in $\mathcal{T}_\Xi$ of the subalgebra generated by the $n$ elements $Y_{\xi_{c_+,i}+\xi_{c_-,i}}$. 

Let $\Lbb_Q^\omega$ denote the intersection of the universal Laurent ring $\Lbb^\Ac_{Q}\subset \mathcal{T}_{\Xi_Q}$ with $\mathcal{T}_{\Xi_Q}^\omega$. By Theorem~\ref{thm:1-step}, we have
\begin{align}
\label{eq:pre-upper-bound}
   \Lbb_Q^\omega = \bigcap_{k=1}^d\left(\mathcal{T}_{\Xi_Q}^\omega\cap\mathcal{T}_{\Xi_{\mu_k(Q)}}^\omega\right). 
\end{align}
We define the reduced Laurent ring based in cluster $Q$ by
\begin{align}
    \label{eq:reduced-laurent-ring}
    \Lbb_{\Xi_Q}^{c} = \Lbb_Q^\omega/(\Lbb_Q^\omega\cap\mathcal{I}_Q)\subseteq \mathcal{T}_{\Xi_Q}^\omega/\mathcal{I}_Q\simeq \mathcal{T}_{\Xi_Q}^c.
\end{align}
We now show this definition is compatible with mutations, and formulate an easily verifiable criterion for membership in $\mathbb{L}^c_Q$. Since the linear combinations of frozen $\mathcal{A}$-variables $\xi_{c_+,i}+\xi_{c_-,i}$ are  skew-orthogonal to all mutable variables, no cluster $\mathcal{X}$-variable $e_k$ can ever be contained in their span. Hence we still have well-defined partial completions $\mathcal{T}^c_{\Xi_Q}((Y_{-e_k}))$ of the quantum torus associated to the quotient lattice, which have the property that the kernels of the natural maps $\mathcal{T}^\omega_{\Xi_Q}\rightarrow\mathcal{T}^c_{\Xi_Q}((Y_{-e_k}))\leftarrow \mathcal{T}^\omega_{\Xi_Q'}$ are precisely $\mathcal{I}_Q,\mathcal{I}_{Q'}$ respectively.
%%%%%%%%%%
%% since they factor as composite of a projection (to $\mathcal{T}_\bi^c$) and an inclusion (of the latter ring into $\mathcal{T}_\bi^c((Y_k))$).
%%%%%%%%%%
%%%%%%%%%%
%%%%%%%
%%%%
% A helpful commutative diagram we may want to include, where $\mu_k^\sharp$ stands for $\mathrm{Ad}(\Psi(Y_{-e_k}))$:
% \[
% \begin{tikzcd}
% \Lbb^\omega_\bi \arrow[r, "\mu_k^\sharp"] \arrow[d] & \Lbb^\omega_{\mu_k(\bi)} \arrow[d] \\
% \mathcal{T}_\bi^c((Y_{-e_k})) \arrow[r, "\mu_k^\sharp"'] & \mathcal{T}_\bi^c((Y_{-e_k}))),
% \end{tikzcd}
% \]
% where the vertical maps have kernels $\mathcal{I}_\bi\cap\Lbb^\omega_\bi,\mathcal{I}_{\bi'}\cap\Lbb^\omega_{\mu_k(\bi)}$ respectively.
%%%%%
%%%%%%%%%
%%%%%%%%%%%%%%%%
\begin{lemma}
\label{lem:reduced-laurent}
If $Q'=\mu_k(Q)$, the quantum mutation map $\mu_k^q$ induces an isomorphism of reduced Laurent rings $\Lbb_{\Xi_Q}^{c}\simeq \Lbb_{{\Xi_Q}}^{c}$.
\end{lemma}
\begin{proof}
    The key point is that, as mentioned above, the generators $1-Y_{\xi_{c_+,i}+\xi_{c_-,i}}$ of $\mathcal{I}_Q$ commute with all unfrozen cluster $\mathcal{X}$-variables $Y_{e_k}$ in any cluster $Q$. But the tori $\mathcal{T}_{\Xi_Q}^\omega,\mathcal{T}_{\Xi_{Q'}}^\omega$ are characterized as the centralizers of the elements $Y_{\xi_{c_+,i}+\xi_{c_-,i}}$. Since the quantum mutation map $\mu^q_k$ from~\eqref{eq:def-Xmutation} is induced by conjugation by $\Psi^q(Y_{-e_k})$, it follows that $\mu_k^q$ maps $\Lbb^\omega_Q$ into $\Lbb^\omega_{Q'}$. For the same reason, the ideal generators $1-Y_{\xi_{c_+,i}+\xi_{c_-,i}}$ are invariant under $\mu^q_k$, and it follows easily from the characterization of $\mathcal{I}_Q,\mathcal{I}_{Q'}$ as kernels given above 
    %%% See the commutative diagram
    that $\mu^q_k(\Lbb_Q^\omega\cap\mathcal{I}_Q)=\Lbb_{Q'}^\omega\cap\mathcal{I}_{Q'}$. This proves the Lemma.
    % More detail: suppose we have $a(1-w)\in \mathcal{I}_\bi$. Then $a$ must also be in the universal Laurent ring $\Lbb$. Indeed, we know that $\mu_k^q(a(1-w))=\mu_k^q(a)(1-w)$ is Laurent, so $\mu_k^q(a)$ as $(1-w)$ is coprime to $(1+Y_k)$ for any mutable $k$. For more variables, can argue similarly setting all but one $w_i$ to 1.  [This is phrased classically, but can upgrade to quantum using Lemma 5.6 of [BZ05] (their $1$ is our $k$)].   % The kernel of the skew form  on $\Lambda^{c_\pm}_{S',c_\pm}$ is isomorphic to the direct sum of that of $\Lambda_{S'}$ with the rank $2n$ sublattice spanned by all $\omega_{c_\pm,i}$. 
\end{proof}
In view of Lemma~\ref{lem:reduced-laurent}, we can unambiguously write $\mathbb{L}_{S';\Ac}^{c}$ for the $\mathcal{A}$-variable version of the reduced universal Laurent ring. By~\eqref{eq:pre-upper-bound}, we have
$$
\mathbb{L}^{c}_{\Xi_Q} = \bigcap_{k=1}^d\frac{\mathcal{T}^\omega_{\Xi_Q}\cap \mathcal{T}^\omega_{\Xi_{\mu_k(Q)}}}{\mathcal{I}_Q\cap \mathcal{T}^\omega_{\Xi_Q}\cap \mathcal{T}^\omega_{\Xi_{\mu_k(Q)}}}\subseteq \frac{\mathcal{T}^\omega_{\Xi_Q}}{\mathcal{I}_Q}\simeq \mathcal{T}_{\Xi_Q}^c.
$$
Since the elements in $\mathcal{T}_{\Xi_Q}^c\subset \mathcal{T}_{\Xi_Q}^c((Y_{-e_k}))$ which remain Laurent polynomials under $\mathrm{Ad}(\Psi(Y_{-e_k}))$ are exactly those contained in the intersection $\mathcal{T}_{\Xi_Q}^c\cap\mathcal{T}_{\Xi_{\mu_k(Q)}}^c$, we obtain
\begin{cor}
\label{cor:cr}
An element $a\in \mathcal{T}_{\Xi_Q}^c$ is contained in $\Lbb_{\Xi_Q}^c$ if and only if for each mutable direction $k$, its image in the partial completion ${\mathcal{T}}_{{\Xi_Q}}^c((Y_{-e_k}))$ remains a Laurent polynomial under $\mathrm{Ad}(\Psi(Y_{-e_k}))$.
\end{cor}

In the same way, we can construct the $\mathcal{X}$-variable version $\Lbb_{\Lambda_Q}^{c}$ of the symplectically reduced universal Laurent ring based in cluster $Q$, which is a subalgebra in $\Lbb_{\Xi_Q}^{c}$. The conclusions of Lemma~\ref{lem:reduced-laurent} and Corollary~\ref{cor:cr} clearly also hold for the $\Lbb_{\Lambda_Q}^{c}$.

\subsection{Local gluing isomorphism for $G=PGL_{n+1}$}

The goal of this section is to construct \emph{local gluing isomorphisms} associated to $c$-isolating triangulations of $S$ and their images $\tri'=\mathcal{C}_c$ under the cutting functor. 

Suppose that $Q_{\tri';c_\pm}$ is a $c_\pm$-isolating quiver subordinate to an ideal triangulation $\tri'$ of $S'$. We write ${Q}^{\omega}_{\tri';c_\pm}$ for the quiver obtained from $Q_{\tri';c_\pm}$ by omitting the frozen $e$-basis vectors $\{e_{c_\pm,i}\}$ and replacing them with frozen $e$-basis vectors
$$
\dot{e}_{c,i} := e_{c_+,i} +\overline{e}_{c_-,i}+\dot{e}_{v^+_i}, \quad 1\leq i \leq n,
$$
where the $\overline{e}_{c_-,i}$ are the vectors defined in \eqref{eq:bar-e-ci}. 
Recall the symplectically reduced lattices $\Lambda^c({{\tri';c_\pm}}),~\Xi^c({{\tri';c_\pm}})$ associated to the quiver $Q_{\tri';c_\pm}$. We write $\mathcal{T}(\tri';c_\pm),~\Tc^{\Ac}(\tri';c_\pm)$ for the corresponding quantum torus algebras. The lattice $\Lambda^c({{\tri';c_\pm}})$ admits the following basis (which we again caution is different from the one determining mutable directions in the cluster algebra):
\begin{align}
\label{eq:edotbasis}
 \hc{\dot{e}_\ell~\big|~ \ell\notin \{(c_\pm,i)\}_{i=1}^n\sqcup \{v^{\pm1}_i\}_{i=1}^{n}} \sqcup \{\dot{e}_{v_i} =\dot{e}_{v^+_i}= -\dot{e}_{v^-_i}\}_{i=1}^n
\sqcup \{\dot{e}_{c,i}\}_{i=1}^n.
\end{align}
We record the pairings
$$
(\dot{e}_{v_i},\dot{e}_{(c,j)}) = \mathfrak{A}_{i,j}, \quad (\dot{e}_{v_i},\dot{e}_{v_j}) = 0=(\dot{e}_{(c,i)},\dot{e}_{(c,j)}),\qquad 1 \leq i,j\leq n.
$$
\begin{notation}
\label{eq:cutnot}
In analogy with the decomposition from Notation~\ref{not:LambdaS-split} for the glued surface $S$, we define a subset of the index set ${I}^{\geq0}_{S'}$ for the quiver ${Q}^{\omega}_{\tri';c_\pm}$ by
$$
{I}^{\geq 0}_{S'} = \{(c,i)\}_{i=1}^n \sqcup \{v^\pm_i\}_{i=1}^{n+1},
$$
% In analogy with the decomposition from Notation~\ref{not:LambdaS-split} for the glued surface $S$, we set
%$$
%\overline{I}^{>0}_{S'} = \{(c,i), v_i\}_{i=1}^n
%$$
and write $I^{<0}_{S'}$ for its complement in the indexing set $I_{S'}$ of the quiver ${Q}^{\omega}_{\tri';c_\pm}$.  Similarly, we write
\begin{align}
\overline{I}^{> 0}_{S'} &= \{e_{v_{n+1}^+},e_{v_{n+1}^-}\}\sqcup \{(c,i), v_i\}_{i=1}^n,\\
\overline{I}^{\geq 0}_{S'} &= \{e_{v_{n+1}^+},e_{v_{n+1}^-}\}\sqcup \overline{I}^{> 0}_{S'},
\end{align}
so that $\{\dot{e}_\ell~|~\ell\in I^{<0}_{S'}\sqcup\overline{I}^{\geq 0}_{S'}\}$ forms a basis for the symplectically reduced lattice $\Lambda^c({{\tri';c_\pm}})$.
%It is also useful to define the vectors
%$$
%e_{v_i} = {e}_{v^+_i} = -{e}_{v^-_i} \in \Lambda_{Q_{\tri';c_\pm}}^c, \quad i=1,\ldots n.
%$$
Again parallel to Notation~\ref{not:LambdaS-split}, we use the partitioning of this basis to make non-orthogonal splittings
\begin{align}
\label{eq:cutsplitlat}
\nonumber \Lambda^c(\tri';c_\pm)&= \Lambda^c_{\leq0}(\tri';c_\pm)\bigoplus \Lambda^c_{>0}(\tri';c_\pm),\\
&= \Lambda^c_{<0}(\tri';c_\pm)\bigoplus \Lambda^c_{\geq0}(\tri';c_\pm).
\end{align}
We write $\Tc_{{\leq0}}(\tri';c_\pm),~\Tc_{{>0}}(\tri';c_\pm)$ for the corresponding quantum tori, and define similarly $\Tc_{{<0}}(\tri';c_\pm),~\Tc_{{\geq0}}(\tri';c_\pm)$.
\end{notation}

Consider the anti-diagonal embedding of the braid group 
$$
B_{n+1}\rightarrow B_{n+1}(c_+)\times B_{n+1}(c_-),\quad \sigma_i \longmapsto (\sigma_i,\sigma_i^{-1}).
$$
Inspecting the formulas~\eqref{eq:braid-e}, we see that the action of this subgroup by cluster transformations descends to the symplectically reduced universal Laurent ring  $\Lbb_{S'}^c$, and moreover factors through the quotient to the Weyl group. From now on, we refer to this action as that of the Weyl group $W(c_\pm)$ associated to the glued pair of tacked circles.

If $\beta$ is an element of the root lattice of $SL_{n+1}$, define a corresponding monomial element of the quantum torus $\mathcal{T}(\tri';c_\pm)$ by\begin{align}
\label{eq:Y-coweights}
Y_{c,\beta^\vee}  = Y_{\sum m_k \dot e_{c,k}} ,\qquad \beta = \sum_{k=1}^n m_k\alpha_k.
\end{align}
%If $\mu$ is an element of the root lattice of $SL_{n+1}$, then it is easy to see from~\eqref{eq:form2} that $Y_{\mu^\vee}$ lies in the quantum torus $\mathcal{T}_{\Xi^c_{S'}}$ and in fact also in $\mathcal{T}_{\Lambda^c_{S'}}$. Indeed, recalling the basis~\eqref{eq:edotbasis} we can notice that
%$$
%Y_{c,\alpha_j^\vee} = Y_{\dot e_{c,j}}.
%$$
Dually, we define
\begin{align}
\label{eq:Y-weights}
Y_{c,\beta} : =Y_{\sum m_k\dot{e}_{v_k}}, \quad \beta = \sum_{k=1}^n m_k\alpha_k.
\end{align}
With these notations we have
\begin{align}
\label{eq:xrel2}
Y_{c,\alpha^\vee}Y_{c,\beta}=q^{2\langle \alpha,\beta\rangle}Y_{c,\beta}Y_{c,\alpha^\vee}.
\end{align}
The Weyl group $W(c_\pm)$ acts on the elements~\eqref{eq:Y-coweights} and~\eqref{eq:Y-weights} via the reflection representation
\begin{align}
\label{eq:weyl-xi}
w(Y_{c,\alpha}) = Y_{c,w(\alpha)}, \qquad w(Y_{c,\alpha^\vee}) = Y_{c,w(\alpha)^\vee}, \quad ~w\in W(c_\pm),
\end{align}
and fixes all $e_\ell$ with $\ell\in I^{>0}_{S'}$.

Recall that each symplectically reduced cluster $\mathcal{X}$-torus $\mathcal{T}_{\Lambda^c_{S'}}$ carries a grading by the root lattice of $SL_{n+1}$, which coincides with the internal grading defined by conjugation with the elements $Y_{c,\alpha}$: 
\begin{align}
\label{eq:internal-grading}
A\in \mathcal{T}_{\Lambda^c_{S'}}\{\beta\} \iff AY_{{c,\alpha}} = q^{2\langle \alpha,\beta\rangle}Y_{{c,\alpha}}A \quad \text{for all } \alpha\in Q_{SL_{n+1}}.
\end{align}
We write
$$
A = \sum_{\beta}A_\beta
$$
for the decomposition of an elment $A$ with respect to this internal grading.
Since the elements $Y_{{c,j}}$ commute with all mutable variables the internal grading is mutation invariant and hence descends to the universal Laurent ring $\Lbb^c_{S'}$. In an isolating cluster $Q_{\tri';c_\pm}$, the grading becomes especially simple to describe: we have
$$
\deg(\dot{e}_\ell) = \begin{cases} 
\alpha_j\quad &\ell = (c,j), ~1\leq j\leq n\\
0\quad &\text{else.} 
\end{cases}
$$

To make contact with the constructions of Section~\ref{sec:AlgWhit} involving $q$-difference operators on the torus of $GL_{n+1}$, it is convenient to consider an extended version $\Lambda_{\geq0}^{c;ext}(\tri')$ of $\Lambda^c_{\geq0}(\tri')$ in which we adjoin an additional frozen direction $e_{c,\eps^\vee_{n+1}}$ whose only nonzero pairings with the other basis elements are given by
$$
(e_{c,\eps^\vee_{n+1}},\dot{e}_{v_n})= 1, \quad (e_{c,\eps^\vee_{n+1}},\dot{e}_{v^\pm_{n+1}})= \mp 1.
$$
In this way, we can associate vectors $e_{c,\la^\vee}$ to arbitrary elements $\la\in\mathbb{Z}^{n+1}$ of the $GL_{n+1}$ weight lattice. 
Similarly, to an arbitrary $\mu=\sum_{j=1}^{n+1}m_j\alpha_j\in\mathbb{Z}^{n+1}$ (where we recall the notation $\alpha_{n+1} = \eps_{n+1}$) we can associate the pair of vectors $e^+_{c,\mu},~e^-_{c,\mu}$ given by
$$
e^\pm_{c,\mu} = m_{n+1}\dot{e}_{v_{n+1}^\pm} \pm \sum_{j=1}^nm_j\dot{e}_{v_{j}}.
$$
In particular, we have
$
e^\pm_{c,\eps_i} = e_{v^\pm_i}.
$
With these definitions, we have
$$
(e^\pm_{c,\mu},e_{c,\la^\vee}) = \pm \langle \la,\mu\rangle.
$$
The lattice  $\Lambda_{\geq0}^{c;ext}(\tri')$ has a basis
\begin{align}
\label{eq:ext-basis}
\{e^+_{c,\eps_i},e_{c,\eps_i^\vee}\}_{i=1}^{n+1} \sqcup \{\zeta' = {e}_{v^+_{n+1}}+{e}_{v^-_{n+1}}\}.
\end{align}
We extend the action of the Weyl group $W(c_\pm)$ in the natural way from $\Lambda_{\geq0}^{c}(\tri')$ to $\Lambda_{\geq0}^{c;ext}(\tri')$, declaring the vector $\zeta'$ to be a fixed point\footnote{Recall that by the gluing quotient relations~\eqref{eq:reduced-lattice} we have ${e}_{v^+_{i}}+{e}_{v^-_{i}}={e}_{v^+_{j}}+{e}_{v^-_{j}}$ for all $i,j$. } while letting the Weyl group act via the reflection representation permuting the basis vectors $\{e^+_{c,\eps_i},e_{c,\eps_i^\vee}\}_{i=1}^{n+1} $.  
Note that the quantum torus $\Tc_{{\geq0}}(\tri';c_\pm)$ is recovered as the centralizer of the element $Y_{e^+_{c,\omega_{n+1}}}$ in $\Tc(\Lambda^{ext}_{S'_{\geq0}})$. Then evidently we have
\begin{lemma}
\label{lem:dqt}
The quantum torus $\Tc(\Lambda^{ext}_{S'_{\geq0}})$ associated to the lattice $\Lambda^{ext}_{S'_{\geq0}}(\tri')$ is $W$-equivariantly isomorphic to $\mathcal{D}_q(T)[Z^{\pm1}]$, where the central subalgebra in the latter generated by $Z$ is identified the corresponding one in generated by $Y_{\zeta'}$. 
\end{lemma}
%Let us fix this isomorphism as \blue{[Did I put the right q-powers here?]}
%$$
%Y_{e_{c,\eps^\vee_{n+1}}}\mapsto D_{n+1}, \quad Y_{e_{c,\alpha^\vee_i}}\mapsto -q^{-1}D_iD_{i+1}^{-1}, \quad i=1,\ldots, n,
%$$
%\begin{align}
%\label{eq:cut-loc-iso}
%Y_{e^+_{c,\eps_i}}\mapsto -q^{-1}Y_{h_+}w_i,\qquad Y_{\zeta'}\mapsto Z.
%\end{align}

%spanned by 
%$\Lambda_{S_{>0}}$ and $e_{c,\eps_{n+1}},e_{c,\eps^\vee_{n+1}}$ with the algebra $\Dc_q(T)$ from~\eqref{eq:DqT} via the isomorphism
%$$
%Y_{c,\la} \mapsto \prod w_j^{\la_j}, \qquad Y_{c,\la^\vee} \mapsto \prod D_j^{\la_j}, \quad \la = (\la_1,\ldots, \la_{n+1})\in \mathbb{Z}^{n+1}.
%$$
%The action~\eqref{eq:weyl-xi} of Weyl group $W(c_\pm)$ via the reflection representation extends naturally to the extended quantum torus associated to $\Lambda_{Q_{\tri';c_\pm}}^{c;ext}$ and the corresponding universal Laurent ring. 

Consider the localization $\Tc_{\Lambda_{S'}^c;\loc}$ of the quantum torus $\Tc_{\Lambda_{S'}^c}$ at the multiplicative set
\begin{align}
    \label{eq:denom-set}
    Ø(c_\pm) = \left\{\prod_j (1-q^{2k_j}Y_{{c,\beta_j}}) \right\}_{\vec k,\vec\beta} \subset \mathcal{T}_{\Lambda^c_{S'}}
\end{align}
where $\vec k$ ranges over all finite length sequences of integers, and $\vec\beta$ over all sequences of positive roots (possibly with repetitions). Note that the denominator set $ Ø(c_\pm)$ is mutation invariant since the $Y_{{c,\beta_j}}$ commute with the subalgebra in $\Tc_{\Lambda_{S'}^c}$ corresponding to the unfrozen directions in $\Lambda_{S'}^c$, and so the definition of $\Tc_{\Lambda_{S'}^c;\loc}$ takes the same form in any cluster. 

Let $\mathfrak{S}_{\alpha,k}\subset \Tc_{\Lambda_{S'}^c}\{0\}$ be the multiplicative set consisting of elements of internal degree zero which are not multiples of the element  $1-q^{2k}Y_{c,\alpha}$ (which we recall is central in $\Tc_{\Lambda_{S'}^c}\{0\}$). Then both $\mathfrak{S}_{\alpha,k}$ and the set of all nonzero elements in $\Tc_{\Lambda_{S'}^c}\{0\}$ form left Ore denominator sets in $\Tc_{\Lambda_{S'}^c}$.  If we write $ \Tc_{\Lambda_{S'}^c;\rat}$ for the Ore localization of $ \Tc_{\Lambda_{S'}^c}$ at the latter denominator set, then just as in formula~\eqref{eq:filtration} of Section~\ref{sec:nildaha}, each divisor 
\begin{align}
\label{eq:def-deltadivisors}
d_{\alpha,k}(c) = \{q^{2k}Y_{c,\alpha}=1\}
\end{align}
defines a filtration on $\Tc_{\Lambda_{S'}^c;\rat}$ with filtered pieces 
$$
\Tc_{\Lambda_{S'}^c;\rat}^{\alpha,k;m} = (1-q^{2k}w_\alpha)^{m}\mathfrak{S}_{\alpha,k}^{-1}\Tc_{\Lambda_{S'}^c}
$$
and projectors $\pi_{\alpha,k;m}$.
We define evaluations and residues for elements of $\Tc_{\Lambda_{S'}^c;\rat}^{\alpha,k;0}$ and $\Tc_{\Lambda_{S'}^c;\rat}^{\alpha,k;-1}$ respectively in the same way as Definition~\ref{def:eval-res}. Once again, since the denominators $ Ø(c_\pm) $ commute with the $Y_{c,\alpha_j}$,  the internal grading~\eqref{eq:internal-grading} extends to $\Tc_{\Lambda_{S'}^c;\loc}$ and moreover preserves each of its filtered pieces relative to the divisors $d_{\alpha,k}$. 

\begin{remark}
\label{rmk:res-hom}
Since $1-q^{2k}Y_{c,\alpha}$ is central in $\Tc_{\Lambda_{S'}^c}\{0\}$, the quotient $\Tc_{\Lambda_{S'}^c;\rat}^{\alpha,k;0}\{0\}/\Tc_{\Lambda_{S'}^c;\rat}^{\alpha,k;1}\{0\}$ is a ring and the restriction of the evaluation map to the internal degree zero piece $\Tc_{\Lambda_{S'}^c;\rat}^{\alpha,k;0}\{0\}$ is a ring homomorphism. 
\end{remark}

Now we return to the context of a $c_\pm$-isolating cluster $Q_{\tri';c_\pm}$. For each element $\beta$ of the root lattice of $G$ we use the element $Y_{c,\beta^\vee}=Y_{c,\beta^\vee}({\tri';c_\pm})$ from~\eqref{eq:Y-coweights} to define an analog of the automorphism from~\eqref{eq:shifts}:
\begin{align}
\label{eq:Yshifts}
\Ad_{Y_{c,\beta^\vee}}\colon A \longmapsto A^{[\beta]} = Y_{c,\beta^\vee} A Y^{-1}_{c,\beta^\vee}.
\end{align}
Then for all $s_\alpha\in W(c_\pm)$ and $A\in\Tc_{\Lambda_{S'}^c;\rat}^{\alpha,k;0}$,  we have (cf Lemma~\ref{lem:symmetry-condition})
\begin{align}
\label{lem:trade}
(s_\alpha A)\big|_{d_{\alpha,k}}=A^{[k\alpha]}\big|_{d_{\alpha,k}}.
\end{align}

Finally, we observe that the constructions above admit analogs for the extended lattice $\Lambda^{c;ext}({\tri';c_\pm})$ in which we replace $\Lambda_{\geq0}^c(\tri')$ by $\Lambda^{c;ext}_{\geq0}(\tri')$ in the decomposition~\eqref{eq:cutsplitlat}. We write $\Tc_{\loc}({\tri';c_\pm})$ for the localized quantum torus in the isolating cluster $Q_{\tri';c_\pm}$,  and $\Tc^{ext}_{\loc}({\tri';c_\pm})$ for its extended version. Similarly, we write $\Tc_{\geq0;\loc}({\tri';c_\pm})$ for the localization at $Ø(c_\pm)$ of the quantum torus $\Tc_{\geq0}({\tri';c_\pm})$.

Now we can state the main definition of this section.
\begin{defn}
\label{def:Tres}
    The (extended) $PGL_{n+1}$ \emph{residue quantum torus} $\Tc_{\res}({\tri';c_\pm})$ (respectively $\Tc^{ext}_{\res}({\tri';c_\pm})$) for the $c_\pm$-isolating cluster $Q_{\tri';c_\pm}$ consists of all elements $A$ of $\Tc_{\loc}({\tri';c_\pm})$ (respectively $\Tc^{ext}_{\loc}({\tri';c_\pm})$ ) such that: 
    \begin{enumerate}
            \item The element $A$ is invariant under the action of the diagonally embedded Weyl group $W(c_\pm)$;
            %Each graded component $A^\mu\in\mathcal{T}_{\Xi^c}^{loc}$
        \item The element $A$ has at worst simple poles at the divisors $d_{\alpha,k}$, and for all elements $\nu$ of the weight lattice of $PGL_{n+1}$ and integers $k$ its residues satisfy
\begin{align}
\label{eq:Lresidue-condition}
\mathrm{res}_{d_{\alpha,k}}\left(A_\nu + A_{s_\alpha(\nu)+k\alpha}Y_{c,(\langle \nu,\alpha\rangle-k)\alpha^\vee}({\tri;c_\pm})\right)=0.
%\left( A(Y_{2\overline\xi_{h_+}+\overline\xi_{h_-}}(Q_{\tri;c_\pm})^{|\langle\mu,\alpha_1^\vee\rangle-k|} + Y_{\overline\xi_{v_{\alpha_1}}}(Q_{\tri;c_\pm})^{|\langle\mu,\alpha_1^\vee\rangle-k|})\right)^{(\mu)} \quad \text{is regular on} \quad \delta_{\alpha_1,k}=\{ q^{2k}Y_{e_{c,\alpha_1}}=1\}.
\end{align}
            \end{enumerate}
\end{defn}
The algebra $\Tc_{\res}({\tri';c_\pm})$ is generated by its subalgebras $\Tc_{<0}(\tri';c_\pm)$ and $\Tc_{\geq0;\res}(\tri';c_\pm)$, where the latter consists of all elements of the localized quantum torus $\Tc_{\geq0;\loc}({\tri';c_\pm})$ satisfying the conditions in Definition~\ref{def:Tres}. 
The following Lemma summarizes the relation of the algebra $\Tc_{\res}({\tri';c_\pm})$ to the residue algebra $\Dres$ from Section~\ref{sec:AlgWhit}.

\begin{lemma}
\label{lem:cutloc}
The identification of quantum tori from Lemma~\ref{lem:dqt} induces an isomorphism of algebras
$$
\Tc^{ext}_{\geq0;\res}(\tri';c_\pm) \simeq \Dres[Z^{\pm1}].
$$
The subalgebra $\Tc_{\geq0;\res}(\tri';c_\pm)$ in in $\Tc^{ext}_{\geq0;\res}(\tri';c_\pm)$ coincides with the centralizer of the element $Y_{e^+_{c,\omega_{n+1}}}$.
\end{lemma}
In particular, the algebra $\Tc^{ext}_{\geq0;\res}(\tri';c_\pm)$ is generated by $Y_{\zeta'}^{\pm1}$ and the elements from point (2) of Lemma~\ref{lem:small-genset}.

\begin{remark}
\label{rmk:A-res-torus}
The construction in Definition~\ref{def:Tres} can also be carried out at the level of $\Xi$-lattices to define the $SL_{n+1}$ residue quantum torus $\Tc^{\Ac}_{\res}(\tri';c_\pm)$ associated to a $c_\pm$-isolating cluster $Q_{\tri';c_\pm}$. The algebra $\Tc^{\Ac}_{\res}(\tri';c_\pm)$ consists of elements of the $Ø(c_\pm)$-localized $\Ac$-quantum torus $\Tc^{\Ac}_{\loc}(\tri';c_\pm)$ satisfying conditions (1) and (2), where $\nu$ in (2) now ranges over the weight lattice of $SL_{n+1}$.

\end{remark}

%\begin{remark}
%\label{rmk:other-res-tori}
%\blue{replace this remark}
%Suppose that $\mu$ is a cluster transformation based at isolating cluster $Q_{\tri',c_\pm}$ and $\mu_\ell$ is a mutation in some direction $\ell\in I_{S'}^{<0}$. Then since such a mutation commutes with the $W(c_\pm)$ action and preserves the elements $Y_{m\alpha^\vee}$ used to formulate the residue condition~\eqref{eq:Lresidue-condition}, we can define in the same way a residue quantum torus $\Tc_{\res}(\mu_\ell(Q_{\tri';c_\pm}))$ for the cluster $\mu_\ell(Q_{\tri',c_\pm})$. The quantum cluster transformation $\mu_\ell$ extends to an isomorphism between the further localized residue quantum tori
%$$
%\mu_\ell \colon S_{Y_\ell}^{-1}\Tc_{\res}(Q_{\tri';c_\pm}) \rightarrow S_{Y'_\ell}^{-1}\Tc_{\res}(\mu(Q_{\tri';c_\pm}))
%$$
%where $S_Y$ is the multiplicative set $\{(1+q^{2k+1}Y)\}_{k\in\mathbb{Z}}$. Since such a mutation $\mu_\ell$ also preserves $\Lambda_{S_{<0}}$, given any cluster transformation $\mu$ consisting only of mutations in directions $I_{S}^{<0}$ we can iterate the process above to define a residue quantum torus $\Tc_{\res}(\mu(Q_{\tri';c_\pm}))$ in the cluster $\mu(Q_{\tri';c_\pm})$.
%\end{remark}

Now suppose that $Q_{\tri;c}$ is an isolating cluster for a simple closed curve $c$ on a marked surface $S$,  $\tri'=\mathcal{C}_c(\tri)$ is the image of $\tri$ under the cutting functor, and $Q_{\tri';c_\pm}$ the corresponding $c_\pm$-isolating quiver on the surface $S'$. 
\begin{defn}
\label{eq:local-S-torus}
We write $\Lbb(\Lambda^{\mathrm{fr}}_{\tri;c})$ for the universal Laurent ring associated to the quiver $(\Lambda_{S}(Q_{\tri;c}),\{e_\ell\}_{\ell\in I_{S}})$, where only the directions $\ell\in I_{S_{>0}}$ are regarded as mutable, and all others are frozen. 
\end{defn}
As an algebra $\Lbb(\Lambda^{\mathrm{fr}}_{\tri;c})$ is generated by the quantum torus $\mathcal{T}_{<0}(\tri;c)$ together with the universal Laurent ring  $\Lbb_{S_{\geq0}}(\tri;c)$ associated 
to the based lattice $(\Lambda_{S_{\geq0}},\{e_\ell\}_{\ell\in I^{S_{\geq0}}})$ where the directions $e_{h_\pm}$ are regarded as frozen and all others mutable. 
Following Proposition~\ref{prop:alg-whit}, we now construct a local gluing isomorphism 
$$
\eta \colon \Tc_{\res}(Q_{\tri';c_\pm}) \longrightarrow \Lbb(\Lambda^{\mathrm{fr}}_{\tri;c}).
$$
Recall from Section~\ref{sec:alg-whit} the space $\Fcr$ of compactly supported $\mathbb{Z}[q^{\pm1}]$-valued functions on $\mathbb{Z}^{n+1}$, along with its subspace $\Vcr$ of functions vanishing outside the dominant cone as defined in~\eqref{eq:Vdef}. We define a faithful representation of the quantum torus $\Tc(Q_{\tri;c})$ on the $\mathbb{Z}[q^{\pm1}]$-module
$$
M[S_{\leq0}] = \mathcal{T}_{\leq0}(\tri;c)\otimes_{\mathbb{Z}[q^{\pm1}]} \Fcr
$$
of $\mathcal{T}_{\leq0}(\tri;c)$-valued functions on $\mathbb{Z}^{n+1}$ as follows. For $\ell\in I^{S_{<0}}$, the generator $Y_{e_\ell}$ multiplies the value of such a function from the left by $Y_{e_\ell}$. For $\ell\in I^{S_{>0}}= \{(s_j,t_j)\}_{j=1}^n$, the generator $Y_{e_\ell}$ acts via the embedding of $\mathcal{T}_{S_{>0}}$ into $\mathcal{T}(\Lambda_{GL_{n+1}})$ described in Remark~\ref{rmk:toda-embed}, where we define the action of the generators $P_i,X_i$ of the latter quantum torus on $\Fcr[S_{\leq0}]$ by
 \begin{align}
\label{eq:toda-torus-action}
(P_{\epsilon_i}\circ \phi)(\la) = -q^{-1}Y_{{h_+}}\phi(\la-\eps_i), \quad i=1,\ldots n+1
\end{align}
$$
 (X_{\alpha_i} \circ\phi)(\la) = -q^{2\langle \la,\alpha_i\rangle+1}\phi(\la), \quad i=1,\ldots n
$$
and
$$
 (X_{\eps_1}\circ\phi)(\la) = q^{2\la_1}\phi(\la).
$$
%$$
%Y_{h_+}\mapsto Y_{h_+}P_{n+1}, \quad Y_{h_-}\mapsto Y_{h_-}P_{1}^{-1},
%$$
 We define the action of the remaining generators $Y_{h_{\pm}}$ of $\Tc(Q_{\tri;c})$ by
\begin{align}
\label{eq:action-def-S}
\nonumber (Y_{h_+}\circ \phi)(\la) &= -q^{-1}Y_{h_+} \phi(\la_1,\ldots,\la_{n+1}-1),\\
 (Y_{h_-}\circ \phi)(\la) &= -qY_{h_-} \phi(\la_1+1,\ldots,\la_{n+1}).
\end{align}

%We do this using the extended version $\Tc^{ext}_{G,S';\res}(Q_{\tri';c_\pm})$. To define its analog on the side of the glued surface $S$,
%%This requires the following preparation on the side of the glued surface $S$.
%recall the non-orthogonal direct sum decomposition
%$$
%\Lambda_{S}(Q_{\tri;c}) = \Lambda_{S_{<0}}(Q_{\tri;c})\bigoplus \Lambda_{S_{\geq0}}(Q_{\tri;c}).
%$$
%introduced in Notation~\ref{not:LambdaS-split}. 

To relate more precisely with the constructions of Section~\ref{sec:AlgWhit}, consider  the based lattice $\Lambda_{GL_{n+1}}$ from Section~\ref{subsec:ltoda} along with its extension $\Lambda_{GL_{n+1};\zeta}$ by the rank 1 lattice $\mathbb{Z}\zeta$, which is equipped with the basis obtained by adjoining the vector $e_h=-\zeta -p_{\eps_1}$. Then if we write $\Lambda_{GL_{n+1};\zeta_*}$ for the quiver where the vector $e_h$ is regarded as a frozen direction, we have an isometric embedding
\begin{align}
\label{eq:pgl-embed}
\Lambda_{S_{\geq0}}(\tri;c) \longrightarrow \Lambda_{GL_{n+1};\zeta_*}
\end{align}
given by
$$
e_{s_i}(\Lambda_{S_{\geq0}})\mapsto e_{s_i}(\Lambda_{GL_{n+1};\zeta_*}),\quad e_{t_i}(\Lambda_{S_{\geq0}})\mapsto e_{t_i}(\Lambda_{GL_{n+1};\zeta_*}), \qquad 1\leq i\leq n
$$
$$
e_{h_+}\mapsto e_{s_{n+1}}+e_{t_{n+1}} ,\quad e_{h_-}\mapsto e_h.
$$
%Hence we can regard  $\Lambda_S$ as being isometrically embedded into the extended lattice 
%$$
%\Lambda_{S}^{ext}=\Lambda_{S_{<0}}\oplus \Lambda_{GL_{n+1};\zeta_*},
%$$
%with skew pairing between the two summands determined by 
%$$
%(e_{h},e_{\mu})_{\Lambda_S^{ext}} = (e_{h_-},e_{\mu})_{\Lambda_S}, \qquad (e_{s_{n+1}},e_{\mu})_{\Lambda_S^{ext}} = (e_{h_+},e_{\mu})_{\Lambda_S}, \quad \in \Lambda_{S_{<0}},
%$$  
%and $(e_{\ell},e_\mu)
%=0$ for all other basis vectors $e_{\ell}$ for $\Lambda_{GL_{n+1};\zeta_*}$ and $\mu\in\Lambda_{S_{<0}}$. 

%We write $\Lbb(Q^{ext}_{\tri;c})$ for the corresponding universal Laurent ring where we regard all basis vectors except $\{e_{s_j},e_{t_j}\}_{j=1}^n$ as being frozen. Hence $\Lbb(Q^{ext}_{\tri;c})$ is generated by its subalgebras $\Lbb(\Lambda_{GL_{n+1};\zeta_*})$ and $\mathcal{T}_{<0}(\tri;c)$.
Then the action $\mathcal{T}(\Lambda_{GL_{n+1};\zeta})$ on $\Fcr[S_{\leq0}]$ where $P_i,X_i$ act via~\eqref{eq:toda-torus-action} and the central generator $Y_\zeta$ acts by
%\begin{align}
%\label{eq:toda-torus-action}
%(P_{\epsilon_i} \phi)(\mu) = -q^{-1}Y_{{h_+}}\phi(\mu-\eps_i), \qquad (X_{\alpha_i} \phi)(\mu) = -q^{2\langle \mu,\alpha_i\rangle+1}\phi(\mu),
%\end{align}
%and
$$
 (Y_\zeta\circ \phi)(\la) =Y_{h_+}Y_{h_-}^{-1}\phi(\la)
$$
restricts to the action of $\mathcal{T}_{{\geq0}}(\tri;c)$ defined using~\eqref{eq:action-def-S}.

%We extend the action of $\Tc_S$ on $M[S_{\leq0}]$ to one of $\mathcal{T}(\Lambda^{ext}_{S})$ in accordance with~\eqref{eq:toda-torus-action} by letting
%\begin{align*}
%(Y_{e_{s_{n+1}}}f)(\la) &= (-1)^nq^{-n}q^{2\lambda_{n+1}}Y_{h_+}\cdot f(\lambda).
%\end{align*}

\begin{lemma}
\label{lem:glue-ext}
The universal Laurent ring $\Lbb(\Lambda_{GL_{n+1};\zeta_*})$ of the extended quiver is given by
\begin{align}
\label{eq:ext-laurent}
\Lbb(\Lambda_{GL_{n+1};\zeta_*})\simeq \Lbb(\Lambda_{GL_{n+1}})[Y_{\zeta}^{\pm1}],
\end{align}
and the embedding~\eqref{eq:pgl-embed} induces an isomorphism
$$
\Lbb_{S_{\geq0}}(\tri;c) \simeq \Lbb(\Lambda_{GL_{n+1};\zeta_*})^{\langle P_{\omega_{n+1}}\rangle},%\quad \bs{\mathrm P} = P_{1}\cdots P_{n+1},
$$
where the right hand side denotes the centralizer of $P_{\omega_{n+1}}$ in $\Lbb(\Lambda_{GL_{n+1};\zeta_*})$.
\end{lemma}
\begin{proof}
Since each $p_{\eps_i}$ is already contained in the lattice $\Lambda_{GL_{n+1}}$, we can perform a change of basis in $\Lambda_{GL_{n+1};\zeta_*}$ modifying only the frozen direction $e_h$ via $e_h\mapsto e_h+p_{\eps_1}=-\zeta$, and so the first claim of the Lemma follows. Since~\eqref{eq:pgl-embed} bijectively maps the set of mutable directions in $\Lambda_{S_{\geq0}}$ to those in $\Lambda_{GL_{n+1};\zeta_*}$, the second claim follows from observing that the image of~\eqref{eq:pgl-embed} is exactly the orthogonal complement to $p_{\omega_{n+1}}= \sum_{i=1}^{n+1}p_{\eps_i}$ in $\Lambda_{GL_{n+1};\zeta_*}$. Indeed, this becomes clear after passing to the basis of the latter given by all mutable directions together with $e_h,e_{t_{n+1}}+e_{s_{n+1}}$ and $e_{s_{n+1}}$ since
$$
(p_{\omega_{n+1}}, e_{s_i})=(p_{\omega_{n+1}}, e_{t_i}) = 0 , \quad i<n+1, \quad (\bs p, e_h)=0=(p_{\omega_{n+1}}, e_{t_{n+1}}+e_{s_{n+1}}),
$$
while $(e_{s_{n+1}},p_{\omega_{n+1}}) = 1$. So~\eqref{eq:pgl-embed} identifies $\mathcal{T}(\Lambda_{S_{\geq0}})$ with the centralizer of $P_{\omega_{n+1}}$ in $\mathcal{T}(\Lambda_{GL_{n+1};\zeta_*})$, and since $P_{\omega_{n+1}}$ commutes with all mutable variables the corresponding statement also holds for the universal Laurent rings.
\end{proof}
It follows from the isomorphism~\eqref{eq:ext-laurent} in the Lemma that the action of $\Lbb(\Lambda^{\mathrm{fr}}_{\tri;c})$ preserves the submodule 
$$
\Vcr[S_{\leq0}] = \mathcal{T}_{\leq0}(\tri;c)\otimes_{\mathbb{Z}[q^{\pm1}]}\Vcr \subset \Fcr[S_{\leq0}]
$$
of all functions vanishing outside the dominant cone. Indeed, this is clear for the subalgebra of multiplication operators $\mathcal{T}_{<0}(\tri;c)$, while by Lemma~\ref{lem:pres} the action of the algebra $\Lbb(\Lambda_{GL_{n+1};\zeta_*})$ (and hence that of $\Lbb_{S_{\geq0}}(\tri;c)$ by~\eqref{eq:ext-laurent}) also preserves $\Vcr[S_{\leq0}]$.
Now consider the action of $\Tc^{ext}_{\loc}(Q_{\tri';c_\pm})$ on 
$\mathcal{T}_{\leq0}(\tri';c_{\pm})\otimes\mathbb{Q}[q^{\pm1}](\bs w)$ induced by letting the subalgebra $\Tc_{<0}(Q_{\tri';c_\pm})$ act by left multiplication operators, and setting
\begin{align}
\label{eq:action-def-Sprime}
(Y_{c,\eps^\vee_{i}}\bullet f)(w) &= f(w_{i}\mapsto q^2w_{i}),\\
%$$
%Y_{c,\alpha_i^\vee}\bullet f = -qf(w_{i}\mapsto q^2w_{i},~w_{i+1}\mapsto q^{-2}w_{i+1}), \quad i=1,\ldots, n
%$$
%$$
(Y_{c,\eps^+_i}\bullet f)(w) &= -q^{-1}Y_{h_+}w_if(w),\\
%$$
%\end{align*}
%$$
(Y_{\zeta'}\bullet f)(w) &= Y_{h_+}Y_{h_-}^{-1}f(w).
\end{align}
It follows from Lemmas~\ref{lem:res-preserved} and~\ref{lem:cutloc} that the action of the extended residue torus $\Tc^{ext}_{\res}(Q_{\tri';c_\pm})$ preserves the subspace
$$
\Vcr[S'_{\leq0}]=\mathcal{T}_{\leq0}(\tri';c_{\pm})\otimes\mathbb{Z}[q^{\pm1}][\bs w]^{sym}. 
$$
Using the gluing diffeomorphism $S\simeq S'/\phi$ to identify the corresponding bicolored subgraphs on $S,S'$ we get an isomorphism of quantum tori 
$$
\iiota_\phi:  \mathcal{T}_{\leq0}(\tri;c)\simeq \mathcal{T}_{\leq0}(\tri';c_\pm).
$$
Combining it with the algebraic Whittaker transform~\eqref{eq:alg-whit}, we get an isomorphism of $\mathbb{Z}[q^{\pm1}]$-modules
\begin{align}
\label{eq:algWhitS}
\iiota_\phi\otimes \mathcal{W}\colon \Vcr[S_{\leq0}] \simeq \Vcr[S'_{\leq0}].
\end{align}

\begin{prop}
\label{prop:local-gluing-pgl}
The identification of faithful representations~\eqref{eq:algWhitS} induces an algebra isomorphism
\begin{align}
\label{eq:local-gluing-pgl}
\eta_c = \eta_{c;\tri}\colon \Lbb(\Lambda^{\mathrm{fr}}_{\tri;c})\simeq \Tc_{\res}(Q_{\tri';c_\pm})
\end{align}
intertwining the Dehn twist automorphisms:
\begin{align}
\label{eq:dehn-intertwining}
\eta_c\circ \tau_c(S)  = \tau_{c_-}(S')\circ \eta_c
\end{align}
\end{prop}

\begin{proof}
Recall that the algebras $\Lbb(\Lambda^{\mathrm{fr}}_{\tri;c})$ and $\Tc_{\res}(Q_{\tri';c_\pm})$ are generated by their factors $\Tc_{<0}(\tri,c)$ (resp. $\Tc_{<0}(\tri';c_\pm)$ ) and $\Lbb_{S_{\geq0}}(\tri;c)$ (resp. $\Tc_{\res}(Q_{\tri';c_\pm})$). Clearly $\Tc_{<0}(\tri,c)$ and $\Tc_{<0}(\tri';c_\pm)$ give the same subalgebra of the endomorphism ring of $\Vcr[S_{\leq0}]$ under the identification ~\eqref{eq:algWhitS}. On the other hand, combining Lemmas~\ref{lem:cutloc} and~\ref{lem:glue-ext} with the intertwining relations for the algebra generators in Proposition~\ref{prop:alg-whit}, we deduce that the same is true of the pair of extended algebras $\Lbb(\Lambda_{GL_{n+1};\zeta_*})$ and $\Tc^{ext}_{\res}(Q_{\tri';c_\pm})$. Since the action of $Y_{e^+_{c,\omega_{n+1}}}$ is identified with that of $P_{\omega_{n+1}}$, the  Proposition follows by the characterization of $ \Tc_{\res}(Q_{\tri';c_\pm})$ and  $\Lbb_{\geq0}(\tri;c)$ as the centralizers of these elements in their extended versions. 
For the last claim, we note that formula \eqref{eq:short-Dehn} identifies the Dehn twist automorphism $\tau_c(S)$ with the restriction to $\Lbb(\Lambda^{\mathrm{fr}}_{\tri;c})$ of the automorphism $\tau$ in~\eqref{eq:tau}, while \eqref{eq:cutdehn} identifies $\tau_c(S')$ with the the restriction to $\Tc_{\res}(Q_{\tri';c_\pm})$  of the automorphism $\widetilde\gamma$. Hence we deduce~\eqref{eq:dehn-intertwining} from the intertwining relation~\eqref{eq:aut-inter}.
\end{proof}

\begin{remark}
As observed in the proof of Proposition~\ref{prop:local-gluing-pgl}, it is immediate from the construction of the isomorphism $\eta_c$ that it is indeed `local' in the sense that
\begin{align}
\label{eq:local-is-local}
\ell\in I_{S}^{<0} \implies \eta_c(Y_\ell) = Y_{\iiota_\phi(\ell)},
\end{align}
and that $\eta_c$ restricts to an isomorphism 
\begin{align}
\label{eq:local-nontriv}
\eta_{c}^{\geq0}\colon \Lbb_{S_{\geq0}}(\tri;c)\simeq \Tc_{\geq0;\res}(\tri';c_\pm).
\end{align}

\end{remark}

\begin{remark}
The particular choice of faithful representation spaces $\Vcr[S_{\leq0}]$ and $\Vcr[S'_{\leq0}]$ used to construct $\eta_c$ is not essential: we get the same isomorphism~\eqref{eq:local-gluing-pgl} if we work instead with functions valued in the skew fraction fields or (partial) completions of the quantum tori associated to $\Lambda_{S_{\leq0}}$ and $\Lambda_{S'_{\leq0}}$ respectively. 
\end{remark}

The quivers $\quiver_{\tri;c}$ and $\quiver^\omega_{\tri';c_\pm}$ are a prototypical example of a more general notion that we call a \emph{locally glueable pair.} Consider the quivers $I^{--}_{\Toda},I^{+-}_{\Toda},I^{++}_{\Toda}$ illustrated in Figure~\ref{fig:Qmm}, and the ones ${I}^{--}_{\mathrm{hat}},{I}^{+-}_{\mathrm{hat}},{I}^{++}_{\mathrm{hat}}$ illustrated in Figure~\ref{fig:Qmmcut}.

\begin{figure}[h]
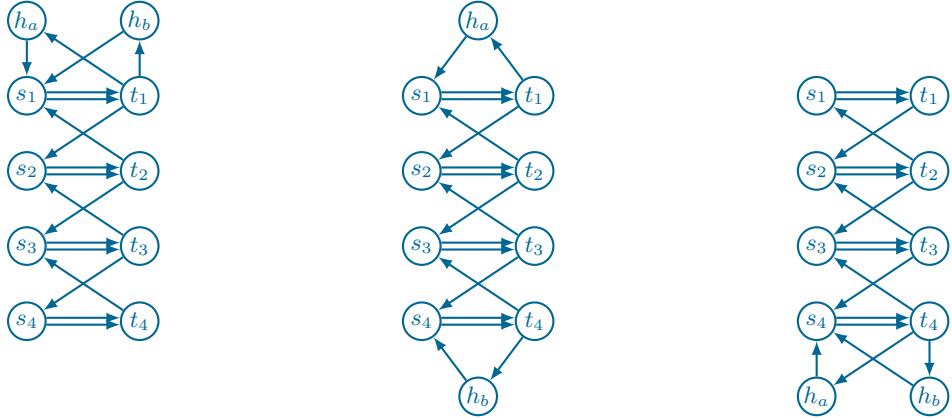

\subfile{fig-Qmm}
\caption{Possible local quivers $I^{--}_{\Toda},I^{+-}_{\Toda},I^{++}_{\Toda}$ in the definition of a locally glueable pair for $PGL_5$.}
\label{fig:Qmm}
\end{figure}

\begin{figure}
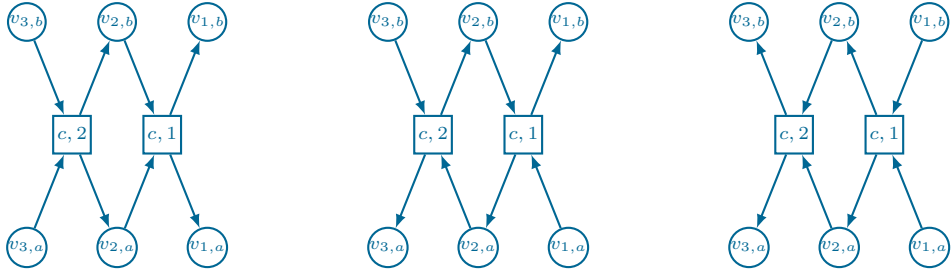

\subfile{fig-Qmmcut.tex}
\caption{Possible local quivers $I^{--}_{\mathrm{hat}},I^{+-}_{\mathrm{hat}},I^{++}_{\mathrm{hat}}$ in the definition of a locally glueable pair for $PGL_3$. }
\label{fig:Qmmcut}
\end{figure}

\begin{defn}
\label{def:glueable-pair}
%\blue{[Can you check I didn't fuck up this definition?]}
A \emph{locally glueable pair} of handle signature $\eps=(\eps_a,\eps_b)\in\{++,-+,--\}$ consists of a pair of quivers $(\quiver,\quiver_\cut)$ together with the following additional data:
\begin{enumerate}
\item
A decomposition of index sets $I = I_{<0}\sqcup I_{\geq0}$ of the quiver $\quiver$, and a bijection $I_{\geq0}\simeq I^{\eps}_{\Toda}$ such that the sublattice in $\Lambda=\Lambda_Q$ corresponding to directions $\{s_j,t_j\}_{j=1}^n$ is orthogonal to the sublattice $\Lambda_{<0}$ corresponding to $I_{<0}$.

\item 
A decomposition of index sets $I_{\cut} = I_{\cut}^{<0}\sqcup I_{\cut}^{\geq0}$ and a bijection  $I_{\cut}^{\geq0}\simeq I^\eps_{\mathrm{hat}}$ such that the sublattice in $\Lambda_{Q^{\cut}}$ corresponding to directions $\{e_{c,j}\}_{j=1}^n$ is orthogonal to the sublattice $\Lambda^{\cut}_{{<0}}$ corresponding to $I_{\cut}^{<0}$, 
%such that the sublattice $\Lambda_{\geq0;\cut}^{\omega}$ spanned by the vectors corresponding to $\{v_{j,a},v_{j,b}\}$ together with the vectors
%$$
%e_{c,j} = e_{c_+,j} - e_{c_-,j} + \eps_b (j-1)(v_{j+1,b}-v_{j,b})
%$$
%is isometric to $\overline{I}_{\cut}^\eps$, the vectors corresponding to $e_{c,j}$ are orthogonal to $\Lambda_{Q^{\cut}_{<0}}$, 
and the action of the Weyl group $W\simeq S_{n+1}$ on the lattice $\Lambda_{\cut}$ fixing all $e_\ell$ for $\ell\in I^{<0}_{\cut}$ and satisfying
$$
\sigma(e_{v_{j,a}}) = e_{v_{\sigma(j),a}}, \quad \sigma(e_{v_{j,b}}) = e_{v_{\sigma(j),b}},
$$
$$
\sigma_i(e_{c,j}) = \begin{cases}
e_{c,i}+e_{c,j}  \quad & |i-j|=1\\
-e_{c,i}\quad & j=i\\
%e_{c,i}+e_{c,i+1} \quad & j=i+1\\
 e_{c,j} \quad & \text{else}
\end{cases}
$$
is an action by isometries of $\Lambda_{\cut}$.

\item A bijection $I_{<0}\sqcup \{h_a,h_b\} \simeq I^{<0}_{\cut}\sqcup\{v_{n+1,a},v_{n+1,b}\}$ such that the induced map between sublattices of $\Lambda$ and $\Lambda_{{\cut}}$ is an isometry.
\end{enumerate}
\end{defn}

It is straightforward to extend the definition of the residue quantum torus from Definition~\ref{def:Tres} to the more general setting of  a locally glueable pair $(\quiver,\quiver_{\cut})$. 
Indeed, setting
$$
\dot{e}_{v_{j,a}} = {e}_{v_{j,a}}-{e}_{v_{j+1,a}} ,\qquad \dot{e}_{v_{j,b}}  = {e}_{v_{j,b}}-{e}_{v_{j+1,b}}, \quad j=1,\ldots, n
$$
the vectors $\{\dot{e}_{v_{j,a}} - \eps_a\eps_b \dot{e}_{v_{j,b}}\}$ lie in the kernel of the skew form on $\Lambda_{\cut}$ and so the form descends to the quotient lattice $\overline{\Lambda}_{\cut}$. We write $\Tc(\overline{\Lambda}_{\cut})$ for the corresponding quantum torus, which again contains elements $Y_{c,\beta^\vee}$ for each element of the root lattice as defined by formula~\eqref{eq:Y-coweights}. By our assumption on the isometric nature of the Weyl group action in part (2) of Definition~\ref{def:glueable-pair}, the elements 
$\dot e_{v_j}\in \overline{\Lambda}_{\cut}$ given by the images of the $\dot{e}_{v_{j,a}}$ in the symplectically reduced lattice have zero pairing with all mutable directions. So we can use them to define mutation-invariant elements $Y_{c,\alpha}$ of  $\Tc(\overline{\Lambda}_{\cut})$ as the unique ones of the form $Y_{\sum_jm_j\dot{e}_{v_j}}$ satisfying the commutation relations~\eqref{eq:xrel2}. These elements define an internal grading as well as an Ore denominator set via \eqref{eq:denom-set}, and by assumption we have an action of the Weyl group on the localized quantum torus $\Tc_{\loc}(\overline{\Lambda}_{\cut})$ by cluster transformations. Hence we can define the residue quantum torus $\Tc_{\res}(\quiver_{\cut})$ associated to $\quiver_{\cut}$ as the set of all elements of $\Tc_{\loc}(\overline{\Lambda}_{\cut})$ satisfying the conditions of Definition~\ref{def:Tres}. Writing $\Lbb(Q^{\mathrm{fr}})$ for the universal Laurent ring associated to the quiver $\quiver^{\mathrm{fr}}$ obtained from $\quiver$ by freezing all directions except $\{s_j,t_j\}_{j=1}^n$, a straightforward adaptation of the construction used to prove Proposition~\ref{prop:local-gluing-pgl} yields a local gluing isomorphism
\begin{align}
\label{eq:general-local-gluing}
\eta \colon \Lbb(\quiver^{\mathrm{fr}}) \simeq \Tc_{\res}(\quiver_{\cut}).
\end{align}

\subsection{Local gluing isomorphism for $G=SL_{n+1}$.}
\label{sec:gluing-sln}
We now briefly discuss how to extend the local gluing map~\eqref{eq:local-gluing-pgl} to the larger quantum torus $\Tc^\Ac(Q_{\tri;c})\supset\Tc(Q_{\tri;c})$ associated to the lattice $\Xi({\tri;c})\supset\Lambda(\tri;c)$. 
%Recall that the $q=1$ specialization of $\Tc^\Ac(Q_{\tri;c})$ is the coordinate ring of a cluster chart on the moduli space $\Pc^\diamond_{G,S}$ for $G=SL_{n+1}$, while the $q=1$ specialization of $\Tc({\tri;c})$ is the coordinate ring of a cluster chart on the corresponding space moduli space for $G=PGL_{n+1}$. 
The main technical\footnote{We would like to have a better general technique for handling locality properties of compatible pairs than the fairly ad hoc one we explain in this section.} issue here is that, unlike the situation for $e$-bases, the skew-pairing between elements of the $\xi$-basis associated to an ideal graph does not have an obviously local form. For this reason, our strategy will be to embed the lattice $\Xi(\tri;c)$ into a lattice obtained by a rational rescaling of $\Lambda(\tri;c)$, then prove that the $PGL_{n+1}$ local gluing map extends to the algebras associated to these rescaled lattices, and then finally check that the resulting isomorphism respects integrality with respect to the original $\Xi$-lattices.

%Recalling the notations from Section~\ref{subsec-quiver-mut}, we write $I_S^{>0}$ for the subset of the index set $I(Q_{\tri;c})$ formed by the faces $\{s_j,t_j\}_{j=1}^n$ of the graph $\Gamma_{\tri;c}$, and $I_S^{\leq0}$ for its complement. Similarly, we write $I_{S}^{\geq0}$ for the $I_{S}^{>0}\sqcup\{h_+,h_-\}$ and $I_S^{<0}$ for its complement in $I(Q_{\tri;c})$. So we have sublattices
%$$
%\Xi^{<0}_{S}\subset \Xi^{\leq 0}_{S}\subset \Xi_S\supset \Xi^{\geq 0}_{S}\supset \Xi^{>0}_S
%$$
%spanned by the corresponding subsets of the basis $\{\xi_i|i\in I(Q_{\tri;c})\}$ for $\Xi_S$. 

Consider the compatible pair $(\Xi^{\mathrm{fr}}(\tri;c),\Lambda^{\mathrm{fr}}(\tri;c))$ obtained from $(\Xi(\tri;c),\Lambda(\tri;c))$ by freezing all directions except for those labelled by $I_S^{>0}$.
As in Definition~\ref{eq:local-S-torus}, we denote the corresponding $\mathcal{A}$-variable universal Laurent ring by
$$
\Lbb(\Xi^{\mathrm{fr}}_{\tri;c})\subset \mathcal{T}^\Ac(\tri;c). 
$$
% Its universal Laurent ring $\Lbb_{\Xi_S^{\mathrm{fr}}}$ satisfies
%$$
%\Lbb_{\Xi_S}\subset \Lbb_{\Xi_S^{\mathrm{fr}}}\subset\mathcal{T}^\Ac(\tri;c). 
%$$
%As a module over $\mathbb{Z}[\red{\nu^{\pm1}}]$, we have 
%\begin{align}
%\label{eq:torus-factor}
%\Tc_{\Xi_S} = \Tc_{\Xi^{<0}_S}\otimes_{\mathbb{Z}[\nu^{\pm1}]}\Tc_{\Xi^{\geq0}_S}.
%\end{align}
%Although the two factors in the decomposition define subalgebras, these subalgebras do not commute since the summands $\Xi^{<0}$ and $\Xi^{\geq0}$ are not skew-orthogonal.

%We give the torus $\widetilde{\mathcal{T}}_{\Xi_S}$ a $\mathbb{Q}$-bigrading by setting
%$$
%\deg(\xi_{h_+}) = (n+1,0) \quad \deg(\xi_{h_-}) = (0,n+1), \quad \deg (\xi_{s_j})= (j,n+1-j) =\deg (\xi_{t_j}).
%$$
%Inspecting the quiver $\widehat Q_{\mathrm{Toda}}^n$ it is clear that each vertex has zero `net flux' with respect to this grading. So the cross-ratio elements $e_{s_j},e_{t_j}$ are of degree $(0,0)$, which implies mutation commutes with projection to graded components and thus $\widetilde\Lbb_{\Xi_S^{\mathrm{fr}}}$ is a graded subring in $\widetilde{\mathcal{T}}_{\Xi_S}$. 

Recall that the relation between the $e$- and $\xi$-bases is encoded by the ensemble matrix $p_S=p_S(\tri;c)$. Consider the block decomposition of $p_S$ with respect to the partition
 $$
 I(Q_{\tri;c}) =I_S^{\leq0} \sqcup I_S^{>0},
 $$
from Notation~\ref{not:LambdaS-split}, which we write as
\begin{align}
\label{eq:block-mat}
p_S \;=\;
\left[\begin{array}{c|c}
A & B \\ \hline
C & D
\end{array}\right].
\end{align}
Here $A$ is the $I_S^{\leq0}\times I_S^{\leq0}$ block, $B$ is the  $I_S^{\leq0}\times I_S^{>0}$ block, etc. Recall that the $(i,j)$-entry of this matrix records the coefficient of $\xi_i$ in the expansion of $e_j$ in the basis $\{\xi_i\}$ for $\Xi$.

Inspecting the quiver $Q_{\tri;c}$, it is clear that the only nonzero columns of $C$ are the two associated to the handles $(h_+,h_-)$, each of which has exactly two nonzero entries: 
$$
C_{s_1,h_+}=-1=C_{s_n,h_-},\qquad C_{t_1,h_+}=1=C_{t_n,h_-}.
$$
Similarly, the only nonzero rows of $B$ are those indexed by $h_+,h_-$:
$$
B_{h_+,s_1}=1=B_{h_-,s_n},\qquad B_{h_+,t_1}=-1=B_{h_-,t_n}.
$$
The matrix $D$ is invertible, with the nonzero inverse entries of $D^{-1}$ being given by 
\begin{align}
\label{eq:dinv}
D^{-1}_{s_i,t_j} = -\mathfrak{A}_{ij}, \quad D^{-1}_{t_i,s_j} = \mathfrak{A}_{ij}
\end{align}
where $\mathfrak{A}$ is the Cartan matrix of type $A_n$.
It follows from these observations that $BD^{-1}C=0$, and so the inverse ensemble matrix is given by
\begin{align}
\label{eq:block-inverse}
(p_S)^{-1} \;=\;
\left[\begin{array}{c|c}
A^{-1} & -A^{-1}BD^{-1} \\ \hline
-D^{-1}CA^{-1} & D^{-1} + D^{-1}CA^{-1}BD^{-1}
\end{array}\right].
\end{align}
The $(i,j)$-entry of this matrix records the coefficient of $e_i$ in the $\mathbb{Q}$-expansion of $\xi_j$ in the basis $\{e_i\}$ for $\Lambda^{\mathbb{Q}}_{Q_{\tri;c}}$.

\begin{lemma}
\label{lem:pin}
Any element of $\Lambda_S^{\mathbb{Q}}$ corresponding to a vector in the column-span of the rank-2 matrix $D^{-1}CA^{-1}$ is a $\mathbb{Q}$-linear combination of
$$
z^{>0}_+ = \sum_{j=1}^n j(e_{s_j}+e_{t_j}), \quad z^{>0}_- = \sum_{j=1}^n j(e_{s_{n-j+1}}+e_{t_{n-j+1}}),
$$
and the projection to $\Lambda_{\geq0}$ of any $\xi_\ell$ with $\ell\in I_{\leq0}$ is a $\mathbb{Q}$-linear combination of the vectors
\begin{align}
\label{eq:zvex}
z_+ = (n+1)e_{h_+} + \sum_{j=1}^n j(e_{s_j}+e_{t_j}), \quad z_- = (n+1)e_{h_-} + \sum_{j=1}^n j(e_{s_{n-j+1}}+e_{t_{n-j+1}}).
\end{align}
\end{lemma}
\begin{proof}
This can be seen by directly computing $D^{-1}C$, which has exactly two nonzero columns in positions $h_\pm$. Alternatively, one can argue as follows: the columns of $D^{-1}CA^{-1}$ are given by the projection of basis vectors $\xi_\ell,\ell\in I_S^{\leq0}$ to $\Lambda^{\mathbb{Q}}_{I_S^{>0}}$ with respect to the (non-orthogonal) direct sum decomposition
\begin{align}
\label{eq:e-sum}
\Lambda^{\mathbb{Q}}_{I_S}= \Lambda^{\mathbb{Q}}_{I_S^{\leq0}} \oplus \Lambda^{\mathbb{Q}}_{I_S^{>0}}
\end{align}
induced by partitioning the basis $\{e_\ell\}$ according to $I_S=I_S^{\leq0}\sqcup I_S^{>0}$. By definition of a compatible pair, any $\xi_\ell$ with $\ell\in I_S^{\leq0}$ must be skew-orthogonal to all $e_{s_j},e_{t_j}$ since the latter vectors are mutable. An easy calculation shows that the skew-orthogonal complement of $\{e_{s_j},e_{t_j}\}_{j=1}^n$ is spanned by the vectors $e_\ell$ with $\ell\in I_{S}^{<0}$ together with the $z_\pm$.
Taking their projections to $\Lambda^{\mathbb{Q}}_{I_S^{>0}}$, we get the claims of the Lemma.
\end{proof}
\begin{remark}
\label{rmk:pin}
Note that the coefficients of $z_+,z_-$ in in the expansion of any  $\overline{\xi}_i$ with $i\in I_S^{\leq0}$  are determined by the matrix $A$, since the coefficients of $e_{h_\pm}$ in that expansion are entries of $A^{-1}$.
\end{remark}

%Using the lemma we can understand the cross-relations between the two factors in the decomposition~\eqref{eq:torus-factor}. 
%\begin{cor}
%For any $i\in I_S^{<0}$, the values of the pairings $(\xi_{h_\pm},\xi_i)$ are determined by the entries of the matrix $A$ in~\eqref{eq:block-mat}.
%\end{cor}
%\begin{proof}
%Write
%$$
%\pi^{\geq0}_{e}: \Lambda_S^\mathbb{Q}\rightarrow \Lambda^{\mathbb{Q}}_{I_S^{\geq0}},\qquad \pi^{<0}_{e}:  \Lambda_S^\mathbb{Q}\rightarrow \Lambda^{\mathbb{Q}}_{I_S^{<0}}
%$$
%for the projectors associated to the decomposition~\eqref{eq:e-sum}. Similarly, we define $\pi^{\leq0}_{e},\pi^{>0}_{e}$. The pairings $(\pi_e^{\leq0}(\xi_{h_{\pm}}),\pi_e^{\leq0}(\xi_{i}))$ are determined by the matrix of the skew-form on $\Lambda^{\mathbb{Q}}_{I_S^{\leq0}}$ in the basis $\{e_i\}$ (which are read off from the matrix $A$), and the expansion coefficients of $\xi_{h_{\pm}},\xi_i$ with respect to this basis (recorded by the matrix $A^{-1}$). On the other hand, we know from the lemma that the projections $\pi^{\geq0}_e(\xi_i)$ and $\pi^{\geq0}_e(\xi_{h_\pm})$ are $\mathbb{Q}$-linear combination of $z_+,z_-$ and hence skew-orthogonal. Since
%$$
%(\pi_e^{<0}(\xi),z_\pm) = (\pi_e^{<0}(\xi),\pi_e^{\leq 0}(z_\pm)), \quad (e_{h_+},e_{h_-})=0
%$$
%it follows that we have
%$$
%(\xi_{h_\pm},\xi_i)=(\pi_e^{\leq0}(\xi_{h_{\pm}}),\pi_e^{\leq0}(\xi_{i}))
%$$
%and the Corollary is proved.
%\end{proof}
It is useful to consider the slightly larger version  ${\mathcal{T}}_{\widetilde{\Xi}_S}$ of $\mathcal{T}_{\Xi_S}$ in which we replace the basis vectors $\xi_{h_\pm}$ by their rescalings $\xi_{h_\pm}/(n+1)$. Here, and in the remainder of this section, we also assume we have extended scalars in the base ring to incorporate the needed fractional powers of $q$. 

%The mutation of elements of  ${\mathcal{T}}_{\widetilde{\Xi}_S}$ in directions $I^{>0}_{S}$  is well-defined since being a compatible pair implies we have $(\xi_{h_\pm},e_{\ell})=0$ for $\ell\in I_{S}^{>0}$. So there is a well-defined universal Laurent ring $\Lbb_{\widetilde{\Xi}_S^{\mathrm{fr}}}$ in the extended quantum torus ${\mathcal{T}}_{\widetilde{\Xi}_S}$, containing $\Lbb_{{\Xi}_S^{\mathrm{fr}}}$ as a subring.
Now consider the following elements of the extended lattice $\widetilde\Xi_S$:
\begin{align}
\label{eq:form1}
\tilde\xi_{s_j} = \xi_{s_j}-\frac{j\xi_{h_+} + (n+1-j)\xi_{h_-}}{n+1}, \qquad \tilde\xi_{t_j} = \xi_{t_j}-\frac{j\xi_{h_+} + (n+1-j)\xi_{h_-}}{n+1}.
\end{align}
We write $\widetilde{\Xi}_{S_{>0}}$ for the lattice spanned by the vectors~\eqref{eq:form1}.%Recall the lattice $\Xi_{SL_{n+1}}$ associated to the $SL_{n+1}$ Toda chain cluster algebra. It is spanned by basis vectors $\xi^{Toda}_\ell$, $\ell\in\{s_j,t_j\}_{j=1}^n$, and we write $e^{Toda}_\ell$, $\ell\in\{s_j,t_j\}_{j=1}^n$ for the corresponding $e$-basis for $\Lambda_{PGL_{n+1}}\subset\Xi_{SL_{n+1}}$.
%In accordance with our convention for labelling the nodes of $\widehat{Q}^{Toda}$, this indexing is such that $(s_1,t_1)$ are adjacent to $h_-$, and $(s_n,t_n)$ to $h_+$.
\begin{lemma}
\label{lem:col-ops}
\begin{enumerate}
\item
There exists an integer $N=M(n+1)$ divisible by $(n+1)$ such that 
$$
\widetilde{\Xi}_S \subset \widetilde\Lambda_S = \frac{1}{N}\Lambda_{<0}(\tri) \oplus \mathbb{Z}\langle\frac{z_{\pm}}{N}\rangle \oplus \Lambda_{>0}(\tri).
$$ 
\item Writing $\widetilde{\Lambda}_{GL_{n+1};\zeta}$ for the lattice spanned by the vectors $\{e_{s_{j}}(\Lambda_{GL_{n+1}}),e_{t_{j}}(\Lambda_{GL_{n+1}})\}_{j=1}^n$ in the basis~\eqref{eq:toda-xi-basis} along with 
\begin{align*}
\widetilde{\xi}_{s_{n+1}}&=  (p_{\omega_{n+1}}+x_{\omega_{n+1}})/N,\\
\widetilde{\xi}_{t_{n+1}}&= x_{\omega_{n+1}}/N,\\
\widetilde\zeta &= \zeta/M,
\end{align*}
the assignments
\begin{align}
\label{eq:st-map}
&\tilde\xi_{s_j} \longmapsto \xi_{s_j}({\Lambda}_{GL_{n+1}})-jM\tilde\xi_{s_{n+1}}({\Lambda}_{GL_{n+1}}),\\
&  \tilde\xi_{t_j} \longmapsto \xi_{t_j}({\Lambda}_{GL_{n+1}})-jM\tilde\xi_{t_{n+1}}({\Lambda}_{GL_{n+1}})
\end{align}
for $1 \le j \le n$ and
$$
\frac{z_-}{N}\longmapsto -\frac{\zeta}{M}-\frac{p_{\omega_{n+1}}}{N}, \qquad \frac{z_+}{N}\longmapsto \frac{p_{\omega_{n+1}}}{N},
$$
define an isometric embedding  into $\widetilde{\Lambda}_{GL_{n+1};\zeta}$ of the sublattice $\widetilde\Lambda_{S_{\geq0}}$ spanned by $\widetilde\Xi_{S_{>0}}$ and the $z_{\pm}/N$ as the orthogonal complement of $p_{\omega_{n+1}}$.
%whose image coincides with the orthogonal complement of $p_{\omega_{n+1}}$, and 
\item 
Under the embedding above, we have
\begin{align}
\label{eq:e-to-e}
e_\ell(\Lambda(\tri;c)) \mapsto e_\ell(\Lambda_{GL_{n+1}}), \quad \ell \in \{s_j,t_j\}_{j=1}^n.
\end{align}
\end{enumerate}
%It restricts to an isometry between the rank-$2n$ sublattice in $\Lambda_S$ spanned by $\{e_{s_j},e_{t_j}\}_{j=1}^n$ and $\Lambda_{SL_{n+1}}$.
\end{lemma}
\begin{proof}
By Lemma~\ref{lem:pin}, we know there exists $M$ such that each $\xi_{\ell}$ with $\ell\in I_{S}^{\leq0}$ lies in   $ \frac{1}{M}\Lambda_{<0}(\tri) \oplus \mathbb{Z}\{\frac{z_{\pm}}{M}\}$. Now consider the block column operation on the matrix $p_S^{-1}$ given by 
$$
\mathrm{col}_2 \mapsto  \mathrm{col}_2-\mathrm{col}_1BD^{-1}.
$$ 
Then the resulting matrix has second column given by $(0,D^{-1})^t$.
Since $B$ is zero except for its rows corresponding to $h_\pm$, recalling the formula~\eqref{eq:dinv} for $D^{-1}$ we see that linear combinations of $e_\ell$ corresponding to these columns are exactly the $\tilde\xi_{s_j},\tilde\xi_{t_j}$, and so taking $N=M(n+1)$ implies the first claim of the Lemma.
By the identification of the Toda quiver as a subquiver in $Q_{\tri;c}$, it is clear that the map $e_\ell(\Lambda(\tri;c))\mapsto e_\ell(\Lambda_{GL_{n+1}}), \ell\in I_{S}^{>0}$ is an isometry. Recall also that $(z_\pm,e_\ell)=0$ for all $\ell\in I_{S}^{>0}$. Since the $\tilde\xi_{s_j},\tilde\xi_{t_j}$ are expressed solely in terms of these $e_\ell$ via the matrix $D^{-1}$, an easy calculation shows that the map~\eqref{eq:st-map} defines an isometric embedding of $\widetilde{\Xi}_{>0}$  into $\widetilde{\Lambda}_{GL_{n+1}}$ satisfying~\eqref{eq:e-to-e}.

%The set formed by  images of $\tilde\xi_{s_j},\tilde\xi_{t_j}$ can be completed to a basis in $\widetilde{\Lambda}_{GL_{n+1},\zeta}$ by adjoining $x_{\omega_{n+1}}/(n+1),p_{\omega_{n+1}}/(n+1),\zeta$. So the rest of the Lemma follows from the observation made in the proof of Lemma~\ref{lem:pin} that the $z_\pm$ are orthogonal to the $e_{s_j},e_{t_j}$ and hence to the $\tilde\xi_{s_j},\tilde\xi_{t_j}$.
\end{proof}
%\begin{cor}
%The map
%\begin{align}
%Y_{\widetilde{\xi}_{s_j}} \mapsto , \quad Y_{\widetilde{\xi}_{t_j}} \mapsto
%\end{align}
%$$
%Y_{z_-/N}\mapsto Y_\zeta ,\qquad Y_{z_+/N}\mapsto P_{\omega_{n+1}}^{\frac{1}{n+1}}
%$$
%is an embedding of algebras.
%\end{cor}
Since each $e_\ell$ with $\ell\in I_{S}^{>0}$ is orthogonal to the $z_\pm$ and $\Lambda_{<0}$, we have a well-defined universal Laurent ring $\Lbb(\widetilde{\Lambda}^{\mathrm{fr}}_{\tri;c})$ with  mutable directions $I_{S}^{>0}$. On the other hand, consider the extended universal Laurent ring $\widetilde{\Lambda}_{GL_{n+1},\zeta}\supset {\Lambda}_{GL_{n+1},\zeta}$ with mutable directions $\{e_{s_j},e_{t_j}\}_{j=1}^n$.
Since by its definition in point (2) of the Lemma $\widetilde{\Lambda}_{GL_{n+1},\zeta}$ has a $\mathbb{Z}$-basis consisting of the mutable directions and orthogonal frozen directions, its universal Laurent ring is is obtained from $\Lbb_{{\Lambda}_{GL_{n+1},\zeta_*}}$ by adjoining the corresponding roots of frozen generators and $q$:
$$
\Lbb_{\widetilde{\Lambda}_{GL_{n+1},\zeta_*}} = \Lbb_{{\Lambda}_{GL_{n+1}}}\langle q^{\frac{1}{N}},Y_{\zeta/M},P_{\omega_{n+1}}^{\frac{1}{N}},X_{\omega_{n+1}}^{\frac{1}{N}}\rangle.
$$
Then as in Lemma~\ref{eq:ext-laurent}, the algebra $\Lbb(\widetilde{\Lambda}^{\mathrm{fr}}_{\tri;c})$  embeds into $\Lbb_{\widetilde{\Lambda}_{GL_{n+1},\zeta_*}} $ as the centralizer of $P_{\omega_{n+1}/N}$.

We extend the action of $\Lbb_{{\Lambda}_{GL_{n+1},\zeta_*}}$ on $\Fcr[S_{\leq0}]$ to one of $\Lbb_{\widetilde{\Lambda}_{GL_{n+1},\zeta_*}}$ on the space of $\widetilde{\Fcr}[S_{\leq0}]$ of $\mathcal{T}_{\Lambda_{\leq0}}\langle q^{\frac{1}{N}}, (-q)^{\frac{1}{2M}},Y_{h_\pm}^{\frac{1}{M}}\rangle$-valued functions on the lattice $\{\la\in \frac{1}{N}\mathbb{Z}^{n+1}~|~ \la_i-\la_{i+1}\in\mathbb{Z}\}$ in the natural way, letting

\begin{align*}
 (P_{\omega_{n+1}}^{\frac{1}{N}}\circ \phi)(\la) &=(-q)^{\frac{1}{M}}Y^{\frac{1}{M}}_{h_+}\phi(\la-\omega_{n+1}/N)\\
 (X_{\omega_{n+1}}^{\frac{1}{N}}\circ \phi)(\la) &=(-q)^{\frac{n}{2M}}q^{2|\lambda|/N}\phi(\la)\\
 (Y_{\zeta/M}\circ \phi)(\la) &=Y^{\frac{1}{M}}_{h_+}Y^{-\frac{1}{M}}_{h_-}\phi(\la).
\end{align*}
 This way we get an action of $\mathcal{T}_{\widetilde\Lambda_{<0}}\otimes \Lbb_{{\Lambda}_{GL_{n+1},\zeta_*}}$ and hence of its subalgebra $\Lbb_{\widetilde\Lambda_{S}}(I_S^{>0})$ on $\widetilde{\Fcr}[S_{\leq0}]$, preserving the subspace $\widetilde{\Vcr}[S_{\leq0}]$ of functions vanishing outside the cone of generalized partitions. The latter space has a Whittaker basis $W_\lambda$ obtained by extending the top Pieri rule to the shift by $\omega_{n+1}/N$, so that $W_{\omega_{n+1}/N}= W_{\omega_{n+1}}^{1/N}=\mathbf{w}^{1/N}$.

%Let $\Lbb_{\widetilde\Lambda^{\mathrm{fr}}_S}$ be the universal Laurent ring associated to the lattice $\widetilde\Lambda_S$ with mutable directions $\{e_\ell\}_{\ell\in I_S^{>0}}$. If we set 
%$$
%\widetilde\Lambda_{<0}(\tri) = \frac{1}{N}\Lambda_{<0}(\tri),\qquad  \widetilde\Lambda_{\geq0}(\tri) = \mathbb{Z}\{\frac{z_{\pm}}{N}\} \oplus \Xi_{S_{>0}},
%$$ 
%then as a module over the base ring we have
%\begin{align}
%\label{eq:ts-split}
%\Lbb_{\widetilde\Lambda^{\mathrm{fr}}_S} =  \Tc_{\widetilde\Lambda_{<0}(\tri)}\otimes \Lbb_{\widetilde\Lambda_{\geq0}(\tri)}.
%\end{align}
%%and by Lemma~\ref{eq:form1} we have
%%$$
%%\Lbb_{\widetilde\Lambda_{\geq0}(\tri)}\simeq \Lbb_{{\Lambda}_{GL_{n+1},\zeta_*}}[X_{\omega_{n+1}}^{\frac{1}{n+1}}]^{\langle P_{\omega_{n+1}} \rangle}.
%%$$
%\green{Can probably omit the stuff immediately above and say isomorphism comes as in $PGL_{n+1}$ setting from previous sec.}

Now let us discuss the symplectically-reduced lattice $\Xi^c({\tri';c_\pm})$ associated to a $c_\pm$-isolating cluster on $S'$. In the corresponding lattice $\Xi({\tri';c_\pm})$ before symplectic reduction, we make the following integral change of basis in accordance with~\eqref{eq:adot-to-a}:
\begin{align}
\label{eq:xidotbasis1}
\dot{\xi}_\ell = \begin{cases}
\sum_{r=1}^{i} \xi_{v^{\pm}_r}, \qquad & \ell = v^\pm_{i},\quad  1\leq i \leq n+1\\
\xi_{\ell},\qquad & \text{otherwise.}
\end{cases}
\end{align}
In this new basis we have
$$
(\dot{e}_{v_j^\pm},\dot{\xi}_{v_i^\pm}) = \delta_{+,-}\delta_{i,j}.
$$
Since the $e_{v_j^\pm}$ are mutable, for any $\ell \in I_{S'}^{\leq0}\setminus \{v_{n+1}^\pm\}$, we have $(e_{v_j^\pm},\xi_\ell)=0$ for all $j$ by definition of a compatible pair. In particular, such $\xi_\ell$ is orthogonal to all elements $\dot{e}_{v_j^\pm}$ and therefore to the elements $\dot{e}_{v_j^+}+\dot{e}_{v_j^-}$ defining the reduction constraint.
%Then the skew-orthogonal complement to the span of all $\{\xi_{c_\pm,i}\}_{i=1}^n$ is spanned by the vectors $\dot{\xi}_\ell$ for $\ell \notin \{v_i^\pm\}_{i=1}^n$, 
  Hence we can choose the following basis in the reduced lattice $\Xi_c=\Xi_{Q_{\tri';c_\pm}}^c$:
\begin{align}
\label{eq:xidotbasis}
{\overline{\xi}}_\ell = \begin{cases}
\dot{\xi}_{v^{\pm}_{n+1}}, \qquad & \ell = v_{n+1}^\pm\\
\dot{\xi}_{v^{+}_j}+\dot{\xi}_{v^{-}_{n+1}}-\dot{\xi}_{v^{-}_j} +j\dot\xi_{(c,j)}, \qquad & \ell = v_{i},\quad  1\leq j \leq n\\
\dot{\xi}_{\ell},\qquad & \text{otherwise.}
\end{cases}
\end{align}
Then we have the pairings $(\overline{\xi}_{v_i},\overline{\xi}_{v_j})=0$, while
$$
(\dot{e}_{v_i},\overline{\xi}_{v_j}) = \delta_{ij}, \quad 1\leq i,j\leq n. %\quad (\dot{e}_{v_i},\dot{e}_{v_j}) =0=(\overline{\xi}_{v_i},\overline{\xi}_{v_j}),\quad 1\leq i,j\leq n.
$$

Hence the degrees of elements of the $\overline{\xi}_\ell$ basis with respect to the internal grading by the $SL_{n+1}$ weight lattice are given by
$$
\deg(\overline{\xi}_\ell) = \begin{cases} 
\omega_j\quad &\ell = v_j, ~1\leq j\leq n\\
0\quad &\text{else.} 
\end{cases}
$$

\begin{remark}
\label{rmk:othersign}
\begin{enumerate}
\item 
Note that despite the minus sign in~\eqref{eq:xidotbasis}, the elements $\{\overline{\xi}_\ell \}$ are (images in $\Xi^c$ of) non-negative integer linear combinations of the original basis $\{\xi_\ell\}$ elements from the compatible pair $\Xi({\tri';c_\pm})$. In particular, the $Y_{\overline{\xi}_\ell}$ are elements of the universally Laurent ring $\Lbb^c$. 
\item Another natural choice of basis vectors in the $v_j$ directions comes from reversing the roles of $+/-$ by taking
$$
\overline{\xi}^o_{v_j} = \dot{\xi}_{v^{-}_j}+\dot{\xi}_{v^{+}_{n+1}}-\dot{\xi}_{v^{+}_j} -j\dot\xi_{(c,j)}.
$$
Like the $Y_{\overline{\xi}_{v_j}}$, the $Y_{\overline{\xi}^o_{v_j}}$  are universally Laurent. 
%We will use these vectors too later on.
\end{enumerate}
\end{remark}
\begin{notation}
\label{not:fr-wts}
For each fundamental weight $\omega_i$ of $SL_{n+1}$, we write $\omega_{c,i}$ for the element of $\Xi^c$ defined by
\begin{align}
\label{eq:fr-wts}
\omega_{c,i} = \xi_{c_-,i} = -\xi_{c_+,i}.
\end{align}
\end{notation}
Similarly to the $PGL_{n+1}$ case, we use the elements~\eqref{eq:fr-wts}  to associate to each element of the weight lattice of $SL_{n+1}$ a quantum torus element
$$
Y_{c,\lambda} =Y_{\sum m_k {\omega}_{c,k}} ,\qquad \lambda = \sum_{k=1}^n m_k\omega_k.
$$
%We also associate elements $Y_{c,\la^\vee}$ to \emph{dominant} weights $\la$ as follows:
%$$
%Y_{c,\lambda^\vee} =Y_{\sum m_k \overline{\xi}_{v_k}} ,\qquad \lambda = \sum_{k=1}^n m_k\omega_k, \quad m_k\geq0.
%$$
We define $\dot{p}_{S'}$ to be the change-of-basis matrix
$$
\dot{e}_j = \sum_i(\dot{p}_{S'})_{i,j} \overline{\xi}_i,
$$
and again consider its block decomposition with respect to the partition $I_{S'} = I_{S'}^{\leq0} \sqcup I_{S'}^{>0}$:
\[
\dot{p}_{S'} \;=\;
\left[\begin{array}{c|c}
\dot{A} & \dot{B} \\ \hline
\dot{C} & \dot{D}
\end{array}\right].
\]
 Now we have a canonical identification $I_S^{<0} \equiv I_{S'}^{<0}$ of index sets (coming from the identification of the corresponding sub graphs induced by the gluing diffeomorphism), which we extend to $I_S^{\leq0} \equiv I_{S'}^{\leq0}$ by identifying
$$
I_S\ni h_\pm \leftrightarrow v^\pm_{n+1}\in I_{S'}.
$$
 Since we have only changed the $e$-basis inside the  $I_{S'}^{>0}$ block, and in the $\xi$-basis have replaced $\xi_{v_{n+1}}^\pm$ by itself plus a linear combination of $\xi$-basis vectors indexed by $I_{S'}^{>0}$, it follows that under the identification of black-white graphs above we have $\dot{A}=A$, where $A$ is the block of $p_S$ from~\eqref{eq:block-mat}.  
 The nonzero entries of $\dot{B}$ and $\dot{C}$ are given by
 $$
 \dot{B}_{h_+,(c,n)} = -1, \qquad  \dot{B}_{h_-,(c,1)} = 1, \quad  \dot{C}_{(c,n),h_+} = 1, \qquad  \dot{C}_{(c,n),h_-} = -1
 $$
 The basis~\eqref{eq:xidotbasis} is chosen so that the only nonzero-entries of the block $\dot{D}$ are again encoded by the $A_n$-Cartan matrix:
$$
(\dot{D})_{v_i,(c,j)} = \mathfrak{A}_{ij}= -(\dot{D})_{(c,j),v_i}.
$$
In particular, we again have $\dot{B}\dot{D}^{-1}\dot{C}=0$ and it follows that 
\begin{align}
\label{eq:pinvdot}
\dot{p}_{S'}^{-1} \;=\;
\left[\begin{array}{c|c}
A^{-1} & -A^{-1}\dot{B}D^{-1} \\ \hline
-\dot{D}^{-1}\dot{C}A^{-1} & \dot{D}^{-1} + \dot{D}^{-1}\dot{C}A^{-1}\dot{B}\dot{D}^{-1}
\end{array}\right].
\end{align}
From the computation of $\dot{D}^{-1}\dot{C}$, it follows as in Lemma~\ref{lem:pin} that the projection to $\Lambda^c_{>0}$ of any $\overline\xi_{i}$ with $i\in I_{S}^{\leq0}$ is a rational multiple of the vector $\sum_{j=1}^nj\dot{e}_{v_j}$. Since all such $\overline\xi_{i}$ are fixed by the action of the Weyl group $W(c_\pm)$, we get
\begin{lemma}
\label{lem:pinprime}
The projection to $\Lambda^c_{\geq0}$ of any $\overline\xi_{i}$ with $i\in I_{S'}^{\leq0}$ is a rational linear combination of the vectors
\begin{align}
\label{eq:zprimevex}
z'_+ = (n+1)e_{h_+} + \sum_{j=1}^nj\dot{e}_{v_j}, \qquad z'_- = (n+1)e_{h_-} - \sum_{j=1}^nj\dot{e}_{v_j}.
\end{align}
\end{lemma}
As in Remark~\ref{rmk:pin},  the coefficients of $z'_+,z'_-$ in in the expansion of such a $\overline{\xi}_i$  are determined by the matrix $A$, since the coefficients of $e_{h_\pm}$ in that expansion are entries of $A^{-1}$.
% The entries of the matrix $\dot D$ are given as follows: writing $(\mathfrak{A}_{ij})$ for the Cartan matrix, we have
% $$
% \dot{e}_{v_i} = -\sum_j \mathfrak{A}_{ij}\overline{\xi}_{c_j}, \quad 1\leq i \leq n
% $$
% while
% $$
% \dot{e}_{c,i}= 
% $$

We again introduce the larger version $\widetilde\Xi_{c}$ of $\Xi_c$ in which we replace the basis vectors $\overline{\xi}_{h_\pm}$ by their rescalings $\overline{\xi}_{h_\pm}/(n+1)$.
% and write $\widetilde{\Tc}_{\res}^\Ac(\tri';c_\pm)$ for the corresponding residue quantum torus.
 By the form~\eqref{eq:pinvdot} of $\dot{p}_{S'}^{-1}$ , we see that $\widetilde{\Xi}_c$ sits inside the lattice
$$
 \widetilde\Lambda^c = \frac{1}{N}\Lambda^c_{<0}(\tri') \oplus \mathbb{Z}\langle\frac{z'_{\pm}}{N}\rangle \oplus \Lambda_{>0}(\tri').
$$
where $N$ is the same as in Lemma~\ref{lem:col-ops}, and the lattice $ \widetilde\Lambda^c_{\geq0} = \langle\frac{z'_{\pm}}{N}\rangle \oplus \Lambda_{>0}(\tri')$ has another basis given by the ${z'_{\pm}}/{N}$ together with the vectors $\{\omega_{c,j}\}_{j=1}^n$ along with the vectors
\begin{align}
\label{eq:form2}
%\nonumber \widetilde{e}_{c,\omega_j} &= \overline{\xi}_{c,j},\\
%\widetilde{e}_{c,\omega^\vee_{j}} &= \overline{\xi}_{v_j}- \frac{j\overline\xi_{h_+} + (n+1-j)\overline\xi_{h_-}}{n+1}, \qquad j=1,\ldots,n
\widetilde{\xi}_{v_j} &= \overline{\xi}_{v_j}- \frac{j\overline\xi_{h_+} + (n+1-j)\overline\xi_{h_-}}{n+1}, \qquad j=1,\ldots,n
\end{align}
obtained by applying the applying the same column operations used in Lemma~\ref{lem:col-ops}. Consider the lattice $\widetilde\Lambda_{\geq0}^{c;ext}(\tri')$ spanned by the vectors $\{e^+_{c,\omega_j},\widetilde{\xi}_{v_j}\}_{j=1}^n$ together with
\begin{align*}
\widetilde{e}^+_{c,\omega_{n+1}}&=  e^+_{c,\omega_{n+1}}/N,\\
\widetilde{e}_{c,\omega^\vee_{n+1}}&=  e_{c,\omega^\vee_{n+1}}/N,\\
\widetilde\zeta' &= \zeta'/M.
\end{align*}
The Weyl group action on  $\Lambda_{\geq0}^{c;ext}(\tri')$ naturally extends to $\widetilde\Lambda_{\geq0}^{c;ext}(\tri')$; note that the ${e}^+_{c,\omega_{n+1}},\widetilde{e}_{c,\omega^\vee_{n+1}}$ and $\zeta'$ are all fixed vectors. Using the formulas~\eqref{eq:braid-xi} for the braid group action, we easily derive:
\begin{lemma}
The map
\begin{align*}
{\omega}_{c,j} &\mapsto \widetilde{e}^+_{c,\omega_{j}}(\widetilde\Lambda_{\geq0}^{c;ext}) - jM\widetilde{e}^+_{c,\omega_{n+1}},\\
%\widetilde{\omega}^\vee_{c,j}
\widetilde{\xi}_{v_j}&\mapsto \widetilde{e}_{c,\omega^\vee_{j}}(\widetilde\Lambda_{\geq0}^{c;ext}) - jM\widetilde{e}^+_{c,\omega^\vee_{n+1}}
\end{align*}
$$
\frac{z'_-}{N}\mapsto -\frac{\zeta'}{M}-\widetilde{e}_{c,\omega^\vee_{n+1}},\quad \frac{z_+}{N}\mapsto \widetilde{e}_{c,\omega^\vee_{n+1}}
$$
defines a Weyl group-equivariant isometric embedding of $ \widetilde\Lambda^c_{\geq0}$ as the orthgonal complement to $\widetilde{e}_{c,\omega^\vee_{n+1}}$ in $\widetilde\Lambda_{\geq0}^{c;ext}(\tri')$.
\end{lemma}
Hence the residue quantum torus  $\widetilde{\Tc}_{\geq0;\res}(\tri';c_\pm)$ associated to the lattice $\widetilde\Lambda^c_{\geq0}$ spanned by the $\{\widetilde{e}_{c,\omega_j}, \widetilde{e}_{c,\omega^\vee_j}\}_{j=1}^n$ together with the $z'_\pm/N$ embeds as the centralizer of $Y_{{e}^+_{c,\omega_{n+1}}}$ in the residue quantum torus $\widetilde{\Tc}^{ext}_{\geq0;\res}(\tri';c_\pm)$ associated to $\Lambda_{\geq0}^{c;ext}(\tri')$. As an algebra,  the latter is just isomorphic  to $\Dres\langle q^{\frac{1}{N}},Y_{\zeta'/M},\mathbf{w}^{\frac{1}{n+1}},\mathbf{D}^{\frac{1}{n+1}}\rangle$. %
%
%Using~\eqref{eq:braid-xi} one can easily check that this basis transforms under the Weyl group action in accordance with the reflection representation: we have
%$$
%s_i(Y_{\widetilde{\xi}_{v_{j}}}) = \begin{cases}
%Y_{\widetilde{\xi}_{v_{j+1}}+\widetilde{\xi}_{v_{j-1}}-\widetilde{\xi}_{v_{j}}} & i=j\\
%Y_{\widetilde{\xi}_{v_{j}}} & \text{else,}
%\end{cases}
%$$
%where we understand $\widetilde{\xi}_{v_{n+1}}=\widetilde{\xi}_{v_{0}}=0$, and the vectors $\widetilde{e}^+_{c,\omega_{n+1}},\widetilde{e}_{c,\omega^\vee_{n+1}},\widetilde\zeta'$ are fixed points.
%So in parallel to Lemma~\ref{lem:cutloc}, the algebra 

%As modules over the base ring (though not as algebras), we have
%\begin{align}
%\label{eq:tsplit}
%{\Tc}_{\widetilde{\Lambda}^c;\res}(\tri';c_\pm) =  \Tc_{\widetilde\Lambda_{<0}(\tri')} \otimes \widehat{\Tc}_{\geq0;\res}^\Ac(\tri';c_\pm).
%\end{align}
%\begin{remark}
%\label{rmk:othersign2}
%If we set
%\begin{align}
%\label{eq:form3}
%\widetilde{\xi}^o_{v_j} = \overline{\xi}^o_{v_j}- \frac{(n+1-j)\overline\xi_{h_+} + j\overline\xi_{h_-}}{n+1},
%\end{align}
%then we compute in the same way that
%\begin{align}
%\label{eq:sumzero}
%\widetilde{\xi}^o_{v_j} = -\widetilde{\xi}_{v_j}.
%\end{align}
%\end{remark}
We again extend the action of $\Tc^{ext}_{\res}(Q_{\tri';c_\pm})$ on $\mathcal{T}_{\leq0}(\tri';c_{\pm})[\bs w^{\pm}]^{\sym}$ to an action of $\widetilde{\Tc}^{ext}_{\geq0;\res}(\tri';c_\pm)$ on 
$$
\widetilde{\Vcr}[S'_{\leq0}]=\mathcal{T}_{\leq0}(\tri';c_{\pm})\otimes\mathbb{Q}[q^{\pm1}][\bs w^{\pm}]^{\sym}[\mathbf{w}^{\frac{1}{N}},q^{\frac{1}{N}},(-q)^{\frac{1}{2M}},Y_{h_\pm}^{\frac{1}{M}}]
$$ 
by setting
\begin{align*}
 (Y_{e^+_{c,\omega_{n+1}}/N}\circ f)(w) &=(-q)^{1/M}Y^{\frac{1}{M}}_{h_+}w_1\ldots w_{n+1}f(w)\\
 (Y_{e_{c,\omega^\vee_{n+1}}/N}\circ f)(\la) &=f(\{w_i\mapsto q^{\frac{1}{N}}w_i\}_{i=1}^n)\\
 (Y_{\zeta'/M}\circ f)(w) &=Y^{\frac{1}{M}}_{h_+}Y^{-\frac{1}{M}}_{h_-}f(w).
\end{align*}
In this way we get an faithful representation of the residue quantum torus $\widetilde{\Tc}_{\geq0;\res}(\tri';c_\pm)$ associated to the lattice $\widetilde\Lambda^c$ on $\widetilde{\Vcr}[S'_{\leq0}]$.
Now as in the $PGL_{n+1}$ case, the algebraic Whittaker transform and the gluing identification $\iiota_\phi$ give rise to an isomorphism
$$
\widetilde{\Vcr}[S_{\leq0}]\simeq \widetilde{\Vcr}[S'_{\leq0}],
$$
and following Proposition~\ref{prop:local-gluing-pgl} this induces an isomorphism of algebras 
\begin{align}
\label{eq:big-iso}
\widetilde{\eta}_c^\Ac\colon\Lbb(\widetilde\Lambda^{\mathrm{fr}}_{\tri;c}) \simeq \widetilde{\Tc}_{\geq0;\res}(\tri';c_\pm)
\end{align}
extending~\eqref{eq:local-gluing-pgl}.
%Let $\widehat{\mathcal{D}}_{\res}=\Dres\langle \sqrt{-1}, q^{1/(n+1)},\mathbf{D}^{1/(n+1)},\mathbf{w}^{1/(n+1)}\rangle$ be the algebra obtained from $\Dres$ by adjoining the square root of -1 along with the  $(n+1)$-th roots of $q$ and the monomials
%$$
%\mathbf{D} = D_1\cdots D_{n+1},\quad \mathbf{w} = w_1\cdots w_{n+1}
%$$ 
%Now 
%Then in parallel to Lemmas~\ref{lem:cutloc} and~\ref{lem:col-ops} we have
%\begin{lemma}
%\label{lem:cut-hom}
%The Weyl group-invariant isomorphism of tori
%$$
%Y_{\widetilde{\xi}_{v_j}} \mapsto (D_1\cdots D_j)\mathbf{D}^{-j/n+1}, \quad Y_{\overline{\xi}_{c,j}} \mapsto (w_1\cdots w_j)\mathbf{w}^{-\frac{j}{n+1}}
%$$
%$$
%Y_{z_-'/N} \mapsto Z , \quad Y_{z_+'/N} \mapsto -q\mathbf{w}^{\frac{1}{n+1}}
%$$
%induces an algebra embedding of the residue quantum torus $\widehat{\Tc}_{\geq0;\res}^\Ac(\tri';c_\pm)$ into $\widehat{\mathcal{D}}_{\res}[Z^{\pm1}]$ \blue{as centralizer}.
%\end{lemma}
%\blue{Replace this bit below by statement about comparinson of extended actions and in parallel with $PGL_{n+1}$-story: passing to centrlizers we get an isomorphism of widetilde rings.}
%\blue{Final step is to show it respects $\Xi$-integrality.}
%We write 
%$$
%\widehat{\eta}\colon \Lbb_{{\Lambda}_{GL_{n+1},\zeta_*}}\langle \sqrt{-1},q^{\frac{1}{n+1}},P_{\omega_{n+1}}^{\frac{1}{n+1}},X_{\omega_{n+1}}^{\frac{1}{n+1}}\rangle \simeq \widehat{\mathcal{D}}_{\res}[Z^{\pm1}]
%$$ 
%for the natural extension of the isomorphism from Proposition~\ref{prop:alg-whit}.
Recall from our Convention~\ref{rmk:tori-roots} that the base ring over which both algebras $\Lbb(\Xi^{\mathrm{fr}}_{\tri;c})$ and  $\Tc^{\Ac}_{\res}(\tri';c_\pm)$ are defined is always taken to contain $(-q)^{\frac{n}{2}}$.
\begin{lemma}
\label{lem:integrality-preserved}
The map~\eqref{eq:big-iso} restricts to an isomorphism
\begin{align}
\label{eq:eta-A}
\eta^\Ac_c \colon \Lbb(\Xi^{\mathrm{fr}}_{\tri;c}) \simeq \Tc^{\Ac}_{\res}(\tri';c_\pm)
\end{align}
under which
\begin{align}
\label{eq:intxi1}
\eta^\Ac_c(H_k(c)) = \chi_{\omega_{k}}(c), \quad \chi_{\omega_k}(c) = \sum_{\nu\in W(\omega_k)} Y_{c,\nu}
\end{align}
and
\begin{align}
\label{eq:intxi2}
\eta^\Ac_c(\xi_{t_k}) = (-q)^{k(n+1-k)/2}\chi^{(0)}_{\omega^\vee_k}(\tri';c_\pm),
\end{align}
\begin{align}
\label{eq:intxi3}
\eta^\Ac_c(\xi_{s_k}) = (-q)^{k(n+1-k)/2}\chi^{(1)}_{\omega^\vee_k}(\tri';c_\pm),
\end{align}
where
$$
 \chi^{(j)}_{\omega_k}(\tri';c_\pm) = \sum_{\sigma\in W/\mathrm{Stab}(\omega_k)}\sigma\left(\prod_{\langle \alpha,\omega_k\rangle<0}\frac{1}{1-Y_{c,\alpha}}Y_{\overline{\xi}_{v_k}+j\omega_{c,k}}(\tri';c_\pm) \right)
$$
with the product being taken over all roots $\alpha$ having negative pairing with $\omega_k$.
%in~\eqref{eq:pres} restricts to give isomorphisms
%\[
%\label{eq:lociso}
%\begin{tikzcd}
%%\iota\otimes \eta_{SL_{n+1}}:
%  & \Lbb_{\Xi^{\mathrm{fr}}_S} 
%    \arrow[r,"\simeq" ]
%  & \mathcal{T}_{\Xi^c_{S'}}^{res} \\[1em]
%{} 
%  & \Lbb_{\Lambda^{\mathrm{fr}}_S} 
%    \arrow[u, hookrightarrow]
%    \arrow[r, "\simeq"]
%  & \mathcal{T}_{\Lambda^c_{S'}}^{res} 
%    \arrow[u, hookrightarrow]
%\end{tikzcd}
%\]
%isomorphically to $\mathcal{T}_{\Xi^c_{S'}}^{res}$.
\end{lemma}
%\green{Should we just make the extension of scalars by $(-q)^{n/2}$ part of the definition of the $\Ac$-quantum tori for the $SL_{n+1}$ moduli space when we first introduce it in Sec 5?}
\begin{proof}
%To see that ~\eqref{eq:local-A-iso} is in fact an algebra homomorphism, recall that 
%the $z_\pm$ (resp. $z'_\pm$) commute with all $e_\ell$ for $\ell\in I_{S}^{>0}$ (resp. $I_{S'}^{>0}$), and that the coefficient $e_{h_\pm}(S)$ in $z_\pm$ is the same as that of $e_{h_\pm}(S')$ in $z'_\pm$. Hence the map
%$$
%\iota\colon \widetilde\Lambda_{\leq0}(S)\simeq \widetilde\Lambda_{\leq0}(S'),\qquad e_\ell(S)\mapsto e_\ell(S'), ~\ell\in I_{S}^{<0}\equiv I_{S'}^{<0}, \quad z_\pm\mapsto z'_\pm
%$$ is an isometry, which implies~\eqref{eq:local-A-iso} is a homomorphism. 
%For the second statement, the image of any element of $\Lbb_{\widetilde\Lambda^{\mathrm{fr}}_S}$ (and hence that of any element of the subring $\Lbb^\Ac(Q^{\mathrm{fr}}_{\tri;c})$) satisfies the residue conditions from Definition~\ref{def:Tres}, and conversely the image of any element of $\Tc^{\Ac}_{\res}(\tri';c_\pm)$ under the inverse of~\ref{eq:local-A-iso} remains Laurent under mutation at $e_\ell$ with $\ell\in I_{S}^{>0}$. 
%Since $\eta(P_{\omega_{n+1}})=w_1\cdots w_{n+1}$,  it follows that it restricts to an isomorphism
What we need to check is that $\widetilde{\eta}_c^\Ac$ respects integrality with respect to the original $\Xi$-lattices: in other words, that after adjoining the indicated fractional powers of $q$ it maps an element of $\Lbb(\widetilde\Lambda^{\mathrm{fr}}_{\tri;c})$ which is a Laurent polynomial in the generators $\{\xi_\ell\}_{\ell\in I_S}$ to an element in $\Tc^{\Ac}_{\res}(\tri';c_\pm)$, and vice versa. It is clear from the constructions of the actions of the two algebras that we have
$$
\widetilde{\eta}_c^\Ac(Y_{z_\pm/N}) = Y_{z'_\pm/N}.
$$
Now recall that from Lemmas~\ref{lem:pin} and~\ref{lem:pinprime} and Remark~\ref{rmk:pin} that any $\xi_\ell$ with $\ell\in I_{\leq0}$ can be expressed as an element of $\Lambda_{S_{<0}}/N$ (resp. $\Lambda_{S'_{<0}}/N$) plus a multiple of $z_\pm/N$ (resp. $z'_\pm/N$), and that the coefficients of the latter vectors are determined by the entries of the  $I_{\leq0}$-blocks of the matrices $p_S^{-1},\dot p_{S'}^{-1}$.
Since we have observed in~\eqref{eq:pinvdot} that these blocks are identical , it follows that
$$
\widetilde{\eta}_c^\Ac(Y_{\xi_{\ell}}) = Y_{\overline{\xi}_{\ell}} \qquad \text{for all } \ell\in I_{\leq0}.
$$
In particular, we have $\widetilde{\eta}_c^\Ac(Y_{\xi_{h_\pm}}) = Y_{\overline{\xi}_{h_\pm}}$. So since the fractional coefficients of $\xi_{h_\pm}$ in formula~\eqref{eq:form1} are identical to those of $\overline{\xi}_{h_\pm}$ in~\eqref{eq:form2}, we may cancel the corresponding quantum torus elements to derive the intertwining formulas~\eqref{eq:intxi2} and~\eqref{eq:intxi3}, from which it follows that Laurent polynomials in the generators $\{\xi_\ell\}_{\ell\in I^{\geq0}_S}$ are mapped to elements of $\Tc^{\Ac}_{\res}(\tri';c_\pm)$ and vice versa. Finally, the intertwining relation for the fundamental Hamiltonians follows from the corresponding relation for the Toda Hamiltonians $H_k$ in Proposition~\ref{prop:alg-whit}.
\end{proof}

\begin{example}
\label{ex:sl2-hom}
If $G=SL_2$, then under the homomorphism from Lemma~\ref{lem:integrality-preserved} we have
\begin{align*}
(-q)^{-\frac{1}{2}}Y_{\widetilde\xi_{t_1}}&\longmapsto \frac{1}{1-Y_{c,\alpha_1}} Y_{\widetilde{\xi}_{v_1}} + \frac{1}{1-Y_{c,-\alpha_1}} Y_{-\widetilde{\xi}_{v_1}},\\
(-q)^{-\frac{1}{2}}Y_{\widetilde\xi_{s_1}}&\longmapsto \frac{1}{1-Y_{c,\alpha_1}} Y_{\widetilde{\xi}_{v_1}+\omega_{c,1}} + \frac{1}{1-Y_{c,-\alpha_1}} Y_{-\widetilde{\xi}_{v_1}-\omega_{c,1}}.
\end{align*}
Since $\eta_c(\xi_{h_\pm})=\overline{\xi}_{h_\pm}$, clearing the common fractional powers on both sides as in the proof of Lemma~\ref{lem:integrality-preserved} we see that
\begin{align*}
(-q)^{-\frac{1}{2}}Y_{\xi_{t_1}}&\longmapsto \frac{1}{1-Y_{c,\alpha_1}} Y_{\overline{\xi}_{v_1}} + \frac{1}{1-Y_{c,-\alpha_1}} Y_{\overline{\xi}_{v_1}^o} \\ %=\mathrm{Sym}\left(\frac{1}{1-Y_{{c,\alpha}}} D_{c,\omega^\vee}\right),\\
(-q)^{-\frac{1}{2}}Y_{\xi_{s_1}}&\longmapsto \frac{1}{1-Y_{c,\alpha_1}} Y_{\overline{\xi}_{v_1}+\omega_{c,1}} + \frac{1}{1-Y_{c,-\alpha_1}} Y_{\overline{\xi}_{v_1}^o-\omega_{c,1}},%=\mathrm{Sym}\left(\frac{Y_{{c,\omega}}}{1-Y_{{c,\alpha}}} D_{c,\omega^\vee}\right).
\end{align*}
where the lattice element $\overline{\xi}_{v_1}^o$ is defined as in Remark~\ref{rmk:othersign}.
%\red{Also do the universally Laurent element }
%$$
%Y_{\xi_s-\xi_t} + Y_{\xi_t-\xi_s}+ Y_{\xi_{h_+}+\xi_{h_-}-\xi_s-\xi_t}.
%$$
%Note that
%$$
% Y_{\xi_s - \xi_t} - Y_{-\xi_s - \xi_t} +Y_{\xi_t -\xi_s}\mapsto Y_{c,\omega} + Y_{c,-\omega}=e_1(\{Y_{c,\lambda}\}).
%$$

\end{example}

\begin{remark}
\label{rmk:laurent}
In the original formulation of the residue conditions~\eqref{eq:Lresidue-condition}, the elements $Y_{c,\alpha^\vee}$ do not lie in the universal Laurent ring $\Lbb(\Lambda^c_{S'})$. But in $\Ac$-variable residue quantum torus $\Tc^{\Ac}_{\res}(\tri';c_\pm)$ one can clear out the denominators in the cross-ratio coordinates $Y_{c,\alpha^\vee}$ to express the residue condition using only elements of the universal Laurent ring $\Lbb(\Xi^c_{S'})$. 
\end{remark}

%There is also an analog of Lemma~\ref{lem:integrality-preserved} for the algebras $\Lbb_{\Lambda^{\mathrm{fr}}_S} \simeq \mathcal{T}_{\Lambda^c_{S'}}^{res}$ associated to the adjoint group $G=PGL_{n+1}$. One straightforward way to do this is as follows: the ring $\Lbb_{\Xi^{\mathrm{fr}}_S}$ has a $\mathbb{Z}^2$-grading induced by the following grading on the quantum torus:
%$$
%\deg(\xi_{s_j})=\deg(\xi_{t_j})=(n+1-j,j),\quad \deg(\xi_{h_+}) = (n+1,0), \quad \deg(\xi_{h_-}) = (0,n+1),
%$$
%and $\deg(\xi_\ell)=0$ otherwise. The fact that this defines a grading on  $\Lbb_{\Xi^{\mathrm{fr}}_S}$ follows from the fact that each mutable cross-ratio $e_{s_j},e_{t_j}$ has degree zero.
\subsection{Global gluing isomorphism}
Now we extend the local gluing isomorphisms from the previous sections to global ones, i.e. to the level of universal Laurent rings. The first step is to define the global analog of the residue quantum torus associated to a $c_\pm$-isolating cluster on $S'$. 
%We write $\Lbb^c[P^+]\subset\Lbb^c$ for the subalgebra graded by the cone $P^+$ of dominant weights, as in the construction described by Remark~\ref{rmk:cones} of Section~\ref{sec:qcv}.

% Now let us set 
%\begin{align}
%\label{def:deltaplus}
%\Delta_{alg}(c_\pm) = \prod_{\alpha>0}\Psi_q(-qY_{\alpha_c})^{-1},
%\end{align}
%where the product is taken over all positive roots of $G$. By the difference equation~\eqref{eq:??}, conjugation by $\Delta_{alg}(c_\pm)$ defines an automorphism of $\mathrm{Frac}(\mathcal{T}^{c})$, and so we may conjugate $\mathbb{L}^{c}$ to an isomorphic subalgebra $\widetilde{\mathbb L}=\Delta_{alg}(c_\pm)^{-1}\mathbb{L}^{c}\Delta_{alg}(c_\pm)$ of the fraction field $\mathrm{Frac}(\mathcal{T}^{c})$ of the quantum torus
% $\mathbb{L}^{c}$. In fact, the algebra $\widetilde{\mathbb L}$  is contained in the Øre localization $\Lbb^c_{rat}$ of the universal Laurent ring $\Lbb^c$ at the multiplicative denominator set
%\begin{align}
%    \label{eq:denom-set}
%    Ø(c_\pm) = \left\{\prod_{r=1} (1-q^{2k_r}Y_{\beta_r}) \right\}_{\vec k,\vec\beta}
%\end{align}
%where $\vec k$ ranges over all finite length sequences of integers, and $\vec\beta$ over all sequences of positive roots (possibly with repetitions). 
%First \red{Explain that localization is mutation invariant.}
Since the elements of the denominator set~\eqref{eq:denom-set} are polynomials in frozen $\Ac$-variables, it follows that $ Ø(c_\pm)$ is invariant under mutation. Similarly, it is preserved under the action of the diagonally embedded Weyl group $W(c_\pm)$ by cluster transformations, since in any cluster Weyl group $W(c_\pm)$ acts by the reflection representation on commutative subalgebra generated by the elements $Y_{\omega_{c,j}}$. 
%Given an arbitrary cluster $Q$, we will write $\mathcal{T}^{\loc}({\Xi_Q^c})$ for the $ Ø(c_\pm)$-localized quantum torus associated to the lattice $\Xi_Q^c$, and $\mathcal{T}^{\loc}({\Lambda_Q^c})$ for the localized quantum torus associated to the lattice $\Lambda_Q^c$.

%$$
%a_{m\alpha}(Q_{\tri;c_\pm})  = \begin{cases}
%Y_{\overline{\xi}_{v_{\alpha_j}}} \quad & m>0\\
%Y_{\overline{\xi}_{-v_{\alpha_j}}} \quad & m<0
%\end{cases}
%$$
%where $\overline{\xi}_{v_{\alpha_j}}$ and $\overline{\xi}_{-v_{\alpha_j}}$ are the elements of $\Xi^c_{S'}$ defined by~\eqref{eq:dual-alpha-xis} and~\eqref{eq:dual-alpha-xis2} respectively.

\begin{defn}
\label{def:Lres}
Let $Q_{\tri';c_\pm}$ be an isolating cluster for $c_\pm$ subordinate to an ideal triangulation $\tri'$.
   The $G=SL_{n+1}$ (resp. $PGL_{n+1}$) \emph{residue universal Laurent ring} $\Lbb_{G,S';\phi}(\tri')$ based at $Q_{\tri';c_\pm}$ consists of all elements $A$ of the residue quantum torus $\Tc_{\res}^\Ac(\tri';c_\pm)$ (resp. $\Tc_{\res}(\tri';c_\pm)$) for which there exists an element $d=d(Y_{c,\alpha_1},\ldots Y_{c,\alpha_n})$ of the denominator set $Ø(c_\pm)$ such that $dA$ lies in the symplectically-reduced universal Laurent ring $\Lbb^c_{\Xi(\tri';c_\pm)}$ from~\eqref{eq:reduced-laurent-ring}.    
%      \begin{enumerate}
%            \item There exists an element $d=d(Y_{c,\alpha_1},\ldots Y_{c,\alpha_n})$ of the denominator set $Ø$ such that $dA^\nu$ is an element of the universal Laurent ring $\Lbb_{S';\Ac}^c$;
%            \item The element $A$ is invariant under the action of the diagonally embedded Weyl group $W_\Delta(c_\pm)$ by cluster transformations.
%            %Each graded component $A^\mu\in\mathcal{T}_{\Xi^c}^{loc}$
%        \item The element $A$ has at worst simple poles at the divisors $\{ q^{2k}Y_{e_{c,\alpha_1}}=1\}$ (i.e. the element $d$ above can be chosen squarefree), and for all weights $\mu$ and integers $k$, 
%\begin{align}
%\label{eq:Lresidue-condition}
%\mathrm{Res}_{\delta_{\alpha,k}}\left(A_\mu + A_{s_\alpha(\mu)+k\alpha}Y_{(\langle \mu,\alpha\rangle-k)\alpha^\vee}(Q_{\tri;c_\pm})\right)=0.
%%\left( A(Y_{2\overline\xi_{h_+}+\overline\xi_{h_-}}(Q_{\tri;c_\pm})^{|\langle\mu,\alpha_1^\vee\rangle-k|} + Y_{\overline\xi_{v_{\alpha_1}}}(Q_{\tri;c_\pm})^{|\langle\mu,\alpha_1^\vee\rangle-k|})\right)^{(\mu)} \quad \text{is regular on} \quad \delta_{\alpha_1,k}=\{ q^{2k}Y_{e_{c,\alpha_1}}=1\}.
%\end{align}
%            \end{enumerate}
\end{defn}
\begin{remark}
As observed in in Remark~\ref{rmk:laurent}, the condition~\eqref{eq:Lresidue-condition} can be equivalently expressed using elements of the universal Laurent ring in the isolating cluster $Q_{\tri;c_\pm}$. When expressed in terms of the coordinates of a general cluster, these elements will be complicated elements of the corresponding quantum torus, and hence the residue conditions will no longer take the simple binomial form from Definition~\ref{def:Tres}.
\end{remark}

\begin{remark}
\label{rmk:genLres}
More generally, given any locally glueable pair $(\quiver,\quiver_{\cut})$ we can define a corresponding residue universal Laurent ring $\Lbb_{\res}(\quiver_{\cut})$ as the set of all elements $A$ of the residue quantum torus $\Tc_{\res}(\quiver_{\cut})$ such that for some $d$ in the denominator set, $dA$ lies in the symplectically reduced universal Laurent ring associated to $\quiver_{\cut}$.
\end{remark}

%\begin{remark}
%\label{rmk:Lres-alg}
%    Condition (1) in the definition above says that $\Lbb_{res}^c$ is contained in the subalgebra $\mathcal{T}_{S\setminus c}\otimes \mathbb{SH}_{q,t=0}$ of $\mathcal{T}_{\bi_{S';c}}^{rat}\simeq \mathcal{T}_{S\setminus c}\otimes \mathcal{D}_q(T_{SL_{n+1}})^{rat}$. In particular, $\Lbb_{res}^c$ forms an algebra. 
%\end{remark}

The definition of $\Lbb_{G,S';\phi}$ is phrased in terms of a particular special isolating cluster coordinate system $Q_{\tri;c_\pm}$ for the pair of tacked circles $c_\pm$. 
%Indeed although they are quantum cluster monomials in the isolating cluster $Q_{\tri;c_\pm}$, in a general cluster the elements $Y_{2\overline\xi_{h_+}+\overline\xi_{h_-}}(Q_{\tri;c_\pm}),Y_{\overline\xi_{v_{\alpha_1}}}(Q_{\tri;c_\pm})$ will be complicated elements of the corresponding quantum torus.
 We now show that the definition is actually independent of the choice of isolating cluster $Q_{\tri;c_\pm}$.

\begin{prop}
\label{prop:residue-well-defined}
    The algebra $\Lbb_{G,S';\phi}$ does not depend on the choice of isolating cluster for the pair of tacked circles $c_\pm$: if $Q_{\tri'_1,c_\pm}$ and $Q_{\tri_2',c_\pm}$ are two such clusters and 
    $$
    \mu_{1,2} \colon Q_{\tri'_2,c_\pm}\rightarrow Q_{\tri'_1,c_\pm}
    $$
     the cluster transformation taking one to the other, we have
    $$
\mu_{1,2}\left(\Lbb_{G,S';\phi}(\tri'_2)\right) = \Lbb_{G,S';\phi}(\tri_1').
    $$
\end{prop}
\begin{proof}
We prove the proposition in the $G=SL_{n+1}$ case, which subsumes the $PGL_{n+1}$ one. Suppose that $A$ is an element of $\Lbb_{G,S';\phi}(\tri'_2)$. Then since the denominator set $Ø(c_\pm)$ is mutation invariant and $dA$ is universally Laurent for some $d\in Ø(c_\pm)$, it follows that $\mu_{1,2}(A)$ is an element of the $Ø(c_\pm)$-localized quantum torus $\Tc_{\res}(\tri';c_\pm)$. Moreover since $A$ is $W(c_\pm)$-invariant and the Weyl group acts by cluster transformations, the same is true of $\mu_{1,2}(A)$. So what remains to be proved is that if $A$ satisfies the residue conditions~\eqref{eq:Lresidue-condition} defined using isolating cluster  $Q_{\tri_2',c_\pm}$, then it also satisfies the analogous conditions defined using isolating cluster $Q_{\tri_1',c_\pm}$. 

Given a root $\alpha$, we abbreviate 
$$
Y_{\alpha^\vee} =Y_{c,\alpha^\vee}({\tri'_1;c})\in\Tc({\tri'_1;c}),\quad Y'_{\alpha^\vee} =Y_{c,\alpha^\vee}({\tri'_2;c_\pm}) \in \Tc(\tri'_2;c)
$$
for the elements defined by~\eqref{eq:Y-coweights} in clusters $Q_{\tri'_1;c}$ and $Q_{\tri'_2;c_\pm}$ respectively. 
%Note that since $\alpha$ is an element of the root (not just weight) lattice, $Y_{\alpha^\vee}$ is indeed an element of the torus $\mathcal{T}_{\Lambda^c}\subset\mathcal{T}_{\Xi^c}$ and not just the covering $\mathcal{T}_{\widetilde\Xi^c}$,  although it is not universally Laurent. 
For $l\in\mathbb{Z}$, set
$$
f_{l\alpha} = \mu_{1,2}(Y'_{l\alpha^\vee})Y_{l\alpha^\vee}^{-1}\in\Frac\left(\Tc_{{\tri_1';c}}\right)_0,
$$
where the subscript indicates the $f_{l\alpha}$ have degree zero with respect to the internal grading~\eqref{eq:internal-grading}.
Since the $Y_{\omega_{c,j}}$ are frozen $\mathscr{A}$-variables, no monomial $Y_{{c,\lambda}}$ can ever appear as the argument of a quantum dilogarithm in the automorphism part of $\mu_{1,2}$, and so $\mu_{1,2}(Y'_{l\alpha^\vee}) $ is regular and nonvanishing at each divisor $d_{\alpha,k}(c)$. In particular, the elements $f_{l\alpha}$ are regular at  all $d_{\alpha,k}(c)$. Note also that we have
\begin{align}
\label{eq:prod1}
1=f_0 = f_{l\alpha}f_{-l\alpha}^{[l\alpha]}. %= f_{l\alpha}[\lambda]f_{-l\alpha}[\lambda+l_\alpha],\quad \lambda\in P_{SL_{n+1}}.
\end{align}
Now the key point is that since the Weyl group $W(c_\pm)$ acts via the reflection representation~\eqref{eq:weyl-xi} in any isolating cluster, for each root $\beta$ and Weyl group element $w$ we have
$$
w\cdot Y_{\beta^\vee}= Y_{w(\beta)^\vee}, \qquad w\cdot Y'_{\beta^\vee} = Y'_{w(\beta)^\vee}.
$$
Since $s_\alpha(\alpha)=-\alpha$, this means that
 \begin{align}
 \label{eq:cocycle-equivariance}
% f_\nu(\bs w) = s_\alpha\cdot f_{s_\alpha(\nu)}(\bs w)=f_{s_\alpha(\nu)}(s_\alpha\cdot\bs w).
s_\alpha \cdot f_{l\alpha} = f_{-l\alpha}.
\end{align}
Recall from Remark~\ref{rmk:res-hom} that on the piece of internal degree zero, the restriction maps $|_{d_{\alpha,k}}$ are ring homomorphisms. Combining with~\eqref{eq:prod1}  and the property~\eqref{lem:trade} we get that for any $r\in\mathbb{Z}$,
% for any $\lambda\in P_{SL_{n+1}}$
\begin{align*}
1 &= f_{l\alpha}\big|_{d_{\alpha,r}}\cdot \left((s_\alpha f_{l\alpha})^{[l\alpha]}\right)\big |_{d_{\alpha,r}}\\
&=f_{l\alpha}\big|_{d_{\alpha,r}}\cdot  f_{l\alpha}^{[(l-r)\alpha]}\big |_{d_{\alpha,r}}
\end{align*}
and in particular $f_{l\alpha}^2\big|_{d_{\alpha,l}}=1.$
Since $f_{l\alpha}$ is obtained from a cluster monomial by conjugation by a product of quantum dilogarithms this forces the sign
\begin{align}
\label{eq:restricts-to-1}
f_{l\alpha}\big|_{d_{\alpha,l}}=1.
\end{align}

Now fix an element weight $\nu$ of the $SL_{n+1}$-weight lattice, and a pair $(\alpha,k)$. We set
$$
m = \langle \nu,\alpha\rangle-k,
$$
so that $s_\alpha(\nu)+k\alpha = \nu-m\alpha$.
%
%$$
%Y_{\xi_{v_{\alpha_1}}} = Y_{\overline\xi_{v_{\alpha_1}}}(Q_{\tri;c}),\qquad Y'_{\overline\xi_{v_{\alpha_1}}} = Y_{\overline\xi_{v_{\alpha_1}}}(Q_{\tri;c_\pm})
%$$
Then the condition~\eqref{eq:Lresidue-condition} for $\mu_{1,2}(A)$ defined relative to isolating cluster $Q(\tri'_1;c_\pm)$ says that
\begin{align*}
\Res_{\alpha,k}\left(\mu_{1,2}(A_{\nu})+\mu_{1,2}(A_{\nu-m\alpha}) Y_{m\alpha^\vee}\right) =0.
\end{align*}
%We suppose this condition holds and show it implies the corresponding one relative to isolating cluster $(\tri',c)$. 
If we take the difference of this condition and the image under $\mu_{1,2}$ of the residue condition for $A$ in the isolating cluster subordinate to  $\tri'_2$, and then multiply from the right by the $d_{\alpha,k}$-regular element $Y_{m\alpha^\vee}^{-1}$, we see that equivalence of the two residue conditions amounts to showing that
\begin{align}
\label{eq:cond-diff}
\Res_{\alpha,k}\left(\mu_{1,2}(A_{\nu-m\alpha})(f_{m\alpha}-1 ) \right) =0.
\end{align}
Working in the extended lattice $\tilde\Xi_c$ if necessary, we can factor
% Recall that $A_{\nu-m\alpha}$ has at worst simple poles at $\delta_{\alpha,k}$, and as observed above $Y_{m\alpha^\vee},\mu_{\tri,\tri'}(Y'_{l\alpha^\vee})\in \Tc^{loc;(\alpha,k)}_{0}$ are regular and nonvanishing there. 
%So splitting 
$
\mu_{1,2}(A_{\nu-m\alpha}) = BY_{(\nu-m\alpha)^\vee},
$
so that $ A_{\nu-m\alpha}f_{m\alpha} = Bf_{m\alpha}^{[\nu-m\alpha]}Y_{(\nu-m\alpha)^\vee}$.
Since $B$ has at most simple poles, formula~\eqref{eq:cond-diff} will follow provided we can show that
$$
f_{m\alpha}^{[\nu-m\alpha]}\big|_{d_{\alpha,k}}=1.
$$
But this follows from conjugating~\eqref{eq:restricts-to-1} by $Y_{(\nu-m\alpha)^\vee}$ and using~\eqref{eq:pi-ad}.
%\begin{align}
%\label{eq:pr1}
%\left(A\big |_{\delta_{\alpha,k}}\right)[\lambda]  =A[\lambda]\big |_{\delta_{\alpha,k+\langle\lambda,\alpha\rangle}}.
%\end{align}.
\end{proof}
\begin{remark}
\label{rmk:props}
\begin{enumerate}
\item The properties of the cluster transformation $\mu_{1,2}$ used in the proof of Proposition~~\ref{prop:residue-well-defined} are that it intertwines the reflection representation of $W(c_\pm)$ by monomial transformations in the two cutting clusters; that it commutes with all $Y_{c,\lambda}$ (and hence preserves the internal grading); and that the image of any $Y_{c,\beta^\vee}'$ is a subtraction-free noncommutative rational fraction regular at each divisor $d_{\alpha,k}(c)$.
 \item  Proposition~\ref{prop:residue-well-defined} can also be proved by a direct calculation using the factorization of the mutation sequence realizing the re-isolation move given in Section~\ref{sec:3-to-3}.
 \end{enumerate}
\end{remark}
Since diffeomorphisms send isolating triangulations to isolating triangulations, it follows from Proposition~\ref{prop:residue-well-defined} that the mapping class group $\Gamma_{S'}$ acts on $\Lbb_{G,S';\phi}$.  This action factors through the quotient by $\ha{\tau_{c_+}\tau_{c_{-}}}$ and thus gives rise to an action of the centralizer $\Gamma_{S;c}$ of $\tau_c$ in $\Gamma_S$ on $\Lbb_{G,S';\phi}$.

\begin{remark}
\label{rmk:Lcirc}
    The intersection of $\Lbb_{G,S';\phi}$ with the degree zero internal-graded piece of the residue quantum torus defines a subalgebra which we denote by $\Lbb_{G,S^\circ_c}$. If $A\in\Lbb_{S^\circ_c}$, the residue conditions imply that $A=A_0$ is regular at each divisor $d_{\alpha,k}$ and is thus an element of the non-localized quantum torus. 
\end{remark}

\begin{theorem}[Gluing universal Laurent rings in the case of an isolatable curve]
    \label{thm:alg-MF}

    % Then the homomorphism $\eta$ in~\eqref{eq:global-W-def} is an isomorphism onto its image $\Lbb^c_W$:
    Suppose $c$ is an oriented isolatable simple closed curve on $S$, and if $G$ has rank $n=1$ suppose that $S$ is not a torus with a single puncture or tacked circle. Then in any $c$-isolating cluster $Q_{S,\tri;c}$ the local isomomorphism $\eta_{c,\tri}^{\Ac}$ from~\eqref{eq:eta-A}  restricts to give isomorphisms of algebras
%    \begin{align}
%    \label{eq:boldW-def}
%        \eta_{c,\tri} \colon \Lbb_{G,S}[(-q)^{\frac{n}{2}}]\simeq \Lbb_{G,S';\phi}[(-q)^{\frac{n}{2}}].
%    \end{align}

\begin{equation}
\label{eq:boldW-def}
\begin{tikzcd}
\Lbb_{SL_{n+1},S} \arrow{r}{\eta_{c,\tri}}  & \Lbb_{SL_{n+1},S';\phi}  \\
\Lbb_{PGL_{n+1},S} \arrow{r}{\eta_{c,\tri}} \arrow[u,hook] & \Lbb_{PGL_{n+1},S';\phi}\arrow[u,hook]
\end{tikzcd}
\end{equation}

%    The isomorphism~\eqref{eq:boldW-def} does not depend on the choice of special coordinate system isolating the curve $c$, and is equivariant with respect to the actions by cluster transformations of the centralizer $\Gamma_{S;c}$ of the Dehn twist $D_c$ on the algebras $\Lbb_{S},\Lbb^c_{res}$ . It restricts to an isomorphism
%        \begin{align}
%    \label{eq:restricted-boldW}
%        \eta_c\colon \Lbb_S^{M_c^*\mathcal{O}_q(T/W)}\simeq \Lbb^c_{S^\circ},
%    \end{align}
%    where $\Lbb_S^{M_c^*\mathcal{O}_q(T/W)}$ is the subalgebra of $\Lbb_S$ consisting of elements which commute with all fundamental Hamiltonians $H_k(c)\in\Lbb_S$ associated to $c$.
\end{theorem}
\begin{proof}
%We first show that the homomorphism defined in~\eqref{eq:global-W-def} is an isomorphism onto its image $\Lbb^c_{res}$.
We prove the theorem using the intertwining relation for Baxter operators under Whittaker transform together with the 1-step criterion for universal Laurentness. Since the proof is the same for both groups $SL_{n+1}$ and $PGL_{n+1}$, in the following we suppress it in the notations, and assume that we have made the necessary extension of scalars in the simply-connected case. First we show that the image of $\eta_c$ is contained in the universal Laurent ring $\Lbb_{S';\phi}$. It is immediate from the local construction~\eqref{eq:algWhitS} that any element $\eta_{c,\tri}(A)$ in the image of $\Lbb_S$ remains Laurent under mutations at all quiver nodes in the cluster $Q_{S',\tri';c_\pm}$ except possibly for the `handles' $e_{v_j^\pm}$ for $j=1,\ldots, n+1$. On the other hand, unless $G$ is of type $A_1$ and $S$ is a torus with one puncture or tacked circle, we have in the cluster $Q_{S;\tri,c}$ two Baxter sequences of mutations associated to the Toda chain subquiver\footnote{The excluded case mentioned in the statement of the theorem is the degenerate one in which case the two vertices $h_\pm$ are amalgamated together.}. 
    For example, the Baxter sequence associated to the `positive' end of the Toda subquiver consists of the mutations 
    $$
\mu_{\mathrm{Baxter},+}=\mu_{s_1}\circ \mu_{t_1}\circ    \cdots \circ\mu_{s_n}\circ \mu_{t_n}\circ \mu_{h_+}.
    $$ 
    With respect to the sign convention~\eqref{eq:mon-mut} for mutation of bases, the cluster transformation $\mu_{\mathrm{Baxter},h_+}$ admits a factorization 
    $$
   \mu_{\mathrm{Baxter},h_+} = \mu_{\mathrm{Baxter},+}'\circ\mu^\sharp_{\mathrm{Baxter},+}.
    $$
%    where
%    $$
%    \mu^\sharp_{\mathrm{Baxter},h_+} = \mathrm{Ad}\left(\prod_{j=1}^{\substack{n \\ \longrightarrow}}\Psi(Y_{-e_{h_+}-\sum_{r=1}^{j-1}e_{s_r})\Psi(zP_{j}^{-1}). \circ \Psi(Y_{-e_{h_+}})\right)
%    $$
    On the side of the surface $S'$, consider the sequence of $n+1$ commuting mutations at the vertices $v^+_1,\ldots,v^+_{n+1}$ of the quiver $Q_{\tri';c_\pm}$. Using the same sign convention~\eqref{eq:mon-mut}, the corresponding cluster transformation factors as
     $$
   \mu_{\mathrm{hat},+} = \mu_{\mathrm{hat},+}'\circ\mu^\sharp_{\mathrm{hat},+},
    $$
    where 
\begin{align}
\label{eq:sharphat}
    \mu^\sharp_{\mathrm{hat},+} =\mathrm{Ad}\left( \prod_{j=1}^{n+1}\Psi(Y_{v^+_j}^{-1})\right).
\end{align}
   So recalling the action~\eqref{eq:action-def-S} of $Y_{h_+},Y_{s_k},Y_{t_k}$ on $M[S_{\leq0}]$ along with that of the $Y_{v_+}$ on $M[S'_{\leq0}]$ defined by~\eqref{eq:action-def-Sprime} and using the intertwining relation~\eqref{eq:aut-inter} for the inverse of the Baxter operator $Q^o(z)$ with $z=Y_{h_+}^{-1}$, we deduce that
\begin{align}
\label{eq:sharp-inter}
\mu^\sharp_{\mathrm{hat},+} \circ\eta_{c,\tri}=   \eta_{c,\tri}\circ \mu^\sharp_{\mathrm{Baxter},+}.
\end{align}
%   
%    the Baxter sequence of mutations $Q^{alg}(z_\pm)$ in the cluster $Q_{S;(\tri,c)}$ is intertwined under $\eta_c$ with  $R(z'_\pm)$, which we recognize as the product of inverses of compact $q$-dilogarithms $\prod_{j=1}^{n+1}\mu_{v_j^\pm}^\flat$ associated to the sequence of $n+1$ commuting mutations $\mu_{\bs v^\pm}$ at  vertices $v^\pm_1,\ldots,v^\pm_{n+1}$ in the cluster $Q_{S';(\tri,c)}$. 
%    
The mutated quiver $\mu_{Baxter,+}(Q_{S;\tri,c})$ again contains a Toda subquiver formed by a subset of basis vectors $\{e_\ell : \ell\in I'_{Toda}\}$, and recall from~\eqref{eq:longseq} that this set of vectors coincides with the image under $\mu^\sharp_{\mathrm{Baxter},+}$ of the vectors $\{e_\ell : \ell\in I^{>0}_{S}\}$ spanning the Toda subquiver of $Q_{S;\tri,c}$:
$$
\{e_\ell ~\big|~\ell\in I'_{Toda}\} = \{\mu_{\mathrm{Baxter},+}'(e_\ell) ~\big|~ \ell\in I^{>0}_{S}\}.
$$
Hence for any $A\in \Lbb_S(\tri;c)$, the fact that $\mu_{Baxter,+}(A)$ is universally Laurent implies that the element $\mu^\sharp_{\mathrm{Baxter},+}(A)$ of the quantum torus $\mathcal{T}(\tri;c)$ remains Laurent under mutation in each direction $\ell\in I^{>0}_{S}$, and so $\mu^\sharp_{\mathrm{Baxter},+}(A)$ lies in $\Lbb(\Xi^{\mathrm{fr}}_{\tri;c})$. Since $\eta_{c,\tri}$ maps $\Lbb(\Xi^{\mathrm{fr}}_{\tri;c})$ into the residue quantum torus, it follows from the intertwining relation~\eqref{eq:sharp-inter} that $\mu^\sharp_{\mathrm{hat},+}(\eta_{c,\tri}(A))$ and hence $\mu_{\mathrm{hat},+}(\eta_{c,\tri}(A))$ is again a Laurent polynomial after clearing its denominator $d\in Ø(c_\pm)$.
%    Moreover, recall from~\eqref{eq:longseq} that after performing the sequence  $\mu^\flat_{Baxter_\pm}$ the set of basis elements associated to the vertices of the Toda subquiver of  $\mu^\flat_{Baxter_\pm}(Q_{S;\tri,c})$ is identical (up to a permutation) with those in the initial Toda subquiver of $Q_{S;\tri,c}$. Hence $\mu^{\flat}_{Baxter_\pm}(A)$ again lies in $\Lbb_{{\Xi}_S^{\mathrm{fr}}}$, and therefore $\mu_{\bs v^\pm}(\eta_c(A))=\eta_c(\mu^\flat_{Baxter_\pm}(A))$ is again a Laurent polynomial.

But since $\mu_{\mathrm{hat},+}$ is a sequence of mutations in distinct commuting directions, it follows that $d\eta_c(A)$ must remain Laurent under each individual mutation $\mu_{v_i^+}$ in the sequence. 
The same argument with the Baxter operator associated to the negative end shows that $d\eta_{c,\tri}(A)$ remains Laurent under each mutation $\mu_{v_i^-}$, and is thus universally Laurent by the 1-step criterion.
This proves that image of the universal Laurent ring $\Lbb_{S}$ under $\eta_{c,\tri}$ is indeed contained in $\Lbb_{S';\phi}$.

In the other direction, it is again immediate from the construction of the local gluing isomorphism that for all $B\in \Lbb_{S';\phi}$, the element $\eta_{c;\tri}^{-1}(B)\in \Lbb(\Xi^{\mathrm{fr}}_{\tri;c})$ remains Laurent under all mutations except possibly for those in directions $h_{\pm}$. Moreover, for any such $B$ we claim that the element $\mu^\sharp_{\mathrm{hat},+}(B)$ is again contained in the residue quantum torus $\Tc^{\Ac}_{\res}(\tri';c_\pm)$. Indeed, by assumption $\mu^\sharp_{\mathrm{hat},+}(B)$ is Laurent after clearing denominators from $Ø(c_\pm)$, and its $W(c_\pm)$ symmetry is immediate from that of the product of quantum dilogarithms on the right-hand-side of~\eqref{eq:sharphat}. The fact that $\mu^\sharp_{\mathrm{hat},+}(B)$ satisfies the residue condition is proved similarly to Proposition~\ref{prop:residue-well-defined}: setting 
$$
g_{m\alpha} = \mu^\sharp_{\mathrm{hat},+}(Y_{m\alpha^\vee})Y^{-1}_{m\alpha^\vee},
$$
one checks directly using the $q$-difference equation~\eqref{eq:q-Gamma} that for any weight $\nu$ we have
$$
g_{m\alpha}^{[\nu-m\alpha]}\big|_{d_{\alpha,k}}=1, \quad m =\langle \nu,\alpha\rangle-k.
$$ 
Now since $\mu^\sharp_{\mathrm{hat},+}(B)$ lies in the residue quantum torus, it again follows from the intertwining relation~\eqref{eq:sharp-inter} that $\eta_{c;\tri}^{-1}(B)$ remains Laurent under the cluster transformation $\mu_{\mathrm{Baxter},+}$. Since the latter consists of mutations at distinct vertices, this implies that $\eta_{c;\tri}^{-1}(B)$ must remain Laurent under the first mutation $h_+$ of the sequence. By a similar argument with $\mu_{\mathrm{hat},-}$ we deduce that  $\eta_{c;\tri}^{-1}(B)$ remains Laurent under mutation at $h_-$, and so by the 1-step criterion $\eta_{c;\tri}^{-1}(B)$ lies in $\Lbb_{S}$, completing the proof of the theorem.
%The other direction is proved \red{-- observe that $\mu_{\mathrm{hat},+}(B)$ is again in residue torus: Weyl symmetry and same argument with $f$ from Prop 9.24 (clean notation). Then the intertwining above shows $\eta_c^{-1}(B)$  stays Laurent under Baxter, which is a sequence of mutations at distinct vertices. In partic, it must remain Laurent under the first mutation at $h_+$. End of proof.}
    \end{proof}
    
\begin{remark}
    The argument used to prove Theorem~\ref{thm:alg-MF} works without modification to show that the local gluing isomorphism associated to general locally glueable pair $(Q,Q_{\cut})$ restricts to an isomorphism
 $$
 \eta \colon \Lbb({\quiver}) \simeq \Lbb_{\res}({\quiver_{\cut}}),
 $$
 where $\Lbb_{\res}({\quiver_{\cut}})$ is defined as in Remark~\ref{rmk:genLres}.
\end{remark}

        % The fact that $\eta$ sends the subalgebra $\Lbb_S^{M_c^*\mathcal{O}_q(T/W)}$ of elements which commute with the $H_k(c)$ isomorphically to $\Lbb^c_{S^\circ}$ then follows from Corollary~\ref{cor:toda-commutant}.  

Our next goal is to show that the isomorphism from Theorem~\ref{thm:alg-MF} does not depend on the choice of isolating cluster $Q_{\tri;c}$ for $c$, and this way prove Theorem~\ref{thm:monodromy-well-defined}. The key technical ingredient is the following proposition on the intertwining of double Baxter operators under local gluing isomorphisms. 

Consider the quiver $\hat\quiver_{n,m}$ obtained from the quiver $\quiver_{n,m}$ in Section~\ref{subsec:bi-baxter} by adding a single additional vertex as indicated in the left pane of Figure~\ref{fig:Qmnhat}.  We decompose its index set as $I_{n,m} = I^{<0}_{n,m}\sqcup I^{\geq0}_{n,m}$, where $I_{n,m}^{<0}=\{s_{j,m},t_{j,m}\}_{j=1}^m$, and fix the identification $I^{\geq0}_{n,m}\simeq I_{\Toda}^{+,-}$ by sending $s_{j,n}\mapsto s_j,t_{j,n}\mapsto t_j$. Hence $\hat\quiver_{n,m}$ forms part of a locally glueable pair $(\hat\quiver_{n,m},\hat\quiver^{\cut}_{n,m})$ of handle signature $\eps = (+,-)$, where the quiver $\hat\quiver^{\cut}_{n,m}$ is illustrated in the left pane of Figure~\ref{fig:Qmnhatcut} for $(n,m)=(2,1)$. We write $\Lbb_{n,m}$ for the universal Laurent ring of $\hat\quiver_{n,m}$,  $\Lbb^{\res}_{n,m}$ for the residue universal Laurent ring for $\hat\quiver^{\cut}_{n,m}$, and $\eta_{n,m} \colon \Lbb_{n,m}\simeq \Lbb^{\res}_{n,m} $ for the corresponding local gluing isomorphism.

\begin{figure}[h]
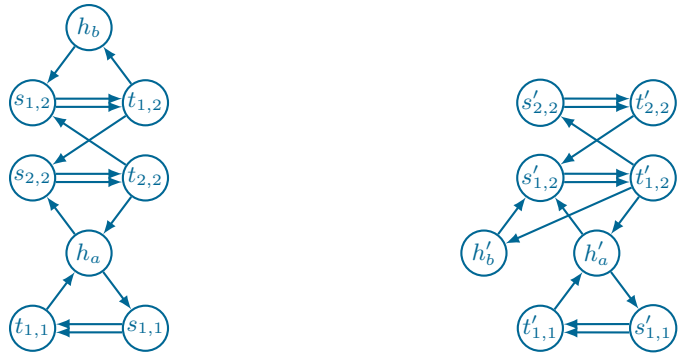

\subfile{fig-Qmnhat.tex}
\caption{The quivers $\hat\quiver_{2,1}$ and $\hat\quiver_{1,2}$. }
\label{fig:Qmnhat}
\end{figure}

%Then if we replace the subquiver corresponding to $I^{\geq0}_{n,m}$ by a copy of $I^{>0}_{S'}$ 
%connected as illustrated in Figure~\ref{fig:Qmnprime} for $(n,m)=(2,1)$, we obtain a quiver $\hat\quiver^{\cut}_{n,m}$.  The quivers with $(\hat\quiver_{n,m},\hat\quiver^{\cut}_{n,m})$ form a locally glueable pair, and choosing the orientation \blue{[i.e. which handle is positive] }  of $I_{n,m}^{\geq0}$ indicated in Figure~\ref{fig:Qmnhat} we get a local gluing isomorphism $\eta_{n,m}$. 
%We denote by $\Lbb_{n,m}$ the universal Laurent ring of $\hat\quiver_{n,m}$, and by $\Lbb^{\res}_{n,m}$ the residue universal Laurent ring for $Q^{\cut}_{n,m}$.
Now consider the quiver $\hat\quiver_{m,n} = \bs\mu_{n,m}(\hat\quiver_{n,m})$ obtained from $\hat\quiver_{m,n}$ by applying the double Baxter mutation sequence $\bs\mu_{n,m}$, illustrated in the right pane of Figure~\ref{fig:Qmnhat} for $(n,m)=(2,1)$. We decompose its index set as $I_{m,n} = I^{<0}_{m,n}\sqcup I^{\geq0}_{m,n}$, where $I_{n,m}^{<0}=\{s'_{j,m},t'_{j,m}\}_{j=1}^m$, and fix the identification $I^{\geq0}_{m,n}\simeq I^{--}_{\Toda}$ of $I^{\geq0}_{m,n}$ with the standard Toda index set with two negative handles as indicated in the Figure.

\begin{figure}[h]
\subfile{fig-Qmnhatcut.tex}
\caption{The quivers $\hat\quiver^{\cut}_{2,1}$ and $\hat\quiver^{\cut}_{1,2}$.  }
\label{fig:Qmnhatcut}
\end{figure}

Observe that the quiver $\hat\quiver^{\cut}_{m,n}$ in the corresponding locally glueable pair $(\hat\quiver_{m,n},\hat\quiver^{\cut}_{m,n})$ can be obtained from $\hat\quiver^{\cut}_{n,m}$ by applying the cluster transformation
$$
\bs\mu^{\cut}_{n,m} = \varsigma_{n,m} \circ \prod_{j=1}^{m+1}\bs\mu_{m,0}(v_{j,a}),
$$
where $\mu_{m,0}(v_{j,a})$ is the $GL_{m+1}\times GL_{1}$ Baxter mutation sequence starting from vertex $v_{j,a}$ in the left pane of Figure~\ref{fig:Qmnhatcut} and $\varsigma_{m,n}$ is the quasi-permutation which is the identity on all mutable basis vectors, and on the frozens $e_{c,i}$ is defined by 
$$
\varsigma_{m,n}\circ (\bs\mu^{\cut}_{n,m})'(e_{c,i})=  e'_{c,i},
$$
where  $(\bs\mu^{\cut}_{n,m})'$ is the composite of isometries~\eqref{eq:monpart} associated to the mutation sequence $\bs\mu^{\cut}_{n,m}$.
%
%Then setting $\hat\quiver^{\cut}_{m,n}= \bs\mu^{\cut}_{n,m}(\hat\quiver_{n,m})$, the quivers $(\hat\quiver_{m,n},\hat\quiver^{\cut}_{m,n})$ again form a locally gluable pair, and we write $\eta_{m,n}$ for the local gluing isomorphism corresponding to the choice of orientation of $I^{\geq0}_{m,n}$ indicated in Figure~\eqref{fig:Qmnhat}.

Writing $\Lbb_{m,n}$ for the universal Laurent ring for $\hat\quiver_{m,n}$ and $\Lbb^{\res}_{m,n}$ for the residue universal Laurent ring for $\hat\quiver^{\cut}_{m,n}$, the cluster transformations above give isomorphisms of algebras
$$
\bs\mu_{m,n}\colon \Lbb_{n,m}\rightarrow \Lbb_{m,n}, \qquad \bs\mu^{\cut}_{m,n}\colon \Lbb^{\res}_{n,m}\rightarrow \Lbb^{\res}_{m,n}.
$$
%consider the quiver $Q'_{n,m}$ illustrated in Figure~\ref{fig:Qnm-cut} for $(n,m)=(3,2)$, consisting of the disjoint union of 
\begin{prop}
\label{prop:double-bax-inter}
The local gluing isomorphisms $\eta_{n,m}$ and $\eta_{m,n}$ intertwine the double Baxter cluster transformation $\bs\mu_{n,m}$ with the cluster transformation $\bs\mu^{\cut}_{n,m}$: we have a commutative square 
    \begin{equation}
        \begin{tikzcd}[column sep=4em]
\Lbb_{n,m} \arrow[r,"\eta_{n,m}"] \arrow[d,"\bs\mu_{n,m}",swap] & \Lbb^{\res}_{n,m} \arrow[d,"\bs\mu^{\cut}_{n,m}"] \\
\Lbb_{m,n} \arrow[r,"\eta_{m,n}"]           & \Lbb^{\res}_{m,n}
\end{tikzcd}
    \end{equation}
\end{prop}

\begin{proof}
First, observe that the restriction of the local gluing isomorphism $\eta_{n,m}\colon\Lbb_{n,m}\simeq \Lbb_{n,m}^{\res}$ can equivalently be constructed using a different pair of faithful representations $\varrho_{n,m},\varrho_{n,m}^{\cut}$ of the two algebras. The representation space $\Vcr = \Vcr_{n,m}$ for $\varrho_{n,m}$ consists of all finitely supported $\mathbb{Q}(q)[t^{\pm1}]((z))$-valued functions of $(\lambda,\nu)\in\mathbb{Z}^{n+1}\times\mathbb{Z}^{m+1}$ which vanish unless both $\lambda$ and $\nu$ lie in their respective dominant cones of partitions, while the underlying space $\Vcr_{\cut}$ for $\varrho^{\cut}_{n,m}$ consists of all finitely supported $\mathbb{Q}(q)[\bs w^{\pm1},t^{\pm1}]^{S_{n+1}}((z))$-valued functions on $\mathbb{Z}^{m+1}$ which vanish unless $\nu$ is a partition. We define an action $\varrho_{n,m}$ of $\Lbb_{n,m}$ on $\Vcr$ by restricting the following action of the ambient quantum torus on the larger space of all $\mathbb{Q}(q)[t^{\pm1}]((z))$-valued functions: if we prepare two sets of canonical pairs $\{P_{j,n},X_{j,n}\}$ and $\{P_{j,m},X_{j,m}\}$ acting via the formulas~\eqref{eq:toda-torus-action}, the action of the generators $\{Y_{s_{j,n}},Y_{t_{j,n}}\}$ and $\{Y_{s_{j,m}},Y_{t_{j,m}}\}$ factors through the embeddings~\eqref{eq:toda-e-basis} to $\mathcal{T}_{GL_{n+1}}$ and $\mathcal{T}_{GL_{m+1}}$,
and we define the action of the handle generators by
$$
\varrho_{n,m}(Y_{h_a})  = z^{-1}P_{n+1,n}P_{1,m}^{-1},\quad \varrho_{n,m}(Y_{h_b})  =tP_{1,n}^{-1}.
$$
That $\Vcr$ is indeed preserved by $\Lbb_{n,m}$ follows as in Lemma~\ref{lem:pres}.
To construct the representation $\varrho^{\cut}_{n,m}$, we again let $\{Y_{s_{j,m}},Y_{t_{j,m}}\}$ act via~\eqref{eq:toda-torus-action}, and define
$$
\varrho_{n,m}(Y_{v_{j,a}})  =-q^{-1}zw_j, \quad  \varrho_{n,m}(Y_{v_{j,b}})  =-qtw^{-1}_j, \quad Y_{c,i} \mapsto D_iD_{i+1}^{-1},
$$
where $D_i$ the the $q$-shift operator from~\eqref{eq:Di}. That $\Vcr_{\cut}$ is preserved by the action of $\Lbb^{\res}_{n,m}$ follows by combining Lemmas~\ref{lem:res-preserved} and~\ref{lem:pres}. Then the $GL_{n+1}$ Whittaker transform 
$$
\mathcal{W}_{GL_{n+1}} \colon \mathbb{V}\simeq \mathbb{V}_{\cut}, \quad f(\lambda,\nu) \mapsto \hat{f}(\nu) = \sum_{\lambda} f(\lambda,\nu)\overline{W(\lambda)}
$$
identifies $\Lbb_{n,m}$ and $\Lbb^{\res}_{n,m}$ as subrings of the endomorphism ring of $\Vcr$, and it is easy to check this identification is the same as that induced by the original definition of the local gluing isomorphism.

In a similar way we define actions $\varrho_{m,n}$ and $\varrho^{\cut}_{m,n}$ of $\Lbb_{m,n}$ and $\Lbb^{\res}_{m,n}$ on the same spaces $\Vcr$ and $\Vcr_{\cut}$.
For $\varrho_{m,n}$, the $\{Y_{s'_{j,n}},Y_{t'_{j,n}}\}$ and $\{Y_{s'_{j,m}},Y_{t'_{j,m}}\}$ and $Y_{h'_b}$ act by the same operators as their un-primed counterparts, while we define the action of the handle generator $Y_{h'_a}$ by
$$
\varrho_{m,n}(Y_{h'_a})  = zP^{-1}_{1,n}P_{m+1,m}.
$$
For $\varrho_{m,n}$, all primed generators act by the same operators as their unprimed counterparts, except for
$$
\varrho_{m,n}(Y_{v_{j,a}}) = -qz^{-1}w_j^{-1}, \quad j=1,\ldots n+1.
$$
Now consider the factorization of the Baxter cluster transformation $\bs\mu_{m,n} =\bs\mu'_{m,n}\circ \bs\mu^\sharp_{m,n}$ where we factor each mutation as in~\eqref{eq:monpart}. As in Section~\ref{subsec:bi-baxter} we write $\Qop_{n,m}$ for the operator on $\mathbb{V}_{n,m}$ obtained by acting with the product of dilogarithms in $\bs\mu^\sharp_{m,n}$ using the representation $\varrho_{n,m}$ (that this operator indeed preserves $\Vcr$ follows from e.g. Corollary~\ref{cor:biToda-coeff}. Similarly, write $\Qop^{\cut}_{n,m}$ for the product of quantum dilogarithms in $(\bs\mu^{\cut}_{m,n})^\sharp$.
Now we observe that the quantum torus representations $\varrho_{n,m}$ and $\varrho_{m,n}$ and their cut counterparts are related via 
$$
\varrho_{n,m} = \varrho_{m,n}\circ \mu_{n,m}', \qquad \varrho^{\cut}_{n,m} = \varrho^{\cut}_{m,n}\circ \varsigma_{n,m}\circ\mu_{n,m}',
$$
so it follows that for all $A\in \Lbb_{n,m}$ and $B\in \Lbb^{\res}_{n,m}$ we have
\begin{align*}
\varrho_{m,n}(\bs\mu_{n,m}(A)) &= \Qop_{n,m}\circ \varrho_{n,m}(A) \circ  \Qop_{n,m}^{-1},\\
\varrho^{\cut}_{m,n}(\bs\mu_{n,m}(B)) &= \Qop^{\cut}_{n,m}\circ \varrho^{\cut}_{n,m}(B) \circ  (\Qop^{\cut}_{n,m})^{-1}.
\end{align*}
Hence the Proposition will follow provided we can show that $\mathcal{W}_{GL_{n+1}} \circ \Qop_{n,m} = \Qop^{\cut}_{n,m}\circ\mathcal{W}_{GL_{n+1}}$. To prove this, we use the $GL_{m+1}$-Whittaker transform: writing 
$$
\Vcr_{\cut,\cut}=\mathbb{Q}(q)[\bs w^{\pm1},\bs{u}^{\pm1},t^{\pm1}]^{S_{n+1}\times S_{m+1}}((z)),
$$ 
we have an isomorphism
$$
\mathcal{W}_{GL_{m+1}} \colon \Vcr_{\cut}\simeq \Vcr_{\cut,
\cut}, \quad \hat{f}(\nu) \mapsto \hat{\hat{f}} = \sum_{\nu} \hat{f}(\nu)\overline{W(\nu)},
$$
and it suffices to prove that $\mathcal{W}_{GL_{m+1}} \mathcal{W}_{GL_{n+1}} \Qop_{n,m} = \mathcal{W}_{GL_{m+1}}\Qop^{\cut}_{n,m}\mathcal{W}_{GL_{n+1}}$. But applying Proposition~\ref{prop:whit-bibax} for $(n,k) = (n,m)$ it follows that 
\begin{align}
\label{eq:cutcut}
\mathcal{W}_{GL_{m+1}} \mathcal{W}_{GL_{n+1}} \Qop_{n,m}  = \prod_{j=1}^{n+1}\prod_{k=1}^{m+1}\Psi(zw_j/u_k)^{-1}\circ \mathcal{W}_{GL_{m+1}} \mathcal{W}_{GL_{n+1}},
\end{align}
while applying the same Proposition $n+1$ times with $(n,k) = (m,0)$ to move $\Qop^{\cut}_{n,m}$ over $\mathcal{W}_{GL_{n+1}}$ also identifies the composite $\mathcal{W}_{GL_{m+1}}\Qop^{\cut}_{n,m}\mathcal{W}_{GL_{n+1}}$ with the right hand side of~\eqref{eq:cutcut}. This completes the proof of the Proposition.
\end{proof}

\begin{theorem}
    \label{thm:alg-MF-coherence}
    % Then the homomorphism $\eta$ in~\eqref{eq:global-W-def} is an isomorphism onto its image $\Lbb^c_W$:
      The isomorphism~\eqref{eq:boldW-def} does not depend on the choice of isolating cylinder for the curve $c$:
      if $\mu_S \colon Q_{\tri_1,c}\rightarrow Q_{\tri_2,c}$ is the cluster transformation connecting two isolating clusters for $c$ on $S$ and 
$\mu_{S'} \colon Q_{\tri'_1,c_\pm}\rightarrow Q_{\tri'_2,c_\pm}$ the one connecting the corresponding clusters on $S'$, then we have a commutative square
    \begin{equation}
    \label{eq:L-square}
        \begin{tikzcd}[column sep=4em]
\Lbb_{S,(\tri_1,c)} \arrow[r,"\eta_{c,\tri_1}"] \arrow[d,"\mu_S"] & \Lbb_{{S';\phi};(\tri_1,c)} \arrow[d,"\mu_{S'}"] \\
\Lbb_{S,(\tri_2,c)} \arrow[r,"\eta_{c,\tri_2}"]           & \Lbb_{{S';\phi};(\tri_2,c)}
\end{tikzcd}
    \end{equation}
    In particular, the fundamental Hamiltonians $H_k(c)$ do not depend on the choice of isolating cylinder for $c$. 
      \end{theorem}
      \begin{proof}
      Recall that any two $c$-isolating triangulations can be connected by a sequence of umbral moves, Dehn twists along $c$, and flips at edges disjoint from the isolating cylinder. The equivariance of $\eta_c$ under the Dehn twist has already been proved in Proposition~\ref{prop:local-gluing-pgl}. In the case of a flip outside the isolating cylinder, we use the fact that the corresponding cluster transformations can be realized as the same (relative to the identification $\iiota_\phi$) sequences of mutations in directions solely within $I_{S}^{<0}$ and $I_{S'}^{<0}$, and so the commutativity of~\eqref{eq:L-square} follows from the locality property~\eqref{eq:local-is-local} of the local gluing isomorphisms. 
To handle the case of an umbral move, we use the factorizations of the cluster transformation $U_{c;e}$ on $S$ and its counterpart $U_{c_+;e_+}$ on $S'$ from Proposition~\ref{prop:umbral-factorization}. The two factorizations have almost identical forms, with the only difference coming in steps $10\circlearrowright 9$ and $13\circlearrowright 12$ in which double Baxter sequences $\bs\mu_{n,n-1}$ and $\bs\mu_{n,n-2}$ in $\mu_S$ are replaced by in $\mu_{S'}$ by $(n+1)$-fold products of Baxter sequences $\bs\mu_{n-1,0}$ and  $\bs\mu_{n-2,0}$ respectively. The virtue of these factorizations is that the intermediate quivers $(Q_{c,k},Q_{c_+,k})$ form {locally glueable pairs} in the sense of Definition~\ref{def:glueable-pair}. As in the case of a flip outside the isolating cylinder, it is clear from the locality property~\eqref{eq:local-is-local} that any piece in the factorization of $\mu_S$ which is not one of the double Baxter operators mentioned above intertwines with its counterpart in $\mu_{S'}$ under the local gluing maps. On the other hand, the intertwining in the remaining cases follows from Proposition~\ref{prop:double-bax-inter}, completing the proof of the commutativity of~\eqref{eq:L-square} in the case of an umbral move.
       
%       Recall that the passage from an isolating cluster for $c$ on $S$ to a refined one is achieved by cluster transformations which act nontrivially only on the factor $\mathcal{T}_{\Xi_{S_{\leq0}}}$ in the factorization~\label{eq:torus-factor}, and similarly on $S'$. In particular, these mutations are also $\eta_c$-equivariant for the same reason as flips disjoint from the isolating cylinder. So combining the factorizations of $\mu_S,\mu_{S'}$ from Propositions~\ref{prop:3-factorization-S} and~\ref{prop:3-factorization-S'} with the intertwining relation~\eqref{eq:alg-q-inter}, we get the commutativity of~\eqref{eq:L-square}.        
             
       With this in hand, the independence of the $H_k(c)$ on the choice of isolating cylinder follows immediately: indeed, we have $\eta_{c,\tri}(H_k(c;\tri)) = \chi_{\omega_k}(c)$ is a Laurent polynomial in frozen $\mathcal{A}$-variables and thus commutes with all mutations on $S'$. So
       $$
       \eta_{c,\tri_2}(\mu_S(H_k(c;\tri_1))) = \mu_{S'}(\eta_{c,\tri_1}(H_k(c;\tri_1)))=\mu_{S'}(\chi_{\omega_k}(c)) = \chi_{\omega_k}(c)= \eta_{c,\tri_2}(H_k(c;\tri_2)),
       $$
and applying the inverse of $\eta_{c,\tri_2}$ the claim follows. 
      \end{proof}
      
      \begin{remark}
Thanks to Theorem~\ref{thm:alg-MF-coherence}, the cutting isomorphism $\eta_c$ makes sense in arbitrary cluster coordinate systems $\quiver,\quiver'$ associated to $S,S'$ respectively. Indeed, pick an isolating cluster $\quiver_{S,\tri,c}$ subordinate to an arbitrary $c$-isolating ideal triangulation $\tri$, and let $Q_{S',\tri,c}$ be the $c_\pm$-isolating cluster subordinate to the triangulation $\tri'=\mathcal{C}_c(\tri)$ of $S'$ obtained by applying the cutting functor to $\tri$. Then if $\mu\colon\quiver\rightarrow \quiver_{S,\tri,c}$ and $\mu_{\cut}\colon\quiver_{S',\tri,c}\rightarrow Q'$ are the corresponding cluster transformations, the theorem guarantees that $\mu_{\cut}\circ\eta_{c}(\tri)\circ \mu^{-1}$ does not depend on the choice of $c$-isolating triangulation $\tri$. 
\end{remark}
      
      \begin{theorem}
          \label{thm:alg-MF-equivariance}
      The isomorphism~\eqref{eq:boldW-def} is equivariant with respect to the actions by cluster transformations of the centralizer $\Gamma_{S;c}$ of the Dehn twist $\tau_c$ on the algebras $\Lbb_{S},\Lbb_{S';\phi}$. It restricts to an isomorphism
        \begin{align}
    \label{eq:restricted-boldW}
        \eta_c\colon \Lbb_S^{M_c^*\mathcal{O}_q(T/W)}\simeq \Lbb^c_{S^\circ},
    \end{align}
    where $\Lbb_S^{M_c^*\mathcal{O}_q(T/W)}$ is the subalgebra of $\Lbb_S$ consisting of elements which commute with all fundamental Hamiltonians $H_k(c)\in\Lbb_S$ associated to $c$.
\end{theorem}
\begin{proof}
Fix a triangulation $\tri\in \widehat{\Pt}_c(S)$ of $S$ in which $c$ has a shadow, and consider $\Lbb_S$ as being embedded in the corresponding quantum torus. Similarly, we consider the triangulation $\tri'=\mathcal{C}_c(\tri)$ of $S'$ obtained by applying the cutting functor to $\tri$, and embed $\Lbb_{S';\phi}$ in the corresponding residue quantum torus. Recall from formula~\eqref{eq:Qc-flip} that the action of a mapping class $\gamma\in \Gamma_{S;c}$ on $\Lbb_S$ is given by a composite $\varsigma^S_\gamma\circ F_{\gamma^{-1}(\tri),\tri}$ where the cluster transformation $F_{\gamma^{-1}(\tri_S),\Delta}$ is the composite of mutations associated to a sequence of flips taking $\tri$ to $\gamma^{-1}(\tri)$, and $\varsigma^S_\gamma$ is the quasi-permutation morphism induced by the diffeomorphism $\gamma$. The action of $\gamma$ on $\Lbb_{S';\phi}$ is given by an analogous composite $\varsigma^{S'}_\gamma\circ F_{\gamma^{-1}(\tri'),\tri'}$. Commutativity of the square
%    \begin{equation}
%    \label{eq:L-square}
%        \begin{tikzcd}
%\Lbb_{S,(\tri,c)} \arrow[r,"\eta_{c,\tri}"] \arrow[d,"\mu_S"] & \Lbb_{{S';\phi};(\tri,c)} \arrow[d,"\mu_{S'}"] \\
%\Lbb_{S,(\gamma^{-1}(\tri),c)} \arrow[r,"\eta_{c,\gamma^{-1}(\tri)}"]           & \Lbb_{{S';\phi};(\gamma^{-1}(\tri),c)}
%\end{tikzcd}
%    \end{equation}
\begin{equation}
\label{eq:sigma-square}
    \begin{tikzcd}[column sep=5em]
\Lbb_{S,(\gamma^{-1}(\tri),c)} \arrow[r,"\eta_{c,\gamma^{-1}(\tri)}"] \arrow[d,"\varsigma^{S}_\gamma"] 
    & \Lbb_{{S';\phi};(\gamma^{-1}(\tri),c)} \arrow[d,"\varsigma^{S'}_\gamma"] \\
\Lbb_{S,(\tri,c)} \arrow[r,"\eta_{c,\tri}"]           
    & \Lbb_{{S';\phi};(\tri,c)}
\end{tikzcd}
\end{equation}
is immediate, since the morphisms $\varsigma^{S}_\gamma,\varsigma^{S'}_\gamma$ are simply relabelings of quiver index sets (which correspond to faces of the bicolored graphs on $S,S'$) induced by the same diffeomorphism $\gamma$. On the other hand, if we apply the previous theorem with 
$$
\tri_1=\tri,\quad\tri_2=\gamma^{-1}(\tri),\quad \mu_S =F_{\gamma^{-1}(\tri),\tri},\quad \mu_{S'}=F_{\gamma^{-1}(\tri'),\tri'},
$$
we get the commutativity of the square
    \begin{equation}
    \label{eq:flip-square}
        \begin{tikzcd}[column sep=5em]
\Lbb_{S,(\tri,c)} \arrow[r,"\eta_{c,\tri}"] \arrow[d,"F_{\gamma^{-1}(\tri),\tri}",swap] & \Lbb_{{S';\phi};(\tri',c)} \arrow[d,"F_{\gamma^{-1}(\tri'),\tri'}"] \\
\Lbb_{S,(\gamma^{-1}(\tri),c)} \arrow[r,"\eta_{c,\gamma^{-1}(\tri)}"]           & \Lbb_{{S';\phi};(\gamma^{-1}(\tri'),c)}
\end{tikzcd}
    \end{equation}
Stacking the second square on top of the first, it follows that the isomorphism $\eta_c$ is equivariant for the action of $\Gamma_{S;c}$ as claimed.

For the final statement about the commutant, recall that the $\Lbb_{S^\circ}^c$ the subalgebra in $\Lbb_{S';\phi}$ consisting of elements which commute with all $Y_{c,\lambda}$. This condition is equivalent to commutativity with all elementary symmetric polynomials $\chi_{\omega_k}(c)$, and the latter are interwined under $\eta_c$ with $H_k(c)\in\Lbb_S$, which proves~\eqref{eq:restricted-boldW}.
\end{proof}

%The proof of the previous Theorem shows that for any two objects $\tri_1,\tri_2$ of $\widehat{\Pt}_c(S)$ -- that is to say, two ideal triangulations of $S$ in which the curve $c$ has a shadow -- and any morphism $\mu:\tri_1\rightarrow\tri_2$ in $\widehat{\Pt}_c(S)$, we have a commutative square 
%    \[
%        \begin{tikzcd}[column sep=5em]
%\Lbb_{S,(\tri_1,c)} \arrow[r,"\eta_{c,\tri}"] \arrow[d,"F_{\gamma^{-1}(\tri),\tri}",swap] & \Lbb_{{S';\phi};(\tri',c)} \arrow[d,"F_{\gamma^{-1}(\tri'),\tri'}"] \\
%\Lbb_{S,(\gamma^{-1}(\tri),c)} \arrow[r,"\eta_{c,\gamma^{-1}(\tri)}"]           & \Lbb_{{S';\phi};(\gamma^{-1}(\tri'),c)}
%\end{tikzcd}
%    \]

\subsection{Non-isolatable curves and multiple cuts} We now discuss the extension of the previous theorems to  non-isolatable simple closed curves $c$, so that $S'=S_+\sqcup S_-$ where $S_-$ is a surface of genus $g>0$ with exactly one tacked circle. In this case, we have no a priori definition of the algebra $\Lbb_{res}^{c}$. Moreover, we do not have an action of the Weyl group by cluster transformations at the level of the cut surface $S'$, since one of its components is the exceptional surface $S_-$. To handle these problems we use the trick of first cutting $S$ along an auxiliary cycle. In the present article we do this for $g>1$, and leave the corner case $g=1$ (which requires separate techniques) for the companion paper.

We start with a quick discussion of how to iterate the constructions from the previous section, allowing us to make multiple cuts on a surface $S$. 
%The simplest such situation is that in which $S$ admits an ideal triangulation $\tri$ containing isolating cylinders for both curves $a$ and $c$, so that $S_{a,c}$ has an ideal triangulation $\tri''$ simultaneously isolating $a_\pm$ and $c_\pm$   In this case, we can apply the constructions of Definition~\ref{} and~\ref{} simulataneously in each isolating cylinder to produce a quivers $Q(\tri;a,c)$ on $S$ and $Q(\tri'';a_\pm,c_\pm)$ on $S_{a,c}$. Writing $\Tc^{a,c}(\tri'';a_\pm,c_\pm)$ the symplectic reduction of the corresponding quantum torus at both pairs of tacked circles $a_\pm,c_\pm$, and $\Tc^{a,c}_{\loc}(S_{a,c};c_\pm)$ the localization of the latter at the denominator set $Ø(a_\pm,c_\pm)=Ø(a_\pm)Ø(c_\pm)$, we define the residue universal Laurent ring 

% {def:c-isol-quiver} {def:cpm-isol-quiver}

\begin{defn}
\label{def:Lresmulticut}
Suppose $S_{a,c}$ is a surface with two pairs of tacked circles $a_\pm,c_\pm$ which admits both an $a_\pm$-isolating triangulation and a (possibly different) $c_\pm$-isolating tringulation.
Let $Q_{S_{a,c};c_\pm}$ be a $c_\pm$-isolating cluster on $S_{a,c}$, $\Tc^{a,c}(S_{a,c})$ the symplectic reduction of the corresponding quantum torus at both pairs of tacked circles $a_\pm,c_\pm$, and $\Tc^{a,c}_{\loc}(S_{a,c};c_\pm)$ the localization of the latter at the denominator set $Ø(a_\pm,c_\pm)=Ø(a_\pm)Ø(c_\pm)$. We define the residue universal Laurent ring $\Lbb_{S_{a,c};\phi_a,\phi_c}$ based in $c_\pm$-isolating cluster $Q_{S_{a,c};c_\pm}$ to be the set of elements $A\in \Tc^{a,c}_{\loc}(S_{a,c};c_\pm)$ such that:
\begin{enumerate}
\item There exists some $d\in Ø(a_\pm,c_\pm)$ such that $dA$ lies in the symplectically reduced universal Laurent ring $\Lbb^{a,c}_{S_{a,c}}$;
\item The element $A$ is invariant under the action by cluster transformations of the product of Weyl groups $W(a_\pm)\times W(c_\pm)$;
\item The element $A$ satisfies the residue conditions~\eqref{eq:Lresidue-condition} for the divisors $d_{\alpha,k}(c)$, and for any $a_\pm$-isolating cluster $Q_{S_{a,c};a_\pm}$ on $S_{a,c}$, its image $\mu(A)$ in $\Tc^{a,c}_{\loc}(S_{a,c};a_\pm)$ under the cluster transformation $\mu: Q_{S_{a,c};c_\pm}\rightarrow Q_{S_{a,c};a_\pm}$ satisfies the corresponding conditions for the divisors $d_{\alpha,k}(a)$.
\end{enumerate}
\end{defn}
\begin{remark}
\begin{enumerate}
\item It follows as in Proposition~\ref{prop:residue-well-defined} that $\Lbb_{S_{a,c};\phi_a,\phi_c}$ does not depend on the choice of $c_\pm$-isolating cluster $Q_{S_{a,c};c_\pm}$. Moreover, it is clear that we have $\Lbb_{S_{a,c};\phi_a,\phi_c}\simeq\Lbb_{S_{c,a};\phi_c,\phi_a}$, the isomorphism being induced by the $\mu$ in part (3) of Definition~\ref{def:Lresmulticut}. 
\item It is straightforward to extend Definition~\ref{def:Lresmulticut} to the case of a surface $S$ with an arbitrary number of pairs $\{a_\pm^{(1)},\ldots, a^{(m)}_\pm\}$ of  tacked circles such that $S$ admits an $a^{(i)}_\pm$-isolating triangulation for all $1\leq i\leq m$.
\end{enumerate}
%\item If $S$ admits an ideal triangulation $\tri$ containing isolating cylinders for both curves $a$ and $c$ (and hence $S_{a,c}$ has an ideal triangulation $\tri''$ simultaneously isolating $a_\pm$ and $c_\pm$), we can apply the constructions of Definition~\ref{def:c-isol-quiver} and~\ref{def:cpm-isol-quiver} simulataneously in each isolating cylinder to produce quivers $Q(\tri;a,c)$ on $S$ and $Q(\tri'';a_\pm,c_\pm)$ on $S_{a,c}$.  Then using the clusters on $S_a,S_c$ obtained from $Q(\tri;a,c)$ by cutting along $a,c$ as the basepoints in the definition of the residue universal Laurent rings, it is clear that we have $\Lbb_{S_{a,c};\phi_a,\phi_c}=\Lbb_{S_{c,a};\phi_c,\phi_a}$.  
\end{remark}
Now suppose that $a$ and $c$ are two disjoint essential simple closed curves on $S$, and let $S_a$ be the surface obtained by cutting $S$ along $a$. We denote by $S_{a,c}$ the surface obtained by further cutting $S_a$ along $c$.  
 If we write $\Lbb_S$ for the universal Laurent ring associated to the pair $(G,S)$, Theorem~\ref{thm:alg-MF} provides canonical isomorphisms of algebras $\eta_{a}\colon \Lbb_S\simeq \Lbb_{S_a;\phi_a}$ and $\eta_c(S_a)\colon \Lbb_{S_a}\simeq \Lbb_{S_{a,c};\phi_c}$. Moreover, since  $\eta_c(S_a)$ sends each $Y_{\xi_{a_\pm,i}}$ to its counterpart on $S_{a,c}$, it follows that $\eta_c(S_a)$ descends to an isomorphism $\Lbb^a_{S_a}\simeq \Lbb_{S_{a,c};\phi_c}$ between the corresponding symplectic reductions for the elements~\eqref{eq:reduced-lattice} associated to the pair $a_\pm$, and induces an isomorphism between as the localizations of the latter at the denominator set $Ø(a_\pm)$. We abuse notations and denote the corresponding maps by the same symbol $\eta_c(S_a)$.  Our goal is to prove that $\eta_c(S_a)$ restricts to an isomorphism of residue universal Laurent rings, which amounts to showing that it intertwines the actions of the Weyl group $W(a_\pm)$ on $S_a$ and $S_{a,c}$, as well as the residue conditions associated to the divisors $d_{\alpha,k}(a)$ on the two surfaces. We first observe this is indeed the true in the special case that $S$ admits an ideal triangulation $\tri$ containing isolating cylinders for both curves $a$ and $c$, where the needed intertwining properties follow immediately from the locality of both the Weyl group action and local gluing isomorphisms with respect to isolating cylinders:
\begin{lemma}
\label{lem:cutsquare}
Suppose that $S$ admits an ideal triangulation $\tri$ containing isolating cylinders for both curves $a$ and $c$. Then the isomorphism $\eta_c\colon \Lbb^a_{S_a}\simeq \Lbb_{S_{a,c};\phi_c}$ is equivariant with respect to the actions of $W(a_\pm)$ by cluster transformations, and induces an isomorphism
\begin{align}
\label{eq:eta-ca}
\eta_c(S_a) \colon \Lbb_{S_a;\phi_a}\simeq \Lbb_{S_{a,c};\phi_a,\phi_c},
\end{align}
The isomorphism~\eqref{eq:eta-ca} and its counterpart $\eta_a(S_c)$ fit into a commutative square
\begin{equation}
\label{eq:2-cut-square}
\begin{tikzcd}[row sep=2.2em, column sep=2.2em]
  & \Lbb_{S} \arrow[dl, "\eta_a(S)"'] \arrow[dr, "\eta_c(S)"] & \\
  \Lbb_{S_{a};\phi_a} \arrow[dr, "\eta_c(S_a)"'] & & \Lbb_{S_{c};\phi_c} \arrow[dl, "\eta_a(S_c)"] \\
  & \Lbb_{S_{a,c};\phi_a,\phi_c} &
\end{tikzcd}
\end{equation}
\end{lemma}
\begin{proof}
Since $S$ admits an ideal triangulation $\tri$ containing isolating cylinders for both curves $a$ and $c$, we can apply the constructions of Definition~\ref{def:c-isol-quiver} and~\ref{def:cpm-isol-quiver} simultaneously in each isolating cylinder to produce quivers $Q(\tri;a,c)$ on $S$, $Q(\tri';a_\pm,c)$ on $S_a$, and $Q(\tri'';a_\pm,c_\pm)$ on $S_{a,c}$. Then the fact that $\eta_c(S_a)$ intertwines the $W(a_\pm)$ action follows from the locality property~\eqref{eq:local-is-local} of $\eta_c(S_a)$ with respect to the isolating cylinder for $c$, and fact that the $W(a_\pm)$ action is realized locally within the isolating cylinder for $a$. For the same reason, it follows that $\eta_c(S_a)$ intertwines the $a_\pm$-residue conditions on $S_a$ and $S_{a,c}$, and thus induces an isomorphism $\Lbb_{S_a;\phi_a}\simeq \Lbb_{S_{a,c};\phi_a,\phi_c}$. The commutativity of the square~\eqref{eq:2-cut-square} is again immediate from the locality of $\eta_a,\eta_c$ with respect to the disjoint isolating cylinders for $a,c$.
\end{proof}
\begin{remark}
Observe that $S$ admits an ideal triangulation containing isolating cylinders for curves $a$ and $c$ with $a\cap c=\emptyset$ if and only if each connected component of the surface $S_{a,c}$ has at least one special point which is not one of the tacks on $a_\pm,c_\pm$. 
\end{remark}
We can bootstrap this Lemma to obtain the following Corollary extending part of its conclusion to the case in which the curves $a$ and $c$ cannot be simultaneously isolated.
\begin{cor}
\label{cor:bootstrap}
Suppose that the surface $S_c$ obtained by cutting $S$ along $c$ has a connected component $S_-$ given by a surface of genus $g>1$ with a single tacked circle $c_-$, and as above let $a$ be a simple closed curve on $S$ with $a\cap c=\emptyset$. Then the isomorphism $\eta_c\colon \Lbb^a_{S_a}\simeq \Lbb_{S_{a,c};\phi_c}$ is equivariant with respect to the actions of $W(a_\pm;S_a)$ and  $W(a_\pm;S_{a,c})$ by cluster transformations, and induces an isomorphism
\begin{align}
\label{eq:eta-ca-2}
\eta_c(S_a) \colon \Lbb_{S_a;\phi_a}\simeq \Lbb_{S_{a,c};\phi_a,\phi_c}.
\end{align}
\end{cor}
\begin{proof}
Since $g(S_-)>1$ we can choose a non-separating curve $d$ on  the connected component of $S_{a,c}$ containing $c_-$ which is disjoint from $c_-$ and $a_\pm$. We abuse notation and denote by the same symbol the corresponding curve on $S_a$, and write $S_{a,d}$ for the surface obtained by cutting $S_a$ along $d$. Then it follows from Lemma~\ref{lem:cutsquare} that we have a commutative square of isomorphisms 
\[
\begin{tikzcd}[row sep=2.2em, column sep=2.2em]
  & \Lbb^a_{S_a} \arrow[dl, "\eta_c(S_a)"'] \arrow[dr, "\eta_d(S_a)"] & \\
  \Lbb^a_{S_{a,c};\phi_c} \arrow[dr, "\eta_d(S_{a,c})"'] & & \Lbb^a_{S_{a,d};\phi_d} \arrow[dl, "\eta_c(S_{a,d})"] \\
  & \Lbb^a_{S_{a,c,d};\phi_c,\phi_d} &
\end{tikzcd}
\]
Moreover, the same Lemma implies that $\eta_d(S_a),\eta_c(S_{a,d}),\eta_d(S_{a,c})$ are all $W(a_\pm)$ equivariant and intertwine residue conditions for the divisors $d_{\alpha,k}(a)$, and so we deduce the same is true of $\eta_c(S_a)$ as claimed. 
%With this established, in order to prove that~\eqref{eq:eta-ca-2} is an isomorphism all that remains to be proved is that the $a_\pm$ residue conditions on $S_a$ are intertwined with the corresponding ones on $S_{a,c}$. But it again follows from Lemma~\ref{lem:cutsquare} that after localizing the square above at $Ø(a_\pm)$, the isomorphisms $\eta_d(S_a),\eta_c(S_{a,d}),\eta_d(S_{a,c})$ intertwine $Ø(c_\pm)$ residue conditions for their source and target algebras, so we again deduce the same is true of $\eta_c(S_a)$, which completes the proof.
\end{proof}

Now we can define the residue universal Laurent ring $\Lbb_{S';\phi_c}$ in the exceptional case that $S'=S_+\sqcup S_-$ where $S_-$ is a surface of genus $g>1$ with exactly one tacked circle. Write $\Lbb^c_{S'}$ for the symplectically reduced universal Laurent ring associated to the surface $S'$ relative to the pair of tacked circles $c_\pm$, and  $\Lbb^{c;\loc}_{S'}$ for its localization at the denominator set $Ø(c_\pm)$. Then given a non-separating simple closed curve $a$ on $S$ disjoint from $c$ which becomes a curve on $S_-$ after cutting along $c$, we have an isomorphism of $Ø(c_\pm)$ localizations
$$
\eta_a(S_c)\colon \Lbb^{c;\loc}_{S'}\simeq \Lbb^{c;\loc}_{S_{a,c};\phi_c}.
$$

\begin{defn}
\label{def:exceptional-lres}
Let $(S,a,c)$ be as in the paragraph above. We define the residue universal Laurent ring $\Lbb_{S';\phi}$ relative to $a$ by
$$
\Lbb_{S';\phi}(a) = \eta_{a}(S_{c})^{-1}\left(\Lbb_{S_{a,c};\phi_a,\phi_c}\right).
$$
\end{defn}
With this definition, it follows from Corollary~\ref{cor:bootstrap} that we have an isomorphism of algebras
\begin{align}
\label{def:exceptional-iso}
\eta_{c;a}  \colon \Lbb_S \simeq \Lbb_{S';\phi}(a), \quad  \eta_{c;a}= \eta_{a}^{-1}(S_c)\circ\eta_c(S_a)\circ\eta_{a}(S).
\end{align}

\begin{theorem}
\label{thm:secondlast}
The algebra $\Lbb_{S';\phi}(a)$ in Definition~\ref{def:exceptional-lres} does not depend on the choice of auxiliary curve $a$: for all non-separating simple closed curves $a_1,a_2$ on $S$ disjoint from $c$ which become curves on $S_-$ after cutting along $c$, we have $\Lbb_{S';\phi}(a_1)=\Lbb_{S';\phi}(a_2)$ and $\eta_{c;a_1}=\eta_{c;a_2}$.
\end{theorem}
\begin{proof}
Consider first the case that  $a_1 \cap a_2 = \varnothing$, so that we have a surface $S_{c,a_1,a_2}$ obtained by cutting $S$ along all three curves. Then we can repeatedly apply Lemma~\eqref{lem:cutsquare}  to get the following chain of equalities
\begin{align}
\label{eq:long-chain}
\nonumber \eta_{c;a_1} = \eta_{a_1}^{-1}\eta_c\eta_{a_1} &= \eta_{a_1}^{-1}\eta_{a_2}^{-1}\eta_{a_2}\eta_c\eta_{a_1} = \eta_{a_1}^{-1}\eta_{a_2}^{-1}\eta_c\eta_{a_2}\eta_{a_1} \\
&= \eta_{a_2}^{-1}\eta_{a_1}^{-1}\eta_c\eta_{a_1}\eta_{a_2} = \eta_{a_2}^{-1}\eta_{a_1}^{-1}\eta_{a_1}\eta_c\eta_{a_2} = \eta_{a_2}^{-1}\eta_c\eta_{a_2} = \eta_{c;a_2}.
\end{align}
On the other hand, if $a_1 \cap a_2 \ne \varnothing$, the assumption that $S_-$ has genus greater than 1 guarantees that there exists a non-separating closed simple curve $b$ such that $a_1 \cap b = a_2 \cap b = \varnothing$. So we have $\eta_{c;a_1} = \eta_{c;b} = \eta_{c;a_2}$, which completes the proof.

\end{proof}

Hence we have a single canonical isomorphism 
\begin{align}
\label{eq:exceptional-iso-canonical}
\eta_c\colon \Lbb_S \simeq \Lbb_{S';\phi}.
\end{align}

\begin{theorem}
\label{thm:last}
The isomorphism~\eqref{eq:exceptional-iso-canonical} is equivariant with respect to the action of the mapping class group $\Gamma_{S;c}$.
\end{theorem}
\begin{proof}
The idea of the proof is to use the freedom to express $\eta_c=\eta_{c,a}$ for any choice of auxiliary curve $c$. Indeed, it follows from Theorem~\ref{thm:alg-MF-equivariance} that $\eta_{c,a}$ intertwines any morphism in $\Ptoh_{a_j,c}(S)$ with the corresponding morphism on the $S'$ side. So the Proposition will follow if we can show that any morphism $\mu$ in $\widehat\Pt_c(S)$ can be factored into a composition $\mu = \mu_k \dots \mu_1$, where for each factor $1 \le j \le k$ there exists an auxiliary curve $a_j$ such that $\mu_j$ is a morphism in $\Ptoh_{a_j,c}(S)$.  To see this, recall that a general morphism $\mu \colon \tri \to \tri'$ in $\widehat\Pt_c(S)$ takes the form $\mu = A_\gamma F_{\gamma^{-1}(\tri'),\tri}$. If $e_c, e'_c$ are the shadows of $c$ in $\tri, \tri'$, then since $\gamma$ preserves $e'_c$ the latter is a shadow of $c$ in $\gamma^{-1}(\tri')$ as well. Consider a non-separating simple closed curve $a \subset S_-(e_c)$ such that $\tri$ is an object in $\Pt^o_{a,c}(S)$, meaning that there exists an arc $e_a \in \tri$ which is a shadow of $a$. Then take any triangulation $\tri''$ of $S$ which agrees with $\gamma^{-1}(\tri')$ on the subsurface $S_+(e_c')$, and in the subsurface $S_-(e_c')$ contains a shadow of $a$ among its arcs: since $a$ and $c$ are disjoint, such a triangulation certainly exists.
% It is not hard to see that there exists another shadow of $a$, an arc $e'_a \notin \tri$, which is compatible with $e_a$ and $e'_c$ but not with $e_c$. 
% Now, for any triangulation $\tri''$ which contains arcs $e'_a, e'_c$ and the subtriangulation $\tri_+(e'_c)$, 
This way we get a morphism $F_{\tri'',\tri}$ in $\Pt_{a,c}(S)$. On the other hand, the morphism $F_{\gamma^{-1}(\tri'),\tri''}$ can be factored into a product of flips $F_{d_r} \dots F_{d_1}$ on $S_-(e'_c)$. Given an arc $d_j$, let $e_j \ne e'_c$ be an edge of the quadrilateral, containing $d_j$ is a diagonal, and $a_j$ be a simple closed curve on $S_-(e'_c)$ of which $e_j$ is a shadow. Then $F_{d_j}$ is a morphism in $\Pt^o_{a_j,c}(S)$. We now turn to the morphism $A_\gamma$. Factoring $\gamma = \gamma_-\gamma_+$, where $\gamma_\pm \in \Gamma_{S_\pm(e'_c)} \subset \Gamma_S$, we see that $A_{\gamma_+}$ is a morphism in $\Ptoh_{a_r,c}$, while $A_{\gamma_-}$ can be further factored into a composition of Dehn twists $\tau_{a_{r+s}} \dots \tau_{a_{r+1}}$ on $S_-(e'_c)$. Evidently, $\tau_{a_j}$ is a morphism in $\Ptoh_{a_j,c}$, and we arrive at the desired factorization
$$
\mu = \tau_{a_{r+s}} \dots \tau_{a_{r+1}} A_{\gamma_+} F_{d_r} \dots F_{d_1} F_+.
$$

\end{proof}

\bibliographystyle{alpha}

\end{document}